\documentclass[numbers=enddot,12pt,final,onecolumn,notitlepage]{scrartcl}%
\usepackage[headsepline,footsepline,manualmark]{scrlayer-scrpage}
\usepackage{amsfonts}
\usepackage{amssymb}
\usepackage{amsmath}
\usepackage{amsthm}
\usepackage{framed}
\usepackage{comment}
\usepackage{color}
\usepackage[breaklinks=true]{hyperref}
\usepackage[sc]{mathpazo}
\usepackage[T1]{fontenc}
\usepackage{needspace}
\usepackage{tabls}
\usepackage{hyperxmp}
\providecommand{\U}[1]{\protect\rule{.1in}{.1in}}
\theoremstyle{definition}
\newtheorem{theo}{Theorem}[section]
\newenvironment{theorem}[1][]
{\begin{theo}[#1]\begin{leftbar}}
{\end{leftbar}\end{theo}}
\newtheorem{lem}[theo]{Lemma}
\newenvironment{lemma}[1][]
{\begin{lem}[#1]\begin{leftbar}}
{\end{leftbar}\end{lem}}
\newtheorem{prop}[theo]{Proposition}
\newenvironment{proposition}[1][]
{\begin{prop}[#1]\begin{leftbar}}
{\end{leftbar}\end{prop}}
\newtheorem{defi}[theo]{Definition}

\newtheorem{remk}[theo]{Remark}
\newenvironment{remark}[1][]
{\begin{remk}[#1]\begin{leftbar}}
{\end{leftbar}\end{remk}}
\newtheorem{coro}[theo]{Corollary}
\newenvironment{corollary}[1][]
{\begin{coro}[#1]\begin{leftbar}}
{\end{leftbar}\end{coro}}
\newtheorem{conv}[theo]{Convention}

\newtheorem{quest}[theo]{Question}
\newenvironment{question}[1][]
{\begin{quest}[#1]\begin{leftbar}}
{\end{leftbar}\end{quest}}
\newtheorem{warn}[theo]{Warning}

\newtheorem{conj}[theo]{Conjecture}
\newenvironment{conjecture}[1][]
{\begin{conj}[#1]\begin{leftbar}}
{\end{leftbar}\end{conj}}
\newtheorem{exam}[theo]{Example}
\newenvironment{example}[1][]
{\begin{exam}[#1]\begin{leftbar}}
{\end{leftbar}\end{exam}}
\newcommand{\silentsection}{\section}
\newenvironment{statement}{\begin{quote}}{\end{quote}}
\newenvironment{fineprint}{\begin{small}}{\end{small}}

\let\sumnonlimits\sum
\let\prodnonlimits\prod
\let\cupnonlimits\bigcup
\renewcommand{\sum}{\sumnonlimits\limits}
\renewcommand{\prod}{\prodnonlimits\limits}
\renewcommand{\bigcup}{\cupnonlimits\limits}
\DeclareMathOperator{\osc}{OSC}
\DeclareMathOperator{\rtb}{R2B}

\newenvironment{noncompile}{}{}
\excludecomment{noncompile}
\ihead{The one-sided cycle shuffles, version 10 March 2024}
\ohead{page \thepage}
\begin{document}

\title{The one-sided cycle shuffles in the symmetric group algebra}
\author{Darij Grinberg and Nadia Lafreni\`{e}re}
\date{
10 March 2024
}
\maketitle

\begin{abstract}
\textbf{Abstract.} We study an infinite family of shuffling operators on the
symmetric group $S_{n}$, which includes the well-studied top-to-random
shuffle. The general shuffling scheme consists of removing one card at a time
from the deck (according to some probability distribution) and re-inserting it
at a position chosen uniformly at random among the positions below. Rewritten
in terms of the group algebra $\mathbb{R}\left[  S_{n}\right]  $, our shuffle
corresponds to right multiplication by a linear combination of the elements
\[
t_{\ell}:=\operatorname*{cyc}\nolimits_{\ell}+\operatorname*{cyc}%
\nolimits_{\ell,\ell+1}+\operatorname*{cyc}\nolimits_{\ell,\ell+1,\ell
+2}+\cdots+\operatorname*{cyc}\nolimits_{\ell,\ell+1,\ldots,n}\in
\mathbb{R}\left[  S_{n}\right]
\]
for all $\ell\in\left\{  1,2,\ldots,n\right\}  $ (where $\operatorname{cyc}%
_{i_{1},i_{2},\ldots,i_{p}}$ denotes the permutation in $S_{n}$ that cycles
through $i_{1},i_{2},\ldots,i_{p}$).

We compute the eigenvalues of these shuffling operators and of all their
linear combinations. In particular, we show that the eigenvalues of right
multiplication by a linear combination $\lambda_{1}t_{1}+\lambda_{2}%
t_{2}+\cdots+\lambda_{n}t_{n}$ (with $\lambda_{1},\lambda_{2},\ldots
,\lambda_{n}\in\mathbb{R}$) are the numbers $\lambda_{1}m_{I,1}+\lambda
_{2}m_{I,2}+\cdots+\lambda_{n}m_{I,n}$, where $I$ ranges over the
\emph{lacunar} subsets of $\left\{  1,2,\ldots,n-1\right\}  $ (i.e., over the
subsets that contain no two consecutive integers), and where $m_{I,\ell}$
denotes the distance from $\ell$ to the next-higher element of $I$ (which
element is understood to be $\ell$ itself if $\ell\in I$, and to be $n+1$ if
$\ell>\max I$). We compute the multiplicities of these eigenvalues and show
that if they are all distinct, the shuffling operator is diagonalizable. To
this purpose, we show that the operators of right multiplication by
$t_{1},t_{2},\ldots,t_{n}$ on $\mathbb{R}\left[  S_{n}\right]  $ are
simultaneously triangularizable, and in fact there is a combinatorially
defined basis (the \textquotedblleft descent-destroying
basis\textquotedblright, as we call it) of $\mathbb{R}\left[  S_{n}\right]  $
in which they are represented by upper-triangular matrices. The results stated
here over $\mathbb{R}$ for convenience are actually stated and proved over an
arbitrary commutative ring $\mathbf{k}$.

We finish by describing a strong stationary time for the random-to-below
shuffle, which is the shuffle in which the card that moves below is selected
uniformly at random, and we give the waiting time for this event to happen.
\medskip

 \medskip


\end{abstract}
\tableofcontents

\section{Introduction}

Card shuffling operators have been studied both from algebraic and
probabilistic point of views. The interest in an algebraic study of those
operators bloomed with the discovery by Diaconis and Shahshahani that the
eigenvalues of some matrices could be used to bound the mixing time of the
shuffles \cite{DiaSha81}, which answers the question ``how many times should
we shuffle a deck of cards to get a well-shuffled deck?''. We now know a
combinatorial description of the eigenvalues of several shuffling operators,
including the transposition shuffle \cite{DiaSha81}, the riffle shuffle
\cite{BayDia92}, the top-to-random shuffle \cite{Phatar91} and the
random-to-random shuffle \cite{DieSal18}, among several others. An interesting
research question is to characterize shuffles whose eigenvalues admit a
combinatorial description. We contribute to this project by describing a new
family of shuffles that do so.

Given a probability distribution $P$ on the set $\left\{  1,2,\ldots
,n\right\}  $, the \emph{one-sided cycle shuffle} corresponding to $P$
consists of picking the card at position $i$ with probability $P\left(
i\right)  $, removing it, and reinserting it at a position weakly below
position $i$, chosen uniformly at random. By varying the probability
distribution, we obtain an infinite family of shuffling operators, whose
eigenvalues can be written as linear combinations of certain combinatorial
numbers with coefficients given by the probability distribution. Special cases
of interest include the top-to-random shuffle, the random-to-below shuffle
(where position $i$ is selected uniformly at random), and the unweighted
one-sided cycle shuffle (where position $i$ is selected with probability
$\dfrac{2\left(  n+1-i\right)  }{n\left(  n+1\right)  }$). A more explicit
description of the shuffles can be found in Section \ref{sec.shuffles}.

Two of our main results -- Corollary \ref{cor.eigen.spec} and Theorem
\ref{thm.eigen.mult} -- give the eigenvalues of all the one-sided cycle
shuffles. These eigenvalues are indexed by what we call \textquotedblleft
lacunar sets\textquotedblright, which are subsets of $\mathbb{Z}$ that do not
contain consecutive integers (see Section \ref{sec.Lacunarity} for details).
As a consequence, all eigenvalues are real, positive and explicitly described.

Most studies of eigenvalues of Markov chains focus on reversible chains, which
means that their transition matrix is symmetric. In that case, eigenvalues can
be used alone for bounding the mixing time of the Markov chain. This is
however not the case for the one-sided cycle shuffles.

Examples of non-reversible Markov chains whose eigenvalues have been studied
include the riffle shuffle \cite{BayDia92}, the top-to-random and
random-to-top shuffles \cite{Phatar91}, the pop shuffles and other `BHR'
shuffling operators \cite{BiHaRo99}, and the top-$m$-to-random shuffles
\cite{DiFiPi92}. All these admit a combinatorial description of their
eigenvalues. It is surprising that non-symmetric matrices admit real
eigenvalues, let alone eigenvalues that can be computed by simple formulas. It
is these surprisingly elegant eigenvalues that have given the impetus for the
present study.

To prove and explain our main results, we decompose the one-sided cycle
shuffles into linear combinations of $n$ operators $t_{1},t_{2},\ldots,t_{n}$,
which we call the \emph{somewhere-to-below shuffles}. Each somewhere-to-below
shuffle $t_{\ell}$ moves the card at position $\ell$ to a position weakly
below it, chosen uniformly at random. We show that the somewhere-to-below
shuffles are simultaneously triangularizable by giving explicitly a basis in
which they can be triangularized. This later gives us the eigenvalues. The
triangularity, in fact, is an understatement; we actually find a filtration
$0=F_{0}\subseteq F_{1}\subseteq F_{2}\subseteq\cdots\subseteq F_{f_{n+1}%
}=\mathbb{Z}\left[  S_{n}\right]  $ of the group ring of $S_{n}$ that is
preserved by all somewhere-to-below shuffles and has the additional property
that each $t_{\ell}$ acts as a scalar on each quotient $F_{i}/F_{i-1}$. Here,
perhaps unexpectedly, $f_{n+1}$ is the $\left(  n+1\right)  $-st Fibonacci
number. Thus, the number of distinct eigenvalues of a one-sided cycle shuffle
is never larger than $f_{n+1}$.

A diversity of algebraic techniques for computing the spectrum of shuffling
operators have appeared recently
\cite{ReSaWe14,DiPaRa14,DieSal18,Lafren19,BaCoMR21,Pang22,NesPen22}. This
paper contributes new algebraic methods to this extensive toolkit.

We end the paper by establishing a strong stationary time for one shuffling
operator in our family, the random-to-below shuffle, which happens in an
expected time of at most $n\left(  \log n+\log\left(  \log n\right)
+\log2\right)  + 1$. The arguments used here are similar to those used to get
a stationary time for the top-to-random shuffle; see Section
\ref{sec.stoppingtime}.
\medskip

This is the arXiv version of the present paper; a somewhat terser writeup
has been published in the \emph{Algebraic Combinatorics} journal.
See also the extended abstract \cite{fps2024sn} for a brief summary of this
and some related work.

\paragraph{Acknowledgements}

The authors would like to thank Eran Assaf, Sarah Brauner, Colin Defant, Persi
Diaconis, Theo Douvropoulos, Maxim Kontsevich, Martin Lorenz, Oliver
Matheau-Raven, Amy Pang, Karol Penson, Victor Reiner and Franco Saliola for
inspiring discussions and insightful comments. This work was made possible
thanks to \cite{sagemath}.

\section{The algebraic setup}

Card shuffling schemes are often understood by mathematicians as drawing,
randomly, a permutation and applying it to a deck of cards. Therefore, our
work takes place in the symmetric group algebra, which we define in this section.

\subsection{\label{subsec.notations}Basic notations}

Let $\mathbf{k}$ be any commutative ring. (In most applications, $\mathbf{k}$
is either $\mathbb{Z}$, $\mathbb{Q}$ or $\mathbb{R}$.)

Let $\mathbb{N}:=\left\{  0,1,2,\ldots\right\}  $ be the set of all
nonnegative integers.

For any integers $a$ and $b$, we set $\left[  a,b\right]  :=\left\{
x\in\mathbb{Z}\ \mid\ a\leq x\leq b\right\}  =\left\{  a,a+1,\ldots,b\right\}
$. This is an empty set if $a>b$.

For each $n\in\mathbb{Z}$, let $\left[  n\right]  :=\left[  1,n\right]
=\left\{  1,2,\ldots,n\right\}  $.

Fix an integer $n\in\mathbb{N}$. Let $S_{n}$ be the $n$-th symmetric group,
i.e., the group of all permutations of $\left[  n\right]  $. We multiply
permutations in the \textquotedblleft continental\textquotedblright\ way: that
is, $\left(  \pi\sigma\right)  \left(  i\right)  =\pi\left(  \sigma\left(
i\right)  \right)  $ for all $\pi,\sigma\in S_{n}$ and $i\in\left[  n\right]
$.

For any $k$ distinct elements $i_{1},i_{2},\ldots,i_{k}$ of $\left[  n\right]
$, we let $\operatorname*{cyc}\nolimits_{i_{1},i_{2},\ldots,i_{k}}$ be the
permutation in $S_{n}$ that sends $i_{1},i_{2},\ldots,i_{k-1},i_{k}$ to
$i_{2},i_{3},\ldots,i_{k},i_{1}$, respectively while leaving all remaining
elements of $\left[  n\right]  $ unchanged. This permutation is known as a
\emph{cycle}. Note that $\operatorname*{cyc}\nolimits_{i}=\operatorname*{id}$
for any single $i\in\left[  n\right]  $.

\subsection{Some elements of $\mathbf{k}\left[  S_{n}\right]  $}

Consider the group algebra $\mathbf{k}\left[  S_{n}\right]  $. In this
algebra, define $n$ elements $t_{1},t_{2},\ldots,t_{n}$ by setting%
\begin{equation}
t_{\ell}:=\operatorname*{cyc}\nolimits_{\ell}+\operatorname*{cyc}%
\nolimits_{\ell,\ell+1}+\operatorname*{cyc}\nolimits_{\ell,\ell+1,\ell
+2}+\cdots+\operatorname*{cyc}\nolimits_{\ell,\ell+1,\ldots,n}\in
\mathbf{k}\left[  S_{n}\right]  \label{eq.def.tl.deftl}%
\end{equation}
for each $\ell\in\left[  n\right]  $. Thus, in particular, $t_{n}%
=\operatorname*{cyc}\nolimits_{n}=\operatorname*{id}=1$ (where $1$ means the
unity of $\mathbf{k}\left[  S_{n}\right]  $). We shall refer to the $n$
elements $t_{1},t_{2},\ldots,t_{n}$ as the \emph{somewhere-to-below shuffles},
due to a probabilistic significance that we will discuss soon.

The first somewhere-to-below shuffle $t_{1}$ is known as the
\emph{top-to-random shuffle}, and has been studied, for example, in
\cite{DiFiPi92}.\footnote{Our $t_{1}$ equals the $B_{1}$ defined in
\cite[(4.4)]{DiFiPi92} (since the cycles $\operatorname*{cyc}\nolimits_{1}%
,\operatorname*{cyc}\nolimits_{1,2},\ldots,\operatorname*{cyc}%
\nolimits_{1,2,\ldots,n}$ are the only permutations $\pi\in S_{n}$ satisfying
$\pi^{-1}\left(  n\right)  >\pi^{-1}\left(  n-1\right)  >\cdots>\pi
^{-1}\left(  2\right)  $).
\par
The (German) diploma thesis \cite{Palmes10} provides a detailed exposition of
the results of \cite[(4.4)]{DiFiPi92} (in particular, \cite[Satz
2.4.6]{Palmes10} is \cite[Theorem 4.2]{DiFiPi92}).
\par
See also \cite{Grinbe18} for an exposition of the most basic algebraic
properties of $t_{1}$ (which is denoted by $\mathbf{A}$ in \cite{Grinbe18}).
An unexpected application to machine learning has recently been given in
\cite[proof of Lemma 29]{Reizen19}.} It shares a lot of properties with its
adjoint operator, the \emph{random-to-top shuffle}, also widely studied
(sometimes with other names, such as the Tsetlin Library or the move-to-front
rule, as in \cite{Hendri72, Donnel91, Phatar91, Fill96, BiHaRo99}), and
described in Section \ref{sect.furtheralg} as $t_{1}^{\prime}$.

We shall study not just the somewhere-to-below shuffles, but also their
$\mathbf{k}$-linear combinations $\lambda_{1}t_{1}+\lambda_{2}t_{2}%
+\cdots+\lambda_{n}t_{n}$ (with $\lambda_{1},\lambda_{2},\ldots,\lambda_{n}%
\in\mathbf{k}$), which we call the \emph{one-sided cycle shuffles}.

\subsection{The card-shuffling interpretation}

For $\mathbf{k}=\mathbb{R}$, the elements $t_{1},t_{2},\ldots,t_{n}$ (and many
other elements of $\mathbf{k}\left[  S_{n}\right]  $) have an interpretation
in terms of card shuffling.

Namely, we consider a permutation $w\in S_{n}$ as a way to order a deck of $n$
cards\footnote{As is customary in card-shuffling combinatorics, the cards are
bijectively numbered $1,2,\ldots,n$; there are no suits, colors or jokers.}
such that the cards are $w\left(  1 \right)  ,w\left(  2 \right)
,\ldots,w\left(  n \right)  $ from top to bottom (so the top card is $w\left(
1 \right)  $, and the bottom card is $w\left(  n \right)  $). Shuffling the
deck corresponds to permuting the cards: A permutation $\sigma\in S_{n}$
transforms a deck order $w\in S_{n}$ into the deck order $w\sigma$ (that is,
the order in which the cards are $w\left(  \sigma(1) \right)  ,w\left(
\sigma(2) \right)  ,\ldots,w\left(  \sigma(n) \right)  $ from top to bottom).

A probability distribution on the $n!$ possible orders of a deck of $n$ cards
can be identified with the element $\sum_{w\in S_{n}}P\left(  w\right)  w$ of
$\mathbb{R}\left[  S_{n}\right]  $, where $P\left(  w\right)  $ is the
probability of the deck having order $w$. Likewise, a nonzero element
$\sum_{\sigma\in S_{n}}P\left(  \sigma\right)  \sigma$ of $\mathbb{R}\left[
S_{n}\right]  $ (with all $P\left(  \sigma\right)  $ being nonnegative reals)
defines a Markov chain on the set of all these $n!$ orders, in which the
transition probability from deck order $w$ to deck order $w\tau$ equals
$\dfrac{P\left(  \tau\right)  }{\sum_{\sigma\in S_{n}}P\left(  \sigma\right)
}$ for each $w,\tau\in S_{n}$. This is an instance of a \emph{right random
walk on a group}, as defined (e.g.) in \cite[Section 2.6]{LePeWi09}.

From this point of view, the top-to-random shuffle $t_{1}$ describes the
Markov chain in which a deck is transformed by picking the topmost card and
moving it into the deck at a position chosen uniformly at random (which may
well be its original, topmost position). This explains the name of $t_{1}$
(and its significance to probabilists). More generally, a somewhere-to-below
shuffle $t_{\ell}$ transforms a deck by picking its $\ell$-th card from the
top and moving it to a weakly lower place (chosen uniformly at random).
Finally, a one-sided cycle shuffle $\lambda_{1}t_{1}+\lambda_{2}t_{2}%
+\cdots+\lambda_{n}t_{n}$ (with $\lambda_{1},\lambda_{2},\ldots,\lambda_{n}%
\in\mathbb{R}_{\geq0}$ being not all $0$) picks a card at random --
specifically, picking the $\ell$-th card from the top with probability
$\dfrac{(n-\ell+1)\lambda_{\ell}}{\sum_{i=1}^{n}(n-i+1)\lambda_{i}}$ -- and
moves it to a weakly lower place (chosen uniformly at random).

\section{\label{sec.shuffles}The one-sided cycle shuffles}

In this section, we shall explore the probabilistic significance of one-sided
cycle shuffles and several particular cases thereof. We begin by a reindexing
of the one-sided cycle shuffles that is particularly convenient for
probabilistic considerations. Note that, since transition matrices of Markov
chains have their rows summing to $1$, the operators, as we describe them in
this section, are scaled to satisfy this property. However, throughout the
paper, the coefficients can sum up to any numbers; multiplying the operators
by the appropriate number would give the corresponding Markov chain.

For a given probability distribution $P$ on the set $\left[  n\right]  $, we
define the \emph{one-sided cycle shuffle governed by }$P$ to be the element%
\[
\osc(P,n):=\frac{P(1)}{n}t_{1}+\frac{P(2)}{n-1}t_{2}+\frac{P(3)}{n-2}%
t_{3}+\cdots+\dfrac{P\left(  n\right)  }{1}t_{n}\in\mathbb{R}\left[
S_{n}\right]  .
\]
This one-sided cycle shuffle gives rise to a Markov chain on the symmetric
group $S_{n}$, which transforms a deck order by selecting a card at random
according to the probability distribution $P$ (more precisely, we pick the
\textbf{position}, not the value of the card, using $P$), and then applying
the corresponding somewhere-to-below shuffle. The transition probability of
this Markov chain is thus given by
\[
Q(\tau,\sigma)=\left\{
\begin{array}
[c]{ll}%
\sum_{i=1}^{n}\frac{P(i)}{n+1-i}, & \text{if }\sigma=\tau;\\
\frac{P(i)}{n+1-i}, & \text{if }\sigma=\tau\cdot\operatorname*{cyc}%
\nolimits_{i,i+1,\ldots,j}\text{ for some $j>i;$}\\
0, & \text{otherwise}.
\end{array}
\right.
\]
The $n!\times n!$-matrix $\left(  Q\left(  \tau,\sigma\right)  \right)
_{\tau,\sigma\in S_{n}}$ is the transition matrix of this Markov chain; when
we talk of the eigenvalues of the Markov chain, we refer to the eigenvalues of
the corresponding transition matrix.

These Markov chains are not reversible, which means that their transition
matrices are not symmetric.

\subsection{Interesting one-sided cycle shuffles}

Some probability distributions on $[n]$ lead to one-sided cycle shuffles that
have an interesting meaning in terms of card shuffling. We shall next consider
three such cases.

\paragraph{The top-to-random shuffle}

The top-to-random shuffle $t_{1}$ is the one-sided cycle shuffle that garnered
the most interest. We obtain it by setting $P(1)=1$, and $P(i)=0$ for all
$i\neq1$.

The transition matrix for the top-to-random shuffle, with $3$ cards $w_{1} :=
w\left(  1 \right)  $, $w_{2} := w\left(  2 \right)  $ and $w_{3} := w\left(
3 \right)  $, is
\[
\operatorname*{T2R}\nolimits_{3}%
=\bordermatrix{ 	&{\scriptstyle w_1w_2w_3} & {\scriptstyle w_1w_3w_2} &
	{\scriptstyle w_2w_1w_3} & {\scriptstyle w_2w_3w_1} &
	{\scriptstyle w_3w_1w_2} & {\scriptstyle w_3w_2w_1} \cr
	{\scriptstyle w_1w_2w_3} & \frac{1}{3} & 0 & \frac{1}{3} & \frac{1}{3} & 0 & 0 \cr
	{\scriptstyle w_1w_3w_2} & 0 & \frac{1}{3} & 0 & 0 & \frac{1}{3} & \frac{1}{3} \cr
	{\scriptstyle w_2w_1w_3} & \frac{1}{3} & \frac{1}{3} & \frac{1}{3} & 0 & 0 & 0 \cr
	{\scriptstyle w_2w_3w_1} & 0 & 0 & 0 & \frac{1}{3} &  \frac{1}{3} & \frac{1}{3} \cr
	{\scriptstyle w_3w_1w_2} & \frac{1}{3} & \frac{1}{3} & 0 & 0 & \frac{1}{3} & 0 \cr
	{\scriptstyle w_3w_2w_1} & 0 & 0 & \frac{1}{3} & \frac{1}{3} & 0 & \frac{1}{3} \cr
}
\]
(where $w_{i}w_{j}w_{k}$ is shorthand for the permutation in $S_{3}$ that
sends $1,2,3$ to $w_{i},w_{j},w_{k}$, respectively).

The eigenvalues of this matrix are known since \cite{Phatar91} to be
$0,\frac{1}{n},\frac{2}{n},\ldots,\frac{n-2}{n},1$, and the multiplicity of
the eigenvalue $\frac{i}{n}$ is the number of permutations in $S_{n}$ that
have exactly $i$ fixed points.\footnote{Actually, \cite{Phatar91} studies a
more general kind of shuffling operators with further parameters $p_{1}%
,p_{2},\ldots,p_{n}$, but these can no longer be seen as random walks on a
group and do not appear to fit into a well-behaved \textquotedblleft
somewhere-to-below shuffle\textquotedblright\ family in the way $t_{1}$ does.}
In other words, the eigenvalues of $t_{1}$ are $0,1,2,\ldots,\mbox{$n-2$},n$
with multiplicities as just said. Other descriptions of the eigenvalues of the
top-to-random shuffle are given in terms of set partitions \cite{BiHaRo99} and
in terms of standard Young tableaux \cite{ReSaWe14}.

\paragraph{The random-to-below shuffle}

The \emph{random-to-below shuffle} consists of picking any card randomly (with
uniform probability), and inserting it anywhere weakly below (with uniform
probability). This is the one-sided cycle shuffle governed by the uniform
distribution (i.e., by the probability distribution $P$ with $P(i)=\dfrac
{1}{n}$ for all $i\in\lbrack n]$). Hence, the random-to-below operator is, in
terms of the somewhere-to-below operators,
\[
\rtb_{n}=\frac{1}{n^{2}}t_{1}+\frac{1}{n(n-1)}t_{2}+\frac{1}{n(n-2)}%
t_{3}+\cdots+\frac{1}{n}t_{n}.
\]

A sample transition matrix for the random-to-below shuffle is given here for a
deck with $3$ cards:
\[
\rtb_{3} =
\bordermatrix{ 	&{\scriptstyle w_1w_2w_3} & {\scriptstyle w_1w_3w_2} &
	{\scriptstyle w_2w_1w_3} & {\scriptstyle w_2w_3w_1} &
	{\scriptstyle w_3w_1w_2} & {\scriptstyle w_3w_2w_1} \cr
	{\scriptstyle w_1w_2w_3} & \frac{11}{18} & \frac{1}{6} & \frac{1}{9} & \frac{1}{9} & 0 & 0 \cr
	{\scriptstyle w_1w_3w_2} & \frac{1}{6} & \frac{11}{18} & 0 & 0 & \frac{1}{9} & \frac{1}{9} \cr
	{\scriptstyle w_2w_1w_3} & \frac{1}{9} & \frac{1}{9} & \frac{11}{18} & \frac{1}{6} & 0 & 0 \cr
	{\scriptstyle w_2w_3w_1} & 0 & 0 & \frac{1}{6} & \frac{11}{18} & \frac{1}{9} & \frac{1}{9} \cr
	{\scriptstyle w_3w_1w_2} & \frac{1}{9} & \frac{1}{9} & 0 & 0 & \frac{11}{18} & \frac{1}{6} \cr
	{\scriptstyle w_3w_2w_1} & 0 & 0 & \frac{1}{9} & \frac{1}{9} & \frac{1}{6} & \frac{11}{18} \cr
}.
\]

A recently studied shuffle admits a similar description, namely the one-sided
transposition shuffle \cite{BaCoMR21}, that picks a card uniformly at random
and swaps it with a card chosen uniformly at random among the cards below.
Despite its similar-sounding description, it is not a one-sided cycle shuffle
(unless $n\leq2$), and a striking difference between the two shuffles is that
the matrix of the one-sided transposition shuffle is symmetric, unlike the one
for random-to-below.

\paragraph{The unweighted one-sided cycle}

Consider a variation of the problem, in which we pick a somewhere-to-below
move uniformly among the possible moves allowed. That is, we choose (with
uniform probability) two integers $i$ and $j$ in $\left[  n\right]  $
satisfying $i\leq j$, and then we apply the cycle $\operatorname*{cyc}%
\nolimits_{i,i+1,\ldots,j}$. Thus, the probability of applying the cycle
$\operatorname*{cyc}\nolimits_{i,i+1,\ldots,j}$ is $\dfrac{2}{n(n+1)}$ for all
$i<j$, and the probability of applying the identity is $\dfrac{2}{n+1}$. This
is the one-sided cycle shuffle governed by the probability distribution $P$
with $P\left(  i\right)  =\dfrac{2\left(  n-i+1\right)  }{n\left(  n+1\right)
}$. For $n=3$, its transition matrix is%
\[
\bordermatrix{  	&{\scriptstyle w_1w_2w_3} & {\scriptstyle w_1w_3w_2} &
	{\scriptstyle w_2w_1w_3} & {\scriptstyle w_2w_3w_1} &
	{\scriptstyle w_3w_1w_2} & {\scriptstyle w_3w_2w_1} \cr
	{\scriptstyle w_1w_2w_3} & \frac{1}{2} & \frac{1}{6} & \frac{1}{6} & \frac{1}{6} & 0 & 0 \cr
	{\scriptstyle w_1w_3w_2} & \frac{1}{6} & \frac{1}{2} & 0 & 0 & \frac{1}{6} & \frac{1}{6} \cr
	{\scriptstyle w_2w_1w_3} & \frac{1}{6} & \frac{1}{6} & \frac{1}{2} & \frac{1}{6} & 0 & 0 \cr
	{\scriptstyle w_2w_3w_1} & 0 & 0 & \frac{1}{6} & \frac{1}{2} & \frac{1}{6} & \frac{1}{6} \cr
	{\scriptstyle w_3w_1w_2} & \frac{1}{6} & \frac{1}{6} & 0 & 0 & \frac{1}{2} & \frac{1}{6} \cr
	{\scriptstyle w_3w_2w_1} & 0 & 0 & \frac{1}{6} & \frac{1}{6} & \frac{1}{6} & \frac{1}{2} \cr
}.
\]

\subsection{Eigenvalues and mixing time results for one-sided cycle shuffles}

\label{sec.eigenvalues_for_osc}

Corollary \ref{cor.eigen.spec} further below describes the eigenvalues for any
one-sided cycle shuffle. For a deck of $n$ cards, the eigenvalues are indexed
by lacunar subsets of $[n-1]$, which are subsets of $[n-1]$ that do not
contain consecutive integers. Given such a subset $I$, we define in Section
\ref{sec.Lacunarity} the nonnegative integers $m_{I,1},m_{I,2},\ldots,m_{I,n}%
$. Then, the eigenvalue of the one-sided cycle shuffle $\osc(P,n)$ indexed by
$I$ is
\[
\frac{P(1)}{n}m_{I,1}+\frac{P(2)}{n-1}m_{I,2}+\cdots+\dfrac{P\left(  n\right)
}{1}m_{I,n}.
\]

A consequence of this description is that all the eigenvalues are nonnegative
reals (and are rational if the $P\left(  1\right)  ,P\left(  2\right)
,\ldots,P\left(  n\right)  $ are). This is a surprising result for a matrix
that is not symmetric.

However, the fact that the matrices are not symmetric means that their
eigenvalues cannot be used alone to bound the mixing time for the one-sided
cycle shuffle. To palliate this, we describe a strong stationary time for the
one-sided cycle shuffles in Section \ref{sec.stoppingtime}. In the specific
case of the random-to-below shuffle, we give the waiting time to achieve it.

\paragraph{Eigenvalues of some interesting one-sided cycle shuffles}

The statement above can be used to find the eigenvalues of any one-sided cycle
shuffle, including the top-to-random shuffle. In this case, the eigenvalues
are given as $\dfrac{m_{I,1}}{n}$. It should become clear, after we define the
numbers $m_{I,1}$ and lacunar sets in Section \ref{sec.Lacunarity}, that the
values that $m_{I,1}$ can take are exactly the integers $0, 1, 2, \ldots, n-2,
n$.

Similarly, Corollary \ref{cor.eigen.spec} (as restated above) yields that the
eigenvalues for the unweighted one-sided cycle shuffle are given by $\dfrac
{2}{n(n+1)} \left(  m_{I,1} + m_{I,2} + \ldots+ m_{I,n}\right)  $, and are
indexed by the lacunar subsets of $[n-1]$. As far as we can tell, there is no
known simple combinatorial expression for the sum $m_{I,1} + m_{I,2} + \cdots+
m_{I,n}$.

\section{The operators in the symmetric group algebra}

We now resume the algebraic study of general one-sided cycle shuffles (with
arbitrary $\mathbf{k}$ and not necessarily governed by a probability
distribution). We will find it more convenient to work with endomorphisms of
the $\mathbf{k}$-module $\mathbf{k}\left[  S_{n}\right]  $ rather than with
$n!\times n!$-matrices.

For each element $x\in\mathbf{k}\left[  S_{n}\right]  $, let $R\left(
x\right)  $ denote the $\mathbf{k}$-linear map%
\begin{align*}
\mathbf{k}\left[  S_{n}\right]   &  \rightarrow\mathbf{k}\left[  S_{n}\right]
,\\
y  &  \mapsto yx.
\end{align*}
This map is known as \textquotedblleft right multiplication by $x$%
\textquotedblright, and is an endomorphism of the free $\mathbf{k}$-module
$\mathbf{k}\left[  S_{n}\right]  $; thus, it makes sense to speak of
eigenvalues, eigenvectors and triangularization.

One of our main results is the following:

\begin{theorem}
\label{thm.Rcomb-main}Let $\lambda_{1},\lambda_{2},\ldots,\lambda_{n}%
\in\mathbf{k}$. Then, the $\mathbf{k}$-module endomorphism $R\left(
\lambda_{1}t_{1}+\lambda_{2}t_{2}+\cdots+\lambda_{n}t_{n}\right)  $ of
$\mathbf{k}\left[  S_{n}\right]  $ can be triangularized -- i.e., there exists
a basis of the $\mathbf{k}$-module $\mathbf{k}\left[  S_{n}\right]  $ such
that this endomorphism is represented by an upper-triangular matrix with
respect to this basis. Moreover, this basis does not depend on $\lambda
_{1},\lambda_{2},\ldots,\lambda_{n}$.
\end{theorem}

We shall eventually describe both the basis and the eigenvalues of this
endomorphism $R\left(  \lambda_{1}t_{1}+\lambda_{2}t_{2}+\cdots+\lambda
_{n}t_{n}\right)  $ explicitly; indeed, both will follow from Theorem
\ref{thm.Rcomb-conc}.

\begin{remark}
\label{rmk.Rcomb}In general, the endomorphism $R\left(  \lambda_{1}%
t_{1}+\lambda_{2}t_{2}+\cdots+\lambda_{n}t_{n}\right)  $ cannot be
diagonalized. For example:

\begin{itemize}
\item If we take $\mathbf{k}=\mathbb{C}$, $n=4$ and $\lambda_{i}=1$ for each
$i\in\left[  n\right]  $ (which is the unweighted one-sided cycle shuffle),
then the minimal polynomial of the endomorphism $R\left(  \lambda_{1}%
t_{1}+\lambda_{2}t_{2}+\cdots+\lambda_{n}t_{n}\right)  $ is $\left(
x-10\right)  \left(  x-6\right)  \left(  x-4\right)  ^{2}\left(  x-2\right)
$, so that this endomorphism is not diagonalizable.

\item If we take $\mathbf{k}=\mathbb{C}$, $n=3$ and $\lambda_{i}=\dfrac{6}{i}$
for each $i\in\left[  n\right]  $, then the minimal polynomial of the
endomorphism $R\left(  \lambda_{1}t_{1}+\lambda_{2}t_{2}+\cdots+\lambda
_{n}t_{n}\right)  $ is $\left(  x-8\right)  ^{2} \left(  x-26\right)  $, so
that this endomorphism is not diagonalizable.
\end{itemize}

Consequently, there is (in general) no basis of $\mathbf{k}\left[
S_{n}\right]  $ such that all the endomorphisms $R\left(  t_{1}\right)
,R\left(  t_{2}\right)  ,\ldots,R\left(  t_{n}\right)  $ are represented by
diagonal matrices with respect to this basis. Triangular matrices are thus the
best one might hope for; and Theorem \ref{thm.Rcomb-main} reveals that this
hope indeed comes true. Eventually, we will see (Theorem
\ref{thm.eigen.diagonalizable}) that the endomorphism $R\left(  \lambda_{1}
t_{1} + \lambda_{2} t_{2} + \cdots+ \lambda_{n} t_{n}\right)  $ is
diagonalizable (over a field) for a sufficiently generic choice of
$\lambda_{1}, \lambda_{2}, \ldots, \lambda_{n}$.
\end{remark}

\section{\label{sec.Lacunarity}Subset basics: Lacunarity, Enclosure and
Non-Shadow}

In order to concretize the claims of Theorem \ref{thm.Rcomb-main}, we shall
introduce some features of sets of integers and a rather famous integer
sequence. The main role will be played by the \textit{lacunar sets}, which
will later index a certain filtration of $\mathbf{k}\left[  S_{n}\right]  $ on
whose subquotients the endomorphisms $R\left(  t_{\ell}\right)  $ act by
scalars. This is especially convenient since the number of lacunar sets is
relatively small (a Fibonacci number).

Let $\left(  f_{0},f_{1},f_{2},\ldots\right)  $ be the \emph{Fibonacci
sequence}. This is the sequence of integers defined recursively by%
\[
f_{0}=0,\ \ \ \ \ \ \ \ \ \ f_{1}=1,\ \ \ \ \ \ \ \ \ \ \text{and}%
\ \ \ \ \ \ \ \ \ \ f_{m}=f_{m-1}+f_{m-2}\text{ for all }m\geq2.
\]

We shall say that a set $I\subseteq\mathbb{Z}$ is \emph{lacunar} if it
contains no two consecutive integers (i.e., there exists no $i\in I$ such that
$i+1\in I$). For instance, the set $\left\{  1,4,6\right\}  $ is lacunar,
while the set $\left\{  1,4,5\right\}  $ is not. Lacunar sets are also known
as \textquotedblleft sparse sets\textquotedblright\ (in \cite{AgNyOr06}) or as
\textquotedblleft Zeckendorf sets\textquotedblright\ (in \cite{Chu19}, at
least when they are finite subsets of $\left\{  1,2,3,\ldots\right\}  $).

It is known (see, e.g., \cite[Proposition 1.4.9]{Grinbe20}) that the number of
lacunar subsets of $\left[  n\right]  $ is the Fibonacci number $f_{n+2}$.
Applying this to $n-1$ instead of $n$, we conclude that the number of lacunar
subsets of $\left[  n-1\right]  $ is $f_{n+1}$ whenever $n>0$. A moment's
thought reveals that this holds for $n=0$ as well (since $\left[  -1\right]
=\varnothing$), and thus holds for each nonnegative integer $n$.

If $I$ is any set of integers, then $I-1$ will denote the set $\left\{
i-1\ \mid\ i\in I\right\}  $; this is again a set of integers. For instance,
$\left\{  2,4,5\right\}  -1=\left\{  1,3,4\right\}  $. Note that a set $I$ is
lacunar if and only if $I\cap\left(  I-1\right)  =\varnothing$.

For any subset $I$ of $\left[  n\right]  $, we define the following:

\begin{itemize}
\item We let $\widehat{I}$ be the set $\left\{  0\right\}  \cup I\cup\left\{
n+1\right\}  $. We shall refer to $\widehat{I}$ as the \emph{enclosure} of $I$.

For example, if $n=5$, then $\widehat{\left\{  2,3\right\}  }=\left\{
0,2,3,6\right\}  $.

\item For any $\ell\in\left[  n\right]  $, we let $m_{I,\ell}$ be the number%
\[
\left(  \text{smallest element of }\widehat{I}\text{ that is }\geq\ell\right)
-\ell\in\left[  0,n+1-\ell\right]  \subseteq\left[  0,n\right]  .
\]
Those numbers $m_{I, \ell}$ already appeared in Subsection
\ref{sec.eigenvalues_for_osc}, as they play a crucial role in the expression
of the eigenvalues of the one-sided cycle shuffles.

For example, if $n=5$ and $I=\left\{  2,3\right\}  $, then%
\[
\left(  m_{I,1},\ m_{I,2},\ m_{I,3},\ m_{I,4},\ m_{I,5}\right)  =\left(
1,\ 0,\ 0,\ 2,\ 1\right)  .
\]
We note that an $\ell\in\left[  n\right]  $ satisfies $m_{I,\ell}=0$ if and
only if $\ell\in\widehat{I}$ (or, equivalently, $\ell\in I$).

\item We let $I^{\prime}$ be the set $\left[  n-1\right]  \setminus\left(
I\cup\left(  I-1\right)  \right)  $. This is the set of all $i\in\left[
n-1\right]  $ satisfying $i\notin I$ and $i+1\notin I$. We shall refer to
$I^{\prime}$ as the \emph{non-shadow} of $I$.

For example, if $n=5$, then $\left\{  2,3\right\}  ^{\prime}=\left[  4\right]
\setminus\left\{  1,2,3\right\}  =\left\{  4\right\}  $.
\end{itemize}

\section{\label{sec.transpositions}The simple transpositions $s_{i}$}

In this section, we will recall the basic properties of simple transpositions
in the symmetric group $S_{n}$, and use them to rewrite the definition
(\ref{eq.def.tl.deftl}) of the somewhere-to-below shuffles.

For any $i\in\left[  n-1\right]  $, we let $s_{i}:=\operatorname*{cyc}%
\nolimits_{i,i+1}\in S_{n}$. This permutation $s_{i}$ is called a \emph{simple
transposition}. It is well-known that $s_{1},s_{2},\ldots,s_{n-1}$ generate
the group $S_{n}$. Moreover, it is known that two simple transpositions
$s_{i}$ and $s_{j}$ commute whenever $\left\vert i-j\right\vert >1$. This
latter fact is known as \emph{reflection locality}.

It is furthermore easy to see that
\begin{equation}
\operatorname*{cyc}\nolimits_{\ell,\ell+1,\ldots,k}=s_{\ell}s_{\ell+1}\cdots
s_{k-1} \label{eq.cyc-via-ss}%
\end{equation}
for each $\ell\leq k$ in $\left[  n\right]  $. Thus, (\ref{eq.def.tl.deftl})
rewrites as follows:%
\begin{align}
t_{\ell}  &  =1+s_{\ell}+s_{\ell}s_{\ell+1}+\cdots+s_{\ell}s_{\ell+1}\cdots
s_{n-1}\nonumber\\
&  =\sum_{j=\ell}^{n}s_{\ell}s_{\ell+1}\cdots s_{j-1}
\label{eq.def.tl.deftl-s-sum}%
\end{align}
for each $\ell\in\left[  n\right]  $.

The following relationship between simple transpositions will later be used in
proving the triangularizability of the somewhere-to-below shuffles:

\begin{lemma}
\label{lem.si-into-cyc}Let $\ell\in\left[  n\right]  $ and $j\in\left[
n\right]  $. Let $i\in\left[  \ell,j-2\right]  $. Then,%
\[
s_{\ell}s_{\ell+1}\cdots s_{j-1}\cdot s_{i}=s_{i+1}\cdot s_{\ell}s_{\ell
+1}\cdots s_{j-1}.
\]

\end{lemma}

\begin{proof}
[Proof of Lemma \ref{lem.si-into-cyc}.]From $i\in\left[  \ell,j-2\right]  $,
we obtain $i\in\left[  \ell,j-1\right]  $ and $i+1\in\left[  \ell,j-1\right]
$ and $\ell\leq i\leq j-2<j$.

It is well-known that
\begin{equation}
\sigma\operatorname*{cyc}\nolimits_{p_{1},p_{2},\ldots,p_{k}}\sigma
^{-1}=\operatorname*{cyc}\nolimits_{\sigma\left(  p_{1}\right)  ,\sigma\left(
p_{2}\right)  ,\ldots,\sigma\left(  p_{k}\right)  }
\label{pf.lem.si-into-cyc.gen}%
\end{equation}
for any $\sigma\in S_{n}$ and any $k$ distinct elements $p_{1},p_{2}%
,\ldots,p_{k}$ of $\left[  n\right]  $.

Let $\sigma=\operatorname*{cyc}\nolimits_{\ell,\ell+1,\ldots,j}$. Then,
$\sigma\left(  i\right)  =i+1$ (since $i\in\left[  \ell,j-1\right]  $) and
$\sigma\left(  i+1\right)  =i+2$ (since $i+1\in\left[  \ell,j-1\right]  $).
However, (\ref{pf.lem.si-into-cyc.gen}) yields%
\[
\sigma\operatorname*{cyc}\nolimits_{i,i+1}\sigma^{-1}=\operatorname*{cyc}%
\nolimits_{\sigma\left(  i\right)  ,\sigma\left(  i+1\right)  }%
=\operatorname*{cyc}\nolimits_{i+1,i+2}%
\]
(since $\sigma\left(  i\right)  =i+1$ and $\sigma\left(  i+1\right)  =i+2$).
In view of $s_{i}=\operatorname*{cyc}\nolimits_{i,i+1}$ and $s_{i+1}%
=\operatorname*{cyc}\nolimits_{i+1,i+2}$, this rewrites as $\sigma s_{i}%
\sigma^{-1}=s_{i+1}$. In other words, $\sigma s_{i}=s_{i+1}\sigma$. In view of
$\sigma=\operatorname*{cyc}\nolimits_{\ell,\ell+1,\ldots,j}=s_{\ell}s_{\ell
+1}\cdots s_{j-1}$, we can rewrite this as $s_{\ell}s_{\ell+1}\cdots
s_{j-1}\cdot s_{i}=s_{i+1}\cdot s_{\ell}s_{\ell+1}\cdots s_{j-1}$. This proves
Lemma \ref{lem.si-into-cyc}.
\end{proof}

\section{The invariant spaces $F\left(  I\right)  $}

\label{sec.invariantspaces}

Recall that our goal is to prove Theorem \ref{thm.Rcomb-main}, which claims
that the one-sided cycle shuffles are triangularizable. To that end, we will
construct a $\mathbf{k}$-submodule filtration of $\mathbf{k}\left[
S_{n}\right]  $ that is preserved by all the somewhere-to-below shuffles. In
this section, we define a first family of submodules $F\left(  I\right)  $ of
$\mathbf{k}[S_{n}]$, which will later serve as building blocks for this filtration.

\subsection{Definition}

For any subset $I$ of $\left[  n\right]  $, we define the following:

\begin{itemize}
\item We let $\operatorname*{sum}I$ denote the sum of all elements of $I$.
This is an integer with $0\leq\operatorname*{sum}I\leq n\left(  n+1\right)
/2$.

\item We let
\[
F\left(  I\right)  :=\left\{  q\in\mathbf{k}\left[  S_{n}\right]
\ \mid\ qs_{i}=q\text{ for all }i\in I^{\prime}\right\}  .
\]
This is a $\mathbf{k}$-submodule of $\mathbf{k}\left[  S_{n}\right]  $.
Intuitively, it can be understood as follows: Writing each permutation $\pi\in
S_{n}$ as the $n$-tuple $\left(  \pi\left(  1\right)  ,\pi\left(  2\right)
,\ldots,\pi\left(  n\right)  \right)  $ (this is called \emph{one-line
notation}), we can view an element $q\in\mathbf{k}\left[  S_{n}\right]  $ as a
$\mathbf{k}$-linear combination of such $n$-tuples. The group $S_{n}$ acts on
such $n$-tuples from the right by permuting positions, and thus acts on their
linear combinations by linearity. An element $q\in\mathbf{k}\left[
S_{n}\right]  $ belongs to $F\left(  I\right)  $ if and only if it is
invariant under permuting any two adjacent positions $i$ and $i+1$ that both
lie outside of $I$. We thus call $F\left(  I\right)  $ an \emph{invariant
space}. \medskip

In terms of shuffling operators, one can think of $F(I)$ as the set of all
random decks (i.e., probability distributions on the $n!$ orderings of a deck)
that are \emph{fully shuffled} within each contiguous interval of
$[n]\backslash I$. This is to be understood as follows: Let $q \in F(I)$, and
let $\sigma\in S_{n}$ be a term appearing in $q$ with coefficient $c$. Let
$[i,j]$ be an interval of $[n]$ containing no element of $I$. Then, for any
permutation $\tau\in S_{n}$ that fixes each element of $[n]\backslash[i,j]$,
the coefficient of $\sigma\tau$ in $q$ is also $c$. Moreover, this property
characterizes the elements $q$ of $F(I)$. \medskip

Note that the set $F\left(  I\right)  $ depends only on $n$ and $I^{\prime}$,
but not on $I$. We nevertheless find it better to index it by $I$.
\end{itemize}

Note that $F\left(  \left[  n\right]  \right)  =\mathbf{k}\left[
S_{n}\right]  $, since $\left[  n\right]  ^{\prime}= \varnothing$. (Also, many
other subsets $I$ of $\left[  n\right]  $ satisfy $F\left(  I\right)
=\mathbf{k}\left[  S_{n}\right]  $. For example, this holds for $I=\left\{
2,4,6,8,\ldots\right\}  \cap\left[  n\right]  $ and for $I=\left\{
1,3,5,7,\ldots\right\}  \cap\left[  n\right]  $ and for $I=\left[  n-1\right]
$. Indeed, all of these sets $I$ satisfy $I^{\prime}=\varnothing$.)

Here are some more examples of the sets $F\left(  I\right)  $:

\begin{example}
\label{exa.F(I).n=3}Let $n=3$. Then, there are $2^{3}=8$ many subsets $I$ of
$\left[  n\right]  =\left[  3\right]  $. We shall compute the non-shadow
$I^{\prime}$ and the invariant space $F\left(  I\right)  $ for each of them:

\begin{itemize}
\item We have $\varnothing^{\prime}=\left[  2\right]  $ and thus
\begin{align*}
F\left(  \varnothing\right)   &  =\left\{  q\in\mathbf{k}\left[  S_{n}\right]
\ \mid\ qs_{i}=q\text{ for all }i\in\left[  2\right]  \right\} \\
&  =\operatorname*{span}\left(  \left[  123\right]  +\left[  132\right]
+\left[  213\right]  +\left[  231\right]  +\left[  312\right]  +\left[
321\right]  \right)  .
\end{align*}
Here, the notation \textquotedblleft$\operatorname*{span}$\textquotedblright%
\ means a $\mathbf{k}$-linear span, whereas the notation $\left[  ijk\right]
$ means the permutation $\sigma\in S_{3}$ that sends $1,2,3$ to $i,j,k$,
respectively. (In our case, we are taking the span of a single vector, but
soon we will see some more complicated spans.)

\item We have $\left\{  1\right\}  ^{\prime}=\left\{  2\right\}  $ and thus
\begin{align*}
F\left(  \left\{  1\right\}  \right)   &  =\left\{  q\in\mathbf{k}\left[
S_{n}\right]  \ \mid\ qs_{2}=q\right\} \\
&  =\operatorname*{span}\left(  \left[  123\right]  +\left[  132\right]
,\ \ \left[  213\right]  +\left[  231\right]  ,\ \ \left[  312\right]
+\left[  321\right]  \right)  .
\end{align*}

\item We have $\left\{  3\right\}  ^{\prime}=\left\{  1\right\}  $ and thus%
\begin{align*}
F\left(  \left\{  3\right\}  \right)   &  =\left\{  q\in\mathbf{k}\left[
S_{n}\right]  \ \mid\ qs_{1}=q\right\} \\
&  =\operatorname*{span}\left(  \left[  123\right]  +\left[  213\right]
,\ \ \left[  132\right]  +\left[  312\right]  ,\ \ \left[  231\right]
+\left[  321\right]  \right)  .
\end{align*}

\item If $I$ is any of the sets $\left\{  2\right\}  $, $\left\{  1,2\right\}
$, $\left\{  1,3\right\}  $, $\left\{  2,3\right\}  $ and $\left\{
1,2,3\right\}  $, then $I^{\prime}=\varnothing$ and thus%
\begin{align*}
F\left(  I\right)   &  =\left\{  q\in\mathbf{k}\left[  S_{n}\right]  \right\}
=\mathbf{k}\left[  S_{n}\right] \\
&  =\operatorname*{span}\left(  \left[  123\right]  ,\ \ \left[  132\right]
,\ \ \left[  213\right]  ,\ \ \left[  231\right]  ,\ \ \left[  312\right]
,\ \ \left[  321\right]  \right)  .
\end{align*}

\end{itemize}
\end{example}

\begin{example}
\label{exa.F(I).n=4}Let $n=4$. Then, $\left\{  1\right\}  ^{\prime}=\left\{
2,3\right\}  $ and thus%
\begin{align*}
F\left(  \left\{  1\right\}  \right)   &  =\left\{  q\in\mathbf{k}\left[
S_{n}\right]  \ \mid\ qs_{i}=q\text{ for all }i\in\left\{  2,3\right\}
\right\} \\
&  =\operatorname*{span}(\left[  1234\right]  +\left[  1243\right]  +\left[
1324\right]  +\left[  1342\right]  +\left[  1423\right]  +\left[  1432\right]
,\\
&  \ \ \ \ \ \ \ \ \ \ \ \ \ \ \ \ \ \ \ \ \left[  2134\right]  +\left[
2143\right]  +\left[  2314\right]  +\left[  2341\right]  +\left[  2413\right]
+\left[  2431\right]  ,\\
&  \ \ \ \ \ \ \ \ \ \ \ \ \ \ \ \ \ \ \ \ \left[  3124\right]  +\left[
3142\right]  +\left[  3214\right]  +\left[  3241\right]  +\left[  3412\right]
+\left[  3421\right]  ,\\
&  \ \ \ \ \ \ \ \ \ \ \ \ \ \ \ \ \ \ \ \ \left[  4123\right]  +\left[
4132\right]  +\left[  4213\right]  +\left[  4231\right]  +\left[  4312\right]
+\left[  4321\right]  ).
\end{align*}
Here, $\left[  ijk\ell\right]  $ means the permutation $\sigma\in S_{4}$ that
sends $1,2,3,4$ to $i,j,k,\ell$, respectively.
\end{example}

In Section \ref{sec.Filtration}, we shall define a filtration of
$\mathbf{k}\left[  S_{n}\right]  $ that requires sorting subsets according to
the sum of their elements. Hence, for each $k\in\mathbb{N}$, we set
\[
F\left(  <k\right)  :=\sum_{\substack{J\subseteq\left[  n\right]
;\\\operatorname*{sum}J<k}}F\left(  J\right)  .
\]

\subsection{Right multiplication by $t_{\ell}-m_{I,\ell}$ moves us down the
$F\left(  I\right)  $-grid}

We now claim the following theorem, which will play a crucial role in our
proof of Theorem \ref{thm.Rcomb-main}:

\begin{theorem}
\label{thm.tl-FI}Let $I\subseteq\left[  n\right]  $ and $\ell\in\left[
n\right]  $. Then,%
\[
F\left(  I\right)  \cdot\left(  t_{\ell}-m_{I,\ell}\right)  \subseteq F\left(
<\operatorname*{sum}I\right)  .
\]
In other words, for each $q\in F\left(  I\right)  $, we have $q\cdot\left(
t_{\ell}-m_{I,\ell}\right)  \in F\left(  <\operatorname*{sum}I\right)  $.
\end{theorem}

This theorem is essential to establishing the triangularization stated in
Theorem \ref{thm.Rcomb-main}, which requires sorting the submodules $F(I)$
according to the sum of elements in $I$.

\begin{proof}
[Proof of Theorem~\ref{thm.tl-FI}.]Fix $q\in F\left(  I\right)  $. We must
prove that $q\cdot\left(  t_{\ell}-m_{I,\ell}\right)  \in F\left(
<\operatorname*{sum}I\right)  $. There are three main parts to our proof. In
the first part, we express $q\cdot\left(  t_{\ell}-m_{I, \ell}\right)  $ as a
sum of products of $q$ with simple transpositions (Equation
\eqref{pf.thm.tl-FI.3}). In the second part, we will break this sum up into
smaller sums (Equation \eqref{pf.thm.tl-FI.5}). In the third and last part, we
will show that each of these smaller sums is in $F\left(  K\right)  $ for some
$K \subseteq\left[  n \right]  $ satisfying $\operatorname{sum} K <
\operatorname{sum} I$ (and therefore in $F\left(  <\operatorname*{sum}%
I\right)  $). This will complete the proof. \medskip

Write the set $I$ in the form $I=\left\{  i_{1}<i_{2}<\cdots<i_{p}\right\}  $,
and furthermore set $i_{0}:=0$ and $i_{p+1}:=n+1$. Then, the enclosure of $I$
is
\[
\widehat{I}=\left\{  0=i_{0}<i_{1}<i_{2}<\cdots<i_{p}<i_{p+1}=n+1\right\}  .
\]

Let $i_{k}$ be the smallest element of $\widehat{I}$ that is greater than or
equal to $\ell$. Thus, $m_{I,\ell}=i_{k}-\ell$ (by the definition of
$m_{I,\ell}$) and%
\begin{equation}
i_{0}<i_{1}<\cdots<i_{k-1}<\ell\leq i_{k}<i_{k+1}<\cdots<i_{p+1}.
\label{pf.thm.tl-FI.ineqs}%
\end{equation}
Note that $k\geq1$ (since $k=0$ would entail $\ell\leq i_{k}=i_{0}=0$, which
is absurd), so that $i_{k}\geq1$. \medskip

From $i_{p+1}=n+1$, we obtain $n=i_{p+1}-1$. Now, multiplying the equality
(\ref{eq.def.tl.deftl-s-sum}) by $q$, we obtain
\begin{align}
qt_{\ell}  &  =q\sum_{j=\ell}^{n}s_{\ell}s_{\ell+1}\cdots s_{j-1}=\sum
_{j=\ell}^{n}qs_{\ell}s_{\ell+1}\cdots s_{j-1}\nonumber\\
&  =\sum_{j=\ell}^{i_{p+1}-1}qs_{\ell}s_{\ell+1}\cdots s_{j-1}%
\ \ \ \ \ \ \ \ \ \ \left(  \text{since }n=i_{p+1}-1\right) \nonumber\\
&  =\sum_{j=\ell}^{i_{k}-1}qs_{\ell}s_{\ell+1}\cdots s_{j-1}+\sum_{j=i_{k}%
}^{i_{p+1}-1}qs_{\ell}s_{\ell+1}\cdots s_{j-1} \label{pf.thm.tl-FI.2}%
\end{align}
(since $\ell\leq i_{k}\leq i_{p+1}$).

Now, from (\ref{pf.thm.tl-FI.ineqs}), it is easy to see that each $u\in\left[
\ell,i_{k}-2\right]  $ belongs to the non-shadow $I^{\prime}$ (since neither
$u$ nor $u+1$ belongs to $I$). Thus, each $u\in\left[  \ell,i_{k}-2\right]  $
satisfies $qs_{u}=q$ (since $q\in F\left(  I\right)  $). By applying this
observation multiple times, we see that $qs_{\ell}s_{\ell+1}\cdots s_{j-1}=q$
for each $j\in\left[  \ell,i_{k}-1\right]  $. Thus,%
\[
\sum_{j=\ell}^{i_{k}-1}\underbrace{qs_{\ell}s_{\ell+1}\cdots s_{j-1}}%
_{=q}=\sum_{j=\ell}^{i_{k}-1}q=\underbrace{\left(  i_{k}-\ell\right)
}_{=m_{I,\ell}}q=m_{I,\ell}q.
\]
Hence, we can rewrite (\ref{pf.thm.tl-FI.2}) as%
\[
qt_{\ell}=m_{I,\ell}q+\sum_{j=i_{k}}^{i_{p+1}-1}qs_{\ell}s_{\ell+1}\cdots
s_{j-1}.
\]
In other words,%
\[
qt_{\ell}-m_{I,\ell}q=\sum_{j=i_{k}}^{i_{p+1}-1}qs_{\ell}s_{\ell+1}\cdots
s_{j-1}.
\]
Since $qt_{\ell}-m_{I,\ell}q=q\cdot\left(  t_{\ell}-m_{I,\ell}\right)  $, we
can rewrite this further as%
\begin{equation}
q\cdot\left(  t_{\ell}-m_{I,\ell}\right)  =\sum_{j=i_{k}}^{i_{p+1}-1}qs_{\ell
}s_{\ell+1}\cdots s_{j-1}. \label{pf.thm.tl-FI.3}%
\end{equation}

Next, recall that $i_{k}<i_{k+1}<\cdots<i_{p+1}$. Hence, the interval $\left[
i_{k},i_{p+1}-1\right]  $ can be written as the disjoint union%
\[
\left[  i_{k},i_{k+1}-1\right]  \sqcup\left[  i_{k+1},i_{k+2}-1\right]
\sqcup\cdots\sqcup\left[  i_{p},i_{p+1}-1\right]  .
\]
Thus, the sum on the right hand side of (\ref{pf.thm.tl-FI.3}) can be split up
as follows:%
\[
\sum_{j=i_{k}}^{i_{p+1}-1}qs_{\ell}s_{\ell+1}\cdots s_{j-1}=\sum_{r=k}%
^{p}\ \ \sum_{j=i_{r}}^{i_{r+1}-1}qs_{\ell}s_{\ell+1}\cdots s_{j-1}.
\]
Therefore, (\ref{pf.thm.tl-FI.3}) can be rewritten as%
\begin{equation}
q\cdot\left(  t_{\ell}-m_{I,\ell}\right)  =\sum_{r=k}^{p}\ \ \sum_{j=i_{r}%
}^{i_{r+1}-1}qs_{\ell}s_{\ell+1}\cdots s_{j-1}. \label{pf.thm.tl-FI.5}%
\end{equation}

Recall that our goal is to prove that $q\cdot\left(  t_{\ell}-m_{I,\ell
}\right)  \in F\left(  <\operatorname*{sum}I\right)  $. In order to do so, we
only need to show that%
\[
\sum_{j=i_{r}}^{i_{r+1}-1}qs_{\ell}s_{\ell+1}\cdots s_{j-1}\in F\left(
<\operatorname*{sum}I\right)  \ \ \ \ \ \ \ \ \ \ \text{for each }r\in\left[
k,p\right]
\]
(because once this is proved, the equality (\ref{pf.thm.tl-FI.5}) will become%
\[
q\cdot\left(  t_{\ell}-m_{I,\ell}\right)  =\sum_{r=k}^{p}\ \ \underbrace{\sum
_{j=i_{r}}^{i_{r+1}-1}qs_{\ell}s_{\ell+1}\cdots s_{j-1}}_{\in F\left(
<\operatorname*{sum}I\right)  }\in\sum_{r=k}^{p}F\left(  <\operatorname*{sum}%
I\right)  \subseteq F\left(  <\operatorname*{sum}I\right)  ,
\]
and we will have achieved our goal).

This is what we shall now do. So let us fix some $r\in\left[  k,p\right]  $.
We set%
\begin{equation}
q^{\prime}:=\sum_{j=i_{r}}^{i_{r+1}-1}qs_{\ell}s_{\ell+1}\cdots s_{j-1}.
\label{pf.thm.tl-FI.q'=}%
\end{equation}
We must show that $q^{\prime}\in F\left(  <\operatorname*{sum}I\right)  $.

To do so, we make extensive use of the facts stated in Section
\ref{sec.transpositions} about simple transpositions, and the rest of the
proof is obtained by dealing with several cases. \medskip

From $r\in\left[  k,p\right]  $, we obtain $k\leq r\leq p$. From $k\leq p$ and
$k\geq1$, we obtain $k\in\left[  p\right]  $, so that $i_{k}\in\left\{
i_{1},i_{2},\ldots,i_{p}\right\}  =I\subseteq\left[  n\right]  $. Therefore,
$i_{k}\leq n$.

Also, from $r\leq p$ and $r\geq k\geq1$, we obtain $r\in\left[  p\right]  $,
so that $i_{r}\in\left\{  i_{1},i_{2},\ldots,i_{p}\right\}  =I\subseteq\left[
n\right]  $. Therefore, $i_{r}\leq n$.

Furthermore, from $k\leq r\leq p$, we obtain $i_{k}\leq i_{r}\leq i_{p}$
(since $i_{1}<i_{2}<\cdots<i_{p}$).

Moreover, from $i_{r}\in\left[  n\right]  $, we obtain $i_{r}\geq1$. From
$i_{0}<i_{1}<i_{2}<\cdots<i_{p}<i_{p+1}=n+1$, we obtain $i_{r+1}\leq n+1$, so
that $i_{r+1}-1\leq n$. Combining this with $i_{r}\geq1$, we conclude that
$\left[  i_{r},i_{r+1}-1\right]  \subseteq\left[  n\right]  $. \medskip

We define a set%
\[
K:=\left(  \left(  I\setminus\left\{  i_{k},i_{k+1},\ldots,i_{r}\right\}
\right)  \cup\left\{  i_{k}-1,\ i_{k+1}-1,\ \ldots,\ i_{r}-1\right\}  \right)
\cap\left[  n\right]  .
\]
Thus, $K$ is obtained from $I$ by replacing the elements $i_{k},i_{k+1}%
,\ldots,i_{r}$ by $i_{k}-1,\ i_{k+1}-1,\ \ldots,\ i_{r}-1$ (and intersecting
the resulting set with $\left[  n\right]  $, which has the effect of removing
$0$ if we have replaced $1$ by $0$). Therefore, $K$ is a subset of $\left[
n\right]  $ and satisfies $\operatorname*{sum}K\leq\operatorname*{sum}%
I-\left(  r-k+1\right)  $ (since $i_{k},i_{k+1},\ldots,i_{r}$ are $r-k+1$
distinct elements of $I$, and we subtracted $1$ from each of
them\footnote{Note that the inequality $\operatorname*{sum}K\leq
\operatorname*{sum}I-\left(  r-k+1\right)  $ is not necessarily an equality,
since some of $i_{k}-1,\ i_{k+1}-1,\ \ldots,\ i_{r}-1$ might already belong to
$I\setminus\left\{  i_{k},i_{k+1},\ldots,i_{r}\right\}  $.}). Hence,
$\operatorname*{sum}K\leq\operatorname*{sum}I-\left(  r-k+1\right)
<\operatorname*{sum}I$ (because $r\geq k$). Thus, $F\left(  K\right)
\subseteq F\left(  <\operatorname*{sum}I\right)  $. Hence, in order to prove
that $q^{\prime} \in F\left(  < \operatorname{sum} I\right)  $, it will
suffice to show the more precise statement that
\[
q^{\prime}\in F\left(  K\right)  .
\]
We shall thus focus on proving this.

In order to prove this, it will clearly suffice to show that $q^{\prime}%
s_{i}=q^{\prime}$ for each $i\in K^{\prime}$, because of the definition of
$F\left(  K \right)  $. So let us fix $i\in K^{\prime}$. We must prove that
$q^{\prime}s_{i}=q^{\prime}$. The rest of the proof is dedicated to that goal.
\medskip

We have $i\in K^{\prime}=\left[  n-1\right]  \setminus\left(  K\cup\left(
K-1\right)  \right)  $ (by the definition of $K^{\prime}$, the non-shadow of
$K$). Thus, $i\in\left[  n-1\right]  $ and $i\notin K\cup\left(  K-1\right)
$. From the latter fact, we conclude that $i\notin K$ and $i+1\notin K$. From
$i\in\left[  n-1\right]  $, we obtain $i+1\in\left[  n\right]  $.

It is easy to see that%
\begin{equation}
i+1\notin I \label{pf.thm.tl-FI.i+1notinI}%
\end{equation}
\footnote{\textit{Proof of (\ref{pf.thm.tl-FI.i+1notinI}):} Assume the
contrary. Thus, $i+1\in I=\left\{  i_{1}<i_{2}<\cdots<i_{p}\right\}  $. In
other words, $i+1=i_{s}$ for some $s\in\left[  p\right]  $. Consider this $s$.
From $i+1=i_{s}$, we obtain $i=i_{s}-1$.
\par
If we had $s\in\left[  k,r\right]  $, then we would have%
\begin{align*}
i  &  =i_{s}-1\in\left\{  i_{k}-1,\ i_{k+1}-1,\ \ldots,\ i_{r}-1\right\}
\ \ \ \ \ \ \ \ \ \ \left(  \text{since }s\in\left[  k,r\right]  \right) \\
&  \subseteq\left(  I\setminus\left\{  i_{k},i_{k+1},\ldots,i_{r}\right\}
\right)  \cup\left\{  i_{k}-1,\ i_{k+1}-1,\ \ldots,\ i_{r}-1\right\}
\end{align*}
and therefore%
\[
i\in\left(  \left(  I\setminus\left\{  i_{k},i_{k+1},\ldots,i_{r}\right\}
\right)  \cup\left\{  i_{k}-1,\ i_{k+1}-1,\ \ldots,\ i_{r}-1\right\}  \right)
\cap\left[  n\right]
\]
(since $i\in\left[  n-1\right]  \subseteq\left[  n\right]  $). This would
contradict the fact that%
\[
i\notin K=\left(  \left(  I\setminus\left\{  i_{k},i_{k+1},\ldots
,i_{r}\right\}  \right)  \cup\left\{  i_{k}-1,\ i_{k+1}-1,\ \ldots
,\ i_{r}-1\right\}  \right)  \cap\left[  n\right]  .
\]
Hence, we cannot have $s\in\left[  k,r\right]  $. Thus, we have either $s<k$
or $s>r$. Therefore, we have $i_{s}\notin\left\{  i_{k},i_{k+1},\ldots
,i_{r}\right\}  $ (because of $i_{1}<i_{2}<\cdots<i_{p}$). In other words,
$i+1\notin\left\{  i_{k},i_{k+1},\ldots,i_{r}\right\}  $ (since $i+1=i_{s}$).
Combining $i+1\in I$ with $i+1\notin\left\{  i_{k},i_{k+1},\ldots
,i_{r}\right\}  $, we obtain%
\[
i+1\in I\setminus\left\{  i_{k},i_{k+1},\ldots,i_{r}\right\}  \subseteq\left(
I\setminus\left\{  i_{k},i_{k+1},\ldots,i_{r}\right\}  \right)  \cup\left\{
i_{k}-1,\ i_{k+1}-1,\ \ldots,\ i_{r}-1\right\}
\]
and therefore%
\begin{align*}
i+1  &  \in\left(  \left(  I\setminus\left\{  i_{k},i_{k+1},\ldots
,i_{r}\right\}  \right)  \cup\left\{  i_{k}-1,\ i_{k+1}-1,\ \ldots
,\ i_{r}-1\right\}  \right)  \cap\left[  n\right]  \ \ \ \ \ \ \ \ \ \ \left(
\text{since }i+1\in\left[  n\right]  \right) \\
&  =K.
\end{align*}
This contradicts $i+1\notin K$. This contradiction shows that our assumption
was false, and thus (\ref{pf.thm.tl-FI.i+1notinI}) is proved.}. Thus, it is
also easy to see that%
\begin{equation}
i\in I^{\prime}\text{ if }i\notin\left[  i_{k},i_{r}\right]
\label{pf.thm.tl-FI.K'subI'}%
\end{equation}
\footnote{\textit{Proof of (\ref{pf.thm.tl-FI.K'subI'}):} Assume that
$i\notin\left[  i_{k},i_{r}\right]  $. We must show that $i\in I^{\prime}$.
\par
Indeed, assume the contrary. Thus, $i\notin I^{\prime}=\left[  n-1\right]
\setminus\left(  I\cup\left(  I-1\right)  \right)  $ (by the definition of
$I^{\prime}$). Since $i\in\left[  n-1\right]  $, this entails that $i\in
I\cup\left(  I-1\right)  $. In other words, $i\in I$ or $i+1\in I$. Since
(\ref{pf.thm.tl-FI.i+1notinI}) yields $i+1\notin I$, we thus must have $i\in
I$. Hence, $i\in I\setminus K$ (since $i\in I$ but $i\notin K$).
\par
The definition of $K$ shows that $I\setminus K\subseteq\left\{  i_{k}%
,i_{k+1},\ldots,i_{r}\right\}  $ (although this inclusion is not necessarily
an equality). Therefore, each element of $I\setminus K$ must belong to
$\left\{  i_{k},i_{k+1},\ldots,i_{r}\right\}  $ and therefore to the interval
$\left[  i_{k},i_{r}\right]  $ as well (since $i_{0}<i_{1}<i_{2}<\cdots
<i_{p}<i_{p+1}$ entails $\left\{  i_{k},i_{k+1},\ldots,i_{r}\right\}
\subseteq\left[  i_{k},i_{r}\right]  $). Hence, from $i\in I\setminus K$, we
obtain $i\in\left[  i_{k},i_{r}\right]  $. But this contradicts $i\notin%
\left[  i_{k},i_{r}\right]  $. This contradiction shows that our assumption
was false. Thus, (\ref{pf.thm.tl-FI.K'subI'}) is proved.}.
Similarly, we can show that
\begin{equation}
i+1\in I^{\prime}\text{ if }i\in\left[  \ell,i_{r}-1\right]
\label{pf.thm.tl-FI.K'subI'2}%
\end{equation}
\footnote{\textit{Proof of (\ref{pf.thm.tl-FI.K'subI'2}):} Assume that
$i\in\left[  \ell,i_{r}-1\right]  $. We must show that $i+1\in I^{\prime}$.
\par
Indeed, assume the contrary. Thus, $i+1\notin I^{\prime}=\left[  n-1\right]
\setminus\left(  I\cup\left(  I-1\right)  \right)  $ (by the definition of
$I^{\prime}$).
\par
From $i\in\left[  \ell,i_{r}-1\right]  $, we obtain $i\geq\ell$ and $i\leq
i_{r}-1$. The latter inequality yields $i+1\leq i_{r}$. However,
(\ref{pf.thm.tl-FI.i+1notinI}) yields $i+1\notin I$. Thus, $i+1\neq i_{r}$
(because if we had $i+1=i_{r}$, then $i+1=i_{r}\in I$ would contradict
$i+1\notin I$). Combining this with $i+1\leq i_{r}$, we obtain $i+1<i_{r}\leq
n$. Hence, $i+1\leq n-1$, so that $i+1\in\left[  n-1\right]  $.
\par
Therefore, from $i+1\notin\left[  n-1\right]  \setminus\left(  I\cup\left(
I-1\right)  \right)  $, we obtain $i+1\in I\cup\left(  I-1\right)  $. In other
words, $i+1\in I$ or $i+1\in I-1$. Since $i+1\notin I$, we thus conclude that
$i+1\in I-1$. Thus, $i+2\in I=\left\{  i_{1}<i_{2}<\cdots<i_{p}\right\}  $. In
other words, there exists some $s\in\left[  p\right]  $ such that $i+2=i_{s}$.
Consider this $s$.
\par
From $i+1<i_{r}$, we obtain $i+1\leq i_{r}-1$, so that $i+2\leq i_{r}$.
Combining this with $i+2>i\geq\ell$, we find that $i+2\in\left[  \ell
,i_{r}\right]  $. Thus, $i_{s}=i+2\in\left[  \ell,i_{r}\right]  $. However,
the only numbers of the form $i_{t}$ (with $t\in\left[  0,p+1\right]  $) that
belong to the interval $\left[  \ell,i_{r}\right]  $ are $i_{k},i_{k+1}%
,\ldots,i_{r}$ (because of (\ref{pf.thm.tl-FI.ineqs})). Hence, from $i_{s}%
\in\left[  \ell,i_{r}\right]  $, we obtain $s\in\left[  k,r\right]  $.
Therefore,%
\begin{align*}
i+1  &  =i_{s}-1\ \ \ \ \ \ \ \ \ \ \left(  \text{since }i+2=i_{s}\right) \\
&  \in\left\{  i_{k}-1,\ i_{k+1}-1,\ \ldots,\ i_{r}-1\right\}
\ \ \ \ \ \ \ \ \ \ \left(  \text{since }s\in\left[  k,r\right]  \right) \\
&  \subseteq\left(  I\setminus\left\{  i_{k},i_{k+1},\ldots,i_{r}\right\}
\right)  \cup\left\{  i_{k}-1,\ i_{k+1}-1,\ \ldots,\ i_{r}-1\right\}  .
\end{align*}
Combined with $i+1\in\left[  n\right]  $, this results in%
\[
i+1\in\left(  \left(  I\setminus\left\{  i_{k},i_{k+1},\ldots,i_{r}\right\}
\right)  \cup\left\{  i_{k}-1,\ i_{k+1}-1,\ \ldots,\ i_{r}-1\right\}  \right)
\cap\left[  n\right]  =K.
\]
But this contradicts $i+1\notin K$. This contradiction shows that our
assumption was false. Thus, (\ref{pf.thm.tl-FI.K'subI'2}) is proved.}.
\medskip

From (\ref{pf.thm.tl-FI.ineqs}) and $r\geq k$, we obtain $\ell\leq
i_{r}<i_{r+1}$. Hence, we are in one of the following five cases:

\textit{Case 1:} We have $i<\ell-1$.

\textit{Case 2:} We have $i=\ell-1$.

\textit{Case 3:} We have $\ell\leq i<i_{r}$.

\textit{Case 4:} We have $i_{r}\leq i<i_{r+1}$.

\textit{Case 5:} We have $i\geq i_{r+1}$. \medskip

For each of these cases, we need to prove that $q^{\prime}s_{i}=q^{\prime}$.

Let us first consider Case 1. In this case, we have $i<\ell-1$. Thus,
$i<\ell-1<\ell\leq i_{k}$, so that $i\notin\left[  i_{k},i_{r}\right]  $.
Hence, from (\ref{pf.thm.tl-FI.K'subI'}), we obtain $i\in I^{\prime}$. Thus,
$qs_{i}=q$ (since $q\in F\left(  I\right)  $). Furthermore, from $i<\ell-1$,
we see that $s_{i}$ commutes with all the permutations $s_{\ell},s_{\ell
+1},\ldots,s_{i_{r+1}-2}$ that appear on the right hand side of
(\ref{pf.thm.tl-FI.q'=}) (by reflection locality). Hence, multiplying the
equality (\ref{pf.thm.tl-FI.q'=}) by $s_{i}$, we find%
\begin{align*}
q^{\prime}s_{i}  &  =\sum_{j=i_{r}}^{i_{r+1}-1}q\underbrace{s_{\ell}s_{\ell
+1}\cdots s_{j-1}\cdot s_{i}}_{\substack{=s_{i}\cdot s_{\ell}s_{\ell+1}\cdots
s_{j-1}\\\text{(since }s_{i}\text{ commutes with all of }s_{\ell},s_{\ell
+1},\ldots,s_{j-1}\text{)}}}=\sum_{j=i_{r}}^{i_{r+1}-1}\underbrace{qs_{i}%
}_{=q}\cdot s_{\ell}s_{\ell+1}\cdots s_{j-1}\\
&  =\sum_{j=i_{r}}^{i_{r+1}-1}qs_{\ell}s_{\ell+1}\cdots s_{j-1}=q^{\prime}.
\end{align*}
We have thus proved $q^{\prime}s_{i}=q^{\prime}$ in Case 1. \medskip

Let us next consider Case 2. In this case, we have $i=\ell-1$. Thus,
$i=\ell-1<\ell\leq i_{k}$, so that $i\notin\left[  i_{k},i_{r}\right]  $.
Hence, from (\ref{pf.thm.tl-FI.K'subI'}), we obtain $i\in I^{\prime}$. Thus,
$qs_{i}=q$ (since $q\in F\left(  I\right)  $). We must prove that $q^{\prime
}s_{i}=q^{\prime}$. This easily follows in the case when $\ell=n$%
\ \ \ \ \footnote{\textit{Proof.} Assume that $\ell=n$. Then, it is easy to
see that the sum on the right hand side of (\ref{pf.thm.tl-FI.q'=}) simplifies
to $q$ (since none of the $s_{\ell},s_{\ell+1},\ldots,s_{n-1}$ factors
actually exist). Hence, (\ref{pf.thm.tl-FI.q'=}) rewrites as $q^{\prime}=q$.
Thus, $q^{\prime}s_{i}=q^{\prime}$ follows from $qs_{i}=q$, qed.}. Hence, for
the rest of Case 2, we WLOG assume that $\ell\neq n$. Therefore, $\ell
\in\left[  n-1\right]  $. Moreover, $\ell=i+1$ (since $i=\ell-1$). Now, it is
easy to see that $\ell\in I^{\prime}$\ \ \ \ \footnote{\textit{Proof.} Assume
the contrary. Thus, $\ell\notin I^{\prime}=\left[  n-1\right]  \setminus
\left(  I\cup\left(  I-1\right)  \right)  $ (by the definition of $I^{\prime}%
$). Hence, $\ell\in I\cup\left(  I-1\right)  $ (since $\ell\in\left[
n-1\right]  $). In other words, $\ell\in I$ or $\ell+1\in I$. Since
$\ell-1=i\in I^{\prime}=\left[  n-1\right]  \setminus\left(  I\cup\left(
I-1\right)  \right)  $, we have $\ell-1\notin I\cup\left(  I-1\right)  $, so
that $\ell-1\notin I$ and $\ell\notin I$. In particular, $\ell\notin I$.
Hence, $\ell+1\in I$ (since we just showed that $\ell\in I$ or $\ell+1\in I$).
Combining $\ell\notin I$ and $\ell+1\in I$, we obtain $i_{k}=\ell+1$ (by the
definition of $i_{k}$). In other words, $i_{k}-1=\ell$. However, $i_{k}-1\in
K$ (by the definition of $K$). In other words, $\ell\in K$ (since
$i_{k}-1=\ell$). But this contradicts $\ell=i+1\notin K$. This contradiction
shows that our assumption was false, qed.}. Hence, $qs_{\ell}=q$ (since $q\in
F\left(  I\right)  $). From $\ell\in I^{\prime}=\left[  n-1\right]
\setminus\left(  I\cup\left(  I-1\right)  \right)  $, we furthermore obtain
$\ell\notin I\cup\left(  I-1\right)  $, so that $\ell\notin I$ and thus
$\ell\neq i_{r}$ (because $i_{r}\in I$). Hence, $\ell<i_{r}$ (since $\ell\leq
i_{k}\leq i_{r}$). Now, (\ref{pf.thm.tl-FI.q'=}) rewrites as%
\begin{align}
q^{\prime}  &  =\sum_{j=i_{r}}^{i_{r+1}-1}\underbrace{qs_{\ell}s_{\ell
+1}\cdots s_{j-1}}_{\substack{=\left(  qs_{\ell}\right)  \cdot s_{\ell
+1}s_{\ell+2}\cdots s_{j-1}\\\text{(since }\ell<i_{r}\leq j\text{)}}%
}=\sum_{j=i_{r}}^{i_{r+1}-1}\underbrace{\left(  qs_{\ell}\right)  }_{=q}\cdot
s_{\ell+1}s_{\ell+2}\cdots s_{j-1}\nonumber\\
&  =\sum_{j=i_{r}}^{i_{r+1}-1}qs_{\ell+1}s_{\ell+2}\cdots s_{j-1}.
\label{pf.thm.tl-FI.C2.q'=}%
\end{align}
From $i=\ell-1<\ell$, we see that $s_{i}$ commutes with all the permutations
$s_{\ell+1},s_{\ell+2},\ldots,s_{i_{r+1}-2}$ that appear on the right hand
side of (\ref{pf.thm.tl-FI.C2.q'=}) (by reflection locality). Hence,
multiplying the equality (\ref{pf.thm.tl-FI.C2.q'=}) by $s_{i}$, we find%
\begin{align*}
q^{\prime}s_{i}  &  =\sum_{j=i_{r}}^{i_{r+1}-1}q\underbrace{s_{\ell+1}%
s_{\ell+2}\cdots s_{j-1}\cdot s_{i}}_{\substack{=s_{i}\cdot s_{\ell+1}%
s_{\ell+2}\cdots s_{j-1}\\\text{(since }s_{i}\text{ commutes with all of
}s_{\ell+1},s_{\ell+2},\ldots,s_{j-1}\text{)}}}=\sum_{j=i_{r}}^{i_{r+1}%
-1}\underbrace{qs_{i}}_{=q}\cdot s_{\ell+1}s_{\ell+2}\cdots s_{j-1}\\
&  =\sum_{j=i_{r}}^{i_{r+1}-1}qs_{\ell+1}s_{\ell+2}\cdots s_{j-1}=q^{\prime
}\ \ \ \ \ \ \ \ \ \ \left(  \text{by (\ref{pf.thm.tl-FI.C2.q'=})}\right)  .
\end{align*}
We have thus proved $q^{\prime}s_{i}=q^{\prime}$ in Case 2. \medskip

Let us now consider Case 3. In this case, we have $\ell\leq i<i_{r}$. It is
easy to see that $i<i_{r}-1$\ \ \ \ \footnote{\textit{Proof.} The construction
of $K$ yields $i_{r}-1\in K$ (unless $i_{r}-1=0$). Hence, we cannot have
$i=i_{r}-1$ (since this would imply $i=i_{r}-1\in K$, which would contradict
$i\notin K$). However, from $i<i_{r}$, we obtain $i\leq i_{r}-1$. Thus,
$i<i_{r}-1$ (since we cannot have $i=i_{r}-1$).}. Hence, $i+1<i_{r}\leq n$, so
that $i+1\in\left[  n-1\right]  $. Also, $i\in\left[  \ell,i_{r}-1\right]  $
(since $\ell\leq i<i_{r}$). Thus, (\ref{pf.thm.tl-FI.K'subI'2}) yields $i+1\in
I^{\prime}$. Hence, $qs_{i+1}=q$ (since $q\in F\left(  I\right)  $).

Let $j\in\left[  i_{r},i_{r+1}-1\right]  $. Then, $i_{r}\leq j\leq i_{r+1}-1$,
so that $i<\underbrace{i_{r}}_{\leq j}-1\leq j-1$. Hence, $i\in\left[
\ell,j-2\right]  $ (since $\ell\leq i$). Also, $j\in\left[  i_{r}%
,i_{r+1}-1\right]  \subseteq\left[  n\right]  $. Therefore,%
\begin{align}
q\underbrace{s_{\ell}s_{\ell+1}\cdots s_{j-1}\cdot s_{i}}_{\substack{=s_{i+1}%
\cdot s_{\ell}s_{\ell+1}\cdots s_{j-1}\\\text{(by Lemma \ref{lem.si-into-cyc}%
)}}}  &  =\underbrace{qs_{i+1}}_{=q}\cdot s_{\ell}s_{\ell+1}\cdots
s_{j-1}\nonumber\\
&  =qs_{\ell}s_{\ell+1}\cdots s_{j-1}. \label{pf.thm.tl-FI.C3.one-term}%
\end{align}

Forget that we fixed $j$. We thus have proved (\ref{pf.thm.tl-FI.C3.one-term})
for each $j\in\left[  i_{r},i_{r+1}-1\right]  $. Now, multiplying the equality
(\ref{pf.thm.tl-FI.q'=}) by $s_{i}$, we find%
\[
q^{\prime}s_{i}=\sum_{j=i_{r}}^{i_{r+1}-1}\underbrace{qs_{\ell}s_{\ell
+1}\cdots s_{j-1}\cdot s_{i}}_{\substack{=qs_{\ell}s_{\ell+1}\cdots
s_{j-1}\\\text{(by (\ref{pf.thm.tl-FI.C3.one-term}))}}}=\sum_{j=i_{r}%
}^{i_{r+1}-1}qs_{\ell}s_{\ell+1}\cdots s_{j-1}=q^{\prime}.
\]
We have thus proved $q^{\prime}s_{i}=q^{\prime}$ in Case 3. \medskip

Next, let us consider Case 4. In this case, we have $i_{r}\leq i<i_{r+1}$. It
is easy to see that the latter inequality can be strengthened to $i<i_{r+1}%
-1$\ \ \ \ \footnote{\textit{Proof.} We have $i\in\left[  n-1\right]  $ and
thus $i<n$. If $r+1=p+1$, then $i_{r+1}=i_{p+1}=n+1$ and thus $i_{r+1}-1=n$,
whence $i<n=i_{r+1}-1$. Thus, for the rest of this proof, we WLOG assume that
we don't have $r+1=p+1$. Hence, $r+1\in\left[  p\right]  $. Thus, $i_{r+1}%
\in\left\{  i_{1}<i_{2}<\cdots<i_{p}\right\}  =I$. If we had $i+1=i_{r+1}$,
then we would thus have $i+1=i_{r+1}\in I$, which would contradict
(\ref{pf.thm.tl-FI.i+1notinI}). Hence, we cannot have $i+1=i_{r+1}$. Thus, we
have $i+1\neq i_{r+1}$, so that $i\neq i_{r+1}-1$. However, $i\leq i_{r+1}-1$
(since $i<i_{r+1}$). Combining these two facts, we obtain $i<i_{r+1}-1$.}. In
other words, $i+1\leq i_{r+1}-1$. Thus, both $i$ and $i+1$ belong to the
interval $\left[  i_{r},i_{r+1}-1\right]  $ (since $i_{r}\leq i<i+1$).

Now, we make the following three claims:

\begin{itemize}
\item \textit{Claim 1:} For any $j\in\left[  i_{r},i_{r+1}-1\right]
\setminus\left\{  i,i+1\right\}  $, we have%
\[
qs_{\ell}s_{\ell+1}\cdots s_{j-1}\cdot s_{i}=qs_{\ell}s_{\ell+1}\cdots
s_{j-1}.
\]

\item \textit{Claim 2:} We have%
\[
qs_{\ell}s_{\ell+1}\cdots s_{i-1}\cdot s_{i}=qs_{\ell}s_{\ell+1}\cdots s_{i}.
\]

\item \textit{Claim 3:} We have
\[
qs_{\ell}s_{\ell+1}\cdots s_{i}\cdot s_{i}=qs_{\ell}s_{\ell+1}\cdots s_{i-1}.
\]

\end{itemize}

Note that Claim 2 is trivial, while Claim 3 follows from $s_{i}^{2}%
=\operatorname*{id}$. Let us now prove Claim 1:

[\textit{Proof of Claim 1:} Fix some $j\in\left[  i_{r},i_{r+1}-1\right]
\setminus\left\{  i,i+1\right\}  $. Thus, $j\in\left[  i_{r},i_{r+1}-1\right]
$ and $j\notin\left\{  i,i+1\right\}  $. The latter fact reveals that either
$j<i$ or $j>i+1$. This means that we are in one of two subcases, which we
consider separately:

\begin{itemize}
\item Let us first consider the subcase when $j<i$. In this subcase, $s_{i}$
commutes with each of $s_{\ell},s_{\ell+1},\ldots,s_{j-1}$ (by reflection
locality). Thus, $s_{\ell}s_{\ell+1}\cdots s_{j-1}\cdot s_{i}=s_{i}\cdot
s_{\ell}s_{\ell+1}\cdots s_{j-1}$. Also, $j<i$ entails $i>j\geq i_{r}$ (since
$j\in\left[  i_{r},i_{r+1}-1\right]  $). Hence, $i\notin\left[  i_{k}%
,i_{r}\right]  $. Therefore, (\ref{pf.thm.tl-FI.K'subI'}) yields $i\in
I^{\prime}$. Thus, $qs_{i}=q$ (since $q\in F\left(  I\right)  $). Now,%
\[
q\underbrace{s_{\ell}s_{\ell+1}\cdots s_{j-1}\cdot s_{i}}_{=s_{i}\cdot
s_{\ell}s_{\ell+1}\cdots s_{j-1}}=\underbrace{qs_{i}}_{=q}\cdot s_{\ell
}s_{\ell+1}\cdots s_{j-1}=qs_{\ell}s_{\ell+1}\cdots s_{j-1}.
\]
We have thus proved Claim 1 in the subcase when $j<i$.

\item Let us now consider the subcase when $j>i+1$. In this subcase, we have
$i<j-1$ and thus $i\leq j-2$. Combining this with $\ell\leq i_{r}\leq i$, we
obtain $i\in\left[  \ell,j-2\right]  $. Hence, Lemma \ref{lem.si-into-cyc}
yields $s_{\ell}s_{\ell+1}\cdots s_{j-1}\cdot s_{i}=s_{i+1}\cdot s_{\ell
}s_{\ell+1}\cdots s_{j-1}$ (since $j\in\left[  i_{r},i_{r+1}-1\right]
\subseteq\left[  n\right]  $). Moreover, from $j\in\left[  i_{r}%
,i_{r+1}-1\right]  \subseteq\left[  n\right]  $, we obtain $j\leq n$, so that
$n\geq j>i+1$. Hence, $i+1<n$, so that $i+1\in\left[  n-1\right]  $.

Furthermore, $i_{r}\leq i<i+1$. On the other hand, from $j>i+1$, we obtain
$i+1<j\leq i_{r+1}-1$ (since $j\in\left[  i_{r},i_{r+1}-1\right]  $), so that
$i+2<i_{r+1}$. Hence, $i_{r}<i+1<i+2<i_{r+1}$. This chain of inequalities
shows that both numbers $i+1$ and $i+2$ lie strictly between the two numbers
$i_{r}$ and $i_{r+1}$, which are two adjacent elements of the enclosure
$\widehat{I}$ (in the sense that there are no further elements of
$\widehat{I}$ between them). Hence, neither $i+1$ nor $i+2$ can belong to
$\widehat{I}$. Thus, neither $i+1$ nor $i+2$ can belong to $I$ (since
$I\subseteq\widehat{I}$). In other words, $i+1\notin I\cup\left(  I-1\right)
$. Since $i+1\in\left[  n-1\right]  $, we thus obtain $i+1\in\left[
n-1\right]  \setminus\left(  I\cup\left(  I-1\right)  \right)  =I^{\prime}$
(by the definition of $I^{\prime}$). Thus, $qs_{i+1}=q$ (since $q\in F\left(
I\right)  $). Now,%
\[
q\underbrace{s_{\ell}s_{\ell+1}\cdots s_{j-1}\cdot s_{i}}_{=s_{i+1}\cdot
s_{\ell}s_{\ell+1}\cdots s_{j-1}}=\underbrace{qs_{i+1}}_{=q}\cdot s_{\ell
}s_{\ell+1}\cdots s_{j-1}=qs_{\ell}s_{\ell+1}\cdots s_{j-1}.
\]
We have thus proved Claim 1 in the subcase when $j>i+1$.
\end{itemize}

We have now covered both possible subcases. Hence, Claim 1 is proved.]

We have now proved all three Claims 1, 2 and 3. Now, consider the sum
$\sum_{j=i_{r}}^{i_{r+1}-1}qs_{\ell}s_{\ell+1}\cdots s_{j-1}$. This sum
contains both an addend for $j=i$ and an addend for $j=i+1$ (since both $i$
and $i+1$ belong to the interval $\left[  i_{r},i_{r+1}-1\right]  $). When we
multiply this sum by $s_{i}$ on the right (i.e., when we replace it by
$\sum_{j=i_{r}}^{i_{r+1}-1}qs_{\ell}s_{\ell+1}\cdots s_{j-1}\cdot s_{i}$), the
addend for $j=i$ becomes $qs_{\ell}s_{\ell+1}\cdots s_{i-1}\cdot
s_{i}=qs_{\ell}s_{\ell+1}\cdots s_{i}$ (by Claim 2), whereas the addend for
$j=i+1$ becomes $qs_{\ell}s_{\ell+1}\cdots s_{i}\cdot s_{i}=qs_{\ell}%
s_{\ell+1}\cdots s_{i-1}$ (by Claim 3), and all remaining addends stay
unchanged (by Claim 1). Hence, multiplying the sum $\sum_{j=i_{r}}^{i_{r+1}%
-1}qs_{\ell}s_{\ell+1}\cdots s_{j-1}$ by $s_{i}$ on the right merely permutes
its addends (specifically, the addend for $j=i$ is swapped with the addend for
$j=i+1$, while all other addends stay unchanged) and therefore does not change
the sum. In other words, we have%
\[
\sum_{j=i_{r}}^{i_{r+1}-1}qs_{\ell}s_{\ell+1}\cdots s_{j-1}\cdot s_{i}%
=\sum_{j=i_{r}}^{i_{r+1}-1}qs_{\ell}s_{\ell+1}\cdots s_{j-1}.
\]
Since $q^{\prime}=\sum_{j=i_{r}}^{i_{r+1}-1}qs_{\ell}s_{\ell+1}\cdots s_{j-1}%
$, this rewrites as $q^{\prime}s_{i}=q^{\prime}$. Thus, we have proved
$q^{\prime}s_{i}=q^{\prime}$ in Case 4. \medskip

Finally, let us consider Case 5. In this case, we have $i\geq i_{r+1}$. Thus,
$i\geq i_{r+1}>i_{r}$ (since $i_{0}<i_{1}<i_{2}<\cdots<i_{p}<i_{p+1}$), so
that $i\notin\left[  i_{k},i_{r}\right]  $. Hence, from
(\ref{pf.thm.tl-FI.K'subI'}), we obtain $i\in I^{\prime}$. Thus, $qs_{i}=q$
(since $q\in F\left(  I\right)  $). Furthermore, from $i\geq i_{r+1}$, we see
that $s_{i}$ commutes with all the permutations $s_{\ell},s_{\ell+1}%
,\ldots,s_{i_{r+1}-2}$ that appear on the right hand side of
(\ref{pf.thm.tl-FI.q'=}) (by reflection locality). Hence, multiplying the
equality (\ref{pf.thm.tl-FI.q'=}) by $s_{i}$, we find%
\begin{align*}
q^{\prime}s_{i}  &  =\sum_{j=i_{r}}^{i_{r+1}-1}q\underbrace{s_{\ell}s_{\ell
+1}\cdots s_{j-1}\cdot s_{i}}_{\substack{=s_{i}\cdot s_{\ell}s_{\ell+1}\cdots
s_{j-1}\\\text{(since }s_{i}\text{ commutes with all of }s_{\ell},s_{\ell
+1},\ldots,s_{j-1}\text{)}}}=\sum_{j=i_{r}}^{i_{r+1}-1}\underbrace{qs_{i}%
}_{=q}\cdot s_{\ell}s_{\ell+1}\cdots s_{j-1}\\
&  =\sum_{j=i_{r}}^{i_{r+1}-1}qs_{\ell}s_{\ell+1}\cdots s_{j-1}=q^{\prime}.
\end{align*}
We have thus proved $q^{\prime}s_{i}=q^{\prime}$ in Case 5. \medskip

We have now proved $q^{\prime}s_{i}=q^{\prime}$ in all five cases. Thus,
$q^{\prime}s_{i}=q^{\prime}$ always holds. As explained above, this completes
the proof of $q^{\prime}\in F\left(  K\right)  $. Therefore, $q^{\prime}\in
F\left(  K\right)  \subseteq F\left(  <\operatorname*{sum}I\right)  $. But
this is precisely what we needed to prove. Thus, Theorem \ref{thm.tl-FI} is proven.
\end{proof}

\section{The Fibonacci filtration}

\label{sec.Filtration}

In this section, we shall build a filtration of $\mathbf{k}\left[
S_{n}\right]  $ by $\mathbf{k}$-submodules that are invariant under the
somewhere-to-below shuffles $R\left(  t_{\ell}\right)  $, which furthermore
has the property that the latter shuffles act as scalars on the subquotients
of the filtration. This filtration will be built up from the submodules
$F\left(  I\right)  $ defined in the previous section, and its properties will
rely on Theorem \ref{thm.tl-FI}.

\subsection{Definition and examples}

Recall from Section \ref{sec.Lacunarity} that the number of lacunar subsets of
$\left[  n-1\right]  $ is $f_{n+1}$. Let $Q_{1},Q_{2},\ldots,Q_{f_{n+1}}$ be
all these $f_{n+1}$ lacunar subsets of $\left[  n-1\right]  $, listed in an
order that satisfies%
\begin{equation}
\operatorname*{sum}\left(  Q_{1}\right)  \leq\operatorname*{sum}\left(
Q_{2}\right)  \leq\cdots\leq\operatorname*{sum}\left(  Q_{f_{n+1}}\right)  .
\label{pf.thm.t-simultri.sum-order}%
\end{equation}
Then, define a $\mathbf{k}$-submodule%
\[
F_{i}:=F\left(  Q_{1}\right)  +F\left(  Q_{2}\right)  +\cdots+F\left(
Q_{i}\right)  \ \ \ \ \ \ \ \ \ \ \text{of }\mathbf{k}\left[  S_{n}\right]
\]
for each $i\in\left[  0,f_{n+1}\right]  $ (so that $F_{0}=0$). We claim the following:

\begin{theorem}
\label{thm.t-simultri}\ 

\begin{enumerate}
\item[\textbf{(a)}] We have%
\[
0=F_{0}\subseteq F_{1}\subseteq F_{2}\subseteq\cdots\subseteq F_{f_{n+1}%
}=\mathbf{k}\left[  S_{n}\right]  .
\]
In other words, the $\mathbf{k}$-submodules $F_{0},F_{1},\ldots,F_{f_{n+1}}$
form a $\mathbf{k}$-module filtration of $\mathbf{k}\left[  S_{n}\right]  $.

\item[\textbf{(b)}] We have $F_{i}\cdot t_{\ell}\subseteq F_{i}$ for each
$i\in\left[  0,f_{n+1}\right]  $ and $\ell\in\left[  n\right]  $.

\item[\textbf{(c)}] For each $i\in\left[  f_{n+1}\right]  $ and $\ell
\in\left[  n\right]  $, we have%
\[
F_{i}\cdot\left(  t_{\ell}-m_{Q_{i},\ell}\right)  \subseteq F_{i-1}.
\]

\end{enumerate}
\end{theorem}

We will eventually prove this theorem; we will also show that each $F_{i}$ is
a free $\mathbf{k}$-module, so that its dimension $\dim F_{i}$ (also known as
its rank) is well-defined whenever $\mathbf{k}\neq0$. First, let us tabulate
the dimensions of the $F_{0},F_{1},\ldots,F_{f_{n+1}}$ for some small values
of $n$:

\begin{example}
Let $n=3$. Then, the lacunar subsets of $\left[  n-1\right]  $ are
$Q_{1}=\varnothing$ and $Q_{2}=\left\{  1\right\}  $ and $Q_{3}=\left\{
2\right\}  $ (this is the only possible ordering that satisfies
(\ref{pf.thm.t-simultri.sum-order}), because no two lacunar subsets of
$\left[  n-1\right]  $ have the same sum). The corresponding $F\left(
I\right)  $'s have already been computed in Example \ref{exa.F(I).n=3}. Here
are some properties of the corresponding $F_{i}$'s:%
\[%
\begin{tabular}
[c]{|c||c|c|c|}\hline
$i$ & $1$ & $2$ & $3$\\\hline\hline
$Q_{i}$ & $\varnothing$ & $\left\{  1\right\}  $ & $\left\{  2\right\}
$\\\hline
$Q_{i}^{\prime}$ & $\left\{  1,2\right\}  $ & $\left\{  2\right\}  $ &
$\varnothing$\\\hline
$\dim F_{i}$ & $1$ & $3$ & $6$\\\hline
$\dim F_{i}-\dim F_{i-1}$ & $1$ & $2$ & $3$\\\hline
\end{tabular}
\ .
\]
Of course, $F_{0}=0$, so we are not showing an $i=0$ column.
\end{example}

\begin{example}
\label{exa.Qi.4}Let $n=4$. Then, the lacunar subsets of $\left[  n-1\right]  $
are $Q_{1}=\varnothing$ and $Q_{2}=\left\{  1\right\}  $ and $Q_{3}=\left\{
2\right\}  $ and $Q_{4}=\left\{  3\right\}  $ and $Q_{5}=\left\{  1,3\right\}
$ (again, there is no other ordering). Here are some properties of the
corresponding $F_{i}$'s:%
\[%
\begin{tabular}
[c]{|c||c|c|c|c|c|}\hline
$i$ & $1$ & $2$ & $3$ & $4$ & $5$\\\hline\hline
$Q_{i}$ & $\varnothing$ & $\left\{  1\right\}  $ & $\left\{  2\right\}  $ &
$\left\{  3\right\}  $ & $\left\{  1,3\right\}  $\\\hline
$Q_{i}^{\prime}$ & $\left\{  1,2,3\right\}  $ & $\left\{  2,3\right\}  $ &
$\left\{  3\right\}  $ & $\left\{  1\right\}  $ & $\varnothing$\\\hline
$\dim F_{i}$ & $1$ & $4$ & $12$ & $18$ & $24$\\\hline
$\dim F_{i}-\dim F_{i-1}$ & $1$ & $3$ & $8$ & $6$ & $6$\\\hline
\end{tabular}
\ \ \
\]

\end{example}

\begin{example}
Let $n=5$. Then, the lacunar subsets of $\left[  n-1\right]  $ are
$Q_{1}=\varnothing$ and $Q_{2}=\left\{  1\right\}  $ and $Q_{3}=\left\{
2\right\}  $ and $Q_{4}=\left\{  3\right\}  $ and $Q_{5}=\left\{  4\right\}  $
and $Q_{6}=\left\{  1,3\right\}  $ and $Q_{7}=\left\{  1,4\right\}  $ and
$Q_{8}=\left\{  2,4\right\}  $ (this is one of two possible orderings; another
can be obtained by swapping $Q_{5}$ with $Q_{6}$). Here are some properties of
the corresponding $F_{i}$'s:%
\[%
\begin{tabular}
[c]{|c||c|c|c|c|c|c|c|c|}\hline
$i$ & $1$ & $2$ & $3$ & $4$ & $5$ & $6$ & $7$ & $8$\\\hline\hline
$Q_{i}$ & $\varnothing$ & $\left\{  1\right\}  $ & $\left\{  2\right\}  $ &
$\left\{  3\right\}  $ & $\left\{  4\right\}  $ & $\left\{  1,3\right\}  $ &
$\left\{  1,4\right\}  $ & $\left\{  2,4\right\}  $\\\hline
$Q_{i}^{\prime}$ & $\left\{  1,2,3,4\right\}  $ & $\left\{  2,3,4\right\}  $ &
$\left\{  3,4\right\}  $ & $\left\{  1,4\right\}  $ & $\left\{  1,2\right\}  $
& $\left\{  4\right\}  $ & $\left\{  2\right\}  $ & $\varnothing$\\\hline
$\dim F_{i}$ & $1$ & $5$ & $20$ & $40$ & $50$ & $70$ & $90$ & $120$\\\hline
$\dim F_{i}-\dim F_{i-1}$ & $1$ & $4$ & $15$ & $20$ & $10$ & $20$ & $20$ &
$30$\\\hline
\end{tabular}
\ .
\]

\end{example}

\begin{example}
Let $n=6$. Then, the lacunar subsets of $\left[  n-1\right]  $ (in one of
several orderings) can be found in the following table:%
\[%
\begin{tabular}
[c]{|c||c|c|c|c|c|c|c|c|c|c|c|c|c|}\hline
$i$ & $1$ & $2$ & $3$ & $4$ & $5$ & $6$ & $7$ & $8$ & $9$ & $10$ & $11$ & $12$
& $13$\\\hline\hline
$Q_{i}$ & $\varnothing$ & $\left\{  1\right\}  $ & $\left\{  2\right\}  $ &
$\left\{  3\right\}  $ & $\left\{  4\right\}  $ & $\left\{  1,3\right\}  $ &
$\left\{  5\right\}  $ & $\left\{  1,4\right\}  $ & $\left\{  1,5\right\}  $ &
$\left\{  2,4\right\}  $ & $\left\{  2,5\right\}  $ & $\left\{  3,5\right\}  $
& $\left\{  1,3,5\right\}  $\\\hline
$d_{i}$ & $1$ & $6$ & $30$ & $75$ & $115$ & $160$ & $175$ & $255$ & $300$ &
$420$ & $540$ & $630$ & $720$\\\hline
$\delta_{i}$ & $1$ & $5$ & $24$ & $45$ & $40$ & $45$ & $15$ & $80$ & $45$ &
$120$ & $120$ & $90$ & $90$\\\hline
\end{tabular}
\ ,
\]
where we set $d_{i}:=\dim F_{i}$ and $\delta_{i}:=\dim F_{i}-\dim F_{i-1}$ for
brevity. (We have not listed the sets $Q_{i}^{\prime}$ to avoid stretching the
table too much.)
\end{example}

When $\mathbf{k}$ is a field, Theorem \ref{thm.t-simultri} entails that the
endomorphisms $R\left(  t_{1}\right)  ,R\left(  t_{2}\right)  ,\ldots,R\left(
t_{n}\right)  $ on $\mathbf{k}\left[  S_{n}\right]  $ can be simultaneously
triangularized (as endomorphisms of the $\mathbf{k}$-module $\mathbf{k}\left[
S_{n}\right]  $). Thus, in particular, any $\mathbf{k}$-linear combination
$R\left(  \lambda_{1}t_{1}+\lambda_{2}t_{2}+\cdots+\lambda_{n}t_{n}\right)  $
of $R\left(  t_{1}\right)  ,R\left(  t_{2}\right)  ,\ldots,R\left(
t_{n}\right)  $ has all its eigenvalues in $\mathbf{k}$. However, we will
later prove this more generally, without assuming that $\mathbf{k}$ is a
field, by explicitly constructing a basis of $\mathbf{k}\left[  S_{n}\right]
$ that triangularizes $R\left(  t_{1}\right)  ,R\left(  t_{2}\right)
,\ldots,R\left(  t_{n}\right)  $.

\subsection{Properties of non-shadows}

So far, it may seem mysterious that the definition of our filtration $F_{0}
\subseteq F_{1} \subseteq F_{2} \subseteq\cdots\subseteq F_{f_{n+1}}$ relies
only on the $F\left(  I\right)  $ for the lacunar subsets $I$ of $\left[
n-1\right]  $, rather than using the $F\left(  I\right)  $ for all subsets $I$
of $\left[  n\right]  $. The reason for this is the observation
(Corollary~\ref{cor.FI-lac-2} further below) that the lacunar subsets $I$ of
$\left[  n-1\right]  $ are ``enough'' (i.e., the $F\left(  I\right)  $ for
which $I$ is not a lacunar subset of $\left[  n-1\right]  $ ``contribute
nothing new'' to the filtration). More precisely, each $F\left(  I\right)  $
(for any $I \subseteq\left[  n\right]  $) is contained in the sum of the
$F\left(  J\right)  $ where $J \subseteq\left[  n-1\right]  $ is lacunar and
satisfies $\operatorname{sum} J \leq\operatorname{sum} I$.

Before we can prove this, we shall show a few combinatorial properties of non-shadows.

\begin{proposition}
\label{prop.K'subI'}Let $I$ be a subset of $\left[  n\right]  $. Let $j\in I$.
Set $K:=\left(  I\setminus\left\{  j\right\}  \right)  \cup\left\{
j-1\right\}  $ if $j>1$, and otherwise set $K:=I\setminus\left\{  j\right\}
$. Then:

\begin{enumerate}
\item[\textbf{(a)}] We have $K^{\prime}\subseteq I^{\prime}\cup\left\{
j\right\}  $.

\item[\textbf{(b)}] If $j+1\in I$, then $K^{\prime}\subseteq I^{\prime}$.
\end{enumerate}
\end{proposition}

\begin{proof}
\textbf{(a)} Let $g\in K^{\prime}\setminus\left\{  j\right\}  $. We shall show
that $g\in I^{\prime}$.

Indeed, we have $g\in K^{\prime}\setminus\left\{  j\right\}  $. In other
words, $g\in K^{\prime}$ and $g\neq j$. Now, $g\in K^{\prime}=\left[
n-1\right]  \setminus\left(  K\cup\left(  K-1\right)  \right)  $ (by the
definition of $K^{\prime}$). In other words, $g\in\left[  n-1\right]  $ and
$g\notin K\cup\left(  K-1\right)  $. From $g\notin K\cup\left(  K-1\right)  $,
we obtain $g\notin K$ and $g+1\notin K$.

However, the construction of $K$ yields $I\setminus\left\{  j\right\}
\subseteq K$.

If we had $g\in I$, then we would have $g\in I\setminus\left\{  j\right\}  $
(since $g\in I$ and $g\neq j$), which would entail $g\in I\setminus\left\{
j\right\}  \subseteq K$, contradicting $g\notin K$. Hence, we cannot have
$g\in I$. Thus, we have $g\notin I$.

We shall now show that $g+1\notin I$. Indeed, let us assume the contrary.
Then, $g+1\in I$. If we had $g+1\neq j$, then we would have $g+1\in
I\setminus\left\{  j\right\}  $ (since $g+1\in I$ and $g+1\neq j$), which
would entail $g+1\in I\setminus\left\{  j\right\}  \subseteq K$, contradicting
$g+1\notin K$. Hence, we cannot have $g+1\neq j$. Thus, we must have $g+1=j$,
so that $g=j-1$ and thus $j-1=g\in\left[  n-1\right]  $. Hence, $j-1\geq1$, so
that $j\geq2$. Thus, the definition of $K$ yields $K=\left(  I\setminus
\left\{  j\right\}  \right)  \cup\left\{  j-1\right\}  $. Consequently,
$j-1\in K$. But this contradicts $j-1=g\notin K$. This contradiction shows
that our assumption was false. Hence, $g+1\notin I$ is proved.

Now, we know that $g\in\left[  n-1\right]  $ satisfies $g\notin I$ and
$g+1\notin I$. In other words, $g\in I^{\prime}$ (by the definition of
$I^{\prime}$).

Forget that we fixed $g$. We thus have shown that $g\in I^{\prime}$ for each
$g\in K^{\prime}\setminus\left\{  j\right\}  $. In other words, $K^{\prime
}\setminus\left\{  j\right\}  \subseteq I^{\prime}$. Hence,%
\[
K^{\prime}\subseteq\underbrace{\left(  K^{\prime}\setminus\left\{  j\right\}
\right)  }_{\subseteq I^{\prime}}\cup\left\{  j\right\}  \subseteq I^{\prime
}\cup\left\{  j\right\}  .
\]
This proves Proposition \ref{prop.K'subI'} \textbf{(a)}. \medskip

\textbf{(b)} Assume that $j+1\in I$. Thus, $j+1\in I\setminus\left\{
j\right\}  $ (since $j+1\neq j$). However, the definition of $K$ yields
$K\supseteq I\setminus\left\{  j\right\}  $. Thus, $j+1\in I\setminus\left\{
j\right\}  \subseteq K$. Hence, $j\in K-1\subseteq K\cup\left(  K-1\right)  $,
so that $j\notin\left[  n-1\right]  \setminus\left(  K\cup\left(  K-1\right)
\right)  $. In other words, $j\notin K^{\prime}$ (since $K^{\prime}=\left[
n-1\right]  \setminus\left(  K\cup\left(  K-1\right)  \right)  $). Hence,
$K^{\prime}\setminus\left\{  j\right\}  =K^{\prime}$ and therefore%
\[
K^{\prime}=\underbrace{K^{\prime}}_{\substack{\subseteq I^{\prime}\cup\left\{
j\right\}  \\\text{(by Proposition \ref{prop.K'subI'} \textbf{(a)})}%
}}\setminus\left\{  j\right\}  \subseteq\left(  I^{\prime}\cup\left\{
j\right\}  \right)  \setminus\left\{  j\right\}  \subseteq I^{\prime}.
\]
This proves Proposition \ref{prop.K'subI'} \textbf{(b)}.
\end{proof}

\begin{proposition}
\label{prop.FI-lac}Let $I\subseteq\left[  n\right]  $. Assume that $I$ is not
a lacunar subset of $\left[  n-1\right]  $. Then, there exists a subset $K$ of
$\left[  n\right]  $ such that $\operatorname*{sum}K<\operatorname*{sum}I$ and
$K^{\prime}\subseteq I^{\prime}$.
\end{proposition}

\begin{proof}
We have assumed that $I$ is not a lacunar subset of $\left[  n-1\right]  $.
Thus, we are in one of the following two cases:

\textit{Case 1:} The set $I$ is not a subset of $\left[  n-1\right]  $.

\textit{Case 2:} The set $I$ is not lacunar.

Let us first consider Case 1. In this case, the set $I$ is not a subset of
$\left[  n-1\right]  $. Hence, we have $n\in I$ (since $I\subseteq\left[
n\right]  $). Let $K:=\left(  I\setminus\left\{  n\right\}  \right)
\cup\left\{  n-1\right\}  $ (or just $K:=I\setminus\left\{  n\right\}  $ in
the case when $n\leq1$). Then,
\begin{align*}
\operatorname*{sum}K  &  \leq\operatorname*{sum}I-n+\left(  n-1\right) \\
&  \ \ \ \ \ \ \ \ \ \ \ \ \ \ \ \ \ \ \ \ \left(  \text{since }n\in I\text{,
but }n-1\text{ may or may not belong to }I\right) \\
&  =\operatorname*{sum}I-1<\operatorname*{sum}I.
\end{align*}
However, Proposition \ref{prop.K'subI'} \textbf{(a)} (applied to $j=n$) yields
$K^{\prime}\subseteq I^{\prime}\cup\left\{  n\right\}  $ (since $n\in I$).
From this, we easily obtain $K^{\prime}\subseteq I^{\prime}$%
\ \ \ \ \footnote{\textit{Proof:} The definition of $K^{\prime}$ yields
$K^{\prime}=\left[  n-1\right]  \setminus\left(  K\cup\left(  K-1\right)
\right)  \subseteq\left[  n-1\right]  $. Combining this with $K^{\prime
}\subseteq I^{\prime}\cup\left\{  n\right\}  $, we obtain%
\[
K^{\prime}\subseteq\left[  n-1\right]  \cap\left(  I^{\prime}\cup\left\{
n\right\}  \right)  =\underbrace{\left(  \left[  n-1\right]  \cap I^{\prime
}\right)  }_{\subseteq I^{\prime}}\cup\underbrace{\left(  \left[  n-1\right]
\cap\left\{  n\right\}  \right)  }_{=\varnothing}\subseteq I^{\prime}.
\]
}. Hence, Proposition \ref{prop.FI-lac} is proved in Case 1.

Let us now consider Case 2. In this case, the set $I$ is not lacunar. In other
words, $I$ contains two consecutive integers $q-1$ and $q$. Consider these
$q-1$ and $q$. Let $K:=\left(  I\setminus\left\{  q-1\right\}  \right)
\cup\left\{  q-2\right\}  $ (or just $K:=I\setminus\left\{  q-1\right\}  $ in
the case when $q-2=0$). Then, $\operatorname*{sum}K<\operatorname*{sum}I$
(similarly to Case 1). However, Proposition \ref{prop.K'subI'} \textbf{(b)}
(applied to $j=q-1$) yields $K^{\prime}\subseteq I^{\prime}$ (since $q-1\in I$
and $\left(  q-1\right)  +1=q\in I$). Hence, Proposition \ref{prop.FI-lac} is
proved in Case 2.

We now have proved Proposition \ref{prop.FI-lac} in both Cases 1 and 2.
\end{proof}

Roughly speaking, Proposition~\ref{prop.FI-lac} tells us that if a subset $I$
of $\left[  n \right]  $ is not a lacunar subset of $\left[  n-1\right]  $,
then we can replace it by a subset $K$ that has a smaller sum (i.e., satisfies
$\operatorname*{sum}K<\operatorname*{sum}I$) and a non-shadow that is
contained in that of $I$. The latter subset $K$ may or may not be a lacunar
subset of $\left[  n-1 \right]  $. If it is not, then we can apply
Proposition~\ref{prop.FI-lac} to it again. Repeatedly applying
Proposition~\ref{prop.FI-lac} like this, we obtain the following corollary:

\begin{corollary}
\label{cor.FI-lac-2}Let $I\subseteq\left[  n\right]  $. Then, there exists a
lacunar subset $J$ of $\left[  n-1\right]  $ such that $\operatorname*{sum}%
J\leq\operatorname*{sum}I$ and $J^{\prime}\subseteq I^{\prime}$.
\end{corollary}

\begin{proof}
We proceed by strong induction on $\operatorname*{sum}I$. Thus, we fix some
$I\subseteq\left[  n\right]  $. We must prove that there exists a lacunar
subset $J$ of $\left[  n-1\right]  $ satisfying $\operatorname*{sum}%
J\leq\operatorname*{sum}I$ and $J^{\prime}\subseteq I^{\prime}$.

If $I$ itself is a lacunar subset of $\left[  n-1\right]  $, then taking $J=I$
suffices. Thus, assume that $I$ is not. Hence, Proposition \ref{prop.FI-lac}
yields that there exists a subset $K$ of $\left[  n\right]  $ such that
$\operatorname*{sum}K<\operatorname*{sum}I$ and $K^{\prime}\subseteq
I^{\prime}$. Consider this $K$. Because of $\operatorname*{sum}%
K<\operatorname*{sum}I$, we can apply the induction hypothesis to $K$ instead
of $I$. We thus conclude that there exists a lacunar subset $J$ of $\left[
n-1\right]  $ such that $\operatorname*{sum}J\leq\operatorname*{sum}K$ and
$J^{\prime}\subseteq K^{\prime}$. This lacunar subset $J$ satisfies
$\operatorname*{sum}J\leq\operatorname*{sum}I$ (since $\operatorname*{sum}%
J\leq\operatorname*{sum}K<\operatorname*{sum}I$) and $J^{\prime}\subseteq
I^{\prime}$ (since $J^{\prime}\subseteq K^{\prime}\subseteq I^{\prime}$).
Hence, it is precisely the kind of subset that we were looking for. This
completes the induction step, and therefore Corollary \ref{cor.FI-lac-2} is proved.
\end{proof}

Corollary~\ref{cor.FI-lac-2} is largely responsible for the fact that the
filtration in Theorem~\ref{thm.t-simultri} uses only the lacunar subsets of
$\left[  n-1 \right]  $ (rather than all subsets of $\left[  n \right]  $).

Next, we observe an essentially obvious fact: If $A$ and $B$ are two subsets
of $\left[  n\right]  $ satisfying $B^{\prime}\subseteq A^{\prime}$, then%
\begin{equation}
F\left(  A\right)  \subseteq F\left(  B\right)  . \label{eq.FI-lac.AB}%
\end{equation}
(This follows directly from the definition of $F\left(  I\right)  $ in terms
of $I^{\prime}$, given at the beginning of Section \ref{sec.invariantspaces}.)

\begin{corollary}
\label{cor.FI-lac-cor}Let $k\in\mathbb{N}$. Then,%
\[
F\left(  <k\right)  =\sum_{\substack{J\subseteq\left[  n-1\right]  \text{ is
lacunar;}\\\operatorname*{sum}J<k}}F\left(  J\right)  .
\]

\end{corollary}

\begin{proof}
The definition of $F\left(  <k\right)  $ yields%
\[
F\left(  <k\right)  =\sum_{\substack{J\subseteq\left[  n\right]
;\\\operatorname*{sum}J<k}}F\left(  J\right)  =\sum_{\substack{I\subseteq
\left[  n\right]  ;\\\operatorname*{sum}I<k}}F\left(  I\right)  .
\]
Now, we shall show the following claim:

\begin{statement}
\textit{Claim 1:} For each $I\subseteq\left[  n\right]  $ satisfying
$\operatorname*{sum}I<k$, there exists some lacunar $J\subseteq\left[
n-1\right]  $ satisfying $\operatorname*{sum}J<k$ and $F\left(  I\right)
\subseteq F\left(  J\right)  $.
\end{statement}

[\textit{Proof of Claim 1:} Let $I\subseteq\left[  n\right]  $ satisfy
$\operatorname*{sum}I<k$. Then, Corollary \ref{cor.FI-lac-2} yields that there
exists a lacunar subset $J$ of $\left[  n-1\right]  $ such that
$\operatorname*{sum}J\leq\operatorname*{sum}I$ and $J^{\prime}\subseteq
I^{\prime}$. This lacunar subset $J$ then clearly satisfies
$\operatorname*{sum}J\leq\operatorname*{sum}I<k$ and $F\left(  I\right)
\subseteq F\left(  J\right)  $ (by (\ref{eq.FI-lac.AB}), applied to $A=I$ and
$B=J$). Thus, Claim 1 follows.] \medskip

Claim 1 shows that each addend of the sum $\sum_{\substack{I\subseteq\left[
n\right]  ;\\\operatorname*{sum}I<k}}F\left(  I\right)  $ is a subset of some
addend of the sum $\sum_{\substack{J\subseteq\left[  n-1\right]  \text{ is
lacunar;}\\\operatorname*{sum}J<k}}F\left(  J\right)  $. Hence, we have%
\[
\sum_{\substack{I\subseteq\left[  n\right]  ;\\\operatorname*{sum}%
I<k}}F\left(  I\right)  \subseteq\sum_{\substack{J\subseteq\left[  n-1\right]
\text{ is lacunar;}\\\operatorname*{sum}J<k}}F\left(  J\right)  .
\]
Combining this inclusion with the reverse inclusion%
\[
\sum_{\substack{J\subseteq\left[  n-1\right]  \text{ is lacunar;}%
\\\operatorname*{sum}J<k}}F\left(  J\right)  \subseteq\sum
_{\substack{I\subseteq\left[  n\right]  ;\\\operatorname*{sum}I<k}}F\left(
I\right)
\]
(which is obvious, since the left hand side is a sub-sum of the right hand
side), we obtain%
\[
\sum_{\substack{I\subseteq\left[  n\right]  ;\\\operatorname*{sum}%
I<k}}F\left(  I\right)  =\sum_{\substack{J\subseteq\left[  n-1\right]  \text{
is lacunar;}\\\operatorname*{sum}J<k}}F\left(  J\right)  .
\]
Thus,%
\[
F\left(  <k\right)  =\sum_{\substack{I\subseteq\left[  n\right]
;\\\operatorname*{sum}I<k}}F\left(  I\right)  =\sum_{\substack{J\subseteq
\left[  n-1\right]  \text{ is lacunar;}\\\operatorname*{sum}J<k}}F\left(
J\right)  .
\]
This proves Corollary \ref{cor.FI-lac-cor}.
\end{proof}

We now have the tools to restrict our study of the $\mathbf{k}$-submodules
$F(I)$ to the sets $I$ that are lacunar subsets of $\left[  n-1 \right]  $.

\subsection{Proof of the filtration}

Using the properties of non-shadows that we just established, we can prove
Theorem \ref{thm.t-simultri}, which gives a filtration of $\mathbf{k}\left[
S_{n} \right]  $ preserved by the somewhere-to-below shuffles.

\begin{proof}
[Proof of Theorem \ref{thm.t-simultri}.]We must establish the following three claims:

\begin{statement}
\textit{Claim 1:} We have $0=F_{0}\subseteq F_{1}\subseteq F_{2}%
\subseteq\cdots\subseteq F_{f_{n+1}}=\mathbf{k}\left[  S_{n}\right]  $.
\end{statement}

\begin{statement}
\textit{Claim 2:} We have $F_{i}\cdot t_{\ell}\subseteq F_{i}$ for each
$i\in\left[  0,f_{n+1}\right]  $ and $\ell\in\left[  n\right]  $.
\end{statement}

\begin{statement}
\textit{Claim 3:} For each $i\in\left[  f_{n+1}\right]  $ and $\ell\in\left[
n\right]  $, we have%
\[
F_{i}\cdot\left(  t_{\ell}-m_{Q_{i},\ell}\right)  \subseteq F_{i-1}.
\]

\end{statement}

First of all, let us show an auxiliary claim:

\begin{statement}
\textit{Claim 0:} Let $k\in\mathbb{N}$. Let $i_{k}$ be the largest
$i\in\left[  f_{n+1}\right]  $ satisfying $\operatorname*{sum}\left(
Q_{i}\right)  <k$ (or $0$ if no such $i$ exists). Then, $F\left(  <k\right)
=F_{i_{k}}$.
\end{statement}

[\textit{Proof of Claim 0:} Recall that $\operatorname*{sum}\left(
Q_{1}\right)  \leq\operatorname*{sum}\left(  Q_{2}\right)  \leq\cdots
\leq\operatorname*{sum}\left(  Q_{f_{n+1}}\right)  $. Thus, the inequality
$\operatorname*{sum}\left(  Q_{i}\right)  <k$ holds for each $i\leq i_{k}$ but
does not hold for any other $i$ (because $i_{k}$ is the largest $i\in\left[
f_{n+1}\right]  $ satisfying $\operatorname*{sum}\left(  Q_{i}\right)  <k$).
Therefore, the lacunar subsets $J$ of $\left[  n-1\right]  $ satisfying
$\operatorname*{sum}J<k$ are precisely $Q_{1},Q_{2},\ldots,Q_{i_{k}}$ (since
$Q_{1},Q_{2},\ldots,Q_{f_{n+1}}$ are all the lacunar subsets of $\left[
n-1\right]  $). Hence,%
\[
\sum_{\substack{J\subseteq\left[  n-1\right]  \text{ is lacunar;}%
\\\operatorname*{sum}J<k}}F\left(  J\right)  =F\left(  Q_{1}\right)  +F\left(
Q_{2}\right)  +\cdots+F\left(  Q_{i_{k}}\right)  =F_{i_{k}}%
\]
(by the definition of $F_{i_{k}}$). However, Corollary \ref{cor.FI-lac-cor}
yields%
\[
F\left(  <k\right)  =\sum_{\substack{J\subseteq\left[  n-1\right]  \text{ is
lacunar;}\\\operatorname*{sum}J<k}}F\left(  J\right)  =F_{i_{k}}.
\]
Thus, Claim 0 is proved.] \medskip

We can now easily prove Claims 1, 3 and 2 in this order:

[\textit{Proof of Claim 1:} From the construction of the modules $F_{i}$, it
is clear that $0=F_{0}\subseteq F_{1}\subseteq F_{2}\subseteq\cdots\subseteq
F_{f_{n+1}}$. We thus only need to prove $F_{f_{n+1}}=\mathbf{k}\left[
S_{n}\right]  $.

Let $k=\dbinom{n}{2}+1$. Then, $\operatorname*{sum}\left[  n\right]
=\dbinom{n}{2}<k$, so that $F\left(  \left[  n\right]  \right)  \subseteq
F\left(  <k\right)  $ (by the definition of $F\left(  <k\right)  $). Let
$i_{k}$ be the largest $i\in\left[  f_{n+1}\right]  $ satisfying
$\operatorname*{sum}\left(  Q_{i}\right)  <k$. Hence, Claim 0 yields $F\left(
<k\right)  =F_{i_{k}}$. Consider this $i_{k}$. However, $F\left(  \left[
n\right]  \right)  =\mathbf{k}\left[  S_{n}\right]  $ because the non-shadow
$[n]^{\prime}= \emptyset$. Thus, $\mathbf{k}\left[  S_{n}\right]  =F\left(
\left[  n\right]  \right)  \subseteq F\left(  <k\right)  =F_{i_{k}}\subseteq
F_{f_{n+1}}$ (because $F_{0}\subseteq F_{1}\subseteq F_{2}\subseteq
\cdots\subseteq F_{f_{n+1}}$). Thus, $F_{f_{n+1}}=\mathbf{k}\left[
S_{n}\right]  $ (since $F_{f_{n+1}}\subseteq\mathbf{k}\left[  S_{n}\right]
$). The proof of Claim 1 is thus finished.] \medskip

[\textit{Proof of Claim 3:} Let $i\in\left[  f_{n+1}\right]  $ and $\ell
\in\left[  n\right]  $. We must prove that $F_{i}\cdot\left(  t_{\ell
}-m_{Q_{i},\ell}\right)  \subseteq F_{i-1}$.

The definition of $F_{i-1}$ yields $F_{i-1}=F\left(  Q_{1}\right)  +F\left(
Q_{2}\right)  +\cdots+F\left(  Q_{i-1}\right)  $. Now, it is easy to see that%
\begin{equation}
F\left(  <\operatorname*{sum}\left(  Q_{k}\right)  \right)  \subseteq F_{i-1}
\label{pf.thm.t-simultri.c3.pf.1}%
\end{equation}
for each $k\in\left[  i\right]  $\ \ \ \ \footnote{\textit{Proof:} Let
$k\in\left[  i\right]  $. Let $j=\operatorname*{sum}\left(  Q_{k}\right)  $.
Let $i_{j}$ be the largest $\mathfrak{i}\in\left[  f_{n+1}\right]  $
satisfying $\operatorname*{sum}\left(  Q_{\mathfrak{i}}\right)  <j$ (or $0$ if
no such $\mathfrak{i}$ exists). Then, Claim 0 (applied to $j$ instead of $k$)
yields $F\left(  <j\right)  =F_{i_{j}}$. In view of $j=\operatorname*{sum}%
\left(  Q_{k}\right)  $, this rewrites as $F\left(  <\operatorname*{sum}%
\left(  Q_{k}\right)  \right)  =F_{i_{j}}$.
\par
However, recall that $i_{j}$ is the largest $\mathfrak{i}\in\left[
f_{n+1}\right]  $ satisfying $\operatorname*{sum}\left(  Q_{\mathfrak{i}%
}\right)  <j$. Thus, $\operatorname*{sum}\left(  Q_{\mathfrak{i}}\right)  <j$
for each $\mathfrak{i}\leq i_{j}$ (because $\operatorname*{sum}\left(
Q_{1}\right)  \leq\operatorname*{sum}\left(  Q_{2}\right)  \leq\cdots
\leq\operatorname*{sum}\left(  Q_{f_{n+1}}\right)  $). Since we \textbf{don't}
have $\operatorname*{sum}\left(  Q_{k}\right)  <j$ (because
$j=\operatorname*{sum}\left(  Q_{k}\right)  $), we thus cannot have $k\leq
i_{j}$. Hence, we have $i_{j}<k$, so that $i_{j}\leq k-1\leq i-1$ (because
$k\leq i$). Hence, $F_{i_{j}}\subseteq F_{i-1}$. Now, $F\left(
<\operatorname*{sum}\left(  Q_{k}\right)  \right)  =F_{i_{j}}\subseteq
F_{i-1}$. This proves (\ref{pf.thm.t-simultri.c3.pf.1}).}.

The definition of $F_{i}$ yields $F_{i}=F\left(  Q_{1}\right)  +F\left(
Q_{2}\right)  +\cdots+F\left(  Q_{i}\right)  =\sum_{k=1}^{i}F\left(
Q_{k}\right)  $. Thus,%
\begin{align*}
F_{i}\cdot\left(  t_{\ell}-m_{Q_{i},\ell}\right)   &  =\sum_{k=1}^{i}F\left(
Q_{k}\right)  \cdot\underbrace{\left(  t_{\ell}-m_{Q_{i},\ell}\right)
}_{=\left(  t_{\ell}-m_{Q_{k},\ell}\right)  +\left(  m_{Q_{k},\ell}%
-m_{Q_{i},\ell}\right)  }\\
&  =\sum_{k=1}^{i}\underbrace{F\left(  Q_{k}\right)  \cdot\left(  \left(
t_{\ell}-m_{Q_{k},\ell}\right)  +\left(  m_{Q_{k},\ell}-m_{Q_{i},\ell}\right)
\right)  }_{\subseteq F\left(  Q_{k}\right)  \cdot\left(  t_{\ell}%
-m_{Q_{k},\ell}\right)  +F\left(  Q_{k}\right)  \cdot\left(  m_{Q_{k},\ell
}-m_{Q_{i},\ell}\right)  }\\
&  \subseteq\sum_{k=1}^{i}\left(  F\left(  Q_{k}\right)  \cdot\left(  t_{\ell
}-m_{Q_{k},\ell}\right)  +F\left(  Q_{k}\right)  \cdot\left(  m_{Q_{k},\ell
}-m_{Q_{i},\ell}\right)  \right) \\
&  =\sum_{k=1}^{i}F\left(  Q_{k}\right)  \cdot\left(  t_{\ell}-m_{Q_{k},\ell
}\right)  +\underbrace{\sum_{k=1}^{i}F\left(  Q_{k}\right)  \cdot\left(
m_{Q_{k},\ell}-m_{Q_{i},\ell}\right)  }_{\substack{=\sum_{k=1}^{i-1}F\left(
Q_{k}\right)  \cdot\left(  m_{Q_{k},\ell}-m_{Q_{i},\ell}\right)
\\\text{(here, we have removed the addend}\\\text{for }k=i\text{, since this
addend is }0\text{)}}}\\
&  =\sum_{k=1}^{i}\underbrace{F\left(  Q_{k}\right)  \cdot\left(  t_{\ell
}-m_{Q_{k},\ell}\right)  }_{\substack{\subseteq F\left(  <\operatorname*{sum}%
\left(  Q_{k}\right)  \right)  \\\text{(by Theorem \ref{thm.tl-FI}%
,}\\\text{applied to }I=Q_{k}\text{)}}}+\sum_{k=1}^{i-1}\underbrace{F\left(
Q_{k}\right)  \cdot\left(  m_{Q_{k},\ell}-m_{Q_{i},\ell}\right)
}_{\substack{\subseteq F\left(  Q_{k}\right)  \\\text{(since }m_{Q_{k},\ell
}-m_{Q_{i},\ell}\text{ is just a scalar)}}}\\
&  \subseteq\sum_{k=1}^{i}\underbrace{F\left(  <\operatorname*{sum}\left(
Q_{k}\right)  \right)  }_{\substack{\subseteq F_{i-1}\\\text{(by
(\ref{pf.thm.t-simultri.c3.pf.1}))}}}+\underbrace{\sum_{k=1}^{i-1}F\left(
Q_{k}\right)  }_{\substack{=F\left(  Q_{1}\right)  +F\left(  Q_{2}\right)
+\cdots+F\left(  Q_{i-1}\right)  \\=F_{i-1}\\\text{(by the definition of
}F_{i-1}\text{)}}}\\
&  \subseteq\sum_{k=1}^{i}F_{i-1}+F_{i-1}\subseteq F_{i-1}.
\end{align*}
This proves Claim 3.] \medskip

[\textit{Proof of Claim 2:} Let $i\in\left[  0,f_{n+1}\right]  $ and $\ell
\in\left[  n\right]  $. We must prove that $F_{i}\cdot t_{\ell}\subseteq
F_{i}$. If $i=0$, then this is clearly true (since $F_{0}=0$). Thus, we WLOG
assume that $i\neq0$. Hence, $i\in\left[  f_{n+1}\right]  $. Thus, Claim 3
yields $F_{i}\cdot\left(  t_{\ell}-m_{Q_{i},\ell}\right)  \subseteq F_{i-1}$.
Now,%
\begin{align*}
F_{i}\cdot\underbrace{t_{\ell}}_{=\left(  t_{\ell}-m_{Q_{i},\ell}\right)
+m_{Q_{i},\ell}}  &  =F_{i}\cdot\left(  \left(  t_{\ell}-m_{Q_{i},\ell
}\right)  +m_{Q_{i},\ell}\right) \\
&  \subseteq\underbrace{F_{i}\cdot\left(  t_{\ell}-m_{Q_{i},\ell}\right)
}_{\subseteq F_{i-1}\subseteq F_{i}}+\underbrace{F_{i}\cdot m_{Q_{i},\ell}%
}_{\substack{\subseteq F_{i}\\\text{(since }m_{Q_{i},\ell}\text{ is a
scalar)}}}\\
&  \subseteq F_{i}+F_{i}\subseteq F_{i}.
\end{align*}
This proves Claim 2.] \medskip

We have now proved all Claims 1, 2 and 3. This proves Theorem \ref{thm.tl-FI}.
\end{proof}

\begin{noncompile}
\silentsection{Optimality of the filtration}

(This section has been removed, since its main result has been reproved later.)

Here is a natural question: Is there a way to shorten the filtration
$F_{0}\subseteq F_{1}\subseteq F_{2}\subseteq\cdots\subseteq F_{f_{n+1}}$ that
we just constructed? In other words, is there any $i\in\left[  f_{n+1}\right]
$ such that $F_{i}=F_{i-1}$ ? The answer to this question is \textquotedblleft
no\textquotedblright\ -- the filtration is already optimal, as long as
$\mathbf{k}\neq0$. The proof of this relies on the following lemma, which is a
converse to Proposition \ref{prop.FI-lac} (and a bit more):

\begin{lemma}
\label{lem.FI-lac-conv}Assume that $\mathbf{k}\neq0$. Let $I\subseteq\left[
n-1\right]  $ be lacunar. Then,%
\begin{equation}
F\left(  I\right)  \not \subseteq F\left(  <\operatorname*{sum}I\right)  .
\label{eq.lem.FI-lac-conv.weak}%
\end{equation}
Even better, we have%
\begin{equation}
F\left(  I\right)  \not \subseteq \sum_{\substack{K\subseteq\left[
n-1\right]  ;\\K\neq I;\\\operatorname*{sum}K\leq\operatorname*{sum}%
I}}F\left(  K\right)  . \label{eq.lem.FI-lac-conv.strong}%
\end{equation}

\end{lemma}

Before we prove this lemma, we need a further lemma about lacunar subsets:

\begin{lemma}
\label{lem.lac-sum-less}Let $I$ and $K$ be two subsets of $\left[  n-1\right]
$ such that $I$ is lacunar and $K\neq I$ and $K^{\prime}\subseteq I^{\prime}$.
Then, $\operatorname*{sum}I<\operatorname*{sum}K$.
\end{lemma}

\begin{proof}
[Proof of Lemma \ref{lem.lac-sum-less}.]First, we observe that $I\setminus
K\subseteq\left(  K\setminus I\right)  -1$.

[\textit{Proof:} Let $i\in I\setminus K$. Thus, $i\in I$ and $i\notin K$. We
cannot have $i\in I^{\prime}$ (since this would contradict $i\in I$).

If we had $i+1\notin K$, then we would have $i\in K^{\prime}$ (since $i\in
I\subseteq\left[  n-1\right]  $ and $i\notin K$ and $i+1\notin K$), which
would entail $i\in K^{\prime}\subseteq I^{\prime}$; but this would contradict
the fact that we cannot have $i\in I^{\prime}$. Thus, we cannot have
$i+1\notin K$. In other words, we have $i+1\in K$. Furthermore, $I$ is
lacunar; thus, from $i\in I$, we obtain $i+1\notin I$. Combining this with
$i+1\in K$, we find $i+1\in K\setminus I$. Hence, $i\in\left(  K\setminus
I\right)  -1$.

Forget that we fixed $i$. We thus have proved that $i\in\left(  K\setminus
I\right)  -1$ for each $i\in I\setminus K$. In other words, $I\setminus
K\subseteq\left(  K\setminus I\right)  -1$.] \medskip

Now, the set $I$ is the union of its two disjoint subsets $I\setminus K$ and
$I\cap K$. Hence,%
\begin{equation}
\operatorname*{sum}I=\operatorname*{sum}\left(  I\setminus K\right)
+\operatorname*{sum}\left(  I\cap K\right)  .
\label{pf.lem.lac-sum-less.sumI=}%
\end{equation}
The same argument (with the roles of $I$ and $K$ swapped) yields%
\begin{equation}
\operatorname*{sum}K=\operatorname*{sum}\left(  K\setminus I\right)
+\operatorname*{sum}\left(  K\cap I\right)  .
\label{pf.lem.lac-sum-less.sumK=}%
\end{equation}

Our goal is to prove that $\operatorname*{sum}I<\operatorname*{sum}K$. If
$I\subseteq K$, then this is obvious (since we have $K\neq I$, so that $I$
must be a \textbf{proper} subset of $K$ in this case). Thus, we WLOG assume
that $I\not \subseteq K$ from now on. Hence, $I\setminus K\neq\varnothing$. In
view of $I\setminus K\subseteq\left(  K\setminus I\right)  -1$, this entails
$\left(  K\setminus I\right)  -1\neq\varnothing$, so that $K\setminus
I\neq\varnothing$. Hence, $\left\vert K\setminus I\right\vert >0$.

Now, from $I\setminus K\subseteq\left(  K\setminus I\right)  -1$, we obtain
\[
\operatorname*{sum}\left(  I\setminus K\right)  \leq\operatorname*{sum}\left(
\left(  K\setminus I\right)  -1\right)  =\operatorname*{sum}\left(  K\setminus
I\right)  -\underbrace{\left\vert K\setminus I\right\vert }_{>0}%
<\operatorname*{sum}\left(  K\setminus I\right)  .
\]

However, (\ref{pf.lem.lac-sum-less.sumI=}) becomes%
\[
\operatorname*{sum}I=\underbrace{\operatorname*{sum}\left(  I\setminus
K\right)  }_{<\operatorname*{sum}\left(  K\setminus I\right)  }%
+\operatorname*{sum}\left(  \underbrace{I\cap K}_{=K\cap I}\right)
<\operatorname*{sum}\left(  K\setminus I\right)  +\operatorname*{sum}\left(
K\cap I\right)  =\operatorname*{sum}K
\]
(by (\ref{pf.lem.lac-sum-less.sumK=})). This proves Lemma
\ref{lem.lac-sum-less}.
\end{proof}

\begin{proof}
[Proof of Lemma \ref{lem.FI-lac-conv}.]If $T$ is any subset of $\left[
n-1\right]  $, then we define $G_{T}$ to be the subgroup of $S_{n}$ generated
by the simple transpositions $s_{i}$ with $i\in T$, and we define the two
elements%
\begin{align*}
\operatorname*{sym}\left(  T\right)   &  :=\sum_{\sigma\in G_{T}}\sigma
\in\mathbf{k}\left[  S_{n}\right]  \ \ \ \ \ \ \ \ \ \ \text{and}\\
\operatorname*{anti}\left(  T\right)   &  :=\sum_{\sigma\in G_{T}%
}\underbrace{\left(  -1\right)  ^{\sigma}}_{\text{the sign of }\sigma}%
\sigma\in\mathbf{k}\left[  S_{n}\right]  .
\end{align*}
(These two elements $\operatorname*{sym}\left(  T\right)  $ and
$\operatorname*{anti}\left(  T\right)  $ are called the \emph{symmetrizer} and
the \emph{antisymmetrizer} of $T$, respectively.)

It is easy to see that%
\begin{equation}
\operatorname*{sym}\left(  T\right)  \cdot\operatorname*{anti}\left(  \left[
n-1\right]  \setminus T\right)  \neq0 \label{pf.lem.FI-lac-conv.asneq0}%
\end{equation}
for each subset $T$ of $\left[  n-1\right]  $.

[\textit{Proof of (\ref{pf.lem.FI-lac-conv.asneq0}):} Let $T$ be a subset of
$\left[  n-1\right]  $. The definitions of $\operatorname*{sym}\left(
T\right)  $ and $\operatorname*{anti}\left(  \left[  n-1\right]  \setminus
T\right)  $ yield%
\begin{align}
\operatorname*{sym}\left(  T\right)  \cdot\operatorname*{anti}\left(  \left[
n-1\right]  \setminus T\right)   &  =\left(  \sum_{\sigma\in G_{T}}%
\sigma\right)  \cdot\left(  \sum_{\tau\in G_{\left[  n-1\right]  \setminus T}%
}\left(  -1\right)  ^{\tau}\tau\right) \nonumber\\
&  =\sum_{\left(  \sigma,\tau\right)  \in G_{T}\times G_{\left[  n-1\right]
\setminus T}}\left(  -1\right)  ^{\tau}\sigma\cdot\tau.
\label{pf.lem.FI-lac-conv.asneq0.pf.1}%
\end{align}
However, the only pair $\left(  \sigma,\tau\right)  \in G_{T}\times G_{\left[
n-1\right]  \setminus T}$ satisfying $\sigma\cdot\tau=\operatorname*{id}$ is
$\left(  \operatorname*{id},\operatorname*{id}\right)  $ (since the two
subgroups $G_{T}$ and $G_{\left[  n-1\right]  \setminus T}$ have intersection
$G_{T}\cap G_{\left[  n-1\right]  \setminus T}=\left\{  \operatorname*{id}%
\right\}  $). Thus, the group element $\operatorname*{id}\in S_{n}$ appears on
the right hand side of (\ref{pf.lem.FI-lac-conv.asneq0.pf.1}) with coefficient
$1$. Therefore, this right hand side is $\neq0$. Thus, $\operatorname*{sym}%
T\cdot\operatorname*{anti}\left(  \left[  n-1\right]  \setminus T\right)
\neq0$. This proves (\ref{pf.lem.FI-lac-conv.asneq0}).] \footnote{By the way,
$\operatorname*{sym}\left(  T\right)  \cdot\operatorname*{anti}\left(  \left[
n-1\right]  \setminus T\right)  $ is the Young symmetrizer of a ribbon-shaped
skew Young diagram; but this is a distraction in the present context.}
\medskip

We will also use the following property of antisymmetrizers: If $T$ is a
subset of $\left[  n-1\right]  $, and if $t\in T$, then the antisymmetrizer
$\operatorname*{anti}\left(  T\right)  $ can be written in the form%
\begin{equation}
\operatorname*{anti}\left(  T\right)  =\left(  \operatorname*{id}%
-s_{t}\right)  \cdot g\ \ \ \ \ \ \ \ \ \ \text{for some }g\in\mathbf{k}%
\left[  S_{n}\right]  . \label{pf.lem.FI-lac-conv.asy-factor}%
\end{equation}
(This can be proved in several ways -- e.g., we can take $g=\sum
_{\substack{\sigma\in G_{T};\\\sigma\text{ is even}}}\sigma$.)

Now, it is easily seen that $\operatorname*{sym}\left(  I^{\prime}\right)  \in
F\left(  I\right)  $ (since generally, we have $\operatorname*{sym}\left(
T\right)  \cdot s_{t}=\operatorname*{sym}\left(  T\right)  $ for each
$T\subseteq\left[  n-1\right]  $ and each $t\in T$). We shall now show that
\begin{equation}
\operatorname*{sym}\left(  I^{\prime}\right)  \notin\sum_{\substack{K\subseteq
\left[  n-1\right]  ;\\K\neq I;\\\operatorname*{sum}K\leq\operatorname*{sum}%
I}}F\left(  K\right)  . \label{pf.lem.FI-lac-conv.3}%
\end{equation}

In order to prove this, we shall show that%
\begin{equation}
\operatorname*{sym}\left(  I^{\prime}\right)  \cdot\operatorname*{anti}\left(
\left[  n-1\right]  \setminus I^{\prime}\right)  \neq0
\label{pf.lem.FI-lac-conv.4}%
\end{equation}
but%
\begin{equation}
\sum_{\substack{K\subseteq\left[  n-1\right]  ;\\K\neq I;\\\operatorname*{sum}%
K\leq\operatorname*{sum}I}}F\left(  K\right)  \cdot\operatorname*{anti}\left(
\left[  n-1\right]  \setminus I^{\prime}\right)  =0.
\label{pf.lem.FI-lac-conv.5}%
\end{equation}

Indeed, (\ref{pf.lem.FI-lac-conv.4}) follows directly from
(\ref{pf.lem.FI-lac-conv.asneq0}).

In order to prove (\ref{pf.lem.FI-lac-conv.5}), it suffices to show that
\[
F\left(  K\right)  \cdot\operatorname*{anti}\left(  \left[  n-1\right]
\setminus I^{\prime}\right)  =0
\]
whenever $K$ is a subset of $\left[  n-1\right]  $ satisfying $K\neq I$ and
$\operatorname*{sum}K\leq\operatorname*{sum}I$. So let us fix such a $K$.
Then, $K^{\prime}\not \subseteq I^{\prime}$ (indeed, if we had $K^{\prime
}\subseteq I^{\prime}$, then Lemma \ref{lem.lac-sum-less} would yield
$\operatorname*{sum}I<\operatorname*{sum}K$, which would contradict
$\operatorname*{sum}K\leq\operatorname*{sum}I$). Hence, there exists some
$k\in K^{\prime}\setminus I^{\prime}$. Consider this $k$. Then, $k\in\left[
n-1\right]  \setminus I^{\prime}$, so that $\operatorname*{anti}\left(
\left[  n-1\right]  \setminus I^{\prime}\right)  $ can be written in the form%
\[
\operatorname*{anti}\left(  \left[  n-1\right]  \setminus I^{\prime}\right)
=\left(  \operatorname*{id}-s_{k}\right)  \cdot g\ \ \ \ \ \ \ \ \ \ \text{for
some }g\in\mathbf{k}\left[  S_{n}\right]
\]
(by (\ref{pf.lem.FI-lac-conv.asy-factor}), applied to $T=\left[  n-1\right]
\setminus I^{\prime}$ and $t=k$). Consider this $g$. On the other hand, since
$k\in K^{\prime}$, we have $F\left(  K\right)  \cdot\left(  \operatorname*{id}%
-s_{k}\right)  =0$ by the definition of $F\left(  K\right)  $. Thus,%
\[
F\left(  K\right)  \cdot\underbrace{\operatorname*{anti}\left(  \left[
n-1\right]  \setminus I^{\prime}\right)  }_{=\left(  \operatorname*{id}%
-s_{k}\right)  \cdot g}=\underbrace{F\left(  K\right)  \cdot\left(
\operatorname*{id}-s_{k}\right)  }_{=0}\cdot g=0,
\]
as desired. Thus, (\ref{pf.lem.FI-lac-conv.5}) is proved.

Having proved (\ref{pf.lem.FI-lac-conv.4}) and (\ref{pf.lem.FI-lac-conv.5}),
we immediately obtain (\ref{pf.lem.FI-lac-conv.3}). Therefore,
(\ref{eq.lem.FI-lac-conv.strong}) holds. However,
(\ref{eq.lem.FI-lac-conv.weak}) follows easily from
(\ref{eq.lem.FI-lac-conv.strong}), because the definition of $F\left(
<\operatorname*{sum}I\right)  $ yields%
\[
F\left(  <\operatorname*{sum}I\right)  =\sum_{\substack{K\subseteq\left[
n-1\right]  ;\\\operatorname*{sum}K<\operatorname*{sum}I}}F\left(  K\right)
\subseteq\sum_{\substack{K\subseteq\left[  n-1\right]  ;\\K\neq
I;\\\operatorname*{sum}K\leq\operatorname*{sum}I}}F\left(  K\right)  .
\]
This completes the proof of Lemma \ref{lem.FI-lac-conv}.
\end{proof}

\begin{corollary}
Assume that $\mathbf{k}\neq0$. Then, $F_{i}\neq F_{i-1}$ for each $i\in\left[
f_{n+1}\right]  $.
\end{corollary}

\begin{proof}
[Proof sketch.]Let $i\in\left[  f_{n+1}\right]  $. Then, $F\left(
Q_{i}\right)  \subseteq F_{i}$. If we had $F_{i}=F_{i-1}$, then we would have%
\begin{align*}
F\left(  Q_{i}\right)   &  \subseteq F_{i}=F_{i-1}\\
&  =F\left(  Q_{1}\right)  +F\left(  Q_{2}\right)  +\cdots+F\left(
Q_{i-1}\right)  \ \ \ \ \ \ \ \ \ \ \left(  \text{by the definition of
}F_{i-1}\right) \\
&  \subseteq\sum_{\substack{K\subseteq\left[  n-1\right]  ;\\K\neq
Q_{i};\\\operatorname*{sum}K\leq\operatorname*{sum}\left(  Q_{i}\right)
}}F\left(  K\right)  \ \ \ \ \ \ \ \ \ \ \left(
\begin{array}
[c]{c}%
\text{since each of }Q_{1},Q_{2},\ldots,Q_{i-1}\text{ is a subset }K\\
\text{of }\left[  n-1\right]  \text{ satisfying }K\neq Q_{i}\text{ and
}\operatorname*{sum}K\leq\operatorname*{sum}\left(  Q_{i}\right)
\end{array}
\right)  .
\end{align*}
However, this is impossible, since (\ref{eq.lem.FI-lac-conv.strong}) (applied
to $I=Q_{i}$) yields%
\[
F\left(  Q_{i}\right)  \not \subseteq \sum_{\substack{K\subseteq\left[
n-1\right]  ;\\K\neq Q_{i};\\\operatorname*{sum}K\leq\operatorname*{sum}%
\left(  Q_{i}\right)  }}F\left(  K\right)  .
\]
Thus, we don't have $F_{i}=F_{i-1}$. Hence, $F_{i}\neq F_{i-1}$, qed.
\end{proof}
\end{noncompile}

\section{The descent-destroying basis of $\mathbf{k}\left[  S_{n}\right]  $}

We will now analyze the filtration $F_{0}\subseteq F_{1}\subseteq
F_{2}\subseteq\cdots\subseteq F_{f_{n+1}}$ from Theorem \ref{thm.t-simultri}
further. We shall show that each of the $\mathbf{k}$-modules $F_{0}%
,F_{1},\ldots,F_{f_{n+1}}$ in this filtration is free, and even better, that
there exists a basis of the $\mathbf{k}$-module $\mathbf{k}\left[
S_{n}\right]  $ such that each $F_{i}$ is spanned by an appropriate subfamily
of this basis.

\subsection{Definition}

To construct this basis, we need the following definitions (some of which are
commonplace in the combinatorics of the symmetric group):

\begin{itemize}
\item The \emph{descent set} of a permutation $w\in S_{n}$ is defined to be
the set of all $i\in\left[  n-1\right]  $ such that $w\left(  i\right)
>w\left(  i+1\right)  $. This set is denoted by $\operatorname*{Des}w$.

For example, the permutation in $S_{4}$ that sends $1,2,3,4$ to $3,2,4,1$ has
descent set $\left\{  1,3\right\}  $.

\item We define a total order $<$ on the set $S_{n}$ as follows: If $u$ and
$v$ are two distinct permutations in $S_{n}$, then we say that $u<v$ if and
only if the smallest $i\in\left[  n\right]  $ satisfying $u\left(  i\right)
\neq v\left(  i\right)  $ satisfies $u\left(  i\right)  <v\left(  i\right)  $.
This relation $<$ is a total order on the set $S_{n}$, and is known as the
\emph{lexicographic order} on $S_{n}$. (If we identify each permutation $w\in
S_{n}$ with the $n$-tuple $\left(  w\left(  1\right)  ,\ w\left(  2\right)
,\ \ldots,\ w\left(  n\right)  \right)  $, then this order is precisely the
lexicographic order on $n$-tuples of integers; this is why it has the same name.)

For example, the smallest permutation in $S_{n}$ with respect to the total
order $<$ is the identity permutation $\operatorname*{id}$, whereas the
largest permutation is the one that sends each $i\in\left[  n\right]  $ to
$n+1-i$.

\item For each $I\subseteq\left[  n-1\right]  $, we let $G\left(  I\right)  $
be the subgroup of $S_{n}$ generated by the subset $\left\{  s_{i}\ \mid\ i\in
I\right\}  $.

For instance, if $n=5$ and $I=\left\{  2,4\right\}  $, then $G\left(
I\right)  =\left\langle s_{2},s_{4}\right\rangle \leq S_{5}$.

\item For each $w\in S_{n}$, we set%
\begin{equation}
a_{w}:=\sum_{\sigma\in G\left(  \operatorname*{Des}w\right)  }w\sigma
\in\mathbf{k}\left[  S_{n}\right]  . \label{eq.aw.def}%
\end{equation}

\end{itemize}

\begin{example}
\label{exa.aw.1}For this example, let $n=3$. We write each permutation $w\in
S_{3}$ as the list $\left[  w\left(  1\right)  \ w\left(  2\right)  \ w\left(
3\right)  \right]  $ (written without commas for brevity, and using square
brackets to distinguish it from a parenthesized integer). Then,%
\begin{align*}
a_{\left[  123\right]  }  &  =\left[  123\right]  ;\\
a_{\left[  132\right]  }  &  =\left[  132\right]  +\left[  123\right]  ;\\
a_{\left[  213\right]  }  &  =\left[  213\right]  +\left[  123\right]  ;\\
a_{\left[  231\right]  }  &  =\left[  231\right]  +\left[  213\right]  ;\\
a_{\left[  312\right]  }  &  =\left[  312\right]  +\left[  132\right]  ;\\
a_{\left[  321\right]  }  &  =\left[  321\right]  +\left[  312\right]
+\left[  231\right]  +\left[  213\right]  +\left[  132\right]  +\left[
123\right]  .
\end{align*}

\end{example}

The quickest way to compute $a_{w}$ for a given permutation $w \in S_{n}$ is
as follows:

\begin{itemize}
\item Break the $n$-tuple $\left(  w\left(  1 \right)  , w\left(  2 \right)  ,
\ldots, w\left(  n \right)  \right)  $ into decreasing blocks by placing a
vertical bar between $w\left(  i \right)  $ and $w\left(  i+1 \right)  $
whenever $w\left(  i \right)  < w\left(  i+1 \right)  $. (For example, if
$\left(  w\left(  1 \right)  , w\left(  2 \right)  , \ldots, w\left(  n
\right)  \right)  = \left(  3, 5, 1, 2, 7, 6, 4 \right)  $, then the result of
this break-up is $\left(  3 \mid5, 1 \mid2 \mid7, 6, 4 \right)  $.)

\item Within each decreasing block, we permute the entries arbitrarily.

\item All resulting $n$-tuples are again interpreted as permutations $v \in
S_{n}$. The $a_{w}$ is the sum of these permutations $v$.
\end{itemize}

\subsection{The lexicographic property}

As Example \ref{exa.aw.1} demonstrates, it seems that an element $a_{w}$ is a
sum of $w$ and several permutations that are smaller than $w$ in the
lexicographic order. This is indeed always the case, and will follow from the
following proposition:

\begin{proposition}
\label{prop.aw.smaller}Let $w\in S_{n}$. Let $\sigma\in G\left(
\operatorname*{Des}w\right)  $ satisfy $\sigma\neq\operatorname*{id}$. Then,
$w\sigma<w$ (with respect to the lexicographic order).
\end{proposition}

Proposition \ref{prop.aw.smaller} is easy to prove with a bit of handwaving,
but trickier to prove formally. We shall thus give a quick informal proof
first, and then a longer, formal proof.

\begin{proof}
[Informal proof of Proposition \ref{prop.aw.smaller}.]Let $i_{1},i_{2}%
,\ldots,i_{p}$ be the elements of the set $\left[  n-1\right]  \setminus
\operatorname*{Des}w$ in increasing order. Furthermore, let $i_{0}=0$ and
$i_{p+1}=n$, so that $0=i_{0}<i_{1}<i_{2}<\cdots<i_{p}<i_{p+1}=n$. Define an
interval%
\[
J_{k}:=\left[  i_{k-1}+1,\ i_{k}\right]  \ \ \ \ \ \ \ \ \ \ \text{for each
}k\in\left[  p+1\right]  .
\]
Then, the $p+1$ intervals $J_{1},J_{2},\ldots,J_{p+1}$ form a set partition of
the interval $\left[  n\right]  $. The permutation $w$ is decreasing on each
of these $p+1$ intervals, and these $p+1$ intervals are actually the
inclusion-maximal intervals with this property.

Now, $\sigma\in G\left(  \operatorname*{Des}w\right)  $ means that the
permutation $\sigma$ preserves each of the $p+1$ intervals $J_{1},J_{2}%
,\ldots,J_{p+1}$ (that is, we have $\sigma\left(  J_{k}\right)  =J_{k}$ for
each $k\in\left[  p+1\right]  $).\ \ \ \ \footnote{Indeed,
$\operatorname*{Des}w=\left[  n-1\right]  \setminus\left\{  i_{1},i_{2}%
,\ldots,i_{p}\right\}  $. Hence, the group $G\left(  \operatorname*{Des}%
w\right)  $ is generated by the simple transpositions $s_{i}$ with
$i\in\left[  n-1\right]  \setminus\left\{  i_{1},i_{2},\ldots,i_{p}\right\}
$. Thus, $\sigma\in G\left(  \operatorname*{Des}w\right)  $ shows that
$\sigma$ is a product of such simple transpositions. However, each such simple
transposition preserves each of the $p+1$ intervals $J_{1},J_{2}%
,\ldots,J_{p+1}$. Thus, so does $\sigma$.} Hence, the permutation $w\sigma$ is
obtained from $w$ by separately permuting the values on each of the $p+1$
intervals $J_{1},J_{2},\ldots,J_{p+1}$. However, recall that $w$ is decreasing
on each of these $p+1$ intervals; thus, if we permute the values of $w$ on
each of these $p+1$ intervals separately, then the permutation $w$ can only
become smaller in the lexicographic order. Hence, $w\sigma\leq w$. Combining
this with $w\sigma\neq w$ (which follows from $\sigma\neq\operatorname*{id}$),
we obtain $w\sigma<w$. This proves Proposition \ref{prop.aw.smaller} (if you
believe this handwaving).
\end{proof}

\begin{fineprint}
Next, we shall give a more formal proof of Proposition \ref{prop.aw.smaller}
for the skeptical reader. This proof will require a further definition and two
lemmas (which might be of independent interest). We begin with the definition:

\begin{itemize}
\item If $w\in S_{n}$, then an \emph{inversion} of $w$ means a pair $\left(
i,j\right)  \in\left[  n\right]  \times\left[  n\right]  $ satisfying $i<j$
and $w\left(  i\right)  >w\left(  j\right)  $. We denote the set of all
inversions of a given permutation $w\in S_{n}$ by $\operatorname*{Inv}w$.
\end{itemize}

Now, we can state our two lemmas:

\begin{lemma}
\label{lem.aw.smaller.1}Let $w\in S_{n}$. Let $\sigma\in G\left(
\operatorname*{Des}w\right)  $. Then, $\operatorname*{Inv}\left(  \left(
w\sigma\right)  ^{-1}\right)  \subseteq\operatorname*{Inv}\left(
w^{-1}\right)  $.
\end{lemma}

\begin{lemma}
\label{lem.aw.smaller.2}Let $u\in S_{n}$ and $v\in S_{n}$ satisfy
$\operatorname*{Inv}\left(  u^{-1}\right)  \subseteq\operatorname*{Inv}\left(
v^{-1}\right)  $. Then, $u\leq v$ (with respect to the lexicographic order).
\end{lemma}

\begin{proof}
[Proof of Lemma \ref{lem.aw.smaller.1}.]Let $\left(  i,j\right)
\in\operatorname*{Inv}\left(  \left(  w\sigma\right)  ^{-1}\right)  $.

We have $\left(  i,j\right)  \in\operatorname*{Inv}\left(  \left(
w\sigma\right)  ^{-1}\right)  $. In other words, $\left(  i,j\right)  $ is an
inversion of $\left(  w\sigma\right)  ^{-1}$. By the definition of an
inversion, this means that $\left(  i,j\right)  \in\left[  n\right]
\times\left[  n\right]  $ and $i<j$ and $\left(  w\sigma\right)  ^{-1}\left(
i\right)  >\left(  w\sigma\right)  ^{-1}\left(  j\right)  $.

Set $a:=w^{-1}\left(  i\right)  $ and $b:=w^{-1}\left(  j\right)  $. We shall
now show that $a>b$.

Indeed, assume the contrary. Thus, $a\leq b$. Since $a\neq b$%
\ \ \ \ \footnote{\textit{Proof.} We have $i<j$, thus $i\neq j$ and therefore
$w^{-1}\left(  i\right)  \neq w^{-1}\left(  j\right)  $. In other words,
$a\neq b$ (since $a=w^{-1}\left(  i\right)  $ and $b=w^{-1}\left(  j\right)
$).}, we thus obtain $a<b$.

From $a=w^{-1}\left(  i\right)  $ and $b=w^{-1}\left(  j\right)  $, we obtain
$w\left(  a\right)  =i$ and $w\left(  b\right)  =j$. Thus, $w\left(  a\right)
=i<j=w\left(  b\right)  $. Hence, there exists some $k\in\left[  a,b-1\right]
\setminus\operatorname*{Des}w$\ \ \ \ \footnote{\textit{Proof.} Assume the
contrary. Thus, there exists no $k\in\left[  a,b-1\right]  \setminus
\operatorname*{Des}w$. In other words, the set $\left[  a,b-1\right]
\setminus\operatorname*{Des}w$ is empty. In other words, $\left[
a,b-1\right]  \subseteq\operatorname*{Des}w$. Hence, each $i\in\left[
a,b-1\right]  $ satisfies $i\in\left[  a,b-1\right]  \subseteq
\operatorname*{Des}w$ and therefore $w\left(  i\right)  >w\left(  i+1\right)
$ (by the definition of $\operatorname*{Des}w$). In other words, we have%
\[
w\left(  a\right)  >w\left(  a+1\right)  >\cdots>w\left(  b-1\right)
>w\left(  b\right)  .
\]
This contradicts $w\left(  a\right)  <w\left(  b\right)  $. This contradiction
shows that our assumption was false, qed.}. Consider this $k$.

From $k\in\left[  a,b-1\right]  \setminus\operatorname*{Des}w\subseteq\left[
a,b-1\right]  $, we obtain $a\leq k\leq b-1<b$. Therefore, $a\in\left[
k\right]  $ but $b\notin\left[  k\right]  $. Moreover, from $k\in\left[
a,b-1\right]  \setminus\operatorname*{Des}w$, we obtain $k\notin%
\operatorname*{Des}w$.

Let $I=\operatorname*{Des}w$. Thus, $k\notin\operatorname*{Des}w=I$. Hence,
$s_{k}$ is not among the generators of the group $G\left(  I\right)  $.

Therefore, it is easy to see that
\begin{equation}
\tau\left(  \left[  k\right]  \right)  =\left[  k\right]
\ \ \ \ \ \ \ \ \ \ \text{for each }\tau\in G\left(  I\right)
\label{pf.lem.aw.smaller.1.tau-pres-k}%
\end{equation}
\footnote{\textit{Proof of (\ref{pf.lem.aw.smaller.1.tau-pres-k}):} We must
show that each element of $G\left(  I\right)  $ preserves the set $\left[
k\right]  $.
\par
We have defined $G\left(  I\right)  $ to be the subgroup of $S_{n}$ generated
by the subset $\left\{  s_{m}\ \mid\ m\in I\right\}  $. Hence, in order to
prove that each element of $G\left(  I\right)  $ preserves the set $\left[
k\right]  $, it suffices to prove that each of the generators $s_{m}$
preserves this set. In other words, it suffices to prove that $s_{m}\left(
\left[  k\right]  \right)  =\left[  k\right]  $ for each $m\in I$.
\par
But this is easy: Let $m\in I$. Then, $m\neq k$ (since $m\in I$ but $k\notin
I$). Hence, we have either $m<k$ or $m>k$. In the former case, the simple
transposition $s_{m}$ swaps the two elements $m$ and $m+1$, which both lie
inside $\left[  k\right]  $; thus, $s_{m}\left(  \left[  k\right]  \right)
=\left[  k\right]  $ in this case. In the latter case, the simple
transposition $s_{m}$ fixes all elements of $\left[  k\right]  $ (since
neither $m$ nor $m+1$ lies in $\left[  k\right]  $); thus, $s_{m}\left(
\left[  k\right]  \right)  =\left[  k\right]  $ in this case as well. Hence,
we have proved that $s_{m}\left(  \left[  k\right]  \right)  =\left[
k\right]  $ in all cases. As explained above, this completes the proof of
(\ref{pf.lem.aw.smaller.1.tau-pres-k}).}. Applying this to $\tau=\sigma$, we
obtain $\sigma\left(  \left[  k\right]  \right)  =\left[  k\right]  $ (since
$\sigma\in G\left(  \underbrace{\operatorname*{Des}w}_{=I}\right)  =G\left(
I\right)  $). Thus, $\sigma^{-1}\left(  \left[  k\right]  \right)  =\left[
k\right]  $ (since $\sigma$ is a bijection). However,%
\[
\left(  w\sigma\right)  ^{-1}\left(  i\right)  =\sigma^{-1}\left(
\underbrace{w^{-1}\left(  i\right)  }_{=a\in\left[  k\right]  }\right)
\in\sigma^{-1}\left(  \left[  k\right]  \right)  =\left[  k\right]  ,
\]
so that $\left(  w\sigma\right)  ^{-1}\left(  i\right)  \leq k$ and therefore
$k\geq\left(  w\sigma\right)  ^{-1}\left(  i\right)  >\left(  w\sigma\right)
^{-1}\left(  j\right)  $. In other words, $\left(  w\sigma\right)
^{-1}\left(  j\right)  <k$, so that $\left(  w\sigma\right)  ^{-1}\left(
j\right)  \in\left[  k\right]  $. Therefore, $\sigma\left(  \left(
w\sigma\right)  ^{-1}\left(  j\right)  \right)  \in\sigma\left(  \left[
k\right]  \right)  =\left[  k\right]  $. In view of $\sigma\left(  \left(
w\sigma\right)  ^{-1}\left(  j\right)  \right)  =\sigma\left(  \sigma
^{-1}\left(  w^{-1}\left(  j\right)  \right)  \right)  =w^{-1}\left(
j\right)  =b$, this rewrites as $b\in\left[  k\right]  $. But this contradicts
$b\notin\left[  k\right]  $. This contradiction shows that our assumption was false.

Hence, $a>b$ is proved. In view of $a=w^{-1}\left(  i\right)  $ and
$b=w^{-1}\left(  j\right)  $, we can rewrite this as $w^{-1}\left(  i\right)
>w^{-1}\left(  j\right)  $. Combining this with $\left(  i,j\right)
\in\left[  n\right]  \times\left[  n\right]  $ and $i<j$, we conclude that
$\left(  i,j\right)  $ is an inversion of $w^{-1}$. In other words, $\left(
i,j\right)  \in\operatorname*{Inv}\left(  w^{-1}\right)  $.

Forget that we fixed $\left(  i,j\right)  $. We thus have shown that $\left(
i,j\right)  \in\operatorname*{Inv}\left(  w^{-1}\right)  $ for each $\left(
i,j\right)  \in\operatorname*{Inv}\left(  \left(  w\sigma\right)
^{-1}\right)  $. In other words, $\operatorname*{Inv}\left(  \left(
w\sigma\right)  ^{-1}\right)  \subseteq\operatorname*{Inv}\left(
w^{-1}\right)  $. Lemma \ref{lem.aw.smaller.1} is thus proven.
\end{proof}

\begin{proof}
[Proof of Lemma \ref{lem.aw.smaller.2}.]We WLOG assume that $u\neq v$ (since
otherwise, the claim is obvious). Thus, there exists some $i\in\left[
n\right]  $ satisfying $u\left(  i\right)  \neq v\left(  i\right)  $. Consider
the \textbf{smallest} such $i$. We shall show that $u\left(  i\right)
<v\left(  i\right)  $. Once this is shown, we will immediately obtain $u<v$
(by the definition of lexicographic order), and thus Lemma
\ref{lem.aw.smaller.2} will follow.

So it remains to prove that $u\left(  i\right)  <v\left(  i\right)  $. For the
sake of contradiction, we assume the contrary. Thus, $u\left(  i\right)  \geq
v\left(  i\right)  $, so that $u\left(  i\right)  >v\left(  i\right)  $ (since
$u\left(  i\right)  \neq v\left(  i\right)  $).

The maps $u$ and $v$ are permutations, and thus are injective.

Recall that $i$ was defined to be the \textbf{smallest} element of $\left[
n\right]  $ satisfying $u\left(  i\right)  \neq v\left(  i\right)  $. Thus,%
\begin{equation}
u\left(  k\right)  =v\left(  k\right)  \ \ \ \ \ \ \ \ \ \ \text{for each
}k<i. \label{pf.lem.aw.smaller.2.smallest}%
\end{equation}

Let $p:=u\left(  i\right)  $ and $q:=v\left(  i\right)  $. Thus, $p>q$ (since
$u\left(  i\right)  >v\left(  i\right)  $), so that $q<p$. Hence, $q\neq p$,
so that $u^{-1}\left(  q\right)  \neq u^{-1}\left(  p\right)  $. Moreover,
$u^{-1}\left(  p\right)  =i$ (since $p=u\left(  i\right)  $). If we had
$u^{-1}\left(  q\right)  <i$, then we would have $u\left(  u^{-1}\left(
q\right)  \right)  =v\left(  u^{-1}\left(  q\right)  \right)  $ (by
(\ref{pf.lem.aw.smaller.2.smallest}), applied to $k=u^{-1}\left(  q\right)
$), so that $v\left(  u^{-1}\left(  q\right)  \right)  =u\left(  u^{-1}\left(
q\right)  \right)  =q=v\left(  i\right)  $ and therefore $u^{-1}\left(
q\right)  =i$ (since the map $v$ is injective); but this would contradict the
very assumption $u^{-1}\left(  q\right)  <i$. Hence, we cannot have
$u^{-1}\left(  q\right)  <i$. Thus, we must have $u^{-1}\left(  q\right)  \geq
i=u^{-1}\left(  p\right)  $. Combining this with $u^{-1}\left(  q\right)  \neq
u^{-1}\left(  p\right)  $, we obtain $u^{-1}\left(  q\right)  >u^{-1}\left(
p\right)  $.

Now we know that $\left(  q,p\right)  \in\left[  n\right]  \times\left[
n\right]  $ satisfies $q<p$ and $u^{-1}\left(  q\right)  >u^{-1}\left(
p\right)  $. In other words, $\left(  q,p\right)  $ is an inversion of
$u^{-1}$. Hence, $\left(  q,p\right)  \in\operatorname*{Inv}\left(
u^{-1}\right)  \subseteq\operatorname*{Inv}\left(  v^{-1}\right)  $. In other
words, $\left(  q,p\right)  $ is an inversion of $v^{-1}$. Hence,
$v^{-1}\left(  q\right)  >v^{-1}\left(  p\right)  $. Since $v^{-1}\left(
q\right)  =i$ (because $q=v\left(  i\right)  $), this rewrites as
$i>v^{-1}\left(  p\right)  $. Thus, $v^{-1}\left(  p\right)  <i$, so that we
can apply (\ref{pf.lem.aw.smaller.2.smallest}) to $k=v^{-1}\left(  p\right)  $
and obtain%
\[
u\left(  v^{-1}\left(  p\right)  \right)  =v\left(  v^{-1}\left(  p\right)
\right)  =p=u\left(  i\right)  .
\]
Hence, $v^{-1}\left(  p\right)  =i$ (because $u$ is injective). In other
words, $p=v\left(  i\right)  $. This contradicts $p>q=v\left(  i\right)  $.
This contradiction shows that our assumption was false. Hence, $u\left(
i\right)  <v\left(  i\right)  $ is proved, and Lemma \ref{lem.aw.smaller.2}
follows as explained above.
\end{proof}

\begin{proof}
[Formal proof of Proposition \ref{prop.aw.smaller}.]Lemma
\ref{lem.aw.smaller.1} yields $\operatorname*{Inv}\left(  \left(
w\sigma\right)  ^{-1}\right)  \subseteq\operatorname*{Inv}\left(
w^{-1}\right)  $. Hence, Lemma \ref{lem.aw.smaller.2} (applied to $u=w\sigma$
and $v=w$) yields $w\sigma\leq w$. However, from $\sigma\neq\operatorname*{id}%
$, we obtain $w\sigma\neq w$ (since $S_{n}$ is a group). Combining this with
$w\sigma\leq w$, we obtain $w\sigma<w$. This proves Proposition
\ref{prop.aw.smaller}.
\end{proof}
\end{fineprint}

\begin{corollary}
\label{cor.aw.lead}Let $w\in S_{n}$. Then,%
\[
a_{w}=w+\left(  \text{a sum of permutations }v\in S_{n}\text{ satisfying
}v<w\right)  .
\]

\end{corollary}

\begin{proof}
The definition of $a_{w}$ yields%
\begin{align*}
a_{w}  &  =\sum_{\sigma\in G\left(  \operatorname*{Des}w\right)  }%
w\sigma=\underbrace{w\operatorname*{id}}_{=w}+\sum_{\substack{\sigma\in
G\left(  \operatorname*{Des}w\right)  ;\\\sigma\neq\operatorname*{id}}%
}w\sigma\ \ \ \ \ \ \ \ \ \ \left(
\begin{array}
[c]{c}%
\text{here, we have split off the}\\
\text{addend for }\sigma=\operatorname*{id}\text{ from the sum }%
\end{array}
\right) \\
&  =w+\underbrace{\sum_{\substack{\sigma\in G\left(  \operatorname*{Des}%
w\right)  ;\\\sigma\neq\operatorname*{id}}}w\sigma}_{\substack{=\left(
\text{a sum of permutations }v\in S_{n}\text{ satisfying }v<w\right)
\\\text{(since Proposition \ref{prop.aw.smaller} shows that each}%
\\\text{addend }w\sigma\text{ of this sum satisfies }w\sigma<w\text{)}}}\\
&  =w+\left(  \text{a sum of permutations }v\in S_{n}\text{ satisfying
}v<w\right)  .
\end{align*}
This proves Corollary \ref{cor.aw.lead}.
\end{proof}

\subsection{The basis property}

Using Corollary \ref{cor.aw.lead}, we can now see that the elements $a_{w}$
for all $w\in S_{n}$ form a basis of $\mathbf{k}\left[  S_{n}\right]  $, and
furthermore, by selecting an appropriate subset of these elements, we can find
a basis of each $F\left(  I\right)  $. To wit, the following two propositions hold:

\begin{proposition}
\label{prop.aw.basis-kSn}The family $\left(  a_{w}\right)  _{w\in S_{n}}$ is a
basis of the $\mathbf{k}$-module $\mathbf{k}\left[  S_{n}\right]  $.
\end{proposition}

\begin{proposition}
\label{prop.aw.basis-FI}For each $I\subseteq\left[  n\right]  $, the family
$\left(  a_{w}\right)  _{w\in S_{n};\ I^{\prime}\subseteq\operatorname*{Des}%
w}$ is a basis of the $\mathbf{k}$-module $F\left(  I\right)  $.
\end{proposition}

We shall derive both Proposition \ref{prop.aw.basis-kSn} and Proposition
\ref{prop.aw.basis-FI} from a more general result. To state the latter, we
introduce another notation:

\begin{itemize}
\item For any subset $I$ of $\left[  n-1\right]  $, we set
\[
Z\left(  I\right)  :=\left\{  q\in\mathbf{k}\left[  S_{n}\right]
\ \mid\ qs_{i}=q\text{ for all }i\in I\right\}  .
\]
This is a $\mathbf{k}$-submodule of $\mathbf{k}\left[  S_{n}\right]  $.
\end{itemize}

The definition of those $\mathbf{k}$-submodules reminds us of the definition
of $F(I)$, so we make the relation between the two notions explicit:

\begin{proposition}
\label{prop.aw.QI=FI'}Let $I\subseteq\left[  n\right]  $. Then, $F\left(
I\right)  =Z\left(  I^{\prime}\right)  $.
\end{proposition}

\begin{proof}
Both $F\left(  I\right)  $ and $Z\left(  I^{\prime}\right)  $ are defined to
be $\left\{  q\in\mathbf{k}\left[  S_{n}\right]  \ \mid\ qs_{i}=q\text{ for
all }i\in I^{\prime}\right\}  $. Thus, we have $F\left(  I\right)  =Z\left(
I^{\prime}\right)  $. This proves Proposition \ref{prop.aw.QI=FI'}.
\end{proof}

Now, we can state the general result from which both Proposition
\ref{prop.aw.basis-kSn} and Proposition \ref{prop.aw.basis-FI} will follow:

\begin{proposition}
\label{prop.aw.basis-QI}Let $I$ be a subset of $\left[  n-1\right]  $. Then,
the family $\left(  a_{w}\right)  _{w\in S_{n};\ I\subseteq\operatorname*{Des}%
w}$ is a basis of the $\mathbf{k}$-module $Z\left(  I\right)  $.
\end{proposition}

\begin{proof}
To prove that the family $\left(  a_{w}\right)  _{w\in S_{n};\ I\subseteq
\operatorname*{Des} w}$ forms a basis of $Z(I)$, there are three items to
prove. First, we shall prove that each element of this family belongs to
$Z(I)$ (Claim 1 below). Then, we will show that this family spans $Z(I)$ (a
consequence of Claim 2 below). Finally, we will show that the (larger) family
$\left(  a_{w}\right)  _{w \in S_{n}}$ is $\mathbf{k}$-linearly independent
(Claim 3). The proofs of these three claims constitute the bulk of the proof
of Proposition \ref{prop.aw.basis-QI}, although an experienced reader will
likely find some (or even all) of them straightforward.\medskip

In the proof that follows, we shall use the notation $\left[  w\right]  q$ for
the coefficient of a permutation $w\in S_{n}$ in an element $q\in
\mathbf{k}\left[  S_{n}\right]  $. (Thus, each $q\in\mathbf{k}\left[
S_{n}\right]  $ satisfies $q=\sum_{w\in S_{n}}\left(  \left[  w\right]
q\right)  w$.) The definition of multiplication in the group algebra
$\mathbf{k}\left[  S_{n}\right]  $ shows that
\begin{equation}
\left[  w\right]  \left(  q\sigma\right)  =\left[  w\sigma^{-1}\right]  q
\label{pf.prop.aw.basis-QI.coeff1}%
\end{equation}
for any $w\in S_{n}$, $\sigma\in S_{n}$ and $q\in\mathbf{k}\left[
S_{n}\right]  $.\medskip

We shall first show that the family $\left(  a_{w}\right)  _{w\in
S_{n};\ I\subseteq\operatorname*{Des}w}$ is a family of vectors in $Z\left(
I\right)  $. In other words, we shall show the following:

\begin{statement}
\textit{Claim 1:} For each $w\in S_{n}$ satisfying $I\subseteq
\operatorname*{Des}w$, we have $a_{w}\in Z\left(  I\right)  $.
\end{statement}

[\textit{Proof of Claim 1:} Let $w\in S_{n}$ satisfy $I\subseteq
\operatorname*{Des}w$. Let $i\in I$. Then, $i\in I\subseteq\operatorname*{Des}%
w$. Hence, $s_{i}$ is one of the generators of the group $G\left(
\operatorname*{Des}w\right)  $ (by the definition of $G\left(
\operatorname*{Des}w\right)  $). Thus, $s_{i}\in G\left(  \operatorname*{Des}%
w\right)  $. However, $G\left(  \operatorname*{Des}w\right)  $ is a group.
Thus, the map $G\left(  \operatorname*{Des}w\right)  \rightarrow G\left(
\operatorname*{Des}w\right)  ,\ \sigma\mapsto\sigma s_{i}$ is a bijection
(since $s_{i}\in G\left(  \operatorname*{Des}w\right)  $).

However, the definition of $a_{w}$ yields $a_{w}=\sum_{\sigma\in G\left(
\operatorname*{Des}w\right)  }w\sigma$. Multiplying this equality by $s_{i}$,
we find%
\[
a_{w}s_{i}=\left(  \sum_{\sigma\in G\left(  \operatorname*{Des}w\right)
}w\sigma\right)  s_{i}=\sum_{\sigma\in G\left(  \operatorname*{Des}w\right)
}w\sigma s_{i}=\sum_{\sigma\in G\left(  \operatorname*{Des}w\right)  }w\sigma
\]
(here, we have substituted $\sigma$ for $\sigma s_{i}$ in the sum, since the
map $G\left(  \operatorname*{Des}w\right)  \rightarrow G\left(
\operatorname*{Des}w\right)  ,\ \sigma\mapsto\sigma s_{i}$ is a bijection).
Comparing this with $a_{w}=\sum_{\sigma\in G\left(  \operatorname*{Des}%
w\right)  }w\sigma$, we obtain $a_{w}s_{i}=a_{w}$.

Now, forget that we fixed $i$. We thus have shown that $a_{w}s_{i}=a_{w}$ for
each $i\in I$. In other words,%
\[
a_{w}\in\left\{  q\in\mathbf{k}\left[  S_{n}\right]  \ \mid\ qs_{i}=q\text{
for all }i\in I\right\}  =Z\left(  I\right)
\]
(by the definition of $Z\left(  I\right)  $). This proves Claim 1.] \medskip

Next, we shall show that the family $\left(  a_{w}\right)  _{w\in
S_{n};\ I\subseteq\operatorname*{Des}w}$ spans the $\mathbf{k}$-module
$Z\left(  I\right)  $. To achieve this, we will first prove the following:

\begin{statement}
\textit{Claim 2:} Let $u\in S_{n}$. Then,\footnote{Here and in the following,
$\operatorname*{span}\left(  \left(  f_{i}\right)  _{i\in I}\right)  $ denotes
the $\mathbf{k}$-linear span of a family $\left(  f_{i}\right)  _{i\in I}$ of
vectors.}
\[
Z\left(  I\right)  \cap\operatorname*{span}\left(  \left(  w\right)  _{w\in
S_{n};\ w\leq u}\right)  \subseteq\operatorname*{span}\left(  \left(
a_{w}\right)  _{w\in S_{n};\ I\subseteq\operatorname*{Des}w}\right)  .
\]

\end{statement}

[\textit{Proof of Claim 2:} We proceed by strong induction on $u$ (using the
lexicographic order as a well-ordering on $S_{n}$). Thus, we fix some
permutation $x\in S_{n}$, and we assume (as induction hypothesis) that Claim 2
has already been proved for each $u<x$. We must then prove Claim 2 for $u=x$.

Using our induction hypothesis, we can easily see that
\begin{equation}
Z\left(  I\right)  \cap\operatorname*{span}\left(  \left(  w\right)  _{w\in
S_{n};\ w<x}\right)  \subseteq\operatorname*{span}\left(  \left(
a_{w}\right)  _{w\in S_{n};\ I\subseteq\operatorname*{Des}w}\right)  .
\label{pf.prop.aw.basis-QI.c2.pf.IH2}%
\end{equation}
\footnote{\textit{Proof of (\ref{pf.prop.aw.basis-QI.c2.pf.IH2}):} If $x$ is
the smallest permutation in $S_{n}$ (with respect to the lexicographic order),
then the family $\left(  w\right)  _{w\in S_{n};\ w<x}$ is empty (since there
is no $w\in S_{n}$ satisfying $w<x$ in this case), and thus its span is
$\operatorname*{span}\left(  \left(  w\right)  _{w\in S_{n};\ w<x}\right)
=0$, so that we have $Z\left(  I\right)  \cap\underbrace{\operatorname*{span}%
\left(  \left(  w\right)  _{w\in S_{n};\ w<x}\right)  }_{=0}=0\subseteq
\operatorname*{span}\left(  \left(  a_{w}\right)  _{w\in S_{n};\ I\subseteq
\operatorname*{Des}w}\right)  $. Hence, if $x$ is the smallest permutation in
$S_{n}$, then (\ref{pf.prop.aw.basis-QI.c2.pf.IH2}) holds. Thus, for the rest
of this proof, we WLOG assume that $x$ is not the smallest permutation in
$S_{n}$. Thus, there exists some $w\in S_{n}$ such that $w<x$. Let $y$ be the
\textbf{largest} such $w$ (this is well-defined, since the lexicographic order
is a total order on the finite set $S_{n}$). Then, the permutations $w\in
S_{n}$ satisfying $w<x$ are precisely the permutations $w\in S_{n}$ satisfying
$w\leq y$. Thus, $\operatorname*{span}\left(  \left(  w\right)  _{w\in
S_{n};\ w<x}\right)  =\operatorname*{span}\left(  \left(  w\right)  _{w\in
S_{n};\ w\leq y}\right)  $. Note also that $y<x$ (by the definition of $y$).
\par
However, our induction hypothesis says that Claim 2 has already been proved
for each $u<x$. Hence, in particular, Claim 2 holds for $u=y$ (since $y<x$).
In other words, we have $Z\left(  I\right)  \cap\operatorname*{span}\left(
\left(  w\right)  _{w\in S_{n};\ w\leq y}\right)  \subseteq
\operatorname*{span}\left(  \left(  a_{w}\right)  _{w\in S_{n};\ I\subseteq
\operatorname*{Des}w}\right)  $. In view of $\operatorname*{span}\left(
\left(  w\right)  _{w\in S_{n};\ w<x}\right)  =\operatorname*{span}\left(
\left(  w\right)  _{w\in S_{n};\ w\leq y}\right)  $, we can rewrite this as
$Z\left(  I\right)  \cap\operatorname*{span}\left(  \left(  w\right)  _{w\in
S_{n};\ w<x}\right)  \subseteq\operatorname*{span}\left(  \left(
a_{w}\right)  _{w\in S_{n};\ I\subseteq\operatorname*{Des}w}\right)  $. This
completes the proof of (\ref{pf.prop.aw.basis-QI.c2.pf.IH2}).}

Our goal is to prove Claim 2 for $u=x$. In other words, our goal is to prove
that $Z\left(  I\right)  \cap\operatorname*{span}\left(  \left(  w\right)
_{w\in S_{n};\ w\leq x}\right)  \subseteq\operatorname*{span}\left(  \left(
a_{w}\right)  _{w\in S_{n};\ I\subseteq\operatorname*{Des}w}\right)  $.

To do so, we let $q\in Z\left(  I\right)  \cap\operatorname*{span}\left(
\left(  w\right)  _{w\in S_{n};\ w\leq x}\right)  $. Thus, $q\in Z\left(
I\right)  $ and $q\in\operatorname*{span}\left(  \left(  w\right)  _{w\in
S_{n};\ w\leq x}\right)  $. From $q\in\operatorname*{span}\left(  \left(
w\right)  _{w\in S_{n};\ w\leq x}\right)  $, we see that $q$ is a $\mathbf{k}%
$-linear combination of the family $\left(  w\right)  _{w\in S_{n};\ w\leq x}%
$. Thus,%
\begin{equation}
\left[  w\right]  q=0\ \ \ \ \ \ \ \ \ \ \text{for every }w\in S_{n}\text{
satisfying }w>x. \label{pf.prop.aw.basis-QI.c2.pf.0}%
\end{equation}

We want to show that $q\in\operatorname*{span}\left(  \left(  a_{w}\right)
_{w\in S_{n};\ I\subseteq\operatorname*{Des}w}\right)  $.

We are in one of the following two cases:

\textit{Case 1:} We have $I\not \subseteq \operatorname*{Des}x$.

\textit{Case 2:} We have $I\subseteq\operatorname*{Des}x$.

First, let us consider Case 1. In this case, we have $I\not \subseteq
\operatorname*{Des}x$. Hence, there exists some $k\in I$ such that
$k\notin\operatorname*{Des}x$. Consider this $k$. Then, $k\in I\subseteq
\left[  n-1\right]  $. Hence, if we had $x\left(  k\right)  >x\left(
k+1\right)  $, then we would have $k\in\operatorname*{Des}x$ (by the
definition of $\operatorname*{Des}x$), which would contradict $k\notin%
\operatorname*{Des}x$. Thus, we cannot have $x\left(  k\right)  >x\left(
k+1\right)  $. Hence, we have $x\left(  k\right)  \leq x\left(  k+1\right)  $.
Since $x\left(  k\right)  \neq x\left(  k+1\right)  $ (because $x$ is a
permutation), we thus find $x\left(  k\right)  <x\left(  k+1\right)  $. Hence,
it is easy to see that $xs_{k}>x$\ \ \ \ \footnote{\textit{Proof.} Let
$y:=xs_{k}$. Then, recalling how $s_{k}$ was defined, we see that all values
of $y$ are equal to the corresponding values of $x$ except for the values at
$k$ and $k+1$, which are swapped. In other words,%
\begin{align*}
&  \left(  y\left(  1\right)  ,y\left(  2\right)  ,\ldots,y\left(  k-1\right)
,y\left(  k\right)  ,y\left(  k+1\right)  ,y\left(  k+2\right)  ,\ldots
,y\left(  n\right)  \right) \\
&  =\left(  x\left(  1\right)  ,x\left(  2\right)  ,\ldots,x\left(
k-1\right)  ,x\left(  k+1\right)  ,x\left(  k\right)  ,x\left(  k+2\right)
,\ldots,x\left(  n\right)  \right)  .
\end{align*}
Thus, the smallest $i\in\left[  n\right]  $ satisfying $x\left(  i\right)
\neq y\left(  i\right)  $ is $k$, and this smallest $i$ satisfies $x\left(
i\right)  <y\left(  i\right)  $ (since we have $x\left(  k\right)  <x\left(
k+1\right)  =y\left(  k\right)  $). Therefore, the definition of lexicographic
order shows that $x<y$. Hence, $x<y=xs_{k}$, so that $xs_{k}>x$.}. Thus,
(\ref{pf.prop.aw.basis-QI.c2.pf.0}) (applied to $w=xs_{k}$) yields $\left[
xs_{k}\right]  q=0$.

On the other hand, $q\in Z\left(  I\right)  $, and therefore $qs_{i}=q$ for
all $i\in I$ (by the definition of $Z\left(  I\right)  $). Applying this to
$i=k$, we obtain $qs_{k}=q$ (since $k\in I$). However,
(\ref{pf.prop.aw.basis-QI.coeff1}) (applied to $w=x$ and $\sigma=s_{k}$)
yields%
\begin{align*}
\left[  x\right]  \left(  qs_{k}\right)   &  =\left[  xs_{k}^{-1}\right]
q=\left[  xs_{k}\right]  q\ \ \ \ \ \ \ \ \ \ \left(  \text{since }s_{k}%
^{-1}=s_{k}\right) \\
&  =0.
\end{align*}
In view of $qs_{k}=q$, this rewrites as $\left[  x\right]  q=0$. In other
words, $\left[  w\right]  q=0$ holds for $w=x$. Combining this with
(\ref{pf.prop.aw.basis-QI.c2.pf.0}), we obtain
\begin{equation}
\left[  w\right]  q=0\ \ \ \ \ \ \ \ \ \ \text{for every }w\in S_{n}\text{
satisfying }w\geq x. \label{pf.prop.aw.basis-QI.c2.pf.c1.1}%
\end{equation}
Hence, $q\in\operatorname*{span}\left(  \left(  w\right)  _{w\in S_{n}%
;\ w<x}\right)  $. Combining this with $q\in Z\left(  I\right)  $, we obtain%
\[
q\in Z\left(  I\right)  \cap\operatorname*{span}\left(  \left(  w\right)
_{w\in S_{n};\ w<x}\right)  \subseteq\operatorname*{span}\left(  \left(
a_{w}\right)  _{w\in S_{n};\ I\subseteq\operatorname*{Des}w}\right)
\]
(by (\ref{pf.prop.aw.basis-QI.c2.pf.IH2})). Hence, we have proved that
$q\in\operatorname*{span}\left(  \left(  a_{w}\right)  _{w\in S_{n}%
;\ I\subseteq\operatorname*{Des}w}\right)  $ in Case 1.

Let us next consider Case 2. In this case, we have $I\subseteq
\operatorname*{Des}x$. Hence, $a_{x}\in Z\left(  I\right)  $ (by Claim 1,
applied to $w=x$). Moreover, $a_{x}$ is an element of the family $\left(
a_{w}\right)  _{w\in S_{n};\ I\subseteq\operatorname*{Des}w}$ (since $x\in
S_{n}$ satisfies $I\subseteq\operatorname*{Des}x$). Hence, $a_{x}%
\in\operatorname*{span}\left(  \left(  a_{w}\right)  _{w\in S_{n}%
;\ I\subseteq\operatorname*{Des}w}\right)  $.

Let $\lambda:=\left[  x\right]  q$. Let $r:=q-\lambda a_{x}\in\mathbf{k}%
\left[  S_{n}\right]  $. Then, $r\in Z\left(  I\right)  $ (since $Z\left(
I\right)  $ is a $\mathbf{k}$-module, and since both $q$ and $a_{x}$ belong to
$Z\left(  I\right)  $). Moreover, Corollary \ref{cor.aw.lead} (applied to
$w=x$) yields%
\[
a_{x}=x+\left(  \text{a sum of permutations }v\in S_{n}\text{ satisfying
}v<x\right)  .
\]
Hence, $\left[  x\right]  \left(  a_{x}\right)  =1$ and%
\begin{equation}
\left[  w\right]  \left(  a_{x}\right)  =0\ \ \ \ \ \ \ \ \ \ \text{for each
}w\in S_{n}\text{ satisfying }w>x. \label{pf.prop.aw.basis-QI.c2.pf.c2.2}%
\end{equation}

Now, from $r=q-\lambda a_{x}$, we obtain
\[
\left[  x\right]  r=\left[  x\right]  \left(  q-\lambda a_{x}\right)  =\left[
x\right]  q-\underbrace{\lambda}_{=\left[  x\right]  q}\cdot
\underbrace{\left[  x\right]  \left(  a_{x}\right)  }_{=1}=\left[  x\right]
q-\left[  x\right]  q=0.
\]
Moreover, for each $w\in S_{n}$ satisfying $w>x$, we have
\begin{align*}
\left[  w\right]  r  &  =\left[  w\right]  \left(  q-\lambda a_{x}\right)
\ \ \ \ \ \ \ \ \ \ \left(  \text{since }r=q-\lambda a_{x}\right) \\
&  =\underbrace{\left[  w\right]  q}_{\substack{=0\\\text{(by
(\ref{pf.prop.aw.basis-QI.c2.pf.0}))}}}-\lambda\cdot\underbrace{\left[
w\right]  \left(  a_{x}\right)  }_{\substack{=0\\\text{(by
(\ref{pf.prop.aw.basis-QI.c2.pf.c2.2}))}}}=0-\lambda\cdot0=0.
\end{align*}
This equality also holds for $w=x$ (since we have just seen that $\left[
x\right]  r=0$). Hence, it holds for all $w\geq x$. Thus, we have shown that
$\left[  w\right]  r=0$ for each $w\in S_{n}$ satisfying $w\geq x$. In other
words, we have $r\in\operatorname*{span}\left(  \left(  w\right)  _{w\in
S_{n};\ w<x}\right)  $. Combining this with $r\in Z\left(  I\right)  $, we
obtain%
\[
r\in Z\left(  I\right)  \cap\operatorname*{span}\left(  \left(  w\right)
_{w\in S_{n};\ w<x}\right)  \subseteq\operatorname*{span}\left(  \left(
a_{w}\right)  _{w\in S_{n};\ I\subseteq\operatorname*{Des}w}\right)
\]
(by (\ref{pf.prop.aw.basis-QI.c2.pf.IH2})). Now, from $r=q-\lambda a_{x}$, we
obtain%
\begin{align*}
q  &  =\underbrace{r}_{\in\operatorname*{span}\left(  \left(  a_{w}\right)
_{w\in S_{n};\ I\subseteq\operatorname*{Des}w}\right)  }+\lambda
\underbrace{a_{x}}_{\in\operatorname*{span}\left(  \left(  a_{w}\right)
_{w\in S_{n};\ I\subseteq\operatorname*{Des}w}\right)  }\\
&  \in\operatorname*{span}\left(  \left(  a_{w}\right)  _{w\in S_{n}%
;\ I\subseteq\operatorname*{Des}w}\right)  +\lambda\operatorname*{span}\left(
\left(  a_{w}\right)  _{w\in S_{n};\ I\subseteq\operatorname*{Des}w}\right) \\
&  \subseteq\operatorname*{span}\left(  \left(  a_{w}\right)  _{w\in
S_{n};\ I\subseteq\operatorname*{Des}w}\right)  \ \ \ \ \ \ \ \ \ \ \left(
\text{since }\operatorname*{span}\left(  \left(  a_{w}\right)  _{w\in
S_{n};\ I\subseteq\operatorname*{Des}w}\right)  \text{ is a }\mathbf{k}%
\text{-module}\right)  .
\end{align*}
Hence, we have proved $q\in\operatorname*{span}\left(  \left(  a_{w}\right)
_{w\in S_{n};\ I\subseteq\operatorname*{Des}w}\right)  $ in Case 2.

Now, we have proved $q\in\operatorname*{span}\left(  \left(  a_{w}\right)
_{w\in S_{n};\ I\subseteq\operatorname*{Des}w}\right)  $ in both Cases 1 and
2. Hence, $q\in\operatorname*{span}\left(  \left(  a_{w}\right)  _{w\in
S_{n};\ I\subseteq\operatorname*{Des}w}\right)  $ always holds.

Forget that we fixed $q$. We thus have shown that $q\in\operatorname*{span}%
\left(  \left(  a_{w}\right)  _{w\in S_{n};\ I\subseteq\operatorname*{Des}%
w}\right)  $ for each $q\in Z\left(  I\right)  \cap\operatorname*{span}\left(
\left(  w\right)  _{w\in S_{n};\ w\leq x}\right)  $. In other words, $Z\left(
I\right)  \cap\operatorname*{span}\left(  \left(  w\right)  _{w\in
S_{n};\ w\leq x}\right)  \subseteq\operatorname*{span}\left(  \left(
a_{w}\right)  _{w\in S_{n};\ I\subseteq\operatorname*{Des}w}\right)  $. In
other words, we have proved Claim 2 for $u=x$. This completes the induction
step. Thus, Claim 2 is proven.] \medskip

Now, it is easy to see that the family $\left(  a_{w}\right)  _{w\in
S_{n};\ I\subseteq\operatorname*{Des}w}$ spans the $\mathbf{k}$-module
$Z\left(  I\right)  $\ \ \ \ \footnote{\textit{Proof.} Let $u$ be the largest
permutation in $S_{n}$ (with respect to the lexicographic order). Thus, every
$w\in S_{n}$ satisfies $w\leq u$.
\par
Let $q\in Z\left(  I\right)  $. Then, $q\in Z\left(  I\right)  \subseteq
\mathbf{k}\left[  S_{n}\right]  =\operatorname*{span}\left(  \left(  w\right)
_{w\in S_{n}}\right)  $ (since the family $\left(  w\right)  _{w\in S_{n}}$ is
a basis of the $\mathbf{k}$-module $\mathbf{k}\left[  S_{n}\right]  $).
However, the family $\left(  w\right)  _{w\in S_{n}}$ is the same as the
family $\left(  w\right)  _{w\in S_{n};\ w\leq u}$ (since every $w\in S_{n}$
satisfies $w\leq u$). Hence, $q\in\operatorname*{span}\left(  \left(
w\right)  _{w\in S_{n};\ w\leq u}\right)  $ (since $q\in\operatorname*{span}%
\left(  \left(  w\right)  _{w\in S_{n}}\right)  $). Combining this with $q\in
Z\left(  I\right)  $, we obtain $q\in Z\left(  I\right)  \cap
\operatorname*{span}\left(  \left(  w\right)  _{w\in S_{n};\ w\leq u}\right)
\subseteq\operatorname*{span}\left(  \left(  a_{w}\right)  _{w\in
S_{n};\ I\subseteq\operatorname*{Des}w}\right)  $ (by Claim 2).
\par
Forget that we fixed $q$. We thus have shown that each $q\in Z\left(
I\right)  $ satisfies $q\in\operatorname*{span}\left(  \left(  a_{w}\right)
_{w\in S_{n};\ I\subseteq\operatorname*{Des}w}\right)  $. In other words,
$Z\left(  I\right)  \subseteq\operatorname*{span}\left(  \left(  a_{w}\right)
_{w\in S_{n};\ I\subseteq\operatorname*{Des}w}\right)  $. In other words, the
family $\left(  a_{w}\right)  _{w\in S_{n};\ I\subseteq\operatorname*{Des}w}$
spans the $\mathbf{k}$-module $Z\left(  I\right)  $ (since Claim 1 shows that
this family is a family of vectors in $Z\left(  I\right)  $). Qed.}. We shall
now show that this family is $\mathbf{k}$-linearly independent. Slightly
better, we will show that the family $\left(  a_{w}\right)  _{w\in S_{n}}$ is
$\mathbf{k}$-linearly independent:

\begin{statement}
\textit{Claim 3:} Let $\left(  \lambda_{w}\right)  _{w\in S_{n}}$ be a family
of elements of $\mathbf{k}$ such that $\sum_{w\in S_{n}}\lambda_{w}a_{w}=0$.
Then, $\lambda_{w}=0$ for each $w\in S_{n}$.
\end{statement}

[\textit{Proof of Claim 3:} This follows by a straightforward triangularity
argument (where the triangularity is provided by Corollary \ref{cor.aw.lead}).
Purely for the sake of completeness, we present the argument in full:

We must prove that%
\begin{equation}
\lambda_{w}=0\ \ \ \ \ \ \ \ \ \ \text{for each }w\in S_{n}.
\label{pf.prop.aw.basis-QI.c3.pf.1}%
\end{equation}

In order to prove (\ref{pf.prop.aw.basis-QI.c3.pf.1}), we proceed by strong
induction on $w$, but this time we use the \textbf{reverse} of the
lexicographic order on $S_{n}$ as our well-ordering. Thus, we fix some $x\in
S_{n}$, and we assume (as the induction hypothesis) that
(\ref{pf.prop.aw.basis-QI.c3.pf.1}) has already been proved for each $w>x$
(not for each $w<x$ as in our previous induction proof). Our goal is then to
prove that (\ref{pf.prop.aw.basis-QI.c3.pf.1}) holds for $w=x$. In other
words, our goal is to prove that $\lambda_{x}=0$.

The induction hypothesis yields that (\ref{pf.prop.aw.basis-QI.c3.pf.1}) holds
for each $w>x$. In other words, $\lambda_{w}=0$ for each $w\in S_{n}$
satisfying $w>x$. Hence, $\sum_{\substack{w\in S_{n};\\w>x}%
}\underbrace{\lambda_{w}}_{=0}a_{w}=\sum_{\substack{w\in S_{n};\\w>x}%
}0a_{w}=0$. However, each $w\in S_{n}$ satisfies exactly one of the three
statements $w<x$ and $w=x$ and $w>x$. Hence, we can split the sum $\sum_{w\in
S_{n}}\lambda_{w}a_{w}$ as follows:%
\[
\sum_{w\in S_{n}}\lambda_{w}a_{w}=\sum_{\substack{w\in S_{n};\\w<x}%
}\lambda_{w}a_{w}+\underbrace{\sum_{\substack{w\in S_{n};\\w=x}}\lambda
_{w}a_{w}}_{=\lambda_{x}a_{x}}+\underbrace{\sum_{\substack{w\in S_{n}%
;\\w>x}}\lambda_{w}a_{w}}_{=0}=\sum_{\substack{w\in S_{n};\\w<x}}\lambda
_{w}a_{w}+\lambda_{x}a_{x}.
\]
Comparing this with $\sum_{w\in S_{n}}\lambda_{w}a_{w}=0$, we obtain%
\[
0=\sum_{\substack{w\in S_{n};\\w<x}}\lambda_{w}a_{w}+\lambda_{x}a_{x}.
\]
Taking the $x$-coefficients on both sides of this equality, we obtain%
\begin{align}
\left[  x\right]  0  &  =\left[  x\right]  \left(  \sum_{\substack{w\in
S_{n};\\w<x}}\lambda_{w}a_{w}+\lambda_{x}a_{x}\right) \nonumber\\
&  =\sum_{\substack{w\in S_{n};\\w<x}}\lambda_{w}\cdot\left[  x\right]
\left(  a_{w}\right)  +\lambda_{x}\cdot\left[  x\right]  \left(  a_{x}\right)
. \label{pf.prop.aw.basis-QI.c3.pf.4}%
\end{align}

Now, let $w\in S_{n}$ be such that $w<x$. Then, $x>w$. However, Corollary
\ref{cor.aw.lead} yields
\[
a_{w}=w+\left(  \text{a sum of permutations }v\in S_{n}\text{ satisfying
}v<w\right)  .
\]
Hence, $\left[  y\right]  \left(  a_{w}\right)  =0$ for all $y\in S_{n}$
satisfying $y>w$. Applying this to $y=x$, we obtain $\left[  x\right]  \left(
a_{w}\right)  =0$ (since $x>w$).

Forget that we fixed $w$. We thus have shown that%
\begin{equation}
\left[  x\right]  \left(  a_{w}\right)  =0\ \ \ \ \ \ \ \ \ \ \text{for each
}w\in S_{n}\text{ satisfying }w<x. \label{pf.prop.aw.basis-QI.c3.pf.3}%
\end{equation}

Also, Corollary \ref{cor.aw.lead} (applied to $w=x$) yields%
\[
a_{x}=x+\left(  \text{a sum of permutations }v\in S_{n}\text{ satisfying
}v<x\right)  .
\]
Hence, $\left[  x\right]  \left(  a_{x}\right)  =1$. Now,
\begin{align*}
0  &  =\left[  x\right]  0=\sum_{\substack{w\in S_{n};\\w<x}}\lambda_{w}%
\cdot\underbrace{\left[  x\right]  \left(  a_{w}\right)  }%
_{\substack{=0\\\text{(by (\ref{pf.prop.aw.basis-QI.c3.pf.3}))}}}+\lambda
_{x}\cdot\underbrace{\left[  x\right]  \left(  a_{x}\right)  }_{=1}%
\ \ \ \ \ \ \ \ \ \ \left(  \text{by (\ref{pf.prop.aw.basis-QI.c3.pf.4}%
)}\right) \\
&  =\underbrace{\sum_{\substack{w\in S_{n};\\w<x}}\lambda_{w}\cdot0}%
_{=0}+\lambda_{x}=\lambda_{x}.
\end{align*}
Thus, $\lambda_{x}=0$. In other words, (\ref{pf.prop.aw.basis-QI.c3.pf.1})
holds for $w=x$. This completes the induction step. Thus,
(\ref{pf.prop.aw.basis-QI.c3.pf.1}) is proved, and Claim 3 follows.] \medskip

Now, we have proved Claim 3. In other words, we have proved that the family
$\left(  a_{w}\right)  _{w\in S_{n}}$ is $\mathbf{k}$-linearly independent.
Hence, its subfamily $\left(  a_{w}\right)  _{w\in S_{n};\ I\subseteq
\operatorname*{Des}w}$ is $\mathbf{k}$-linearly independent as well (since a
subfamily of a $\mathbf{k}$-linearly independent family must itself be
$\mathbf{k}$-linearly independent family). Since we also know that this
subfamily spans the $\mathbf{k}$-module $Z\left(  I\right)  $, we thus
conclude that this subfamily is a basis of $Z\left(  I\right)  $. This proves
Proposition \ref{prop.aw.basis-QI}.
\end{proof}

\begin{proof}
[Proof of Proposition \ref{prop.aw.basis-kSn}.]The definition of $Z\left(
\varnothing\right)  $ yields%
\[
Z\left(  \varnothing\right)  =\left\{  q\in\mathbf{k}\left[  S_{n}\right]
\ \mid\ qs_{i}=q\text{ for all }i\in\varnothing\right\}  =\mathbf{k}\left[
S_{n}\right]
\]
(because the statement \textquotedblleft$qs_{i}=q$ for all $i\in\varnothing
$\textquotedblright\ is vacuously true for each $q\in\mathbf{k}\left[
S_{n}\right]  $). However, Proposition \ref{prop.aw.basis-QI} (applied to
$I=\varnothing$) yields that the family $\left(  a_{w}\right)  _{w\in
S_{n};\ \varnothing\subseteq\operatorname*{Des}w}$ is a basis of the
$\mathbf{k}$-module $Z\left(  \varnothing\right)  $. Since the family $\left(
a_{w}\right)  _{w\in S_{n};\ \varnothing\subseteq\operatorname*{Des}w}$ is
nothing other than the family $\left(  a_{w}\right)  _{w\in S_{n}}$ (because
the statement \textquotedblleft$\varnothing\subseteq\operatorname*{Des}%
w$\textquotedblright\ holds for each $w\in S_{n}$), we can rewrite this as
follows: The family $\left(  a_{w}\right)  _{w\in S_{n}}$ is a basis of the
$\mathbf{k}$-module $Z\left(  \varnothing\right)  $. In other words, the
family $\left(  a_{w}\right)  _{w\in S_{n}}$ is a basis of the $\mathbf{k}%
$-module $\mathbf{k}\left[  S_{n}\right]  $ (since $Z\left(  \varnothing
\right)  =\mathbf{k}\left[  S_{n}\right]  $). This proves Proposition
\ref{prop.aw.basis-kSn}.
\end{proof}

\begin{proof}
[Proof of Proposition \ref{prop.aw.basis-FI}.]Let $I\subseteq\left[  n\right]
$. Then, Proposition \ref{prop.aw.QI=FI'} yields $F\left(  I\right)  =Z\left(
I^{\prime}\right)  $.

However, Proposition \ref{prop.aw.basis-QI} (applied to $I^{\prime}$ instead
of $I$) yields that the family $\left(  a_{w}\right)  _{w\in S_{n}%
;\ I^{\prime}\subseteq\operatorname*{Des}w}$ is a basis of the $\mathbf{k}%
$-module $Z\left(  I^{\prime}\right)  $. Since $F\left(  I\right)  =Z\left(
I^{\prime}\right)  $, we can rewrite this as follows: The family $\left(
a_{w}\right)  _{w\in S_{n};\ I^{\prime}\subseteq\operatorname*{Des}w}$ is a
basis of the $\mathbf{k}$-module $F\left(  I\right)  $. This proves
Proposition \ref{prop.aw.basis-FI}.
\end{proof}

We refer to the basis $\left(  a_{w}\right)  _{w\in S_{n}}$ of $\mathbf{k}%
\left[  S_{n}\right]  $ as the \emph{descent-destroying basis}, due to how
$a_{w}$ is defined in terms of \textquotedblleft removing\textquotedblright%
\ descents from $w$. As with any basis, we can ask the following rather
natural question about it:

\begin{question}
How can we explicitly expand a permutation $v\in S_{n}$ in the basis $\left(
a_{w}\right)  _{w\in S_{n}}$ of $\mathbf{k}\left[  S_{n}\right]  $ ?
\end{question}

\begin{example}
For this example, let $n=4$. We write each permutation $w\in S_{4}$ as the
list $\left[  w\left(  1\right)  \ w\left(  2\right)  \ w\left(  3\right)
\ w\left(  4\right)  \right]  $ (written without commas for brevity, and using
square brackets to distinguish it from a parenthesized integer). Then,%
\[
\left[  3412\right]  =a_{\left[  1234\right]  }-a_{\left[  1324\right]
}+a_{\left[  1342\right]  }+a_{\left[  3124\right]  }-a_{\left[  3142\right]
}+a_{\left[  3412\right]  }.
\]

\end{example}

We note that it is \textbf{not} generally true that when we express a
permutation $v\in S_{n}$ as a $\mathbf{k}$-linear combination of the basis
$\left(  a_{w}\right)  _{w\in S_{n}}$, all coefficients will belong to
$\left\{  0,1,-1\right\}  $. However, the smallest $n$ for which this is not
the case is $n=8$, which suggests that the coefficients are not too complicated.

\section{$Q$-indices and bases of $F_{i}$}

\subsection{Definition}

We can now use our basis $\left(  a_{w}\right)  _{w\in S_{n}}$ and its
subfamilies $\left(  a_{w}\right)  _{w\in S_{n};\ I^{\prime}\subseteq
\operatorname*{Des}w}$ to obtain a basis for each piece $F_{i}$ of the
filtration $F_{0}\subseteq F_{1}\subseteq F_{2}\subseteq\cdots\subseteq
F_{f_{n+1}}$. First, for the sake of convenience, we define a certain
permutation statistic we call the \textquotedblleft$Q$-index\textquotedblright%
. It is worth pointing out that this \textquotedblleft$Q$%
-index\textquotedblright\ will depend on the way how we numbered the lacunar
subsets of $\left[  n-1\right]  $ by $Q_{1},Q_{2},\ldots,Q_{f_{n+1}}$, so it
is not really a natural permutation statistic. We will show in Proposition
\ref{prop.Qind.equivalent}, however, that the assignment of the lacunar set
$Q_{i}$ (where $i$ is the $Q$-index of $w$) to a permutation $w$ is canonical
(i.e., does not depend on the numbering of the lacunar subsets).

First, we prove a lemma:

\begin{lemma}
\label{lem.Qind.exists}Let $w\in S_{n}$. Then, there exists some $i\in\left[
f_{n+1}\right]  $ such that $Q_{i}^{\prime}\subseteq\operatorname*{Des}w$.
\end{lemma}

\begin{proof}
Let $I=\left\{  j\in\left[  n-1\right]  \ \mid\ j\equiv n-1\operatorname{mod}%
2\right\}  $. Then, $I$ is a lacunar subset of $\left[  n-1\right]  $ (in
fact, $I$ is lacunar since all elements of $I$ have the same parity). Thus,
there exists some $i\in\left[  f_{n+1}\right]  $ such that $I=Q_{i}$ (since
$Q_{1},Q_{2},\ldots,Q_{f_{n+1}}$ are all lacunar subsets of $\left[
n-1\right]  $). Consider this $i$. We shall show that $Q_{i}^{\prime}%
\subseteq\operatorname*{Des}w$.

The definition of $I$ yields that each element of $\left[  n-1\right]  $ is
either in $I$ (if it has the same parity as $n-1$) or in $I-1$ (if it has
not). In other words, $\left[  n-1\right]  \subseteq I\cup\left(  I-1\right)
$. The definition of $I^{\prime}$ yields $I^{\prime}=\left[  n-1\right]
\setminus\left(  I\cup\left(  I-1\right)  \right)  =\varnothing$ (since
$\left[  n-1\right]  \subseteq I\cup\left(  I-1\right)  $). In view of
$I=Q_{i}$, this rewrites as $Q_{i}^{\prime}=\varnothing$. Hence,
$Q_{i}^{\prime}=\varnothing\subseteq\operatorname*{Des}w$. This proves Lemma
\ref{lem.Qind.exists}.
\end{proof}

Now, we can define the $Q$-index:

\begin{itemize}
\item If $w\in S_{n}$ is any permutation, then the $Q$\emph{-index} of $w$ is
defined to be the \textbf{smallest} $i\in\left[  f_{n+1}\right]  $ such that
$Q_{i}^{\prime}\subseteq\operatorname*{Des}w$. (This is well-defined, because
Lemma \ref{lem.Qind.exists} shows that such an $i$ exists.) We denote the
$Q$-index of $w$ by $\operatorname*{Qind}w$.
\end{itemize}

\begin{example}
For this example, let $n=4$. Recall Example \ref{exa.Qi.4}, in which we listed
all the lacunar subsets of $\left[  3\right]  $ in order. Let $w\in S_{n}$ be
the permutation such that $\left(  w\left(  1\right)  ,w\left(  2\right)
,\ldots,w\left(  n\right)  \right)  =\left(  4,3,1,2\right)  $. Then,
$\operatorname*{Des}w=\left\{  1,2\right\}  $. Hence, $Q_{4}^{\prime}=\left\{
1\right\}  \subseteq\operatorname*{Des}w$, but it is easy to see that
$Q_{i}^{\prime}\not \subseteq \operatorname*{Des}w$ for all $i<4$. Hence, the
smallest $i\in\left[  f_{n+1}\right]  $ such that $Q_{i}^{\prime}%
\subseteq\operatorname*{Des}w$ is $4$. In other words, the $Q$-index of $w$ is
$4$. In other words, $\operatorname*{Qind}w=4$.
\end{example}

\subsection{An equivalent description}

As we said, the $Q$-index of a permutation $w\in S_{n}$ depends on the
ordering of $Q_{1},Q_{2},\ldots,Q_{f_{n+1}}$. However, the dependence is not
as strong as it might appear from the definition; indeed, we have the
following alternative characterization:

\begin{proposition}
\label{prop.Qind.equivalent}Let $w\in S_{n}$ and $i\in\left[  f_{n+1}\right]
$. Then, $\operatorname*{Qind}w=i$ if and only if $Q_{i}^{\prime}%
\subseteq\operatorname*{Des}w\subseteq\left[  n-1\right]  \setminus Q_{i}$.
\end{proposition}

Before we prove this proposition, we need two further lemmas about lacunar subsets:

\begin{lemma}
\label{lem.lac-sum-less2}Let $I$ and $K$ be two subsets of $\left[
n-1\right]  $ such that $I$ is lacunar and $K\neq I$ and $K^{\prime}%
\subseteq\left[  n-1\right]  \setminus I$. Then, $\operatorname*{sum}%
I<\operatorname*{sum}K$.
\end{lemma}

\begin{proof}
[Proof of Lemma \ref{lem.lac-sum-less2}.]First, we observe that $I\setminus
K\subseteq\left(  K\setminus I\right)  -1$.

[\textit{Proof:} Let $i\in I\setminus K$. Thus, $i\in I$ and $i\notin K$.

If we had $i+1\notin K$, then we would have $i\in K^{\prime}$ (since $i\in
I\subseteq\left[  n-1\right]  $ and $i\notin K$ and $i+1\notin K$), which
would entail $i\in K^{\prime}\subseteq\left[  n-1\right]  \setminus I$; but
this would contradict $i\in I$. Thus, we cannot have $i+1\notin K$. In other
words, we have $i+1\in K$. Furthermore, $I$ is lacunar; thus, from $i\in I$,
we obtain $i+1\notin I$. Combining this with $i+1\in K$, we find $i+1\in
K\setminus I$. Hence, $i\in\left(  K\setminus I\right)  -1$.

Forget that we fixed $i$. We thus have proved that $i\in\left(  K\setminus
I\right)  -1$ for each $i\in I\setminus K$. In other words, $I\setminus
K\subseteq\left(  K\setminus I\right)  -1$.] \medskip

Now, the set $I$ is the union of its two disjoint subsets $I\setminus K$ and
$I\cap K$. Hence,%
\begin{equation}
\operatorname*{sum}I=\operatorname*{sum}\left(  I\setminus K\right)
+\operatorname*{sum}\left(  I\cap K\right)  .
\label{pf.lem.lac-sum-less2.sumI=}%
\end{equation}
The same argument (with the roles of $I$ and $K$ swapped) yields%
\begin{equation}
\operatorname*{sum}K=\operatorname*{sum}\left(  K\setminus I\right)
+\operatorname*{sum}\left(  K\cap I\right)  .
\label{pf.lem.lac-sum-less2.sumK=}%
\end{equation}

Our goal is to prove that $\operatorname*{sum}I<\operatorname*{sum}K$. If
$I\subseteq K$, then this is obvious (since we have $K\neq I$, so that $I$
must be a \textbf{proper} subset of $K$ in this case). Thus, we WLOG assume
that $I\not \subseteq K$ from now on. Hence, $I\setminus K\neq\varnothing$. In
view of $I\setminus K\subseteq\left(  K\setminus I\right)  -1$, this entails
$\left(  K\setminus I\right)  -1\neq\varnothing$, so that $K\setminus
I\neq\varnothing$. Hence, $\left\vert K\setminus I\right\vert >0$.

Now, from $I\setminus K\subseteq\left(  K\setminus I\right)  -1$, we obtain
\[
\operatorname*{sum}\left(  I\setminus K\right)  \leq\operatorname*{sum}\left(
\left(  K\setminus I\right)  -1\right)  =\operatorname*{sum}\left(  K\setminus
I\right)  -\underbrace{\left\vert K\setminus I\right\vert }_{>0}%
<\operatorname*{sum}\left(  K\setminus I\right)  .
\]

However, (\ref{pf.lem.lac-sum-less2.sumI=}) becomes%
\[
\operatorname*{sum}I=\underbrace{\operatorname*{sum}\left(  I\setminus
K\right)  }_{<\operatorname*{sum}\left(  K\setminus I\right)  }%
+\operatorname*{sum}\left(  \underbrace{I\cap K}_{=K\cap I}\right)
<\operatorname*{sum}\left(  K\setminus I\right)  +\operatorname*{sum}\left(
K\cap I\right)  =\operatorname*{sum}K
\]
(by (\ref{pf.lem.lac-sum-less2.sumK=})). This proves Lemma
\ref{lem.lac-sum-less2}.
\end{proof}

\begin{lemma}
\label{lem.lac-IR}Let $I$ be a subset of $\left[  n\right]  $. Let $j\in I$.
Then, there exists a lacunar subset $K$ of $\left[  n-1\right]  $ satisfying
$\operatorname*{sum}K<\operatorname*{sum}I$ and $K^{\prime}\subseteq
I^{\prime}\cup\left\{  j\right\}  $.
\end{lemma}

\begin{proof}
Set $R:=\left(  I\setminus\left\{  j\right\}  \right)  \cup\left\{
j-1\right\}  $ if $j>1$, and otherwise set $R:=I\setminus\left\{  j\right\}
$. Thus, the set $R$ is obtained from $I$ by replacing the element $j$ (which
was in $I$, because $j\in I$) by the smaller element $j-1$ (unless $j=1$, in
which case $j$ is just removed). In either case, we therefore have
$\operatorname*{sum}R<\operatorname*{sum}I$. Also, it is easy to see that
$R\subseteq\left[  n\right]  $ and $R^{\prime}\subseteq I^{\prime}\cup\left\{
j\right\}  $ (by Proposition \ref{prop.K'subI'} \textbf{(a)}, applied to
$K=R$). Thus, Corollary \ref{cor.FI-lac-2} (applied to $R$ instead of $I$)
yields that there exists a lacunar subset $J$ of $\left[  n-1\right]  $ such
that $\operatorname*{sum}J\leq\operatorname*{sum}R$ and $J^{\prime}\subseteq
R^{\prime}$. Consider this $J$. Then, $\operatorname*{sum}J\leq
\operatorname*{sum}R<\operatorname*{sum}I$ and $J^{\prime}\subseteq R^{\prime
}\subseteq I^{\prime}\cup\left\{  j\right\}  $. Hence, there exists a lacunar
subset $K$ of $\left[  n-1\right]  $ satisfying $\operatorname*{sum}%
K<\operatorname*{sum}I$ and $K^{\prime}\subseteq I^{\prime}\cup\left\{
j\right\}  $ (namely, $K=J$). This proves Lemma \ref{lem.lac-IR}.
\end{proof}

\begin{proof}
[Proof of Proposition \ref{prop.Qind.equivalent}.]$\Longrightarrow:$ Assume
that $\operatorname*{Qind}w=i$.\ We must prove that $Q_{i}^{\prime}%
\subseteq\operatorname*{Des}w\subseteq\left[  n-1\right]  \setminus Q_{i}$.

In view of the definition of the $Q$-index, our assumption
$\operatorname*{Qind}w=i$ means that $Q_{i}^{\prime}\subseteq
\operatorname*{Des}w$ and that $i$ is the smallest element of $\left[
f_{n+1}\right]  $ with this property. The latter statement means that
\begin{equation}
Q_{k}^{\prime}\not \subseteq \operatorname*{Des}w\ \ \ \ \ \ \ \ \ \ \text{for
each }k<i. \label{pf.lem.Qind.equivalent.to.1}%
\end{equation}

Now, let $j\in\left(  \operatorname*{Des}w\right)  \cap Q_{i}$. We shall
derive a contradiction.

Indeed, we have $j\in\left(  \operatorname*{Des}w\right)  \cap Q_{i}\subseteq
Q_{i}$. Hence, Lemma \ref{lem.lac-IR} (applied to $I=Q_{i}$) shows that there
exists a lacunar subset $K$ of $\left[  n-1\right]  $ satisfying
$\operatorname*{sum}K<\operatorname*{sum}\left(  Q_{i}\right)  $ and
$K^{\prime}\subseteq Q_{i}^{\prime}\cup\left\{  j\right\}  $. Consider this
$K$. Since $K$ is a lacunar subset of $\left[  n-1\right]  $, we have
$K=Q_{k}$ for some $k\in\left[  f_{n+1}\right]  $ (since the lacunar subsets
of $\left[  n-1\right]  $ are $Q_{1},Q_{2},\ldots,Q_{f_{n+1}}$). Consider this
$k$. Thus, $Q_{k}=K$, so that $\operatorname*{sum}\left(  Q_{k}\right)
=\operatorname*{sum}K<\operatorname*{sum}\left(  Q_{i}\right)  $. However, if
we had $i\leq k$, then we would have $\operatorname*{sum}\left(  Q_{i}\right)
\leq\operatorname*{sum}\left(  Q_{k}\right)  $ (by
(\ref{pf.thm.t-simultri.sum-order})), which would contradict
$\operatorname*{sum}\left(  Q_{k}\right)  <\operatorname*{sum}\left(
Q_{i}\right)  $. Thus, we cannot have $i\leq k$. Hence, we must have $i>k$, so
that $k<i$. Therefore, (\ref{pf.lem.Qind.equivalent.to.1}) yields
$Q_{k}^{\prime}\not \subseteq \operatorname*{Des}w$. In other words,
$K^{\prime}\not \subseteq \operatorname*{Des}w$ (since $Q_{k}=K$).

However, $K^{\prime}\subseteq\underbrace{Q_{i}^{\prime}}_{\subseteq
\operatorname*{Des}w}\cup\left\{  j\right\}  \subseteq\left(
\operatorname*{Des}w\right)  \cup\left\{  j\right\}  =\operatorname*{Des}w$
(since $j\in\left(  \operatorname*{Des}w\right)  \cap Q_{i}\subseteq
\operatorname*{Des}w$). This contradicts $K^{\prime}\not \subseteq
\operatorname*{Des}w$.

Forget that we fixed $j$. We thus have obtained a contradiction for each
$j\in\left(  \operatorname*{Des}w\right)  \cap Q_{i}$. Hence, there exists no
such $j$. In other words, the set $\left(  \operatorname*{Des}w\right)  \cap
Q_{i}$ is empty. In other words, $\operatorname*{Des}w$ is disjoint from
$Q_{i}$. Hence, $\operatorname*{Des}w\subseteq\left[  n-1\right]  \setminus
Q_{i}$ (since $\operatorname*{Des}w\subseteq\left[  n-1\right]  $). Combining
this with $Q_{i}^{\prime}\subseteq\operatorname*{Des}w$, we obtain
$Q_{i}^{\prime}\subseteq\operatorname*{Des}w\subseteq\left[  n-1\right]
\setminus Q_{i}$. Thus, we have proved the \textquotedblleft$\Longrightarrow
$\textquotedblright\ direction of Proposition \ref{prop.Qind.equivalent}.
\medskip

$\Longleftarrow:$ Assume that $Q_{i}^{\prime}\subseteq\operatorname*{Des}%
w\subseteq\left[  n-1\right]  \setminus Q_{i}$. We must prove that
$\operatorname*{Qind}w=i$.

We shall show that $Q_{k}^{\prime}\not \subseteq \operatorname*{Des}w$ for
each $k<i$. Indeed, let us fix a positive integer $k<i$. Thus,
$\operatorname*{sum}\left(  Q_{k}\right)  \leq\operatorname*{sum}\left(
Q_{i}\right)  $ (by (\ref{pf.thm.t-simultri.sum-order})). Also, from $k<i$, we
obtain $Q_{k}\neq Q_{i}$ (since the sets $Q_{1},Q_{2},\ldots,Q_{f_{n+1}}$ are
distinct). Also, the set $Q_{i}$ is lacunar (since the sets $Q_{1}%
,Q_{2},\ldots,Q_{f_{n+1}}$ are lacunar).

Now, assume (for the sake of contradiction) that $Q_{k}^{\prime}%
\subseteq\operatorname*{Des}w$. Then, $Q_{k}^{\prime}\subseteq
\operatorname*{Des}w\subseteq\left[  n-1\right]  \setminus Q_{i}$. Therefore,
Lemma \ref{lem.lac-sum-less2} (applied to $I=Q_{i}$ and $K=Q_{k}$) yields
$\operatorname*{sum}\left(  Q_{i}\right)  <\operatorname*{sum}\left(
Q_{k}\right)  $. This contradicts $\operatorname*{sum}\left(  Q_{k}\right)
\leq\operatorname*{sum}\left(  Q_{i}\right)  $. This contradiction shows that
our assumption (that $Q_{k}^{\prime}\subseteq\operatorname*{Des}w$) was false.
Hence, we have $Q_{k}^{\prime}\not \subseteq \operatorname*{Des}w$.

Forget that we fixed $k$. We thus have shown that $Q_{k}^{\prime
}\not \subseteq \operatorname*{Des}w$ for each $k<i$. Since we also know that
$Q_{i}^{\prime}\subseteq\operatorname*{Des}w$ (by assumption), we thus
conclude that $i$ is the \textbf{smallest} element of $\left[  f_{n+1}\right]
$ such that $Q_{i}^{\prime}\subseteq\operatorname*{Des}w$. In other words, $i$
is the $Q$-index of $w$ (since this is how the $Q$-index of $w$ is defined).
In other words, $i=\operatorname*{Qind}w$. That is, $\operatorname*{Qind}w=i$.
Thus, we have proved the \textquotedblleft$\Longleftarrow$\textquotedblright%
\ direction of Proposition \ref{prop.Qind.equivalent}.
\end{proof}

\subsection{Bases of the $F_{i}$ and $F_{i}/F_{i-1}$}

\begin{theorem}
\label{thm.aw.freeness}Recall the $\mathbf{k}$-module filtration
$0=F_{0}\subseteq F_{1}\subseteq F_{2}\subseteq\cdots\subseteq F_{f_{n+1}%
}=\mathbf{k}\left[  S_{n}\right]  $ from Theorem \ref{thm.t-simultri}. Then:

\begin{enumerate}
\item[\textbf{(a)}] For each $i\in\left[  0,f_{n+1}\right]  $, the
$\mathbf{k}$-module $F_{i}$ is free with basis $\left(  a_{w}\right)  _{w\in
S_{n};\ \operatorname*{Qind}w\leq i}$.

\item[\textbf{(b)}] For each $i\in\left[  f_{n+1}\right]  $, the $\mathbf{k}%
$-module $F_{i}/F_{i-1}$ is free with basis $\left(  \overline{a_{w}}\right)
_{w\in S_{n};\ \operatorname*{Qind}w=i}$. Here, $\overline{x}$ denotes the
projection of an element $x\in F_{i}$ onto the quotient $F_{i}/F_{i-1}$.
\end{enumerate}
\end{theorem}

\begin{proof}
\textbf{(a)} Proposition \ref{prop.aw.basis-kSn} yields that the family
$\left(  a_{w}\right)  _{w\in S_{n}}$ is a basis of the $\mathbf{k}$-module
$\mathbf{k}\left[  S_{n}\right]  $. Hence, this family $\left(  a_{w}\right)
_{w\in S_{n}}$ is $\mathbf{k}$-linearly independent.

Let $i\in\left[  0,f_{n+1}\right]  $. For each $k\in\left[  i\right]  $, we
have%
\begin{equation}
F\left(  Q_{k}\right)  =\operatorname*{span}\left(  \left(  a_{w}\right)
_{w\in S_{n};\ Q_{k}^{\prime}\subseteq\operatorname*{Des}w}\right)
\label{pf.thm.aw.freeness.a.1}%
\end{equation}
(since Proposition \ref{prop.aw.basis-FI} (applied to $I=Q_{k}$) shows that
the family $\left(  a_{w}\right)  _{w\in S_{n};\ Q_{k}^{\prime}\subseteq
\operatorname*{Des}w}$ is a basis of the $\mathbf{k}$-module $F\left(
Q_{k}\right)  $). However, the definition of $F_{i}$ yields%
\begin{align}
F_{i}  &  =F\left(  Q_{1}\right)  +F\left(  Q_{2}\right)  +\cdots+F\left(
Q_{i}\right)  =\sum_{k=1}^{i}\underbrace{F\left(  Q_{k}\right)  }%
_{\substack{=\operatorname*{span}\left(  \left(  a_{w}\right)  _{w\in
S_{n};\ Q_{k}^{\prime}\subseteq\operatorname*{Des}w}\right)  \\\text{(by
(\ref{pf.thm.aw.freeness.a.1}))}}}\nonumber\\
&  =\sum_{k=1}^{i}\operatorname*{span}\left(  \left(  a_{w}\right)  _{w\in
S_{n};\ Q_{k}^{\prime}\subseteq\operatorname*{Des}w}\right) \nonumber\\
&  =\operatorname*{span}\left(  \left(  a_{w}\right)  _{w\in S_{n}%
;\ Q_{k}^{\prime}\subseteq\operatorname*{Des}w\text{ for some }k\in\left[
i\right]  }\right)  \label{pf.thm.aw.freeness.a.2}%
\end{align}
(since the sum of the spans of some families of vectors is the span of the
union of these families). However, if $w\in S_{n}$ is a permutation, then the
statement \textquotedblleft$Q_{k}^{\prime}\subseteq\operatorname*{Des}w$ for
some $k\in\left[  i\right]  $\textquotedblright\ is equivalent to the
statement \textquotedblleft$\operatorname*{Qind}w\leq i$\textquotedblright%
\ (since $\operatorname*{Qind}w$ is defined as the \textbf{smallest}
$j\in\left[  f_{n+1}\right]  $ such that $Q_{j}^{\prime}\subseteq
\operatorname*{Des}w$). Thus, the family $\left(  a_{w}\right)  _{w\in
S_{n};\ Q_{k}^{\prime}\subseteq\operatorname*{Des}w\text{ for some }%
k\in\left[  i\right]  }$ is precisely the family $\left(  a_{w}\right)  _{w\in
S_{n};\ \operatorname*{Qind}w\leq i}$. Hence, we can rewrite
(\ref{pf.thm.aw.freeness.a.2}) as follows:%
\[
F_{i}=\operatorname*{span}\left(  \left(  a_{w}\right)  _{w\in S_{n}%
;\ \operatorname*{Qind}w\leq i}\right)  .
\]
In other words, the family $\left(  a_{w}\right)  _{w\in S_{n}%
;\ \operatorname*{Qind}w\leq i}$ spans the $\mathbf{k}$-module $F_{i}$.
Furthermore, this family is $\mathbf{k}$-linearly independent (since it is a
subfamily of the $\mathbf{k}$-linearly independent family $\left(
a_{w}\right)  _{w\in S_{n}}$). Thus, this family is a basis of the
$\mathbf{k}$-module $F_{i}$. In other words, the $\mathbf{k}$-module $F_{i}$
is free with basis $\left(  a_{w}\right)  _{w\in S_{n};\ \operatorname*{Qind}%
w\leq i}$. This proves Theorem \ref{thm.aw.freeness} \textbf{(a)}. \medskip

\textbf{(b)} For each $i\in\left[  0,f_{n+1}\right]  $, we let $A\left(
i\right)  $ denote the set of all permutations $w\in S_{n}$ satisfying
$\operatorname*{Qind}w\leq i$. Clearly, $A\left(  0\right)  \subseteq A\left(
1\right)  \subseteq\cdots\subseteq A\left(  f_{n+1}\right)  $.

Let $i\in\left[  f_{n+1}\right]  $. Then, the permutations $w\in S_{n}$
satisfying $\operatorname*{Qind}w\leq i$ are precisely the permutations $w\in
A\left(  i\right)  $ (by the definition of $A\left(  i\right)  $). Hence, the
family $\left(  a_{w}\right)  _{w\in S_{n};\ \operatorname*{Qind}w\leq i}$ is
precisely the family $\left(  a_{w}\right)  _{w\in A\left(  i\right)  }$.

However, Theorem \ref{thm.aw.freeness} \textbf{(a)} yields that the
$\mathbf{k}$-module $F_{i}$ is free with basis $\left(  a_{w}\right)  _{w\in
S_{n};\ \operatorname*{Qind}w\leq i}$. In other words, the $\mathbf{k}$-module
$F_{i}$ is free with basis $\left(  a_{w}\right)  _{w\in A\left(  i\right)  }$
(since the family $\left(  a_{w}\right)  _{w\in S_{n};\ \operatorname*{Qind}%
w\leq i}$ is precisely the family $\left(  a_{w}\right)  _{w\in A\left(
i\right)  }$). The same argument (applied to $i-1$ instead of $i$) yields that
the $\mathbf{k}$-module $F_{i-1}$ is free with basis $\left(  a_{w}\right)
_{w\in A\left(  i-1\right)  }$. Note that $A\left(  i-1\right)  \subseteq
A\left(  i\right)  $ and that $F_{i-1}$ is a $\mathbf{k}$-submodule of $F_{i}$.

However, the following fact is simple and well-known:

\begin{statement}
\textit{Fact 1:} Let $B$ and $C$ be two sets such that $C\subseteq B$. Let $U$
be a $\mathbf{k}$-module that is free with a basis $\left(  f_{w}\right)
_{w\in B}$. Let $V$ be a $\mathbf{k}$-submodule of $U$ that is free with basis
$\left(  f_{w}\right)  _{w\in C}$. Then, the $\mathbf{k}$-module $U/V$ is free
with basis $\left(  \overline{f_{w}}\right)  _{w\in B\setminus C}$. Here,
$\overline{x}$ denotes the projection of an element $x\in U$ onto the quotient
$U/V$.
\end{statement}

We apply Fact 1 to $B=A\left(  i\right)  $ and $C=A\left(  i-1\right)  $ and
$U=F_{i}$ and $V=F_{i-1}$. As a consequence, we conclude that the $\mathbf{k}%
$-module $F_{i}/F_{i-1}$ is free with basis $\left(  \overline{a_{w}}\right)
_{w\in A\left(  i\right)  \setminus A\left(  i-1\right)  }$.

However,
\begin{align*}
&  A\left(  i\right)  \setminus A\left(  i-1\right) \\
&  =\left\{  w\in A\left(  i\right)  \ \mid\ w\notin A\left(  i-1\right)
\right\} \\
&  =\left\{  w\in S_{n} \ \mid\ w \in A\left(  i\right)  \text{ but not } w\in
A\left(  i-1\right)  \right\} \\
&  =\left\{  w\in S_{n}\ \mid\ \operatorname*{Qind}w\leq i\text{ but not
}\operatorname*{Qind}w\leq i-1\right\} \\
&  \ \ \ \ \ \ \ \ \ \ \ \ \ \ \ \ \ \ \ \ \left(
\begin{array}
[c]{c}%
\text{since }A\left(  i\right)  \text{ is the set of all }w\in S_{n}\text{
satisfying }\operatorname*{Qind}w\leq i\text{,}\\
\text{whereas }A\left(  i-1\right)  \text{ is the set of all }w\in S_{n}\text{
satisfying }\operatorname*{Qind}w\leq i-1
\end{array}
\right) \\
&  =\left\{  w\in S_{n}\ \mid\ \operatorname*{Qind}w=i\right\}
\end{align*}
(since a $w\in S_{n}$ satisfies \textquotedblleft$\operatorname*{Qind}w\leq i$
but not $\operatorname*{Qind}w\leq i-1$\textquotedblright\ if and only if it
satisfies $\operatorname*{Qind}w=i$). Thus, the family $\left(  \overline
{a_{w}}\right)  _{w\in A\left(  i\right)  \setminus A\left(  i-1\right)  }$ is
exactly the family $\left(  \overline{a_{w}}\right)  _{w\in S_{n}%
;\ \operatorname*{Qind}w=i}$. Hence, the $\mathbf{k}$-module $F_{i}/F_{i-1}$
is free with basis $\left(  \overline{a_{w}}\right)  _{w\in S_{n}%
;\ \operatorname*{Qind}w=i}$ (because we have previously showed that the
$\mathbf{k}$-module $F_{i}/F_{i-1}$ is free with basis $\left(  \overline
{a_{w}}\right)  _{w\in A\left(  i\right)  \setminus A\left(  i-1\right)  }$).
This proves Theorem \ref{thm.aw.freeness} \textbf{(b)}.
\end{proof}

\subsection{Our filtration has no equal terms}

For our next corollary, we need a simple existence result:

\begin{lemma}
\label{lem.Qind.surj}Let $i\in\left[  f_{n+1}\right]  $. Then, there exists
some permutation $w\in S_{n}$ satisfying $\operatorname*{Qind}w=i$.
\end{lemma}

\begin{proof}
We shall construct such a permutation $w$ as follows:

Let $J:=\left[  n-1\right]  \setminus Q_{i}$. Thus, $J$ is a subset of
$\left[  n-1\right]  $.

Let $m:=\left\vert J\right\vert $. Let $w\in S_{n}$ be the permutation that
sends the $m$ elements of $J$ (from smallest to largest) to the $m$ numbers
$n,n-1,n-2,\ldots,n-m+1$ (in this order) while sending the remaining $n-m$
elements of $\left[  n\right]  $ (from smallest to largest) to the $n-m$
numbers $1,2,\ldots,n-m$ (in this order). For example, if $n=8$ and
$J=\left\{  2,4,5\right\}  $, then $m=3$ and $\left(  w\left(  1\right)
,w\left(  2\right)  ,\ldots,w\left(  n\right)  \right)  =\left(
1,8,2,7,6,3,4,5\right)  $. The definition of $w$ easily yields that
$\operatorname*{Des}w=J$.

Thus, we have $\operatorname*{Des}w=J=\left[  n-1\right]  \setminus Q_{i}$.
The definition of $Q_{i}^{\prime}$ yields
\[
Q_{i}^{\prime}=\left[  n-1\right]  \setminus\underbrace{\left(  Q_{i}%
\cup\left(  Q_{i}-1\right)  \right)  }_{\supseteq Q_{i}}\subseteq\left[
n-1\right]  \setminus Q_{i}=J=\operatorname*{Des}w.
\]
Combining this with $\operatorname*{Des}w\subseteq\operatorname*{Des}w=\left[
n-1\right]  \setminus Q_{i}$, we obtain $Q_{i}^{\prime}\subseteq
\operatorname*{Des}w\subseteq\left[  n-1\right]  \setminus Q_{i}$. However,
the latter chain of inclusions is equivalent to $\operatorname*{Qind}w=i$
(because of Proposition \ref{prop.Qind.equivalent}). Thus, we have
$\operatorname*{Qind}w=i$.

So we have constructed a permutation $w\in S_{n}$ satisfying
$\operatorname*{Qind}w=i$. As explained above, this proves Lemma
\ref{lem.Qind.surj}.
\end{proof}

Combining Lemma \ref{lem.Qind.surj} with Theorem \ref{thm.aw.freeness}, we
obtain the following corollary (which, roughly speaking, says that our
filtration $F_{0}\subseteq F_{1}\subseteq F_{2}\subseteq\cdots\subseteq
F_{f_{n+1}}$ cannot be shortened):

\begin{corollary}
\label{cor.Fi.distinct}Assume that $\mathbf{k}\neq0$. Then, $F_{i}\neq
F_{i-1}$ for each $i\in\left[  f_{n+1}\right]  $.
\end{corollary}

\begin{proof}
Let $i\in\left[  f_{n+1}\right]  $. We must prove that $F_{i}\neq F_{i-1}$. In
other words, we must prove that $F_{i}/F_{i-1}\neq0$ (since $F_{i-1}$ is a
$\mathbf{k}$-submodule of $F_{i}$). However, Theorem \ref{thm.aw.freeness}
\textbf{(b)} yields that the $\mathbf{k}$-module $F_{i}/F_{i-1}$ is free with
basis $\left(  \overline{a_{w}}\right)  _{w\in S_{n};\ \operatorname*{Qind}%
w=i}$. Hence, in order to prove that $F_{i}/F_{i-1}\neq0$, it suffices to show
that this basis $\left(  \overline{a_{w}}\right)  _{w\in S_{n}%
;\ \operatorname*{Qind}w=i}$ is nonempty. In other words, it suffices to show
that there exists some permutation $w\in S_{n}$ satisfying
$\operatorname*{Qind}w=i$. However, this follows from Lemma
\ref{lem.Qind.surj}. Thus, Corollary \ref{cor.Fi.distinct} is proved.
\end{proof}

\section{Triangularizing the endomorphism}

We are now ready to prove Theorem \ref{thm.Rcomb-main}, made concrete as follows:

\begin{theorem}
\label{thm.Rcomb-conc}Let $w\in S_{n}$ and $\ell\in\left[  n\right]  $. Let
$i=\operatorname*{Qind}w$. Then,%
\[
a_{w}t_{\ell}=m_{Q_{i},\ell}a_{w}+\left(  \text{a }\mathbf{k}\text{-linear
combination of }a_{v}\text{'s for }v\in S_{n}\text{ satisfying }%
\operatorname*{Qind}v<i\right)  .
\]

\end{theorem}

This theorem shows that for each $\ell\in\left[  n\right]  $, the $n!\times
n!$-matrix that represents the endomorphism $R\left(  t_{\ell}\right)  $ of
$\mathbf{k}\left[  S_{n}\right]  $ with respect to the basis $\left(
a_{w}\right)  _{w\in S_{n}}$ is upper-triangular if we order the set $S_{n}$
by increasing $Q$-index (note that this is not the lexicographic order!).
Thus, the same holds for any $\mathbf{k}$-linear combination%
\[
R\left(  \lambda_{1}t_{1}+\lambda_{2}t_{2}+\cdots+\lambda_{n}t_{n}\right)
=\lambda_{1}R\left(  t_{1}\right)  +\lambda_{2}R\left(  t_{2}\right)
+\cdots+\lambda_{n}R\left(  t_{n}\right)  .
\]
Theorem \ref{thm.Rcomb-main} therefore follows, if we can prove Theorem
\ref{thm.Rcomb-conc}. We shall do this in a moment; first, let us give an example:

\begin{example}
For this example, let $n=4$. We write each permutation $w\in S_{4}$ as the
list $\left[  w\left(  1\right)  \ w\left(  2\right)  \ w\left(  3\right)
\ w\left(  4\right)  \right]  $ (written without commas for brevity, and using
square brackets to distinguish it from a parenthesized integer). Then,%
\[
a_{\left[  4312\right]  }t_{2}=a_{\left[  4312\right]  }%
+\underbrace{a_{\left[  4321\right]  }-a_{\left[  4231\right]  }-a_{\left[
3241\right]  }-a_{\left[  2143\right]  }}_{\substack{\text{this is a
}\mathbf{k}\text{-linear combination of }a_{v}\text{'s}\\\text{for }v\in
S_{n}\text{ satisfying }\operatorname*{Qind}v<i\text{, where }%
i=\operatorname*{Qind}\left[  4312\right]  }}.
\]
Indeed, Example \ref{exa.Qi.4} tells us that $\operatorname*{Qind}\left[
4312\right]  =4$, whereas $\operatorname*{Qind}\left[  4321\right]  =1$ and
$\operatorname*{Qind}\left[  4231\right]  =\operatorname*{Qind}\left[
3241\right]  =\operatorname*{Qind}\left[  2143\right]  =3$.
\end{example}

\begin{proof}
[Proof of Theorem \ref{thm.Rcomb-conc}.]Theorem \ref{thm.aw.freeness}
\textbf{(a)} yields that the $\mathbf{k}$-module $F_{i}$ is free with basis
$\left(  a_{v}\right)  _{v\in S_{n};\ \operatorname*{Qind}v\leq i}$. (Here, we
have renamed the index $w$ from Theorem \ref{thm.aw.freeness} \textbf{(a)} as
$v$ in order to avoid confusion with the already-fixed permutation $w$.)

Now, $w\in S_{n}$ and $\operatorname*{Qind}w\leq i$ (since
$\operatorname*{Qind}w=i$). Hence, $a_{w}$ is an element of the family
$\left(  a_{v}\right)  _{v\in S_{n};\ \operatorname*{Qind}v\leq i}$. Since the
latter family $\left(  a_{v}\right)  _{v\in S_{n};\ \operatorname*{Qind}v\leq
i}$ is a basis of $F_{i}$, this entails that $a_{w}\in F_{i}$. Hence,%
\[
\underbrace{a_{w}}_{\in F_{i}}\cdot\left(  t_{\ell}-m_{Q_{i},\ell}\right)  \in
F_{i}\cdot\left(  t_{\ell}-m_{Q_{i},\ell}\right)  \subseteq F_{i-1}%
\ \ \ \ \ \ \ \ \ \ \left(  \text{by Theorem \ref{thm.t-simultri}
\textbf{(c)}}\right)  .
\]

However, Theorem \ref{thm.aw.freeness} \textbf{(a)} (applied to $i-1$ instead
of $i$) yields that the $\mathbf{k}$-module $F_{i-1}$ is free with basis
$\left(  a_{v}\right)  _{v\in S_{n};\ \operatorname*{Qind}v\leq i-1}$. (Here,
again, we have renamed the index $w$ from Theorem \ref{thm.aw.freeness}
\textbf{(a)} as $v$ in order to avoid confusion with the already-fixed
permutation $w$.) Thus, in particular, $\left(  a_{v}\right)  _{v\in
S_{n};\ \operatorname*{Qind}v\leq i-1}$ is a basis of the $\mathbf{k}$-module
$F_{i-1}$. Hence, $F_{i-1}=\operatorname*{span}\left(  \left(  a_{v}\right)
_{v\in S_{n};\ \operatorname*{Qind}v\leq i-1}\right)  $. Now,
\[
a_{w}\cdot\left(  t_{\ell}-m_{Q_{i},\ell}\right)  \in F_{i-1}%
=\operatorname*{span}\left(  \left(  a_{v}\right)  _{v\in S_{n}%
;\ \operatorname*{Qind}v\leq i-1}\right)  =\operatorname*{span}\left(  \left(
a_{v}\right)  _{v\in S_{n};\ \operatorname*{Qind}v < i}\right)
\]
(since the condition \textquotedblleft$\operatorname*{Qind}v\leq
i-1$\textquotedblright\ is equivalent to \textquotedblleft%
$\operatorname*{Qind}v<i$\textquotedblright). In other words,%
\[
a_{w}\cdot\left(  t_{\ell}-m_{Q_{i},\ell}\right)  =\left(  \text{a }%
\mathbf{k}\text{-linear combination of }a_{v}\text{'s for }v\in S_{n}\text{
satisfying }\operatorname*{Qind}v<i\right)  .
\]
In view of $a_{w}\cdot\left(  t_{\ell}-m_{Q_{i},\ell}\right)  =a_{w}t_{\ell
}-m_{Q_{i},\ell}a_{w}$, this can be rewritten as%
\[
a_{w}t_{\ell}-m_{Q_{i},\ell}a_{w}=\left(  \text{a }\mathbf{k}\text{-linear
combination of }a_{v}\text{'s for }v\in S_{n}\text{ satisfying }%
\operatorname*{Qind}v<i\right)  .
\]
Equivalently,%
\[
a_{w}t_{\ell}=m_{Q_{i},\ell}a_{w}+\left(  \text{a }\mathbf{k}\text{-linear
combination of }a_{v}\text{'s for }v\in S_{n}\text{ satisfying }%
\operatorname*{Qind}v<i\right)  .
\]
This proves Theorem \ref{thm.Rcomb-conc}.
\end{proof}

\section{The eigenvalues of the endomorphism}

\subsection{An annihilating polynomial}

We have now shown enough to easily obtain a polynomial that annihilates any
given $\mathbf{k}$-linear combination $\lambda_{1}t_{1}+\lambda_{2}%
t_{2}+\cdots+\lambda_{n}t_{n}$ of the shuffles $t_{1},t_{2},\ldots,t_{n}$ (and
therefore the corresponding endomorphism $R\left(  \lambda_{1}t_{1}%
+\lambda_{2}t_{2}+\cdots+\lambda_{n}t_{n}\right)  $):

\begin{theorem}
\label{thm.eigen.annih-pol}Let $\lambda_{1},\lambda_{2},\ldots,\lambda_{n}%
\in\mathbf{k}$. Let $t:=\lambda_{1}t_{1}+\lambda_{2}t_{2}+\cdots+\lambda
_{n}t_{n}$. Then,%
\[
\prod_{\substack{I\subseteq\left[  n-1\right]  \text{ is}\\\text{lacunar}%
}}\left(  t-\left(  \lambda_{1}m_{I,1}+\lambda_{2}m_{I,2}+\cdots+\lambda
_{n}m_{I,n}\right)  \right)  =0.
\]
(Here, the product on the left hand side is well-defined, since all its
factors $t-\left(  \lambda_{1}m_{I,1}+\lambda_{2}m_{I,2}+\cdots+\lambda
_{n}m_{I,n}\right)  $ lie in the commutative subalgebra $\mathbf{k}\left[
t\right]  $ of $\mathbf{k}\left[  S_{n}\right]  $ and therefore commute with
each other.)
\end{theorem}

\begin{proof}
For each $i\in\left[  f_{n+1}\right]  $, we set
\[
g_{i}:=\lambda_{1}m_{Q_{i},1}+\lambda_{2}m_{Q_{i},2}+\cdots+\lambda
_{n}m_{Q_{i},n}=\sum_{\ell=1}^{n}\lambda_{\ell}m_{Q_{i},\ell}\in\mathbf{k}.
\]

First, we shall show that
\begin{equation}
F_{i}\cdot\left(  t-g_{i}\right)  \subseteq F_{i-1}%
\ \ \ \ \ \ \ \ \ \ \text{for each }i\in\left[  f_{n+1}\right]  .
\label{pf.thm.eigen.annih-pol.1}%
\end{equation}

[\textit{Proof of (\ref{pf.thm.eigen.annih-pol.1}):} Let $i\in\left[
f_{n+1}\right]  $. From $t=\lambda_{1}t_{1}+\lambda_{2}t_{2}+\cdots
+\lambda_{n}t_{n}=\sum_{\ell=1}^{n}\lambda_{\ell}t_{\ell}$ and $g_{i}%
=\sum_{\ell=1}^{n}\lambda_{\ell}m_{Q_{i},\ell}$, we obtain%
\[
t-g_{i}=\sum_{\ell=1}^{n}\lambda_{\ell}t_{\ell}-\sum_{\ell=1}^{n}\lambda
_{\ell}m_{Q_{i},\ell}=\sum_{\ell=1}^{n}\lambda_{\ell}\left(  t_{\ell}%
-m_{Q_{i},\ell}\right)  .
\]
Therefore,%
\begin{align*}
F_{i}\cdot\left(  t-g_{i}\right)   &  =F_{i}\cdot\sum_{\ell=1}^{n}%
\lambda_{\ell}\left(  t_{\ell}-m_{Q_{i},\ell}\right)  =\sum_{\ell=1}%
^{n}\lambda_{\ell}\underbrace{F_{i}\cdot\left(  t_{\ell}-m_{Q_{i},\ell
}\right)  }_{\substack{\subseteq F_{i-1}\\\text{(by Theorem
\ref{thm.t-simultri} \textbf{(c)})}}}\\
&  \subseteq\sum_{\ell=1}^{n}\lambda_{\ell}F_{i-1}\subseteq F_{i-1}%
\ \ \ \ \ \ \ \ \ \ \left(  \text{since }F_{i-1}\text{ is a }\mathbf{k}%
\text{-module}\right)  .
\end{align*}
This proves (\ref{pf.thm.eigen.annih-pol.1}).] \medskip

Next, we claim that%
\begin{equation}
F_{m}\cdot\prod_{j=1}^{m}\left(  t-g_{j}\right)
=0\ \ \ \ \ \ \ \ \ \ \text{for each }m\in\left[  0,f_{n+1}\right]  .
\label{pf.thm.eigen.annih-pol.2}%
\end{equation}
(Here, the product $\prod_{j=1}^{m}\left(  t-g_{j}\right)  $ is well-defined,
since all its factors $t-g_{j}$ lie in the commutative subalgebra
$\mathbf{k}\left[  t\right]  $ of $\mathbf{k}\left[  S_{n}\right]  $ and
therefore commute with each other.)

[\textit{Proof of (\ref{pf.thm.eigen.annih-pol.2}):} We proceed by induction
on $m$:

\textit{Induction base:} For $m=0$, the equality
(\ref{pf.thm.eigen.annih-pol.2}) says that $F_{0}\cdot\left(  \text{empty
product}\right)  =0$, which is true (since $F_{0}=0$).

\textit{Induction step:} Let $i\in\left[  f_{n+1}\right]  $. Assume (as the
induction hypothesis) that (\ref{pf.thm.eigen.annih-pol.2}) holds for $m=i-1$.
We must prove that (\ref{pf.thm.eigen.annih-pol.2}) holds for $m=i$.

We have%
\[
F_{i}\cdot\underbrace{\prod_{j=1}^{i}\left(  t-g_{j}\right)  }_{=\left(
t-g_{i}\right)  \cdot\prod_{j=1}^{i-1}\left(  t-g_{j}\right)  }%
=\underbrace{F_{i}\cdot\left(  t-g_{i}\right)  }_{\substack{\subseteq
F_{i-1}\\\text{(by (\ref{pf.thm.eigen.annih-pol.1}))}}}\cdot\prod_{j=1}%
^{i-1}\left(  t-g_{j}\right)  \subseteq F_{i-1}\cdot\prod_{j=1}^{i-1}\left(
t-g_{j}\right)  =0
\]
(since we assumed that (\ref{pf.thm.eigen.annih-pol.2}) holds for $m=i-1$).
Hence, $F_{i}\cdot\prod_{j=1}^{i}\left(  t-g_{j}\right)  =0$. In other words,
(\ref{pf.thm.eigen.annih-pol.2}) holds for $m=i$. This completes the induction
step. Thus, the proof of (\ref{pf.thm.eigen.annih-pol.2}) is complete.]
\medskip

Now, recall that $Q_{1},Q_{2},\ldots,Q_{f_{n+1}}$ are all the lacunar subsets
of $\left[  n-1\right]  $, listed without repetition. Hence,%
\begin{align*}
&  \prod_{\substack{I\subseteq\left[  n-1\right]  \text{ is}\\\text{lacunar}%
}}\left(  t-\left(  \lambda_{1}m_{I,1}+\lambda_{2}m_{I,2}+\cdots+\lambda
_{n}m_{I,n}\right)  \right) \\
&  =\prod_{j=1}^{f_{n+1}}\left(  t-\underbrace{\left(  \lambda_{1}m_{Q_{j}%
,1}+\lambda_{2}m_{Q_{j},2}+\cdots+\lambda_{n}m_{Q_{j},n}\right)
}_{\substack{=g_{j}\\\text{(by the definition of }g_{j}\text{)}}}\right) \\
&  =\prod_{j=1}^{f_{n+1}}\left(  t-g_{j}\right)  =\underbrace{1}%
_{\substack{\in\mathbf{k}\left[  S_{n}\right]  =F_{f_{n+1}}\\\text{(since
}F_{f_{n+1}}=\mathbf{k}\left[  S_{n}\right]  \\\text{(by Theorem
\ref{thm.t-simultri} \textbf{(a)}))}}}\cdot\prod_{j=1}^{f_{n+1}}\left(
t-g_{j}\right)  \in F_{f_{n+1}}\cdot\prod_{j=1}^{f_{n+1}}\left(
t-g_{j}\right)  =0
\end{align*}
(by (\ref{pf.thm.eigen.annih-pol.2}), applied to $m=f_{n+1}$). In other words,%
\[
\prod_{\substack{I\subseteq\left[  n-1\right]  \text{ is}\\\text{lacunar}%
}}\left(  t-\left(  \lambda_{1}m_{I,1}+\lambda_{2}m_{I,2}+\cdots+\lambda
_{n}m_{I,n}\right)  \right)  =0.
\]
This proves Theorem \ref{thm.eigen.annih-pol}.
\end{proof}

\subsection{The spectrum}

We can now describe the spectrum of $R\left(  \lambda_{1}t_{1}+\lambda
_{2}t_{2}+\cdots+\lambda_{n}t_{n}\right)  $ when $\mathbf{k}$ is a field:

\begin{corollary}
\label{cor.eigen.spec}Let $\lambda_{1},\lambda_{2},\ldots,\lambda_{n}%
\in\mathbf{k}$. Assume that $\mathbf{k}$ is a field. Then,%
\begin{align*}
&  \operatorname*{Spec}\left(  R\left(  \lambda_{1}t_{1}+\lambda_{2}%
t_{2}+\cdots+\lambda_{n}t_{n}\right)  \right) \\
&  =\left\{  \lambda_{1}m_{I,1}+\lambda_{2}m_{I,2}+\cdots+\lambda_{n}%
m_{I,n}\ \mid\ I\subseteq\left[  n-1\right]  \text{ is lacunar}\right\}  .
\end{align*}

Here, $\operatorname*{Spec}f$ denotes the spectrum (i.e., the set of all
eigenvalues) of a $\mathbf{k}$-linear operator $f$.
\end{corollary}

An interesting fact here is that the number of distinct eigenvalues cannot
exceed the number of lacunar subsets of $[n-1]$, which was shown in Section
\ref{sec.Lacunarity} to be the Fibonacci number $f_{n+1}$. This is a
surprisingly low number compared to the number of distinct eigenvalues that
$R\left(  a\right)  $ can have for an arbitrary $a \in\mathbf{k}\left[  S_{n}
\right]  $. In fact, the latter number is the number of involutions of $[n]$,
or equivalently the number of standard Young tableaux with $n$
cells.\footnote{This is due to the fact that (when $\mathbf{k}$ is a
$\mathbb{Q}$-algebra) $\mathbf{k}\left[  S_{n} \right]  $ decomposes into a
direct sum of \emph{Specht modules} indexed by partitions of $n$, and that the
Specht module corresponding to the partition $\lambda$ appears $f^{\lambda}$
many times, where $f^{\lambda}$ is the number of standard tableaux of shape
$\lambda$. Since $R\left(  a \right)  $ acts by the same endomorphism on all
copies of a single Specht module, but can act independently on all
non-isomorphic Specht modules, we see that the maximum number of distinct
eigenvalues of $R\left(  a \right)  $ equals the sum of the dimensions of all
non-isomorphic Specht modules. But this number is the number of standard
tableaux with $n$ cells, i.e., the number of involutions of $\left[  n
\right]  $.}


\begin{proof}
[Proof of Corollary \ref{cor.eigen.spec}.]Let
\[
\rho:=R\left(  \lambda_{1}t_{1}+\lambda_{2}t_{2}+\cdots+\lambda_{n}%
t_{n}\right)  :\mathbf{k}\left[  S_{n}\right]  \rightarrow\mathbf{k}\left[
S_{n}\right]  .
\]

Let $w_{1},w_{2},\ldots,w_{n!}$ be the $n!$ permutations in $S_{n}$, ordered
in such a way that%
\begin{equation}
\operatorname*{Qind}\left(  w_{1}\right)  \leq\operatorname*{Qind}\left(
w_{2}\right)  \leq\cdots\leq\operatorname*{Qind}\left(  w_{n!}\right)  .
\label{pf.cor.eigen.spec.qind-leq}%
\end{equation}
(This ordering is not the lexicographic order!)

Proposition \ref{prop.aw.basis-kSn} says that the family $\left(
a_{w}\right)  _{w\in S_{n}}$ is a basis of the $\mathbf{k}$-module
$\mathbf{k}\left[  S_{n}\right]  $. In other words, the list $\left(
a_{w_{1}},a_{w_{2}},\ldots,a_{w_{n!}}\right)  $ is a basis of the $\mathbf{k}%
$-module $\mathbf{k}\left[  S_{n}\right]  $ (since this list is just a
reindexing of the family $\left(  a_{w}\right)  _{w\in S_{n}}$). We shall
refer to this basis as the \emph{a-basis}. Let $M=\left(  \mu_{i,j}\right)
_{i,j\in\left[  n!\right]  }$ be the matrix that represents the endomorphism
$\rho$ with respect to this a-basis $\left(  a_{w_{1}},a_{w_{2}}%
,\ldots,a_{w_{n!}}\right)  $. Then, for each $j\in\left[  n!\right]  $, we
have%
\begin{equation}
\rho\left(  a_{w_{j}}\right)  =\sum_{k=1}^{n!}\mu_{k,j}a_{w_{k}}.
\label{pf.cor.eigen.spec.1}%
\end{equation}
On the other hand,%
\begin{align}
\rho\left(  a_{w_{j}}\right)   &  =\left(  R\left(  \lambda_{1}t_{1}%
+\lambda_{2}t_{2}+\cdots+\lambda_{n}t_{n}\right)  \right)  \left(  a_{w_{j}%
}\right) \nonumber\\
&  \ \ \ \ \ \ \ \ \ \ \ \ \ \ \ \ \ \ \ \ \left(  \text{since }\rho=R\left(
\lambda_{1}t_{1}+\lambda_{2}t_{2}+\cdots+\lambda_{n}t_{n}\right)  \right)
\nonumber\\
&  =a_{w_{j}}\cdot\underbrace{\left(  \lambda_{1}t_{1}+\lambda_{2}t_{2}%
+\cdots+\lambda_{n}t_{n}\right)  }_{=\sum_{\ell=1}^{n}\lambda_{\ell}t_{\ell}%
}\nonumber\\
&  \ \ \ \ \ \ \ \ \ \ \ \ \ \ \ \ \ \ \ \ \left(  \text{by the definition of
}R\left(  \lambda_{1}t_{1}+\lambda_{2}t_{2}+\cdots+\lambda_{n}t_{n}\right)
\right) \nonumber\\
&  =a_{w_{j}}\cdot\sum_{\ell=1}^{n}\lambda_{\ell}t_{\ell}=\sum_{\ell=1}%
^{n}\lambda_{\ell}a_{w_{j}}t_{\ell}. \label{pf.cor.eigen.spec.2}%
\end{align}

Define an element $g_{i}\in\mathbf{k}$ for each $i\in\left[  f_{n+1}\right]  $
as in the proof of Theorem \ref{thm.eigen.annih-pol}.

We shall now prove the following two properties of our matrix $M=\left(
\mu_{i,j}\right)  _{i,j\in\left[  n!\right]  }$:

\begin{statement}
\textit{Claim 1:} We have $\mu_{j,j}=g_{\operatorname*{Qind}\left(
w_{j}\right)  }$ for each $j\in\left[  n!\right]  .$
\end{statement}

\begin{statement}
\textit{Claim 2:} For any $j,k\in\left[  n!\right]  $ satisfying $k>j$, we
have $\mu_{k,j}=0$.
\end{statement}

[\textit{Proof of Claim 1:} Let $j\in\left[  n!\right]  $. We must prove that
$\mu_{j,j}=g_{\operatorname*{Qind}\left(  w_{j}\right)  }$.

The equality (\ref{pf.cor.eigen.spec.1}) shows that $\mu_{j,j}$ is the
coefficient of $a_{w_{j}}$ when $\rho\left(  a_{w_{j}}\right)  $ is expanded
as a $\mathbf{k}$-linear combination of the a-basis.

Let $i:=\operatorname*{Qind}\left(  w_{j}\right)  $. Then,
(\ref{pf.cor.eigen.spec.2}) becomes%
\begin{align*}
&  \rho\left(  a_{w_{j}}\right) \\
&  =\sum_{\ell=1}^{n}\lambda_{\ell}\underbrace{a_{w_{j}}t_{\ell}%
}_{\substack{=m_{Q_{i},\ell}a_{w_{j}}+\left(  \text{a }\mathbf{k}\text{-linear
combination of }a_{v}\text{'s for }v\in S_{n}\text{ satisfying }%
\operatorname*{Qind}v<i\right)  \\\text{(by Theorem \ref{thm.Rcomb-conc},
applied to }w=w_{j}\text{)}}}\\
&  =\sum_{\ell=1}^{n}\lambda_{\ell}\left(  m_{Q_{i},\ell}a_{w_{j}}+\left(
\text{a }\mathbf{k}\text{-linear combination of }a_{v}\text{'s for }v\in
S_{n}\text{ satisfying }\operatorname*{Qind}v<i\right)  \right) \\
&  =\sum_{\ell=1}^{n}\lambda_{\ell}m_{Q_{i},\ell}a_{w_{j}}+\left(  \text{a
}\mathbf{k}\text{-linear combination of }a_{v}\text{'s for }v\in S_{n}\text{
satisfying }\operatorname*{Qind}v<i\right)  .
\end{align*}
In view of%
\[
\sum_{\ell=1}^{n}\lambda_{\ell}m_{Q_{i},\ell}a_{w_{j}}=\underbrace{\left(
\sum_{\ell=1}^{n}\lambda_{\ell}m_{Q_{i},\ell}\right)  }_{\substack{=g_{i}%
\\\text{(by the definition of }g_{i}\text{)}}}a_{w_{j}}=g_{i}a_{w_{j}},
\]
we can rewrite this as%
\begin{equation}
\rho\left(  a_{w_{j}}\right)  =g_{i}a_{w_{j}}+\left(  \text{a }\mathbf{k}%
\text{-linear combination of }a_{v}\text{'s for }v\in S_{n}\text{ satisfying
}\operatorname*{Qind}v<i\right)  . \label{pf.cor.eigen.spec.3}%
\end{equation}

The right hand side of (\ref{pf.cor.eigen.spec.3}) is clearly a $\mathbf{k}%
$-linear combination of the a-basis. Let us compute the coefficient of
$a_{w_{j}}$ in this combination. Indeed, the first addend $g_{i}a_{w_{j}}$
clearly contributes $g_{i}$ to this coefficient. On the other hand, the
$\mathbf{k}$-linear combination of $a_{v}$'s for $v\in S_{n}$ satisfying
$\operatorname*{Qind}v<i$ does not contain $a_{w_{j}}$ (because $w_{j}$ is not
a $v\in S_{n}$ satisfying $\operatorname*{Qind}v<i$%
\ \ \ \ \footnote{\textit{Proof.} We have $\operatorname*{Qind}\left(
w_{j}\right)  =i$. Thus, we do not have $\operatorname*{Qind}\left(
w_{j}\right)  <i$. Hence, $w_{j}$ is not a $v\in S_{n}$ satisfying
$\operatorname*{Qind}v<i$.}), and thus does not contribute to the coefficient
of $a_{w_{j}}$ on the right hand side of (\ref{pf.cor.eigen.spec.3}). Thus,
the total coefficient with which the basis element $a_{w_{j}}$ appears on the
right hand side of (\ref{pf.cor.eigen.spec.3}) is $g_{i}$. Thus, the equality
(\ref{pf.cor.eigen.spec.3}) expresses $\rho\left(  a_{w_{j}}\right)  $ as a
$\mathbf{k}$-linear combination of the a-basis, and the basis element
$a_{w_{j}}$ appears in this combination with coefficient $g_{i}$. Hence, when
$\rho\left(  a_{w_{j}}\right)  $ is expanded as a $\mathbf{k}$-linear
combination of the a-basis, the basis element $a_{w_{j}}$ appears with
coefficient $g_{i}$. In other words, $\mu_{j,j}=g_{i}$ (since $\mu_{j,j}$ is
the coefficient of $a_{w_{j}}$ when $\rho\left(  a_{w_{j}}\right)  $ is
expanded as a $\mathbf{k}$-linear combination of the a-basis). In view of
$i=\operatorname*{Qind}\left(  w_{j}\right)  $, this rewrites as $\mu
_{j,j}=g_{\operatorname*{Qind}\left(  w_{j}\right)  }$. This completes our
proof of Claim 1.] \medskip

[\textit{Proof of Claim 2:} Let $j,k\in\left[  n!\right]  $ satisfy $k>j$. We
must prove that $\mu_{k,j}=0$.

The equality (\ref{pf.cor.eigen.spec.1}) shows that $\mu_{k,j}$ is the
coefficient of $a_{w_{k}}$ when $\rho\left(  a_{w_{j}}\right)  $ is expanded
as a $\mathbf{k}$-linear combination of the a-basis. Thus, our goal is to show
that this coefficient is $0$ (since we must prove that $\mu_{k,j}=0$). In
other words, our goal is to show that when $\rho\left(  a_{w_{j}}\right)  $ is
expanded as a $\mathbf{k}$-linear combination of the a-basis, the basis
element $a_{w_{k}}$ appears with coefficient $0$.

Let $i:=\operatorname*{Qind}\left(  w_{j}\right)  $. Just as in the proof of
Claim 1, we obtain the equality (\ref{pf.cor.eigen.spec.3}). The right hand
side of this equality is clearly a $\mathbf{k}$-linear combination of the
a-basis. Let us see whether the element $a_{w_{k}}$ of the a-basis appears in
this combination. Indeed, $a_{w_{k}}$ clearly does not appear in the addend
$g_{i}a_{w_{j}}$, because $k\neq j$ (since $k>j$). Furthermore, $a_{w_{k}}$
does not appear in the $\mathbf{k}$-linear combination of $a_{v}$'s for $v\in
S_{n}$ satisfying $\operatorname*{Qind}v<i$ either, because $w_{k}$ is not a
$v\in S_{n}$ satisfying $\operatorname*{Qind}v<i$%
\ \ \ \ \footnote{\textit{Proof.} From $k>j$, we obtain $j\leq k$ and thus
$\operatorname*{Qind}\left(  w_{j}\right)  \leq\operatorname*{Qind}\left(
w_{k}\right)  $ (by (\ref{pf.cor.eigen.spec.qind-leq})). Hence,
$\operatorname*{Qind}\left(  w_{k}\right)  \geq\operatorname*{Qind}\left(
w_{j}\right)  =i$. Thus, we do not have $\operatorname*{Qind}\left(
w_{k}\right)  <i$. Hence, $w_{k}$ is not a $v\in S_{n}$ satisfying
$\operatorname*{Qind}v<i$.}. Hence, $a_{w_{k}}$ appears nowhere on the right
hand side of (\ref{pf.cor.eigen.spec.3}). Thus, the equality
(\ref{pf.cor.eigen.spec.3}) expresses $\rho\left(  a_{w_{j}}\right)  $ as a
$\mathbf{k}$-linear combination of the a-basis, but without the basis element
$a_{w_{k}}$ ever appearing in this combination. Hence, when $\rho\left(
a_{w_{j}}\right)  $ is expanded as a $\mathbf{k}$-linear combination of the
a-basis, the basis element $a_{w_{k}}$ appears with coefficient $0$. This
completes our proof of Claim 2.] \medskip

Claim 2 shows that the matrix $M$ is upper-triangular. Hence, its eigenvalues
are its diagonal entries. In other words,
\[
\operatorname*{Spec}M=\left\{  \text{all diagonal entries of }M\right\}
=\left\{  \mu_{j,j}\ \mid\ j\in\left[  n!\right]  \right\}  =\left\{
g_{\operatorname*{Qind}\left(  w_{j}\right)  }\ \mid\ j\in\left[  n!\right]
\right\}
\]
(since Claim 1 yields that $\mu_{j,j}=g_{\operatorname*{Qind}\left(
w_{j}\right)  }$ for each $j\in\left[  n!\right]  $).

The values $\operatorname*{Qind}w$ for all $w\in S_{n}$ belong to the set
$\left[  f_{n+1}\right]  $ (by the definition of $\operatorname*{Qind}w$).
Conversely, each element $i$ of $\left[  f_{n+1}\right]  $ can be written as
$\operatorname*{Qind}w$ for at least one permutation $w\in S_{n}$ (by Lemma
\ref{lem.Qind.surj}). Combining these two observations, we obtain%
\[
\left\{  \operatorname*{Qind}w\ \mid\ w\in S_{n}\right\}  =\left[
f_{n+1}\right]  .
\]

Now, recall that the matrix $M$ represents the endomorphism $\rho$ with
respect to the basis $\left(  a_{w_{1}},a_{w_{2}},\ldots,a_{w_{n!}}\right)  $.
Hence, its eigenvalues are the eigenvalues of the latter endomorphism. In
other words, $\operatorname*{Spec}M=\operatorname*{Spec}\rho$. In view of
$\rho=R\left(  \lambda_{1}t_{1}+\lambda_{2}t_{2}+\cdots+\lambda_{n}%
t_{n}\right)  $, this rewrites as $\operatorname*{Spec}M=\operatorname*{Spec}%
\left(  R\left(  \lambda_{1}t_{1}+\lambda_{2}t_{2}+\cdots+\lambda_{n}%
t_{n}\right)  \right)  $. Hence,
\begin{align*}
&  \operatorname*{Spec}\left(  R\left(  \lambda_{1}t_{1}+\lambda_{2}%
t_{2}+\cdots+\lambda_{n}t_{n}\right)  \right) \\
&  =\operatorname*{Spec}M\\
&  =\left\{  g_{\operatorname*{Qind}\left(  w_{j}\right)  }\ \mid\ j\in\left[
n!\right]  \right\} \\
&  =\left\{  g_{\operatorname*{Qind}w}\ \mid\ w\in S_{n}\right\}
\ \ \ \ \ \ \ \ \ \ \left(  \text{since }w_{1},w_{2},\ldots,w_{n!}\text{ are
the }n!\text{ permutations in }S_{n}\right) \\
&  =\left\{  g_{i}\ \mid\ i\in\left[  f_{n+1}\right]  \right\}
\ \ \ \ \ \ \ \ \ \ \left(  \text{since }\left\{  \operatorname*{Qind}%
w\ \mid\ w\in S_{n}\right\}  =\left[  f_{n+1}\right]  \right) \\
&  =\left\{  \lambda_{1}m_{Q_{i},1}+\lambda_{2}m_{Q_{i},2}+\cdots+\lambda
_{n}m_{Q_{i},n}\ \mid\ i\in\left[  f_{n+1}\right]  \right\} \\
&  \ \ \ \ \ \ \ \ \ \ \ \ \ \ \ \ \ \ \ \ \left(  \text{since }g_{i}\text{ is
defined as }\lambda_{1}m_{Q_{i},1}+\lambda_{2}m_{Q_{i},2}+\cdots+\lambda
_{n}m_{Q_{i},n}\right) \\
&  =\left\{  \lambda_{1}m_{I,1}+\lambda_{2}m_{I,2}+\cdots+\lambda_{n}%
m_{I,n}\ \mid\ I\subseteq\left[  n-1\right]  \text{ is lacunar}\right\}
\end{align*}
(since $Q_{1},Q_{2},\ldots,Q_{f_{n+1}}$ are exactly the lacunar subsets $I$ of
$\left[  n-1\right]  $). This proves Corollary \ref{cor.eigen.spec}.
\end{proof}

\subsection{Diagonalizability}

We have already seen in Remark \ref{rmk.Rcomb} that the endomorphism $R\left(
\lambda_{1}t_{1}+\lambda_{2}t_{2}+\cdots+\lambda_{n}t_{n}\right)  $ of
$\mathbf{k}\left[  S_{n}\right]  $ may fail to be diagonalizable (even if
$\mathbf{k}=\mathbb{C}$). However, in a large class of cases, it is diagonalizable:

\begin{theorem}
\label{thm.eigen.diagonalizable}Let $\lambda_{1},\lambda_{2},\ldots
,\lambda_{n}\in\mathbf{k}$. Assume that $\mathbf{k}$ is a field. Assume that
the elements $\lambda_{1}m_{I,1}+\lambda_{2}m_{I,2}+\cdots+\lambda_{n}m_{I,n}$
for all lacunar subsets $I\subseteq\left[  n-1\right]  $ are distinct. Then,
the endomorphism $R\left(  \lambda_{1}t_{1}+\lambda_{2}t_{2}+\cdots
+\lambda_{n}t_{n}\right)  $ of $\mathbf{k}\left[  S_{n}\right]  $ is diagonalizable.
\end{theorem}

In order to prove Theorem \ref{thm.eigen.diagonalizable}, we will need a
slightly apocryphal concept from algebra:

\begin{itemize}
\item A $\mathbf{k}$\emph{-algebra antihomomorphism} from a $\mathbf{k}%
$-algebra $A$ to a $\mathbf{k}$-algebra $B$ means a $\mathbf{k}$-linear map
$f:A\rightarrow B$ that satisfies $f\left(  1\right)  =1$ and%
\[
f\left(  a_{1}a_{2}\right)  =f\left(  a_{2}\right)  f\left(  a_{1}\right)
\ \ \ \ \ \ \ \ \ \ \text{for all }a_{1},a_{2}\in A.
\]

\end{itemize}

Thus, a $\mathbf{k}$-algebra antihomomorphism from a $\mathbf{k}$-algebra $A$
to a $\mathbf{k}$-algebra $B$ is the same as a $\mathbf{k}$-algebra
homomorphism from $A^{\operatorname*{op}}$ to $B$, where
$A^{\operatorname*{op}}$ is the opposite algebra of $A$ (that is, the
$\mathbf{k}$-algebra $A$ with its multiplication reversed).

It is well-known that $\mathbf{k}$-algebra homomorphisms preserve univariate
polynomials: That is, if $f$ is a $\mathbf{k}$-algebra homomorphism from a
$\mathbf{k}$-algebra $A$ to a $\mathbf{k}$-algebra $B$, and if $P\in
\mathbf{k}\left[  X\right]  $ is a polynomial, then $f\left(  P\left(
u\right)  \right)  =P\left(  f\left(  u\right)  \right)  $ for any $u\in A$.
The same holds for $\mathbf{k}$-algebra antihomomorphisms:

\begin{proposition}
\label{prop.antihom.unipol}Let $f$ be a $\mathbf{k}$-algebra antihomomorphism
from a $\mathbf{k}$-algebra $A$ to a $\mathbf{k}$-algebra $B$. Let
$P\in\mathbf{k}\left[  X\right]  $ be a polynomial. Then, $f\left(  P\left(
u\right)  \right)  =P\left(  f\left(  u\right)  \right)  $ for any $u\in A$.
\end{proposition}

\begin{proof}
This can be proved in the same way as the analogous result about $\mathbf{k}%
$-algebra homomorphisms.
\end{proof}

\begin{proof}
[Proof of Theorem \ref{thm.eigen.diagonalizable}.]Consider the endomorphism
ring $\operatorname*{End}\nolimits_{\mathbf{k}}\left(  \mathbf{k}\left[
S_{n}\right]  \right)  $ of the $\mathbf{k}$-algebra $\mathbf{k}\left[
S_{n}\right]  $.

We have defined an endomorphism $R\left(  x\right)  \in\operatorname*{End}%
\nolimits_{\mathbf{k}}\left(  \mathbf{k}\left[  S_{n}\right]  \right)  $ of
the $\mathbf{k}$-module $\mathbf{k}\left[  S_{n}\right]  $ for each
$x\in\mathbf{k}\left[  S_{n}\right]  $. Thus, we obtain a map%
\begin{align*}
R:\mathbf{k}\left[  S_{n}\right]   &  \rightarrow\operatorname*{End}%
\nolimits_{\mathbf{k}}\left(  \mathbf{k}\left[  S_{n}\right]  \right)  ,\\
x  &  \mapsto R\left(  x\right)  .
\end{align*}
It is well-known (and straightforward to check) that this map $R$ is a
$\mathbf{k}$-algebra antihomomorphism (i.e., a $\mathbf{k}$-linear map
satisfying $R\left(  1\right)  =1$ and $R\left(  xy\right)  =R\left(
y\right)  \cdot R\left(  x\right)  $ for all $x,y\in\mathbf{k}\left[
S_{n}\right]  $). In fact, $R$ is the standard right action of the
$\mathbf{k}$-algebra $\mathbf{k}\left[  S_{n}\right]  $ on itself.

Let%
\[
t:=\lambda_{1}t_{1}+\lambda_{2}t_{2}+\cdots+\lambda_{n}t_{n}\in\mathbf{k}%
\left[  S_{n}\right]  .
\]
Let $\rho$ be the endomorphism $R\left(  t\right)  $ of $\mathbf{k}\left[
S_{n}\right]  $. We shall show that $\rho$ is diagonalizable.

A univariate polynomial $P\in\mathbf{k}\left[  X\right]  $ is said to be
\emph{split separable} if it can be factored as a product of distinct monic
polynomials of degree $1$ (that is, if it can be written as $P=\prod_{j=1}%
^{k}\left(  X-p_{j}\right)  $, where $p_{1},p_{2},\ldots,p_{k}$ are $k$
\textbf{distinct} elements of $\mathbf{k}$).

Let $P$ be the polynomial $\prod_{\substack{I\subseteq\left[  n-1\right]
\text{ is}\\\text{lacunar}}}\left(  X-\left(  \lambda_{1}m_{I,1}+\lambda
_{2}m_{I,2}+\cdots+\lambda_{n}m_{I,n}\right)  \right)  \in\mathbf{k}\left[
X\right]  $. This polynomial $P$ is split separable, since we assumed that the
elements $\lambda_{1}m_{I,1}+\lambda_{2}m_{I,2}+\cdots+\lambda_{n}m_{I,n}$ for
all lacunar subsets $I\subseteq\left[  n-1\right]  $ are distinct.

Moreover, the definition of $P$ yields
\[
P\left(  t\right)  =\prod_{\substack{I\subseteq\left[  n-1\right]  \text{
is}\\\text{lacunar}}}\left(  t-\left(  \lambda_{1}m_{I,1}+\lambda_{2}%
m_{I,2}+\cdots+\lambda_{n}m_{I,n}\right)  \right)  =0
\]
by Theorem \ref{thm.eigen.annih-pol}. However, $R$ is a $\mathbf{k}$-algebra
antihomomorphism. Hence, Proposition \ref{prop.antihom.unipol} (applied to
$A=\mathbf{k}\left[  S_{n}\right]  $, $B=\operatorname*{End}%
\nolimits_{\mathbf{k}}\left(  \mathbf{k}\left[  S_{n}\right]  \right)  $ and
$f=R$) yields that $R\left(  P\left(  u\right)  \right)  =P\left(  R\left(
u\right)  \right)  $ for any $u\in\mathbf{k}\left[  S_{n}\right]  $. Applying
this to $u=t$, we obtain $R\left(  P\left(  t\right)  \right)  =P\left(
\underbrace{R\left(  t\right)  }_{=\rho}\right)  =P\left(  \rho\right)  $.
Hence, $P\left(  \rho\right)  =R\left(  \underbrace{P\left(  t\right)  }%
_{=0}\right)  =R\left(  0\right)  =0$. Therefore, the minimal polynomial of
$\rho$ divides $P$. (Note that the minimal polynomial of $\rho$ is indeed
well-defined, since $\rho$ is an endomorphism of the finite-dimensional
$\mathbf{k}$-vector space $\mathbf{k}\left[  S_{n}\right]  $.)

It is easy to see that any polynomial $Q\in\mathbf{k}\left[  X\right]  $ that
divides a split separable polynomial must itself be split separable. Hence,
the minimal polynomial of $\rho$ is split separable (since this minimal
polynomial divides $P$, but we know that $P$ is split separable).

Now, recall the following fact (see, e.g., \cite[Theorem 4.11]{Conrad22} or
\cite[\S 6.4, Theorem 6]{HofKun71} or \cite[Proposition 3.8]{StoLui19}): If
the minimal polynomial of an endomorphism of a finite-dimensional $\mathbf{k}%
$-vector space is split separable, then this endomorphism is diagonalizable.
Hence, the endomorphism $\rho$ is diagonalizable (since the minimal polynomial
of $\rho$ is split separable). In other words, $R\left(  \lambda_{1}%
t_{1}+\lambda_{2}t_{2}+\cdots+\lambda_{n}t_{n}\right)  $ is diagonalizable
(since $\rho=R\left(  \underbrace{t}_{=\lambda_{1}t_{1}+\lambda_{2}%
t_{2}+\cdots+\lambda_{n}t_{n}}\right)  =R\left(  \lambda_{1}t_{1}+\lambda
_{2}t_{2}+\cdots+\lambda_{n}t_{n}\right)  $). This proves Theorem
\ref{thm.eigen.diagonalizable}.
\end{proof}

Note that Theorem \ref{thm.eigen.diagonalizable} is not an \textquotedblleft
if and only if\textquotedblright\ statement. We do not know if there is an
easy way to characterize when $R\left(  \lambda_{1}t_{1}+\lambda_{2}%
t_{2}+\cdots+\lambda_{n}t_{n}\right)  $ is diagonalizable.

\begin{remark}
Let $I$ be a subset of $\left[  n\right]  $. Then, the numbers $m_{I,1}%
,m_{I,2},\ldots,m_{I,n}$ together uniquely determine $I$. Indeed, a moment's
thought reveals that%
\[
I=\left\{  \ell\in\left[  n\right]  \ \mid\ m_{I,\ell}=0\right\}  .
\]
Hence, if $\mathbf{k}$ is a field of characteristic $0$, then the main
assumption of Theorem \ref{thm.eigen.diagonalizable} (viz., that the elements
$\lambda_{1}m_{I,1}+\lambda_{2}m_{I,2}+\cdots+\lambda_{n}m_{I,n}$ for all
lacunar subsets $I\subseteq\left[  n-1\right]  $ are distinct) will be
satisfied for any \textquotedblleft sufficiently\textquotedblright\ generic
$\lambda_{1},\lambda_{2},\ldots,\lambda_{n}\in\mathbf{k}$.
\end{remark}

\begin{example}
\label{ex.r2b.diagonalizability} We cannot use Theorem
\ref{thm.eigen.diagonalizable} to show that the random-to-below shuffle is
always diagonalizable. For example, when $n=12$, two lacunar sets
($\{1,6,8,10\}$ and $\{6,8,11\}$) yield $\sum_{\ell=1}^{n} \frac{m_{I,\ell}%
}{n+1-\ell} = \frac{13 573}{3960}$. This is the smallest example we could
find, meaning that the shuffle is certainly diagonalizable when $\mathbf{k} =
\mathbb{Q}$ and $n\leq11$. It remains an open question whether the
random-to-below shuffle is diagonalizable.
\end{example}

\begin{example}
\label{ex.t2r.diagonalizability} There are diagonalizable one-sided cycle
shuffles that do not satisfy the hypotheses of Theorem
\ref{thm.eigen.diagonalizable}. For example, it is known since \cite[Theorem
4.1]{DiFiPi92} that the top-to-random shuffle ($t_{1}$) is diagonalizable. In
our notation, it corresponds to $\lambda_{1} = 1$ and $\lambda_{2} =
\lambda_{3} = \ldots= \lambda_{n} =0$, which does not satisfy the conditions
of Theorem \ref{thm.eigen.diagonalizable} in general.
\end{example}

\begin{question}
\label{qu.criterion.diagonalizability} Can a necessary and sufficient
criterion be found for the diagonalizability of a one-sided shuffle (as
opposed to the merely sufficient one in Theorem \ref{thm.eigen.diagonalizable})?
\end{question}

\section{The multiplicities of the eigenvalues}

\subsection{The dimensions of $F_{i}/F_{i-1}$, explicitly}

In Theorem \ref{thm.aw.freeness} \textbf{(b)}, we have given bases for all the
quotient $\mathbf{k}$-modules $F_{i}/F_{i-1}$. The sizes of these bases are
the dimensions of these quotient $\mathbf{k}$-modules. Let us now characterize
these dimensions more explicitly:

\begin{theorem}
\label{thm.deltai.main}Let $i\in\left[  f_{n+1}\right]  $. Let $\delta_{i}$ be
the number of all permutations $w\in S_{n}$ satisfying $\operatorname*{Qind}%
w=i$. Then:

\begin{enumerate}
\item[\textbf{(a)}] The $\mathbf{k}$-module $F_{i}/F_{i-1}$ is free and has
dimension (i.e., rank) equal to $\delta_{i}$. (Here, of course, $F_{0}%
\subseteq F_{1}\subseteq F_{2}\subseteq\cdots\subseteq F_{f_{n+1}}$ is the
filtration from Theorem \ref{thm.t-simultri}.)

\item[\textbf{(b)}] The number $\delta_{i}$ equals the number of all
permutations $w\in S_{n}$ that satisfy%
\[
w\left(  j\right)  <w\left(  j+1\right)  \ \ \ \ \ \ \ \ \ \ \text{for all
}j\in Q_{i}%
\]
and%
\[
w\left(  j\right)  >w\left(  j+1\right)  \ \ \ \ \ \ \ \ \ \ \text{for all
}j\in Q_{i}^{\prime}.
\]

\item[\textbf{(c)}] Write the set $Q_{i}$ in the form $Q_{i}=\left\{
i_{1}<i_{2}<\cdots<i_{p}\right\}  $, and set $i_{0}=1$ and $i_{p+1}=n+1$. Let
$j_{k}=i_{k}-i_{k-1}$ for each $k\in\left[  p+1\right]  $. Then,%
\begin{equation}
\delta_{i}=\dbinom{n}{j_{1},j_{2},\ldots,j_{p+1}}\cdot\prod_{k=2}^{p+1}\left(
j_{k}-1\right)  . \label{eq.thm.deltai.main.c.eq}%
\end{equation}
Here, $\dbinom{n}{j_{1},j_{2},\ldots,j_{p+1}}$ denotes the multinomial
coefficient $\dfrac{n!}{j_{1}!j_{2}!\cdots j_{p+1}!}$.

\item[\textbf{(d)}] We have $\delta_{i}\mid n!$.
\end{enumerate}
\end{theorem}

\begin{proof}
\textbf{(a)} Theorem \ref{thm.aw.freeness} \textbf{(b)} shows that the
$\mathbf{k}$-module $F_{i}/F_{i-1}$ is free with basis $\left(  \overline
{a_{w}}\right)  _{w\in S_{n};\ \operatorname*{Qind}w=i}$. Hence, its dimension
is the number of all permutations $w\in S_{n}$ satisfying
$\operatorname*{Qind}w=i$. But this latter number is $\delta_{i}$ (by the
definition of $\delta_{i}$). This proves Theorem \ref{thm.deltai.main}
\textbf{(a)}. \medskip

\textbf{(b)} For any permutation $w\in S_{n}$, we have the following chain of
equivalences:%
\begin{align*}
&  \ \left(  \operatorname*{Qind}w=i\right) \\
&  \Longleftrightarrow\ \left(  Q_{i}^{\prime}\subseteq\operatorname*{Des}%
w\subseteq\left[  n-1\right]  \setminus Q_{i}\right)
\ \ \ \ \ \ \ \ \ \ \left(  \text{by Proposition \ref{prop.Qind.equivalent}%
}\right) \\
&  \Longleftrightarrow\ \left(  \underbrace{Q_{i}^{\prime}\subseteq
\operatorname*{Des}w}_{\Longleftrightarrow\ \left(  j\in\operatorname*{Des}%
w\text{ for all }j\in Q_{i}^{\prime}\right)  }\text{ and }%
\underbrace{\operatorname*{Des}w\subseteq\left[  n-1\right]  \setminus Q_{i}%
}_{\substack{\Longleftrightarrow\ \left(  \operatorname*{Des}w\text{ is
disjoint from }Q_{i}\right)  \\\text{(since }\operatorname*{Des}%
w\subseteq\left[  n-1\right]  \text{ always holds)}}}\right) \\
&  \Longleftrightarrow\ \left(  \left(  j\in\operatorname*{Des}w\text{ for all
}j\in Q_{i}^{\prime}\right)  \text{ and }\underbrace{\left(
\operatorname*{Des}w\text{ is disjoint from }Q_{i}\right)  }%
_{\Longleftrightarrow\ \left(  j\notin\operatorname*{Des}w\text{ for all }j\in
Q_{i}\right)  }\right) \\
&  \Longleftrightarrow\ \left(  \left(  \underbrace{j\in\operatorname*{Des}%
w}_{\substack{\Longleftrightarrow\ \left(  w\left(  j\right)  >w\left(
j+1\right)  \right)  \\\text{(by the definition of }\operatorname*{Des}%
w\text{)}}}\text{ for all }j\in Q_{i}^{\prime}\right)  \text{ and }\left(
\underbrace{j\notin\operatorname*{Des}w}_{\substack{\Longleftrightarrow
\ \left(  w\left(  j\right)  \leq w\left(  j+1\right)  \right)  \\\text{(by
the definition of }\operatorname*{Des}w\text{)}}}\text{ for all }j\in
Q_{i}\right)  \right) \\
&  \Longleftrightarrow\ \left(  \left(  w\left(  j\right)  >w\left(
j+1\right)  \text{ for all }j\in Q_{i}^{\prime}\right)  \text{ and }\left(
\underbrace{w\left(  j\right)  \leq w\left(  j+1\right)  }%
_{\substack{\Longleftrightarrow\ \left(  w\left(  j\right)  <w\left(
j+1\right)  \right)  \\\text{(since }w\left(  j\right)  \neq w\left(
j+1\right)  \\\text{(because }w\text{ is a permutation))}}}\text{ for all
}j\in Q_{i}\right)  \right) \\
&  \Longleftrightarrow\ \left(  \left(  w\left(  j\right)  >w\left(
j+1\right)  \text{ for all }j\in Q_{i}^{\prime}\right)  \text{ and }\left(
w\left(  j\right)  <w\left(  j+1\right)  \text{ for all }j\in Q_{i}\right)
\right) \\
&  \Longleftrightarrow\ \left(  \left(  w\left(  j\right)  <w\left(
j+1\right)  \text{ for all }j\in Q_{i}\right)  \text{ and }\left(  w\left(
j\right)  >w\left(  j+1\right)  \text{ for all }j\in Q_{i}^{\prime}\right)
\right)  .
\end{align*}
Thus, $\delta_{i}$ equals the number of all permutations $w\in S_{n}$
satisfying%
\[
\left(  w\left(  j\right)  <w\left(  j+1\right)  \text{ for all }j\in
Q_{i}\right)  \text{ and }\left(  w\left(  j\right)  >w\left(  j+1\right)
\text{ for all }j\in Q_{i}^{\prime}\right)
\]
(because $\delta_{i}$ was defined as the number of all permutations $w\in
S_{n}$ satisfying $\operatorname*{Qind}w=i$). This proves Theorem
\ref{thm.deltai.main} \textbf{(b)}. \medskip

\textbf{(c)} We introduce a bit of terminology: If $K=\left[  u,v\right]  $ is
an interval of $\mathbb{Z}$, and if $T$ is an arbitrary subset of $\mathbb{Z}%
$, then a map $f:K\rightarrow T$ will be called \emph{up-decreasing} if it
satisfies%
\[
f\left(  u\right)  <f\left(  u+1\right)  >f\left(  u+2\right)  >f\left(
u+3\right)  >\cdots>f\left(  v\right)
\]
(that is, if it is increasing on $\left[  u,u+1\right]  $ and decreasing on
$\left[  u+1,v\right]  $). For instance, the map $\left[  5\right]
\rightarrow\left[  -3,0\right]  $ that sends each $k\in\left[  5\right]  $ to
$-\left\vert k-2\right\vert $ is up-decreasing.

The following fact is easy to see:

\begin{statement}
\textit{Claim 1:} Let $h\geq2$ be an integer. Let $K=\left[  u,v\right]  $ be
an interval of $\mathbb{Z}$ having size $\left\vert K\right\vert =v-u+1=h$.
Let $T$ be a subset of $\mathbb{Z}$ that has size $h$. Then, the number of
up-decreasing bijections $f:K\rightarrow T$ is $h-1$.
\end{statement}

[\textit{Proof of Claim 1:} We WLOG assume that $K=\left[  h\right]  $ and
$T=\left[  h\right]  $, because we can otherwise rename the elements of $K$
and of $T$ while preserving their relative order. Thus, the bijections
$f:K\rightarrow T$ are precisely the permutations of $\left[  h\right]  $, and
we must show that the number of up-decreasing permutations of $\left[
h\right]  $ is $h-1$.

But this is easy to show: An up-decreasing permutation of $\left[  h\right]  $
is a permutation $f$ of $\left[  h\right]  $ satisfying $f\left(  1\right)
<f\left(  2\right)  >f\left(  3\right)  >f\left(  4\right)  >\cdots>f\left(
h\right)  $. Thus, any up-decreasing permutation $f$ of $\left[  h\right]  $
is uniquely determined by its first value $f\left(  1\right)  $, because its
remaining values must be the remaining elements of $\left[  h\right]  $ in
decreasing order (to ensure that $f\left(  2\right)  >f\left(  3\right)
>f\left(  4\right)  >\cdots>f\left(  h\right)  $ holds). The first value
$f\left(  1\right)  $ cannot be $h$ (since this would violate $f\left(
1\right)  <f\left(  2\right)  $), but can be any of the other $h-1$ elements
of $\left[  h\right]  $. Thus, there are $h-1$ choices for $f\left(  1\right)
$, and each of these choices leads to a unique up-decreasing permutation $f$
of $\left[  h\right]  $. Hence, there are $h-1$ such permutations in total.
This completes the proof of Claim 1.] \medskip

Recall that $i_{1}<i_{2}<\cdots<i_{p}$ are the $p$ elements of $Q_{i}%
\subseteq\left[  n-1\right]  $, and we have furthermore set $i_{0}=1$ and
$i_{p+1}=n+1$. Hence,
\[
1=i_{0}\leq i_{1}<i_{2}<\cdots<i_{p}<i_{p+1}=n+1.
\]

Define an interval%
\[
J_{k}:=\left[  i_{k-1},\ i_{k}-1\right]  \ \ \ \ \ \ \ \ \ \ \text{for each
}k\in\left[  p+1\right]  .
\]
Then, the interval $\left[  n\right]  $ is the disjoint union $J_{1}\sqcup
J_{2}\sqcup\cdots\sqcup J_{p+1}$. We have%
\begin{equation}
Q_{i}=\left\{  i_{1},i_{2},\ldots,i_{p}\right\}
\label{pf.thm.deltai.main.c.Qi=}%
\end{equation}
and%
\begin{equation}
Q_{i}^{\prime}=\left\{  1,\ 2,\ \ldots,\ i_{1}-2\right\}  \cup\bigcup
_{k=2}^{p+1}\left\{  i_{k-1}+1,\ i_{k-1}+2,\ \ldots,\ i_{k}-2\right\}  .
\label{pf.thm.deltai.main.c.Qi'=}%
\end{equation}
Note further that each $k\in\left[  p+1\right]  $ satisfies $\left\vert
J_{k}\right\vert =i_{k}-i_{k-1}$ (since $J_{k}=\left[  i_{k-1},\ i_{k}%
-1\right]  $) and therefore $\left\vert J_{k}\right\vert =i_{k}-i_{k-1}=j_{k}%
$. Furthermore, note that $j_{1},j_{2},\ldots,j_{p+1}$ are nonnegative
integers (since each $k\in\left[  p+1\right]  $ satisfies $j_{k}%
=i_{k}-\underbrace{i_{k-1}}_{\leq i_{k}}\geq i_{k}-i_{k}=0$). Finally, it is
easy to see that
\begin{equation}
j_{k}\geq2\ \ \ \ \ \ \ \ \ \ \text{for each }k\in\left[  2,p+1\right]  .
\label{pf.thm.deltai.main.c.geq2}%
\end{equation}

\begin{fineprint}
[\textit{Proof of (\ref{pf.thm.deltai.main.c.geq2}):} Let $k\in\left[
2,p+1\right]  $. Then, both $k-1$ and $k$ belong to $\left[  p+1\right]  $.

The set $Q_{i}$ is a lacunar subset of $\left[  n-1\right]  $ (since
$Q_{1},Q_{2},\ldots,Q_{f_{n+1}}$ are all the lacunar subsets of $\left[
n-1\right]  $). Thus, the set $Q_{i}\cup\left\{  n+1\right\}  $ is lacunar as
well (since each element of $Q_{i}$ is $\leq n-1$ and thus differs by at least
$2$ from the new element $n+1$). Hence, any two distinct elements of the set
$Q_{i}\cup\left\{  n+1\right\}  $ differ by at least $2$.

However, from $Q_{i}=\left\{  i_{1}<i_{2}<\cdots<i_{p}\right\}  $ and
$i_{p+1}=n+1$, we obtain $Q_{i}\cup\left\{  n+1\right\}  =\left\{  i_{1}%
<i_{2}<\cdots<i_{p}<i_{p+1}\right\}  $ (since $Q_{i}\subseteq\left[
n-1\right]  $). Therefore, $i_{k-1}$ and $i_{k}$ are two distinct elements of
the set $Q_{i}\cup\left\{  n+1\right\}  $ (since both $k-1$ and $k$ belong to
$\left[  p+1\right]  $). Consequently, $i_{k-1}$ and $i_{k}$ differ by at
least $2$ (since any two distinct elements of the set $Q_{i}\cup\left\{
n+1\right\}  $ differ by at least $2$). In other words, $i_{k}-i_{k-1}\geq2$
(since $i_{k-1}<i_{k}$). But the definition of $j_{k}$ yields $j_{k}%
=i_{k}-i_{k-1}\geq2$. This proves (\ref{pf.thm.deltai.main.c.geq2}).] \medskip
\end{fineprint}

Now, Theorem \ref{thm.deltai.main} \textbf{(b)} shows that $\delta_{i}$ is the
number of all permutations $w\in S_{n}$ that satisfy%
\begin{equation}
w\left(  j\right)  <w\left(  j+1\right)  \ \ \ \ \ \ \ \ \ \ \text{for all
}j\in Q_{i} \label{pf.thm.deltai.main.c.3a}%
\end{equation}
and%
\begin{equation}
w\left(  j\right)  >w\left(  j+1\right)  \ \ \ \ \ \ \ \ \ \ \text{for all
}j\in Q_{i}^{\prime}. \label{pf.thm.deltai.main.c.3b}%
\end{equation}
In view of (\ref{pf.thm.deltai.main.c.Qi=}) and
(\ref{pf.thm.deltai.main.c.Qi'=}), we can rewrite this as follows: $\delta
_{i}$ is the number of all permutations $w\in S_{n}$ that satisfy%
\[
w\left(  1\right)  >w\left(  2\right)  >w\left(  3\right)  >\cdots>w\left(
i_{1}-1\right)
\]
and%
\[
w\left(  i_{k-1}\right)  <w\left(  i_{k-1}+1\right)  >w\left(  i_{k-1}%
+2\right)  >w\left(  i_{k-1}+3\right)  >\cdots>w\left(  i_{k}-1\right)
\]
for each $k\in\left[  2,p+1\right]  $. In other words, $\delta_{i}$ is the
number of all permutations $w\in S_{n}$ such that the restriction
$w\mid_{J_{1}}$ is strictly decreasing whereas the restrictions $w\mid_{J_{2}%
},\ \ w\mid_{J_{3}},\ \ \ldots,\ \ w\mid_{J_{p+1}}$ are up-decreasing (since
$J_{k}=\left[  i_{k-1},\ i_{k}-1\right]  $ for each $k\in\left[  p+1\right]
$). We can construct such a permutation $w$ as follows:

\begin{itemize}
\item First, we choose the \textbf{sets} $w\left(  J_{k}\right)  $ for all
$k\in\left[  p+1\right]  $. In doing so, we must ensure that these $p+1$ sets
are disjoint and cover the entire set $\left[  n\right]  $, and have the size
$\left\vert w\left(  J_{k}\right)  \right\vert =\left\vert J_{k}\right\vert
=j_{k}$ for each $k$. Thus, there are $\dbinom{n}{j_{1},j_{2},\ldots,j_{p+1}}$
many options at this step.

\item At this point, the restriction $w\mid_{J_{1}}$ is already uniquely
determined, since $w\mid_{J_{1}}$ has to be strictly decreasing and its image
$w\left(  J_{1}\right)  $ is already chosen.

\item Now, for each $k\in\left[  2,p+1\right]  $, we choose the restriction
$w\mid_{J_{k}}$. This restriction has to be an up-decreasing bijection from
the interval $J_{k}$ to the (already chosen) set $w\left(  J_{k}\right)  $,
which has size $\left\vert w\left(  J_{k}\right)  \right\vert =\left\vert
J_{k}\right\vert =j_{k}$; thus, by Claim 1 (applied to $h=j_{k}$ and $K=J_{k}$
and $T=w\left(  J_{k}\right)  $), there are $j_{k}-1$ options for this
restriction $w\mid_{J_{k}}$ (since (\ref{pf.thm.deltai.main.c.geq2}) yields
$j_{k}\geq2$). Hence, in total, we have $\prod_{k=2}^{p+1}\left(
j_{k}-1\right)  $ options at this step.
\end{itemize}

\noindent Altogether, the total number of possibilities to perform this
construction is thus $\dbinom{n}{j_{1},j_{2},\ldots,j_{p+1}}\cdot\prod
_{k=2}^{p+1}\left(  j_{k}-1\right)  $. Hence,
\[
\delta_{i}=\dbinom{n}{j_{1},j_{2},\ldots,j_{p+1}}\cdot\prod_{k=2}^{p+1}\left(
j_{k}-1\right)  .
\]
This proves Theorem \ref{thm.deltai.main} \textbf{(c)}. \medskip

\textbf{(d)} Define the integers $i_{0},i_{1},\ldots,i_{p+1}$ and $j_{1}%
,j_{2},\ldots,j_{p+1}$ as in Theorem \ref{thm.deltai.main} \textbf{(c)}. Then,
we have $j_{k}\geq2$ for each $k\in\left[  2,p+1\right]  $ (in fact, this is
the inequality (\ref{pf.thm.deltai.main.c.geq2}), which has been shown in our
above proof of Theorem \ref{thm.deltai.main} \textbf{(c)}). Hence, for each
$k\in\left[  2,p+1\right]  $, we have%
\[
j_{k}!=\underbrace{1\cdot2\cdot\cdots\cdot\left(  j_{k}-2\right)  }_{=\left(
j_{k}-2\right)  !}\cdot\left(  j_{k}-1\right)  \cdot j_{k}=\left(
j_{k}-2\right)  !\cdot\left(  j_{k}-1\right)  \cdot j_{k}%
\]
and therefore%
\begin{equation}
j_{k}-1=\dfrac{j_{k}!}{\left(  j_{k}-2\right)  !\cdot j_{k}}.
\label{pf.thm.deltai.main.d.2}%
\end{equation}

The definition of a multinomial coefficient yields%
\[
\dbinom{n}{j_{1},j_{2},\ldots,j_{p+1}}=\dfrac{n!}{j_{1}!j_{2}!\cdots j_{p+1}%
!}=\dfrac{n!}{\prod_{k=1}^{p+1}j_{k}!}=\dfrac{n!}{j_{1}!\prod_{k=2}^{p+1}%
j_{k}!}.
\]
From (\ref{eq.thm.deltai.main.c.eq}), we now obtain%
\begin{align*}
\delta_{i}  &  =\underbrace{\dbinom{n}{j_{1},j_{2},\ldots,j_{p+1}}}%
_{=\dfrac{n!}{j_{1}!\prod_{k=2}^{p+1}j_{k}!}}\cdot\prod_{k=2}^{p+1}%
\underbrace{\left(  j_{k}-1\right)  }_{\substack{=\dfrac{j_{k}!}{\left(
j_{k}-2\right)  !\cdot j_{k}}\\\text{(by (\ref{pf.thm.deltai.main.d.2}))}%
}}=\dfrac{n!}{j_{1}!\prod_{k=2}^{p+1}j_{k}!}\cdot\prod_{k=2}^{p+1}\dfrac
{j_{k}!}{\left(  j_{k}-2\right)  !\cdot j_{k}}\\
&  =\dfrac{n!}{j_{1}!}\cdot\prod_{k=2}^{p+1}\left(  \dfrac{j_{k}!}{\left(
j_{k}-2\right)  !\cdot j_{k}}/j_{k}!\right)  =\dfrac{n!}{j_{1}!\cdot
\prod_{k=2}^{p+1}\left(  \left(  j_{k}-2\right)  !\cdot j_{k}\right)  }.
\end{align*}
Thus, we obtain $\delta_{i}\mid n!$ (since the denominator $j_{1}!\cdot
\prod_{k=2}^{p+1}\left(  \left(  j_{k}-2\right)  !\cdot j_{k}\right)  $ in
this equality is clearly an integer). This proves Theorem
\ref{thm.deltai.main} \textbf{(d)}.
\end{proof}

\subsection{The multiplicities of the eigenvalues}

Finally, we can find the algebraic multiplicities of the eigenvalues of the
endomorphism $R\left(  \lambda_{1}t_{1}+\lambda_{2}t_{2}+\cdots+\lambda
_{n}t_{n}\right)  $ (when $\mathbf{k}$ is a field and $\lambda_{1},\lambda
_{2},\ldots,\lambda_{n}\in\mathbf{k}$ are arbitrary). Roughly speaking, we
want to claim that each eigenvalue $\lambda_{1}m_{I,1}+\lambda_{2}%
m_{I,2}+\cdots+\lambda_{n}m_{I,n}$ (where $I\subseteq\left[  n-1\right]  $ is
a lacunar subset) has algebraic multiplicity $\delta_{i}$, where $i\in\left[
f_{n+1}\right]  $ is chosen such that $I=Q_{i}$ (and where $\delta_{i}$ is as
in Theorem \ref{thm.deltai.main}). This is not fully precise; indeed, if some
lacunar subsets $I\subseteq\left[  n-1\right]  $ produce the same eigenvalues
$\lambda_{1}m_{I,1}+\lambda_{2}m_{I,2}+\cdots+\lambda_{n}m_{I,n}$, then their
respective $\delta_{i}$'s need to be added together to form the right
algebraic multiplicity. The technically correct statement of our claim is thus
as follows:

\begin{theorem}
\label{thm.eigen.mult}Assume that $\mathbf{k}$ is a field. Let $\lambda
_{1},\lambda_{2},\ldots,\lambda_{n}\in\mathbf{k}$. For each $i\in\left[
f_{n+1}\right]  $, let $\delta_{i}$ be the number of all permutations $w\in
S_{n}$ satisfying $\operatorname*{Qind}w=i$. For each $i\in\left[
f_{n+1}\right]  $, we set
\[
g_{i}:=\lambda_{1}m_{Q_{i},1}+\lambda_{2}m_{Q_{i},2}+\cdots+\lambda
_{n}m_{Q_{i},n}=\sum_{\ell=1}^{n}\lambda_{\ell}m_{Q_{i},\ell}\in\mathbf{k}.
\]

Let $\kappa\in\mathbf{k}$. Then, the algebraic multiplicity of $\kappa$ as an
eigenvalue of $R\left(  \lambda_{1}t_{1}+\lambda_{2}t_{2}+\cdots+\lambda
_{n}t_{n}\right)  $ equals%
\[
\sum_{\substack{i\in\left[  f_{n+1}\right]  ;\\g_{i}=\kappa}}\delta_{i}.
\]

\end{theorem}

\begin{proof}
We shall use the notations introduced in the proof of Corollary
\ref{cor.eigen.spec}. In that proof, we have shown that the matrix $M$ is upper-triangular.

Recall that the eigenvalues of a triangular matrix are its diagonal entries,
and moreover, the algebraic multiplicity of an eigenvalue is the number of
times that it appears on the main diagonal. We can apply this fact to the
matrix $M$ (since $M$ is upper-triangular), and thus conclude that%
\begin{align*}
&  \left(  \text{the algebraic multiplicity of }\kappa\text{ as an eigenvalue
of }M\right) \\
&  =\left(  \text{the number of times that }\kappa\text{ appears on the main
diagonal of }M\right) \\
&  =\left(  \text{the number of }j\in\left[  n!\right]  \text{ such that }%
\mu_{j,j}=\kappa\right)  \ \ \ \ \ \ \ \ \ \ \left(  \text{since }M=\left(
\mu_{i,j}\right)  _{i,j\in\left[  n!\right]  }\right) \\
&  =\left(  \text{the number of }j\in\left[  n!\right]  \text{ such that
}g_{\operatorname*{Qind}\left(  w_{j}\right)  }=\kappa\right) \\
&  \ \ \ \ \ \ \ \ \ \ \ \ \ \ \ \ \ \ \ \ \left(
\begin{array}
[c]{c}%
\text{since Claim 1 from the proof of Corollary \ref{cor.eigen.spec}}\\
\text{yields that } \mu_{j,j}=g_{\operatorname*{Qind}\left(  w_{j}\right)
}\text{ for each }j\in\left[  n!\right]
\end{array}
\right) \\
&  =\left(  \text{the number of }w\in S_{n}\text{ such that }%
g_{\operatorname*{Qind}w}=\kappa\right) \\
&  \ \ \ \ \ \ \ \ \ \ \ \ \ \ \ \ \ \ \ \ \left(  \text{since }w_{1}%
,w_{2},\ldots,w_{n!}\text{ are the }n!\text{ permutations in }S_{n}\right) \\
&  =\sum_{\substack{i\in\left[  f_{n+1}\right]  ;\\g_{i}=\kappa}%
}\underbrace{\left(  \text{the number of }w\in S_{n}\text{ such that
}\operatorname*{Qind}w=i\right)  }_{\substack{=\delta_{i}\\\text{(by the
definition of }\delta_{i}\text{)}}}\\
&  \ \ \ \ \ \ \ \ \ \ \ \ \ \ \ \ \ \ \ \ \left(  \text{here we have split
the sum up according to the value of }\operatorname*{Qind}w\right) \\
&  =\sum_{\substack{i\in\left[  f_{n+1}\right]  ;\\g_{i}=\kappa}}\delta_{i}.
\end{align*}
This proves Theorem \ref{thm.eigen.mult}.
\end{proof}

\section{\label{sect.furtheralg}Further algebraic consequences}

In this section, we shall derive some more corollaries from the above. To be
more specific, we first study the algebraic properties of the antipode of the
one-sided cycle shuffle $\lambda_{1}t_{1}+\lambda_{2}t_{2}+\cdots+\lambda
_{n}t_{n}$; this corresponds to the reversal of the corresponding Markov
chain. Then, we discuss the endomorphism $L\left(  \lambda_{1}t_{1}%
+\lambda_{2}t_{2}+\cdots+\lambda_{n}t_{n}\right)  $ corresponding to left
multiplication (as opposed to right multiplication, which we have studied
before) by the shuffle. We next use our notions of $Q$-index and non-shadow to
subdivide the Boolean algebra of the set $\left[  n-1 \right]  $ into Boolean
intervals indexed by the lacunar subsets of $\left[  n-1 \right]  $. Finally,
we explore what known results about the top-to-random shuffle our results can
and cannot prove.

\subsection{\label{subsect.furtheralg.btr}Below-to-somewhere shuffles}

We have so far been considering the somewhere-to-below shuffles $t_{1}%
,t_{2},\ldots,t_{n}$, which are sums of cycles. If we invert these cycles
(i.e., reverse the order of cycling), we obtain new elements of $\mathbf{k}%
\left[  S_{n}\right]  $, which may be called the \textquotedblleft
below-to-somewhere shuffles\textquotedblright. Here is their precise definition:

For each $\ell\in\left[  n\right]  $, we define the element%
\begin{equation}
t_{\ell}^{\prime}:=\operatorname*{cyc}\nolimits_{\ell}+\operatorname*{cyc}%
\nolimits_{\ell+1,\ell}+\operatorname*{cyc}\nolimits_{\ell+2,\ell+1,\ell
}+\cdots+\operatorname*{cyc}\nolimits_{n,n-1,\ldots,\ell}\in\mathbf{k}\left[
S_{n}\right]  . \label{eq.def.tl.deftl'}%
\end{equation}
In terms of card shuffling, this element $t_{\ell}^{\prime}$ corresponds to
randomly picking a card from the bottommost $n-\ell+1$ positions in the deck
(with uniform probabilities) and moving it to position $\ell$. Thus, we call
$t_{1}^{\prime},t_{2}^{\prime},\ldots,t_{n}^{\prime}$ the
\emph{below-to-somewhere shuffles}. The first of them, $t_{1}^{\prime}$, is
known as the \emph{random-to-top shuffle} (as it picks a random card and
surfaces it to the top of the deck).

It is natural to ask whether our above properties of $t_{1},t_{2},\ldots
,t_{n}$ have analogues for these new elements $t_{1}^{\prime},t_{2}^{\prime
},\ldots,t_{n}^{\prime}$. For example, an analogue of Theorem
\ref{thm.eigen.annih-pol} holds:

\begin{theorem}
\label{thm.eigen.annih-pol'}Let $\lambda_{1},\lambda_{2},\ldots,\lambda_{n}%
\in\mathbf{k}$. Let $t^{\prime}:=\lambda_{1}t_{1}^{\prime}+\lambda_{2}%
t_{2}^{\prime}+\cdots+\lambda_{n}t_{n}^{\prime}$. Then,%
\[
\prod_{\substack{I\subseteq\left[  n-1\right]  \text{ is}\\\text{lacunar}%
}}\left(  t^{\prime}-\left(  \lambda_{1}m_{I,1}+\lambda_{2}m_{I,2}%
+\cdots+\lambda_{n}m_{I,n}\right)  \right)  =0.
\]

\end{theorem}

Theorem \ref{thm.eigen.annih-pol'} can actually be deduced from Theorem
\ref{thm.eigen.annih-pol} pretty easily:

Let $S$ be the $\mathbf{k}$-linear map $\mathbf{k}\left[  S_{n}\right]
\rightarrow\mathbf{k}\left[  S_{n}\right]  $ that sends each permutation $w\in
S_{n}$ to its inverse $w^{-1}$. This map $S$ is known as the \emph{antipode}
of the group algebra $\mathbf{k}\left[  S_{n}\right]  $ (see, e.g.,
\cite[Example 2.2.8]{Meusbu21}); it is an involution (i.e., it satisfies
$S\circ S=\operatorname*{id}$) and a $\mathbf{k}$-algebra antihomomorphism
(i.e., it is $\mathbf{k}$-linear and satisfies $S\left(  1\right)  =1$ and
$S\left(  uv\right)  =S\left(  v\right)  \cdot S\left(  u\right)  $ for all
$u,v\in\mathbf{k}\left[  S_{n}\right]  $). For any $k$ distinct elements
$i_{1},i_{2},\ldots,i_{k}$ of $\left[  n\right]  $, we have
\begin{align}
S\left(  \operatorname*{cyc}\nolimits_{i_{1},i_{2},\ldots,i_{k}}\right)   &
=\left(  \operatorname*{cyc}\nolimits_{i_{1},i_{2},\ldots,i_{k}}\right)
^{-1}\ \ \ \ \ \ \ \ \ \ \left(  \text{by the definition of }S\right)
\nonumber\\
&  =\operatorname*{cyc}\nolimits_{i_{k},i_{k-1},\ldots,i_{1}}.
\label{eq.bts.Scyc}%
\end{align}
Hence, for each $\ell\in\left[  n\right]  $, we have%
\begin{align}
S\left(  t_{\ell}\right)   &  =S\left(  \operatorname*{cyc}\nolimits_{\ell
}+\operatorname*{cyc}\nolimits_{\ell,\ell+1}+\operatorname*{cyc}%
\nolimits_{\ell,\ell+1,\ell+2}+\cdots+\operatorname*{cyc}\nolimits_{\ell
,\ell+1,\ldots,n}\right)  \ \ \ \ \ \ \ \ \ \ \left(  \text{by
(\ref{eq.def.tl.deftl})}\right) \nonumber\\
&  =S\left(  \operatorname*{cyc}\nolimits_{\ell}\right)  +S\left(
\operatorname*{cyc}\nolimits_{\ell,\ell+1}\right)  +S\left(
\operatorname*{cyc}\nolimits_{\ell,\ell+1,\ell+2}\right)  +\cdots+S\left(
\operatorname*{cyc}\nolimits_{\ell,\ell+1,\ldots,n}\right) \nonumber\\
&  =\operatorname*{cyc}\nolimits_{\ell}+\operatorname*{cyc}\nolimits_{\ell
+1,\ell}+\operatorname*{cyc}\nolimits_{\ell+2,\ell+1,\ell}+\cdots
+\operatorname*{cyc}\nolimits_{n,n-1,\ldots,\ell}\ \ \ \ \ \ \ \ \ \ \left(
\text{by (\ref{eq.bts.Scyc})}\right) \nonumber\\
&  =t_{\ell}^{\prime}\ \ \ \ \ \ \ \ \ \ \left(  \text{by
(\ref{eq.def.tl.deftl'})}\right)  . \label{eq.def.Stl}%
\end{align}
Thus, we can obtain properties of $t_{1}^{\prime},t_{2}^{\prime},\ldots
,t_{n}^{\prime}$ by applying the map $S$ to corresponding properties of
$t_{1},t_{2},\ldots,t_{n}$. In particular, we can obtain Theorem
\ref{thm.eigen.annih-pol'} this way:

\begin{proof}
[Proof of Theorem \ref{thm.eigen.annih-pol'}.]Let $t:=\lambda_{1}t_{1}%
+\lambda_{2}t_{2}+\cdots+\lambda_{n}t_{n}$. Thus,%
\begin{align*}
S\left(  t\right)   &  =S\left(  \lambda_{1}t_{1}+\lambda_{2}t_{2}%
+\cdots+\lambda_{n}t_{n}\right) \\
&  =\lambda_{1}S\left(  t_{1}\right)  +\lambda_{2}S\left(  t_{2}\right)
+\cdots+\lambda_{n}S\left(  t_{n}\right)  \ \ \ \ \ \ \ \ \ \ \left(
\text{since }S\text{ is }\mathbf{k}\text{-linear}\right) \\
&  =\lambda_{1}t_{1}^{\prime}+\lambda_{2}t_{2}^{\prime}+\cdots+\lambda
_{n}t_{n}^{\prime}\ \ \ \ \ \ \ \ \ \ \left(  \text{by (\ref{eq.def.Stl}%
)}\right) \\
&  =t^{\prime}\ \ \ \ \ \ \ \ \ \ \left(  \text{by the definition of
}t^{\prime}\right)  .
\end{align*}

Now, let $P$ be the polynomial $\prod_{\substack{I\subseteq\left[  n-1\right]
\text{ is}\\\text{lacunar}}}\left(  X-\left(  \lambda_{1}m_{I,1}+\lambda
_{2}m_{I,2}+\cdots+\lambda_{n}m_{I,n}\right)  \right)  \in\mathbf{k}\left[
X\right]  $. Then,%
\[
P\left(  t\right)  =\prod_{\substack{I\subseteq\left[  n-1\right]  \text{
is}\\\text{lacunar}}}\left(  t-\left(  \lambda_{1}m_{I,1}+\lambda_{2}%
m_{I,2}+\cdots+\lambda_{n}m_{I,n}\right)  \right)  =0
\]
(by Theorem \ref{thm.eigen.annih-pol}). Thus, $S\left(  P\left(  t\right)
\right)  =S\left(  0\right)  =0$.

However, $S$ is a $\mathbf{k}$-algebra antihomomorphism. Thus, Proposition
\ref{prop.antihom.unipol} (applied to $A=\mathbf{k}\left[  S_{n}\right]  $,
$B=\mathbf{k}\left[  S_{n}\right]  $, $f=S$ and $u=t$) yields that
\[
S\left(  P\left(  t\right)  \right)  =P\left(  \underbrace{S\left(  t\right)
}_{=t^{\prime}}\right)  =P\left(  t^{\prime}\right)  =\prod
_{\substack{I\subseteq\left[  n-1\right]  \text{ is}\\\text{lacunar}}}\left(
t^{\prime}-\left(  \lambda_{1}m_{I,1}+\lambda_{2}m_{I,2}+\cdots+\lambda
_{n}m_{I,n}\right)  \right)
\]
(by the definition of $P$). Comparing this with $S\left(  P\left(  t\right)
\right)  =0$, we obtain%
\[
\prod_{\substack{I\subseteq\left[  n-1\right]  \text{ is}\\\text{lacunar}%
}}\left(  t^{\prime}-\left(  \lambda_{1}m_{I,1}+\lambda_{2}m_{I,2}%
+\cdots+\lambda_{n}m_{I,n}\right)  \right)  =0.
\]
This proves Theorem \ref{thm.eigen.annih-pol'}.
\end{proof}

A more interesting question is to find an analogue of Theorem
\ref{thm.Rcomb-main} for the below-to-somewhere shuffles: Is there a basis of
the $\mathbf{k}$-module $\mathbf{k}\left[  S_{n}\right]  $ with respect to
which the $\mathbf{k}$-module endomorphisms $R\left(  \lambda_{1}t_{1}%
^{\prime}+\lambda_{2}t_{2}^{\prime}+\cdots+\lambda_{n}t_{n}^{\prime}\right)  $
are represented by triangular matrices for all $\lambda_{1},\lambda_{2}%
,\ldots,\lambda_{n}\in\mathbf{k}$ ? Again, the answer is \textquotedblleft
yes\textquotedblright, but this basis is no longer the descent-destroying
basis $\left(  a_{w}\right)  _{w\in S_{n}}$ (ordered by increasing $Q$-index);
instead, it is the \textbf{dual} basis to $\left(  a_{w}\right)  _{w\in S_{n}%
}$ with respect to a certain bilinear form (ordered by \textbf{decreasing}
$Q$-index). Let us elaborate on this now.\footnote{Note that, with respect to
the \textbf{standard} basis $\left(  w \right)  _{w \in S_{n}}$ of
$\mathbf{k}\left[  S_{n} \right]  $, the matrix representing the endomorphism
$R\left(  \lambda_{1}t_{1}^{\prime}+\lambda_{2}t_{2}^{\prime}+\cdots
+\lambda_{n}t_{n}^{\prime}\right)  $ is the transpose of the matrix
representing the endomorphism $R\left(  \lambda_{1}t_{1}+\lambda_{2}%
t_{2}+\cdots+\lambda_{n}t_{n}\right)  $. However, neither of these two
matrices is triangular.}

First, we recall some concepts from linear algebra (although we are working at
a slightly unusual level of generality, since we do not require $\mathbf{k}$
to be a field):

\begin{itemize}
\item The \emph{dual} of a $\mathbf{k}$-module $U$ is defined to be the
$\mathbf{k}$-module $\operatorname*{Hom}\nolimits_{\mathbf{k}}\left(
U,\mathbf{k}\right)  $ of all $\mathbf{k}$-linear maps from $U$ to
$\mathbf{k}$. We denote this dual by $U^{\vee}$.

\item A \emph{bilinear form} on two $\mathbf{k}$-modules $U$ and $V$ is
defined to be a map $f:U\times V\rightarrow\mathbf{k}$ that is $\mathbf{k}%
$-linear in each of its two arguments. A bilinear form $f:U\times
V\rightarrow\mathbf{k}$ canonically induces a $\mathbf{k}$-module
homomorphism
\begin{align*}
f^{\circ}:V  &  \rightarrow U^{\vee},\\
v  &  \mapsto\left(  \text{the map }U\rightarrow\mathbf{k}\text{ that sends
each }u\in U\text{ to }f\left(  u,v\right)  \right)  .
\end{align*}
A bilinear form $f:U\times V\rightarrow\mathbf{k}$ is called
\emph{nondegenerate} if the $\mathbf{k}$-module homomorphism $f^{\circ
}:V\rightarrow U^{\vee}$ is an isomorphism.

\item If $U$ and $V$ are two $\mathbf{k}$-modules with bases $\left(
u_{w}\right)  _{w\in W}$ and $\left(  v_{w}\right)  _{w\in W}$,
respectively\footnote{Note that the bases must have the same indexing set in
this definition.}, and if $f:U\times V\rightarrow\mathbf{k}$ is a bilinear
form, then we say that the basis $\left(  v_{w}\right)  _{w\in W}$ is
\emph{dual} to $\left(  u_{w}\right)  _{w\in W}$ with respect to $f$ if and
only if we have%
\[
\left(  f\left(  u_{p},v_{q}\right)  =\left[  p=q\right]
\ \ \ \ \ \ \ \ \ \ \text{for all }p,q\in W\right)  .
\]
Here, we are using the \emph{Iverson bracket notation}: For each statement
$\mathcal{A}$, we let $\left[  \mathcal{A}\right]  $ denote the truth value of
$\mathcal{A}$ (that is, $1$ if $\mathcal{A}$ is true and $0$ if $\mathcal{A}$
is false).
\end{itemize}

The following three general facts about dual bases are easy and known:

\begin{proposition}
\label{prop.bilf.nondeg-if-bases}Let $U$ and $V$ be two $\mathbf{k}$-modules,
and let $f:U\times V\rightarrow\mathbf{k}$ be a bilinear form. Let $\left(
u_{w}\right)  _{w\in W}$ be a basis of the $\mathbf{k}$-module $U$ such that
the set $W$ is finite. Let $\left(  v_{w}\right)  _{w\in W}$ be a basis of the
$\mathbf{k}$-module $V$ that is dual to $\left(  u_{w}\right)  _{w\in W}$.
Then, the bilinear form $f$ is nondegenerate.
\end{proposition}

\begin{proof}
[Proof sketch.]Recall that $\left(  u_{w}\right)  _{w\in W}$ is a basis of
$U$. For each $w\in W$, let $c_{w}:U\rightarrow\mathbf{k}$ be the map that
sends each $u\in U$ to the $u_{w}$-coordinate of $u$ with respect to this
basis. This map $c_{w}$ is $\mathbf{k}$-linear and thus belongs to $U^{\vee}$.
Now, it is easy to see that $\left(  c_{w}\right)  _{w\in W}$ is a basis of
$U^{\vee}$ (since $W$ is finite).

However, the basis $\left(  v_{w}\right)  _{w\in W}$ is dual to $\left(
u_{w}\right)  _{w\in W}$. Thus, $f^{\circ}\left(  v_{w}\right)  =c_{w}$ for
each $w\in W$ (since any $w,p\in W$ satisfy $\left(  f^{\circ}\left(
v_{w}\right)  \right)  \left(  u_{p}\right)  =f\left(  u_{p},v_{w}\right)
=\left[  p=w\right]  =c_{w}\left(  u_{p}\right)  $). In other words, the map
$f^{\circ}$ sends the basis $\left(  v_{w}\right)  _{w\in W}$ of $V$ to the
basis $\left(  c_{w}\right)  _{w\in W}$ of $U^{\vee}$. This entails that
$f^{\circ}$ is an isomorphism (since any $\mathbf{k}$-linear map that sends a
basis of its domain to a basis of its target must be an isomorphism). In other
words, $f$ is nondegenerate. This proves Proposition
\ref{prop.bilf.nondeg-if-bases}.
\end{proof}

\begin{proposition}
\label{prop.bilf.dual-bases}Let $U$ and $V$ be two $\mathbf{k}$-modules, and
let $f:U\times V\rightarrow\mathbf{k}$ be a nondegenerate bilinear form. Let
$\left(  u_{w}\right)  _{w\in W}$ be a basis of the $\mathbf{k}$-module $U$,
where $W$ is a finite set. Then, there is a unique basis of $V$ that is dual
to $\left(  u_{w}\right)  _{w\in W}$ with respect to $f$.
\end{proposition}

\begin{proof}
[Proof sketch.]Since $f$ is nondegenerate, the map $f^{\circ}:V\rightarrow
U^{\vee}$ is an isomorphism. Thus, we can WLOG assume that $V=U^{\vee}$ and
that $f$ is the standard pairing between $U$ and $U^{\vee}$ (that is, the
bilinear form $U\times U^{\vee}\rightarrow\mathbf{k}$ that sends each pair
$\left(  u,f\right)  $ to $f\left(  u\right)  \in\mathbf{k}$). Now, recall
that $\left(  u_{w}\right)  _{w\in W}$ is a basis of $U$. For each $w\in W$,
let $c_{w}:U\rightarrow\mathbf{k}$ be the map that sends each $u\in U$ to the
$u_{w}$-coordinate of $u$ with respect to this basis. This map $c_{w}$ is
$\mathbf{k}$-linear and thus belongs to $U^{\vee}$. Now, it is easy to see
that $\left(  c_{w}\right)  _{w\in W}$ is a basis of $U^{\vee}=V$ that is dual
to $\left(  u_{w}\right)  _{w\in W}$ with respect to $f$, and moreover it is
the only such basis. Hence, Proposition \ref{prop.bilf.dual-bases} follows.
\end{proof}

\begin{proposition}
\label{prop.bilf.dual-expansions}Let $U$ and $V$ be two $\mathbf{k}$-modules,
and let $f:U\times V\rightarrow\mathbf{k}$ be a bilinear form. Let $\left(
u_{w}\right)  _{w\in W}$ be a basis of the $\mathbf{k}$-module $U$ such that
the set $W$ is finite. Let $\left(  v_{w}\right)  _{w\in W}$ be a basis of the
$\mathbf{k}$-module $V$ that is dual to $\left(  u_{w}\right)  _{w\in W}$. Then:

\begin{enumerate}
\item[\textbf{(a)}] For any $u\in U$, we have%
\[
u=\sum_{w\in W}f\left(  u,v_{w}\right)  u_{w}.
\]

\item[\textbf{(b)}] For any $v\in V$, we have%
\[
v=\sum_{w\in W}f\left(  u_{w},v\right)  v_{w}.
\]

\end{enumerate}
\end{proposition}

\begin{proof}
\textbf{(a)} Let $u\in U$. Recall that $\left(  u_{w}\right)  _{w\in W}$ is a
basis of the $\mathbf{k}$-module $U$. Thus, we can write $u$ as a $\mathbf{k}%
$-linear combination of this basis. In other words, there exists a family
$\left(  \lambda_{w}\right)  _{w\in W}\in\mathbf{k}^{W}$ of scalars such that%
\begin{equation}
u=\sum_{w\in W}\lambda_{w}u_{w}. \label{pf.prop.bilf.dual-expansions.a.u=sum}%
\end{equation}
Consider this family.

We have assumed that the basis $\left(  v_{w}\right)  _{w\in W}$ is dual to
$\left(  u_{w}\right)  _{w\in W}$. In other words, we have
\begin{equation}
\left(  f\left(  u_{p},v_{q}\right)  =\left[  p=q\right]
\ \ \ \ \ \ \ \ \ \ \text{for all }p,q\in W\right)  .
\label{pf.prop.bilf.dual-expansions.a.dual}%
\end{equation}

Now, for each $q\in W$, we have%
\begin{align*}
f\left(  u,v_{q}\right)   &  =f\left(  \sum_{w\in W}\lambda_{w}u_{w}%
,\ v_{q}\right)  \ \ \ \ \ \ \ \ \ \ \left(  \text{since }u=\sum_{w\in
W}\lambda_{w}u_{w}\right) \\
&  =\sum_{w\in W}\lambda_{w}\underbrace{f\left(  u_{w},v_{q}\right)
}_{\substack{=\left[  w=q\right]  \\\text{(by
(\ref{pf.prop.bilf.dual-expansions.a.dual}), applied to }p=w\text{)}%
}}\ \ \ \ \ \ \ \ \ \ \left(  \text{since }f\text{ is a bilinear form}\right)
\\
&  =\sum_{w\in W}\lambda_{w}\left[  w=q\right]  =\lambda_{q}%
\underbrace{\left[  q=q\right]  }_{=1}+\sum_{\substack{w\in W;\\w\neq
q}}\lambda_{w}\underbrace{\left[  w=q\right]  }_{\substack{=0\\\text{(since
}w\neq q\text{)}}}\\
&  \ \ \ \ \ \ \ \ \ \ \ \ \ \ \ \ \ \ \ \ \left(  \text{here, we have split
off the addend for }w=q\text{ from the sum}\right) \\
&  =\lambda_{q}+\underbrace{\sum_{\substack{w\in W;\\w\neq q}}\lambda_{w}%
0}_{=0}=\lambda_{q}.
\end{align*}
Renaming the variable $q$ as $w$ in this result, we obtain the following: For
each $w\in W$, we have
\begin{equation}
f\left(  u,v_{w}\right)  =\lambda_{w}.
\label{pf.prop.bilf.dual-expansions.a.f=lam}%
\end{equation}

Now, (\ref{pf.prop.bilf.dual-expansions.a.u=sum}) becomes%
\[
u=\sum_{w\in W}\underbrace{\lambda_{w}}_{\substack{=f\left(  u,v_{w}\right)
\\\text{(by (\ref{pf.prop.bilf.dual-expansions.a.f=lam}))}}}u_{w}=\sum_{w\in
W}f\left(  u,v_{w}\right)  u_{w}.
\]
This proves Proposition \ref{prop.bilf.dual-expansions} \textbf{(a)}. \medskip

\textbf{(b)} This is analogous to the proof of part \textbf{(a)} (but, of
course, the obvious changes need to be made -- e.g., the equality
(\ref{pf.prop.bilf.dual-expansions.a.u=sum}) is replaced by $v=\sum_{w\in
W}\lambda_{w}v_{w}$, and the equality
(\ref{pf.prop.bilf.dual-expansions.a.f=lam}) is replaced by $f\left(
u_{w},v\right)  =\lambda_{w}$).
\end{proof}

Now, we apply the above to the $\mathbf{k}$-module $\mathbf{k}\left[
S_{n}\right]  $. We define a bilinear form $f:\mathbf{k}\left[  S_{n}\right]
\times\mathbf{k}\left[  S_{n}\right]  \rightarrow\mathbf{k}$ by setting%
\begin{equation}
f\left(  p,q\right)  =\left[  p=q\right]  \ \ \ \ \ \ \ \ \ \ \text{for all
}p,q\in S_{n}. \label{eq.bilf.onkSn}%
\end{equation}
(This defines a unique bilinear form, since $\left(  w\right)  _{w\in S_{n}}$
is a basis of the $\mathbf{k}$-module $\mathbf{k}\left[  S_{n}\right]  $.)
Clearly, the basis $\left(  w\right)  _{w\in S_{n}}$ of $\mathbf{k}\left[
S_{n}\right]  $ is dual to itself with respect to this form $f$. Thus,
Proposition \ref{prop.bilf.nondeg-if-bases} (applied to $U=\mathbf{k}\left[
S_{n}\right]  $, $V=\mathbf{k}\left[  S_{n}\right]  $, $W=S_{n}$, $\left(
u_{w}\right)  _{w\in W}=\left(  w\right)  _{w\in S_{n}}$ and $\left(
v_{w}\right)  _{w\in W}=\left(  w\right)  _{w\in S_{n}}$) yields that the
bilinear form $f$ is nondegenerate. Hence, Proposition
\ref{prop.bilf.dual-bases} (applied to $U=\mathbf{k}\left[  S_{n}\right]  $,
$V=\mathbf{k}\left[  S_{n}\right]  $, $W=S_{n}$ and $\left(  u_{w}\right)
_{w\in W}=\left(  a_{w}\right)  _{w\in S_{n}}$) yields that there is a unique
basis of $\mathbf{k}\left[  S_{n}\right]  $ that is dual to $\left(
a_{w}\right)  _{w\in S_{n}}$ with respect to $f$ (since Proposition
\ref{prop.aw.basis-kSn} tells us that $\left(  a_{w}\right)  _{w\in S_{n}}$ is
a basis of $\mathbf{k}\left[  S_{n}\right]  $). Let us denote this basis by
$\left(  b_{w}\right)  _{w\in S_{n}}$. Thus, the basis $\left(  b_{w}\right)
_{w\in S_{n}}$ is dual to $\left(  a_{w}\right)  _{w\in S_{n}}$; in other
words, we have%
\begin{equation}
f\left(  a_{p},b_{q}\right)  =\left[  p=q\right]
\ \ \ \ \ \ \ \ \ \ \text{for all }p,q\in S_{n}. \label{eq.bilf.onkSn-ab}%
\end{equation}

Now, we claim the following analogue to Theorem \ref{thm.Rcomb-conc}:

\begin{theorem}
\label{thm.Rcomb-conc'}Let $w\in S_{n}$ and $\ell\in\left[  n\right]  $. Let
$i=\operatorname*{Qind}w$. Then,%
\[
b_{w}t_{\ell}^{\prime}=m_{Q_{i},\ell}b_{w}+\left(  \text{a }\mathbf{k}%
\text{-linear combination of }b_{v}\text{'s for }v\in S_{n}\text{ satisfying
}\operatorname*{Qind}v>i\right)  .
\]

\end{theorem}

Once we have proved Theorem \ref{thm.Rcomb-conc'}, it will follow that if we
order the basis $\left(  b_{w}\right)  _{w\in S_{n}}$ in the order of
decreasing $Q$-index, the endomorphisms $R\left(  t_{1}^{\prime}\right)
,R\left(  t_{2}^{\prime}\right)  ,\ldots,R\left(  t_{n}^{\prime}\right)  $
(and thus also their linear combinations $R\left(  \lambda_{1}t_{1}^{\prime
}+\lambda_{2}t_{2}^{\prime}+\cdots+\lambda_{n}t_{n}^{\prime}\right)  $) will
be represented by upper-triangular matrices. The analogue of Theorem
\ref{thm.Rcomb-main} for below-to-somewhere shuffles will thus follow. So it
remains to prove Theorem \ref{thm.Rcomb-conc'}. In order to do so, we need a
simple lemma about the bilinear form $f:\mathbf{k}\left[  S_{n}\right]
\times\mathbf{k}\left[  S_{n}\right]  \rightarrow\mathbf{k}$ defined by
(\ref{eq.bilf.onkSn}):

\begin{lemma}
\label{lem.S.self-adj}We have%
\[
f\left(  u,vS\left(  x\right)  \right)  =f\left(  ux,v\right)
\ \ \ \ \ \ \ \ \ \ \text{for all }x,u,v\in\mathbf{k}\left[  S_{n}\right]  .
\]

\end{lemma}

\begin{proof}
Let $x,u,v\in\mathbf{k}\left[  S_{n}\right]  $. We must prove the equality
$f\left(  u,vS\left(  x\right)  \right)  =f\left(  ux,v\right)  $. Both sides
of this equality depend $\mathbf{k}$-linearly on each of the three elements
$x,u,v$ (since the map $f$ is $\mathbf{k}$-linear in each argument, whereas
the map $S$ is $\mathbf{k}$-linear). Hence, in order to prove this equality,
we can WLOG assume that all of $x,u,v$ belong to the basis $\left(  w\right)
_{w\in S_{n}}$ of the $\mathbf{k}$-module $\mathbf{k}\left[  S_{n}\right]  $.
Assume this.

Thus, $x,u,v\in S_{n}$. The definition of $S$ now yields $S\left(  x\right)
=x^{-1}$ (since $x\in S_{n}$). Moreover, $vx^{-1}\in S_{n}$ (since $v$ and
$x^{-1}$ belong to $S_{n}$) and $ux\in S_{n}$ (since $u$ and $x$ belong to
$S_{n}$). Furthermore, (\ref{eq.bilf.onkSn}) (applied to $p=u$ and $q=vx^{-1}%
$) yields $f\left(  u,vx^{-1}\right)  =\left[  u=vx^{-1}\right]  $ (since $u$
and $vx^{-1}$ belong to $S_{n}$). Likewise, (\ref{eq.bilf.onkSn}) (applied to
$p=ux$ and $q=v$) yields $f\left(  ux,v\right)  =\left[  ux=v\right]  $.

However, the two statements $u=vx^{-1}$ and $ux=v$ are clearly equivalent.
Thus, their truth values are equal. In other words, $\left[  u=vx^{-1}\right]
=\left[  ux=v\right]  $. Combining what we have shown above, we obtain%
\[
f\left(  u,v\underbrace{S\left(  x\right)  }_{=x^{-1}}\right)  =f\left(
u,vx^{-1}\right)  =\left[  u=vx^{-1}\right]  =\left[  ux=v\right]  =f\left(
ux,v\right)
\]
(since $f\left(  ux,v\right)  =\left[  ux=v\right]  $). This is precisely the
equality that we wanted to prove. Thus, Lemma \ref{lem.S.self-adj} is proved.
\end{proof}

\begin{proof}
[Proof of Theorem \ref{thm.Rcomb-conc'}.]Forget that we fixed $w$ and $i$ (but
keep $\ell$ fixed). For each $u\in S_{n}$, define two elements%
\[
\widetilde{a}_{u}:=a_{u}t_{\ell}-m_{Q_{\operatorname*{Qind}u},\ell}%
a_{u}\ \ \ \ \ \ \ \ \ \ \text{and}\ \ \ \ \ \ \ \ \ \ \widetilde{b}%
_{u}:=b_{u}t_{\ell}^{\prime}-m_{Q_{\operatorname*{Qind}u},\ell}b_{u}%
\]
of $\mathbf{k}\left[  S_{n}\right]  $.

We know that the family $\left(  a_{w}\right)  _{w\in S_{n}}$ is a basis of
the $\mathbf{k}$-module $\mathbf{k}\left[  S_{n}\right]  $; we called this
basis the descent-destroying basis. We also know that $\left(  b_{w}\right)
_{w\in S_{n}}$ is a basis of $\mathbf{k}\left[  S_{n}\right]  $ that is dual
to $\left(  a_{w}\right)  _{w\in S_{n}}$ with respect to $f$. Thus,
Proposition \ref{prop.bilf.dual-expansions} \textbf{(a)} (applied to
$U=\mathbf{k}\left[  S_{n}\right]  $, $V=\mathbf{k}\left[  S_{n}\right]  $,
$W=S_{n}$, $\left(  u_{w}\right)  _{w\in W}=\left(  a_{w}\right)  _{w\in
S_{n}}$ and $\left(  v_{w}\right)  _{w\in W}=\left(  b_{w}\right)  _{w\in
S_{n}}$) shows that each $u\in\mathbf{k}\left[  S_{n}\right]  $ satisfies%
\begin{equation}
u=\sum_{w\in S_{n}}f\left(  u,b_{w}\right)  a_{w}.
\label{pf.thm.Rcomb-conc'.u=sum}%
\end{equation}
Furthermore, Proposition \ref{prop.bilf.dual-expansions} \textbf{(b)} (applied
to $U=\mathbf{k}\left[  S_{n}\right]  $, $V=\mathbf{k}\left[  S_{n}\right]  $,
$W=S_{n}$, $\left(  u_{w}\right)  _{w\in W}=\left(  a_{w}\right)  _{w\in
S_{n}}$ and $\left(  v_{w}\right)  _{w\in W}=\left(  b_{w}\right)  _{w\in
S_{n}}$) shows that each $v\in\mathbf{k}\left[  S_{n}\right]  $ satisfies%
\begin{equation}
v=\sum_{w\in S_{n}}f\left(  a_{w},v\right)  b_{w}.
\label{pf.thm.Rcomb-conc'.v=sum}%
\end{equation}

For each $u\in S_{n}$, we have%
\begin{equation}
\widetilde{a}_{u}=\sum_{w\in S_{n}}f\left(  \widetilde{a}_{u},b_{w}\right)
a_{w} \label{pf.thm.Rcomb-conc'.1}%
\end{equation}
(by (\ref{pf.thm.Rcomb-conc'.u=sum}), applied to $\widetilde{a}_{u}$ instead
of $u$).

For each $v\in S_{n}$, we have%
\[
\widetilde{b}_{v}=\sum_{w\in S_{n}}f\left(  a_{w},\widetilde{b}_{v}\right)
b_{w}%
\]
(by (\ref{pf.thm.Rcomb-conc'.v=sum}), applied to $\widetilde{b}_{v}$ instead
of $v$). Renaming the indices $v$ and $w$ as $w$ and $v$ in this sentence, we
obtain the following: For each $w\in S_{n}$, we have%
\begin{equation}
\widetilde{b}_{w}=\sum_{v\in S_{n}}f\left(  a_{v},\widetilde{b}_{w}\right)
b_{v}. \label{pf.thm.Rcomb-conc'.2}%
\end{equation}

We shall now prove the following:

\begin{statement}
\textit{Claim 1:} Let $u,w\in S_{n}$ be such that $\operatorname*{Qind}%
w\geq\operatorname*{Qind}u$. Then, $f\left(  \widetilde{a}_{u},b_{w}\right)
=0$.
\end{statement}

\begin{statement}
\textit{Claim 2:} Let $u,w\in S_{n}$. Then, $f\left(  a_{u},\widetilde{b}%
_{w}\right)  =f\left(  \widetilde{a}_{u},b_{w}\right)  $.
\end{statement}

[\textit{Proof of Claim 1:} Let $j=\operatorname*{Qind}u$. By assumption, we
have $\operatorname*{Qind}w\geq\operatorname*{Qind}u=j$. Thus, $w$ does not
satisfy $\operatorname*{Qind}w<j$.

Theorem \ref{thm.Rcomb-conc} (applied to $u$ and $j$ instead of $w$ and $i$)
yields
\[
a_{u}t_{\ell}=m_{Q_{j},\ell}a_{u}+\left(  \text{a }\mathbf{k}\text{-linear
combination of }a_{v}\text{'s for }v\in S_{n}\text{ satisfying }%
\operatorname*{Qind}v<j\right)
\]
(since $j=\operatorname*{Qind}u$). In other words,%
\[
a_{u}t_{\ell}-m_{Q_{j},\ell}a_{u}=\left(  \text{a }\mathbf{k}\text{-linear
combination of }a_{v}\text{'s for }v\in S_{n}\text{ satisfying }%
\operatorname*{Qind}v<j\right)  .
\]
In view of
\[
\widetilde{a}_{u}=a_{u}t_{\ell}-m_{Q_{\operatorname*{Qind}u},\ell}a_{u}%
=a_{u}t_{\ell}-m_{Q_{j},\ell}a_{u}\ \ \ \ \ \ \ \ \ \ \left(  \text{since
}\operatorname*{Qind}u=j\right)  ,
\]
we can rewrite this as
\[
\widetilde{a}_{u}=\left(  \text{a }\mathbf{k}\text{-linear combination of
}a_{v}\text{'s for }v\in S_{n}\text{ satisfying }\operatorname*{Qind}%
v<j\right)  .
\]
This equality shows that $\widetilde{a}_{u}$ can be written as a $\mathbf{k}%
$-linear combination of the descent-destroying basis, and the only basis
elements that appear (with nonzero coefficients) in this combination are the
$a_{v}$ for $v\in S_{n}$ satisfying $\operatorname*{Qind}v<j$. Hence, if $v\in
S_{n}$ does not satisfy $\operatorname*{Qind}v<j$, then $a_{v}$ does not
appear in the expansion of $\widetilde{a}_{u}$ as a $\mathbf{k}$-linear
combination of the descent-destroying basis. Applying this to $v=w$, we
conclude that $a_{w}$ does not appear in the expansion of $\widetilde{a}_{u}$
as a $\mathbf{k}$-linear combination of the descent-destroying basis (since
$w\in S_{n}$ does not satisfy $\operatorname*{Qind}w<j$). In other words, the
coefficient of $a_{w}$ when $\widetilde{a}_{u}$ is expanded as a $\mathbf{k}%
$-linear combination of the descent-destroying basis is $0$.

However, the equality (\ref{pf.thm.Rcomb-conc'.1}) shows that $f\left(
\widetilde{a}_{u},b_{w}\right)  $ is the coefficient of $a_{w}$ when
$\widetilde{a}_{u}$ is expanded as a $\mathbf{k}$-linear combination of the
descent-destroying basis. But we have just shown that this coefficient is $0$.
Thus, we conclude that $f\left(  \widetilde{a}_{u},b_{w}\right)  =0$. This
proves Claim 1.] \medskip

[\textit{Proof of Claim 2:} The definition of $\widetilde{a}_{u}$ yields
$\widetilde{a}_{u}=a_{u}t_{\ell}-m_{Q_{\operatorname*{Qind}u},\ell}a_{u}$.
Thus,%
\begin{align}
f\left(  \widetilde{a}_{u},b_{w}\right)   &  =f\left(  a_{u}t_{\ell
}-m_{Q_{\operatorname*{Qind}u},\ell}a_{u},\ b_{w}\right) \nonumber\\
&  =f\left(  a_{u}t_{\ell},b_{w}\right)  -m_{Q_{\operatorname*{Qind}u},\ell
}\underbrace{f\left(  a_{u},b_{w}\right)  }_{\substack{=\left[  u=w\right]
\\\text{(by (\ref{eq.bilf.onkSn-ab}), applied to }p=u\\\text{and }q=w\text{)}%
}}\ \ \ \ \ \ \ \ \ \ \left(  \text{since }f\text{ is a bilinear form}\right)
\nonumber\\
&  =f\left(  a_{u}t_{\ell},b_{w}\right)  -m_{Q_{\operatorname*{Qind}u},\ell
}\left[  u=w\right]  . \label{pf.thm.Rcomb-conc'.c2.pf.1}%
\end{align}

However, it is easy to see that%
\begin{equation}
m_{Q_{\operatorname*{Qind}u},\ell}\left[  u=w\right]
=m_{Q_{\operatorname*{Qind}w},\ell}\left[  u=w\right]
\label{pf.thm.Rcomb-conc'.c2.pf.2}%
\end{equation}
\footnote{\textit{Proof of (\ref{pf.thm.Rcomb-conc'.c2.pf.2}):} If $u=w$, then
(\ref{pf.thm.Rcomb-conc'.c2.pf.2}) is obvious. Hence, we WLOG assume that
$u\neq w$. Thus, $\left[  u=w\right]  =0$. Hence, both sides of
(\ref{pf.thm.Rcomb-conc'.c2.pf.2}) equal $0$ (since they contain the factor
$\left[  u=w\right]  =0$). Thus, (\ref{pf.thm.Rcomb-conc'.c2.pf.2}) holds,
qed.}.

On the other hand, the definition of $\widetilde{b}_{w}$ yields $\widetilde{b}%
_{w}=b_{w}t_{\ell}^{\prime}-m_{Q_{\operatorname*{Qind}w},\ell}b_{w}$. Hence,%
\begin{align*}
f\left(  a_{u},\widetilde{b}_{w}\right)   &  =f\left(  a_{u},\ b_{w}t_{\ell
}^{\prime}-m_{Q_{\operatorname*{Qind}w},\ell}b_{w}\right) \\
&  =f\left(  a_{u},b_{w}\underbrace{t_{\ell}^{\prime}}_{\substack{=S\left(
t_{\ell}\right)  \\\text{(by (\ref{eq.def.Stl}))}}}\right)
-m_{Q_{\operatorname*{Qind}w},\ell}\underbrace{f\left(  a_{u},b_{w}\right)
}_{\substack{=\left[  u=w\right]  \\\text{(by (\ref{eq.bilf.onkSn-ab}),
applied to }p=u\\\text{and }q=w\text{)}}}\\
&  \ \ \ \ \ \ \ \ \ \ \ \ \ \ \ \ \ \ \ \ \left(  \text{since }f\text{ is a
bilinear form}\right) \\
&  =\underbrace{f\left(  a_{u},b_{w}S\left(  t_{\ell}\right)  \right)
}_{\substack{=f\left(  a_{u}t_{\ell},b_{w}\right)  \\\text{(by Lemma
\ref{lem.S.self-adj}, applied to }a_{u}\text{, }b_{w}\text{ and }t_{\ell
}\\\text{instead of }u\text{, }v\text{ and }x\text{)}}%
}-\underbrace{m_{Q_{\operatorname*{Qind}w},\ell}\left[  u=w\right]
}_{\substack{=m_{Q_{\operatorname*{Qind}u},\ell}\left[  u=w\right]
\\\text{(by (\ref{pf.thm.Rcomb-conc'.c2.pf.2}))}}}\\
&  =f\left(  a_{u}t_{\ell},b_{w}\right)  -m_{Q_{\operatorname*{Qind}u},\ell
}\left[  u=w\right]  =f\left(  \widetilde{a}_{u},b_{w}\right)
\end{align*}
(by (\ref{pf.thm.Rcomb-conc'.c2.pf.1})). This proves Claim 2.] \medskip

Now, let $w\in S_{n}$. Let $i=\operatorname*{Qind}w$. Then, the definition of
$\widetilde{b}_{w}$ yields $\widetilde{b}_{w}=b_{w}t_{\ell}^{\prime
}-m_{Q_{\operatorname*{Qind}w},\ell}b_{w}=b_{w}t_{\ell}^{\prime}-m_{Q_{i}%
,\ell}b_{w}$ (since $\operatorname*{Qind}w=i$). However,
(\ref{pf.thm.Rcomb-conc'.2}) yields%
\begin{align*}
\widetilde{b}_{w}  &  =\sum_{v\in S_{n}}\underbrace{f\left(  a_{v}%
,\widetilde{b}_{w}\right)  }_{\substack{=f\left(  \widetilde{a}_{v}%
,b_{w}\right)  \\\text{(by Claim 2, applied to }u=v\text{)}}}b_{v}=\sum_{v\in
S_{n}}f\left(  \widetilde{a}_{v},b_{w}\right)  b_{v}\\
&  =\sum_{\substack{v\in S_{n};\\\operatorname*{Qind}w\geq\operatorname*{Qind}%
v}}\underbrace{f\left(  \widetilde{a}_{v},b_{w}\right)  }%
_{\substack{=0\\\text{(by Claim 1, applied to }u=v\text{)}}}b_{v}%
+\sum_{\substack{v\in S_{n};\\\operatorname*{Qind}w<\operatorname*{Qind}%
v}}f\left(  \widetilde{a}_{v},b_{w}\right)  b_{v}\\
&  \ \ \ \ \ \ \ \ \ \ \ \ \ \ \ \ \ \ \ \ \left(
\begin{array}
[c]{c}%
\text{since each }v\in S_{n}\text{ satisfies either }\operatorname*{Qind}%
w\geq\operatorname*{Qind}v\\
\text{or }\operatorname*{Qind}w<\operatorname*{Qind}v\text{ (but not both)}%
\end{array}
\right) \\
&  =\underbrace{\sum_{\substack{v\in S_{n};\\\operatorname*{Qind}%
w\geq\operatorname*{Qind}v}}0b_{v}}_{=0}+\sum_{\substack{v\in S_{n}%
;\\\operatorname*{Qind}w<\operatorname*{Qind}v}}f\left(  \widetilde{a}%
_{v},b_{w}\right)  b_{v}=\sum_{\substack{v\in S_{n};\\\operatorname*{Qind}%
w<\operatorname*{Qind}v}}f\left(  \widetilde{a}_{v},b_{w}\right)  b_{v}\\
&  =\left(  \text{a }\mathbf{k}\text{-linear combination of }b_{v}\text{'s for
}v\in S_{n}\text{ satisfying }\operatorname*{Qind}w<\operatorname*{Qind}%
v\right) \\
&  =\left(  \text{a }\mathbf{k}\text{-linear combination of }b_{v}\text{'s for
}v\in S_{n}\text{ satisfying }i<\operatorname*{Qind}v\right) \\
&  \ \ \ \ \ \ \ \ \ \ \ \ \ \ \ \ \ \ \ \ \left(  \text{since }%
\operatorname*{Qind}w=i\right) \\
&  =\left(  \text{a }\mathbf{k}\text{-linear combination of }b_{v}\text{'s for
}v\in S_{n}\text{ satisfying }\operatorname*{Qind}v>i\right)
\end{align*}
(since $i<\operatorname*{Qind}v$ is equivalent to $\operatorname*{Qind}v>i$).
In view of $\widetilde{b}_{w}=b_{w}t_{\ell}^{\prime}-m_{Q_{i},\ell}b_{w}$,
this can be rewritten as%
\[
b_{w}t_{\ell}^{\prime}-m_{Q_{i},\ell}b_{w}=\left(  \text{a }\mathbf{k}%
\text{-linear combination of }b_{v}\text{'s for }v\in S_{n}\text{ satisfying
}\operatorname*{Qind}v>i\right)  .
\]
In other words,
\[
b_{w}t_{\ell}^{\prime}=m_{Q_{i},\ell}b_{w}+\left(  \text{a }\mathbf{k}%
\text{-linear combination of }b_{v}\text{'s for }v\in S_{n}\text{ satisfying
}\operatorname*{Qind}v>i\right)  .
\]
This proves Theorem \ref{thm.Rcomb-conc'}.
\end{proof}

\subsection{\label{subsect.furtheralg.L}Left multiplication}

For each element $x\in\mathbf{k}\left[  S_{n}\right]  $, let $L\left(
x\right)  $ denote the $\mathbf{k}$-linear map%
\begin{align*}
\mathbf{k}\left[  S_{n}\right]   &  \rightarrow\mathbf{k}\left[  S_{n}\right]
,\\
y  &  \mapsto xy.
\end{align*}
This is a \textquotedblleft left\textquotedblright\ analogue to the right
multiplication map $R\left(  x\right)  $. It is interesting to study from a
shuffling perspective, as this corresponds to shuffling on the labels of a
permutation instead of shuffling on the positions. Thus, having studied
$R\left(  \lambda_{1}t_{1}+\lambda_{2}t_{2}+\cdots+\lambda_{n}t_{n}\right)  $
in detail, we may wonder which of our results extend to $L\left(  \lambda
_{1}t_{1}+\lambda_{2}t_{2}+\cdots+\lambda_{n}t_{n}\right)  $. In particular,
does an analogue of Theorem \ref{thm.Rcomb-main} hold for $L\left(
\lambda_{1}t_{1}+\lambda_{2}t_{2}+\cdots+\lambda_{n}t_{n}\right)  $ instead of
$R\left(  \lambda_{1}t_{1}+\lambda_{2}t_{2}+\cdots+\lambda_{n}t_{n}\right)  $ ?

The answer is \textquotedblleft yes\textquotedblright, and in fact it turns
out that this question is equivalent to the analogous question for $R\left(
\lambda_{1}t_{1}^{\prime}+\lambda_{2}t_{2}^{\prime}+\cdots+\lambda_{n}%
t_{n}^{\prime}\right)  $ answered (in the positive) in Subsection
\ref{subsect.furtheralg.btr}, because the endomorphisms $L\left(  \lambda
_{1}t_{1}+\lambda_{2}t_{2}+\cdots+\lambda_{n}t_{n}\right)  $ and $R\left(
\lambda_{1}t_{1}^{\prime}+\lambda_{2}t_{2}^{\prime}+\cdots+\lambda_{n}%
t_{n}^{\prime}\right)  $ are conjugate via the antipode $S$. More generally,
the following holds:\footnote{Recall that $S$ is the $\mathbf{k}$-linear map
from $\mathbf{k}\left[  S_{n}\right]  $ to $\mathbf{k}\left[  S_{n}\right]  $
that sends each $w\in S_{n}$ to $w^{-1}$.}

\begin{proposition}
\label{prop.L.LxSx}Let $x\in\mathbf{k}\left[  S_{n}\right]  $. Then, the
endomorphisms $L\left(  x\right)  $ and $R\left(  S\left(  x\right)  \right)
$ of $\mathbf{k}\left[  S_{n}\right]  $ are mutually conjugate in the
endomorphism ring $\operatorname*{End}\nolimits_{\mathbf{k}}\left(
\mathbf{k}\left[  S_{n}\right]  \right)  $ of the $\mathbf{k}$-module
$\mathbf{k}\left[  S_{n}\right]  $. Namely, we have%
\begin{equation}
R\left(  S\left(  x\right)  \right)  =S\circ\left(  L\left(  x\right)
\right)  \circ S^{-1}. \label{eq.prop.L.LxSx.1}%
\end{equation}

\end{proposition}

\begin{proof}
Let $y\in\mathbf{k}\left[  S_{n}\right]  $. Recall that $S$ is an involution;
thus, $S$ is invertible. Hence, $S^{-1}$ exists. Moreover, recall that $S$ is
a $\mathbf{k}$-algebra antihomomorphism; thus, we have%
\begin{equation}
S\left(  xz\right)  =S\left(  z\right)  S\left(  x\right)
\ \ \ \ \ \ \ \ \ \ \text{for each }z\in\mathbf{k}\left[  S_{n}\right]  .
\label{pf.prop.L.LxSx.1}%
\end{equation}

Now, comparing%
\[
\left(  R\left(  S\left(  x\right)  \right)  \right)  \left(  y\right)
=yS\left(  x\right)  \ \ \ \ \ \ \ \ \ \ \left(  \text{by the definition of
}R\left(  S\left(  x\right)  \right)  \right)
\]
with%
\begin{align*}
\left(  S\circ\left(  L\left(  x\right)  \right)  \circ S^{-1}\right)  \left(
y\right)   &  =S\left(  \underbrace{\left(  L\left(  x\right)  \right)
\left(  S^{-1}\left(  y\right)  \right)  }_{\substack{=xS^{-1}\left(
y\right)  \\\text{(by the definition of }L\left(  x\right)  \text{)}}}\right)
=S\left(  xS^{-1}\left(  y\right)  \right) \\
&  =\underbrace{S\left(  S^{-1}\left(  y\right)  \right)  }_{=y}S\left(
x\right)  \ \ \ \ \ \ \ \ \ \ \left(  \text{by (\ref{pf.prop.L.LxSx.1}),
applied to }z=S^{-1}\left(  y\right)  \right) \\
&  =yS\left(  x\right)  ,
\end{align*}
we obtain $\left(  R\left(  S\left(  x\right)  \right)  \right)  \left(
y\right)  =\left(  S\circ\left(  L\left(  x\right)  \right)  \circ
S^{-1}\right)  \left(  y\right)  $.

Forget that we fixed $y$. We thus have shown that $\left(  R\left(  S\left(
x\right)  \right)  \right)  \left(  y\right)  =\left(  S\circ\left(  L\left(
x\right)  \right)  \circ S^{-1}\right)  \left(  y\right)  $ for each
$y\in\mathbf{k}\left[  S_{n}\right]  $. In other words, $R\left(  S\left(
x\right)  \right)  =S\circ\left(  L\left(  x\right)  \right)  \circ S^{-1}$.
Hence, the endomorphisms $L\left(  x\right)  $ and $R\left(  S\left(
x\right)  \right)  $ of $\mathbf{k}\left[  S_{n}\right]  $ are mutually
conjugate in the endomorphism ring $\operatorname*{End}\nolimits_{\mathbf{k}%
}\left(  \mathbf{k}\left[  S_{n}\right]  \right)  $ of the $\mathbf{k}$-module
$\mathbf{k}\left[  S_{n}\right]  $. This proves Proposition \ref{prop.L.LxSx}.
\end{proof}

\begin{corollary}
\label{cor.L.conj}Let $\lambda_{1},\lambda_{2},\ldots,\lambda_{n}\in
\mathbf{k}$. Then, the endomorphisms $L\left(  \lambda_{1}t_{1}+\lambda
_{2}t_{2}+\cdots+\lambda_{n}t_{n}\right)  $ and $R\left(  \lambda_{1}%
t_{1}^{\prime}+\lambda_{2}t_{2}^{\prime}+\cdots+\lambda_{n}t_{n}^{\prime
}\right)  $ of $\mathbf{k}\left[  S_{n}\right]  $ are mutually conjugate in
the endomorphism ring $\operatorname*{End}\nolimits_{\mathbf{k}}\left(
\mathbf{k}\left[  S_{n}\right]  \right)  $ of the $\mathbf{k}$-module
$\mathbf{k}\left[  S_{n}\right]  $. Namely, we have%
\[
R\left(  \lambda_{1}t_{1}^{\prime}+\lambda_{2}t_{2}^{\prime}+\cdots
+\lambda_{n}t_{n}^{\prime}\right)  =S\circ\left(  L\left(  \lambda_{1}%
t_{1}+\lambda_{2}t_{2}+\cdots+\lambda_{n}t_{n}\right)  \right)  \circ S^{-1}.
\]

\end{corollary}

\begin{proof}
It is easy to see that the map
\begin{align*}
R:\mathbf{k}\left[  S_{n}\right]   &  \rightarrow\operatorname*{End}%
\nolimits_{\mathbf{k}}\left(  \mathbf{k}\left[  S_{n}\right]  \right)  ,\\
x  &  \mapsto R\left(  x\right)
\end{align*}
is $\mathbf{k}$-linear. Hence,%
\[
R\left(  \lambda_{1}t_{1}^{\prime}+\lambda_{2}t_{2}^{\prime}+\cdots
+\lambda_{n}t_{n}^{\prime}\right)  =\lambda_{1}R\left(  t_{1}^{\prime}\right)
+\lambda_{2}R\left(  t_{2}^{\prime}\right)  +\cdots+\lambda_{n}R\left(
t_{n}^{\prime}\right)  =\sum_{\ell=1}^{n}\lambda_{\ell}R\left(  t_{\ell
}^{\prime}\right)  .
\]
Similarly,%
\[
L\left(  \lambda_{1}t_{1}+\lambda_{2}t_{2}+\cdots+\lambda_{n}t_{n}\right)
=\sum_{\ell=1}^{n}\lambda_{\ell}L\left(  t_{\ell}\right)  .
\]
Hence,%
\begin{align*}
S\circ\left(  L\left(  \lambda_{1}t_{1}+\lambda_{2}t_{2}+\cdots+\lambda
_{n}t_{n}\right)  \right)  \circ S^{-1}  &  =S\circ\left(  \sum_{\ell=1}%
^{n}\lambda_{\ell}L\left(  t_{\ell}\right)  \right)  \circ S^{-1}\\
&  =\sum_{\ell=1}^{n}\lambda_{\ell}S\circ\left(  L\left(  t_{\ell}\right)
\right)  \circ S^{-1}%
\end{align*}
(since composition of $\mathbf{k}$-linear maps is $\mathbf{k}$-bilinear).
Comparing this with%
\begin{align*}
R\left(  \lambda_{1}t_{1}^{\prime}+\lambda_{2}t_{2}^{\prime}+\cdots
+\lambda_{n}t_{n}^{\prime}\right)   &  =\sum_{\ell=1}^{n}\lambda_{\ell
}R\left(  \underbrace{t_{\ell}^{\prime}}_{\substack{=S\left(  t_{\ell}\right)
\\\text{(by (\ref{eq.def.Stl}))}}}\right)  =\sum_{\ell=1}^{n}\lambda_{\ell
}\underbrace{R\left(  S\left(  t_{\ell}\right)  \right)  }_{\substack{=S\circ
\left(  L\left(  t_{\ell}\right)  \right)  \circ S^{-1}\\\text{(by
(\ref{eq.prop.L.LxSx.1}),}\\\text{applied to }x=t_{\ell}\text{)}}}\\
&  =\sum_{\ell=1}^{n}\lambda_{\ell}S\circ\left(  L\left(  t_{\ell}\right)
\right)  \circ S^{-1},
\end{align*}
we obtain%
\[
R\left(  \lambda_{1}t_{1}^{\prime}+\lambda_{2}t_{2}^{\prime}+\cdots
+\lambda_{n}t_{n}^{\prime}\right)  =S\circ\left(  L\left(  \lambda_{1}%
t_{1}+\lambda_{2}t_{2}+\cdots+\lambda_{n}t_{n}\right)  \right)  \circ S^{-1}.
\]
Thus, the endomorphisms $L\left(  \lambda_{1}t_{1}+\lambda_{2}t_{2}%
+\cdots+\lambda_{n}t_{n}\right)  $ and $R\left(  \lambda_{1}t_{1}^{\prime
}+\lambda_{2}t_{2}^{\prime}+\cdots+\lambda_{n}t_{n}^{\prime}\right)  $ of
$\mathbf{k}\left[  S_{n}\right]  $ are mutually conjugate in the endomorphism
ring $\operatorname*{End}\nolimits_{\mathbf{k}}\left(  \mathbf{k}\left[
S_{n}\right]  \right)  $ of the $\mathbf{k}$-module $\mathbf{k}\left[
S_{n}\right]  $. This proves Corollary \ref{cor.L.conj}.
\end{proof}

Using Corollary \ref{cor.L.conj}, we can derive properties of $L\left(
\lambda_{1}t_{1}+\lambda_{2}t_{2}+\cdots+\lambda_{n}t_{n}\right)  $ from
properties of $R\left(  \lambda_{1}t_{1}^{\prime}+\lambda_{2}t_{2}^{\prime
}+\cdots+\lambda_{n}t_{n}^{\prime}\right)  $ by conjugating with $S^{-1}$. In
particular, we obtain an analogue of Theorem \ref{thm.Rcomb-main} for
$L\left(  \lambda_{1}t_{1}+\lambda_{2}t_{2}+\cdots+\lambda_{n}t_{n}\right)  $
instead of $R\left(  \lambda_{1}t_{1}+\lambda_{2}t_{2}+\cdots+\lambda_{n}%
t_{n}\right)  $, since we already know (from Subsection
\ref{subsect.furtheralg.btr}) that such an analogue exists for $R\left(
\lambda_{1}t_{1}^{\prime}+\lambda_{2}t_{2}^{\prime}+\cdots+\lambda_{n}%
t_{n}^{\prime}\right)  $. Thus, we shall not discuss $L\left(  \lambda
_{1}t_{1}+\lambda_{2}t_{2}+\cdots+\lambda_{n}t_{n}\right)  $ any further.

\begin{noncompile}
The following is true but obsolete.

Using some representation theory, we can answer Question
\ref{quest.Rcomb-mainL} in the case when $\mathbf{k}$ is a field of
characteristic $0$. Indeed, we have a more general fact:

\begin{proposition}
\label{prop.char0-L=R}Assume that $\mathbf{k}$ is a field of characteristic
$0$. Let $x\in\mathbf{k}\left[  S_{n}\right]  $ be arbitrary. Then, the
endomorphisms $L\left(  x\right)  $ and $R\left(  x\right)  $ of
$\mathbf{k}\left[  S_{n}\right]  $ are conjugate in $\operatorname*{End}%
\nolimits_{\mathbf{k}}\left(  \mathbf{k}\left[  S_{n}\right]  \right)  $ (that
is, the matrices representing them are similar).
\end{proposition}

\begin{proof}
[Proof sketch.]It is well-known (see, e.g., \cite[\S 7.2, between Proposition
1 and Lemma 3]{Fulton97}) that all representations of $S_{n}$ in
characteristic $0$ are defined over $\mathbb{Q}$. Hence, by the
Artin--Wedderburn theorem, $\mathbb{Q}\left[  S_{n}\right]  $ is isomorphic to
a direct product of matrix algebras over $\mathbb{Q}$. (This can also be
proved directly, by constructing a system of matrix units in $\mathbb{Q}%
\left[  S_{n}\right]  $; see, e.g., \cite[Theorem 2.6]{GarEge20} for this.) By
tensoring this isomorphism with $\mathbf{k}$ over $\mathbb{Q}$, we conclude
that $\mathbf{k}\left[  S_{n}\right]  $ is isomorphic to a direct product of
matrix algebras over $\mathbf{k}$. Hence, in order to prove Proposition
\ref{prop.char0-L=R}, it suffices to show the analogous result for matrix
algebras over $\mathbf{k}$: That is, it suffices to show that if
$A\in\mathbf{k}^{m\times m}$ is any matrix, then the endomorphisms $L\left(
A\right)  $ and $R\left(  A\right)  $ of the matrix algebra $\mathbf{k}%
^{m\times m}$ (where $L\left(  A\right)  $ sends any matrix $B$ to $AB$,
whereas $R\left(  A\right)  $ sends it to $BA$) are conjugate in
$\operatorname*{End}\nolimits_{\mathbf{k}}\left(  \mathbf{k}^{m\times
m}\right)  $.

However, a well-known fact (\cite[Theorem 1]{TauZas59}) says that any matrix
over a field is similar to its transpose. Thus, $A$ is similar to $A^{T}$. In
other words, there exists an invertible matrix $M\in\mathbf{k}^{m\times m}$
such that $A=MA^{T}M^{-1}$. Consider this $M$.

Now, let $\Phi:\mathbf{k}^{m\times m}\rightarrow\mathbf{k}^{m\times m}$ be the
$\mathbf{k}$-linear map sending each matrix $B\in\mathbf{k}^{m\times m}$ to
$MB^{T}M^{-1}$. Then, $\Phi\in\operatorname*{End}\nolimits_{\mathbf{k}}\left(
\mathbf{k}^{m\times m}\right)  $, and it is straightforward to see that each
matrix $B\in\mathbf{k}^{m\times m}$ satisfies $A\Phi\left(  B\right)
=\Phi\left(  BA\right)  $. Hence, $L\left(  A\right)  \circ\Phi=\Phi\circ
R\left(  A\right)  $, so that $L\left(  A\right)  =\Phi\circ R\left(
A\right)  \circ\Phi^{-1}$. This shows that $L\left(  A\right)  $ and $R\left(
A\right)  $ are conjugate in $\operatorname*{End}\nolimits_{\mathbf{k}}\left(
\mathbf{k}^{m\times m}\right)  $, and thus Proposition \ref{prop.char0-L=R} is complete.
\end{proof}
\end{noncompile}

\subsection{A Boolean interval partition of $\mathcal{P}\left(  \left[
n-1\right]  \right)  $}

Our results on $Q$-indices and lacunar subsets shown above quickly lead to a
curious result, which may be of independent interest (similar results appear
in \cite{AgNyOr06} and other references on peak algebras and cd-indices):

\begin{corollary}
\label{cor.lac.bool}Let $J$ be a subset of $\left[  n-1\right]  $. Then, there
exists a unique lacunar subset $I$ of $\left[  n-1\right]  $ satisfying
$I^{\prime}\subseteq J\subseteq\left[  n-1\right]  \setminus I$.
\end{corollary}

\begin{proof}
First of all, we observe that there exists a permutation $w\in S_{n}$
satisfying $\operatorname*{Des}w=J$ (indeed, we have already constructed such
a $w$ in our proof of Lemma \ref{lem.Qind.surj}\footnote{Arguably, the set $J$
in the proof of Lemma \ref{lem.Qind.surj} was not an arbitrary subset of
$\left[  n-1\right]  $, but a specially constructed one; however, the
construction of $w$ works equally well for any $J$.}). Fix such a $w$.

There exists a unique $i\in\left[  f_{n+1}\right]  $ such that
$\operatorname*{Qind}w=i$ (since $\operatorname*{Qind}w$ is a well-defined
element of $\left[  f_{n+1}\right]  $). In view of Proposition
\ref{prop.Qind.equivalent}, we can rewrite this as follows: There exists a
unique $i\in\left[  f_{n+1}\right]  $ such that $Q_{i}^{\prime}\subseteq
\operatorname*{Des}w\subseteq\left[  n-1\right]  \setminus Q_{i}$. In view of
$\operatorname*{Des}w=J$, we can rewrite this as follows: There exists a
unique $i\in\left[  f_{n+1}\right]  $ such that $Q_{i}^{\prime}\subseteq
J\subseteq\left[  n-1\right]  \setminus Q_{i}$. Since $Q_{1},Q_{2}%
,\ldots,Q_{f_{n+1}}$ are all the lacunar subsets of $\left[  n-1\right]  $
(listed without repetition), we can rewrite this as follows: There exists a
unique lacunar subset $I$ of $\left[  n-1\right]  $ satisfying $I^{\prime
}\subseteq J\subseteq\left[  n-1\right]  \setminus I$. Corollary
\ref{cor.lac.bool} is thus proven.
\end{proof}

We can rewrite Corollary \ref{cor.lac.bool} in the language of Boolean
interval partitions (see \cite[\S 4.4]{Grinbe21}): Namely, it says that there
is a Boolean interval partition of the powerset $\mathcal{P}\left(  \left[
n-1\right]  \right)  $ whose blocks are the intervals $\left[  I^{\prime
},\ \left[  n-1\right]  \setminus I\right]  $ for all lacunar subsets $I$ of
$\left[  n-1\right]  $.

\subsection{Consequences for the top-to-random shuffle}

Let us briefly comment on what our above results yield for the top-to-random
shuffle $t_{1}$. It is easy to derive from Corollary \ref{cor.eigen.spec} that
when $\mathbf{k}$ is a field, we have%
\[
\operatorname*{Spec}\left(  R\left(  t_{1}\right)  \right)  =\left\{
m_{I,1}\ \mid\ I\subseteq\left[  n-1\right]  \text{ is lacunar}\right\}
=\left\{  0,1,\ldots,n-2,n\right\}
\]
(the latter equality sign here is a consequence of the definition of $m_{I,1}$
and the fact that $\widehat{I}\subseteq\left\{  0,1,\ldots,n-1,n+1\right\}
$). This, of course, is a fairly well-known result (e.g., being part of
\cite[Theorem 4.1]{DiFiPi92}). Unfortunately, the fact that $R\left(
t_{1}\right)  $ is diagonalizable when $\mathbf{k}$ is a field of
characteristic $0$ (see, e.g., \cite[Theorem 4.1]{DiFiPi92}) cannot be
recovered from our above results (as the assumptions of Theorem
\ref{thm.eigen.diagonalizable} are not satisfied when $n\geq4$ and
$\lambda_{2}=\lambda_{3}=\cdots=\lambda_{n}=0$).

\section{\label{sec.stoppingtime}Strong stationary time for the
random-to-below shuffle}

We now leave the realm of algebra for some probabilistic analysis of the
one-sided cycle shuffles.

We shall start this section by recalling how a strong stationary time for the
top-to-random shuffle has been obtained (\cite{AldDia86}). Using a similar but
subtler strategy, we will then describe a strong stationary time for the
one-sided cycle shuffles, and compute its waiting time in the specific case of
the random-to-below shuffle.

\subsection{Strong stationary time for the top-to-random shuffle}

A stopping time for the top-to-random shuffle can be obtained using the
following clever argument: At any given time, the cards that have already been
moved from the top position will appear in a uniformly random relative order.
Hence, once all cards have been moved from the top position, all permutations
of the deck are equally likely. To estimate the time for this event to happen,
we follow the position of the card that is originally at the bottom of the
deck. This card occasionally moves up a position, but never moves down until
it reaches the top of the deck. It moves from the bottommost position to the
next-higher one with probability $\dfrac{1}{n}$, then to one position higher
with probability $\dfrac{2}{n}$, etc., until (as we said) it reaches the top.
One iteration of the top-to-random shuffle later, the deck will be fully
mixed, therefore giving a strong stationary time. The waiting time for this
event can be easily seen to approach $n\log n$. Details can be found in the
introduction of \cite{AldDia86}, or in \cite[\S 6.1 and \S 6.5.3]{LePeWi09}.

\subsection{A similar argument for the one-sided cycle shuffles}

A similar argument can be used for the one-sided cycle shuffles. However,
unlike for the top-to-random shuffle, we do not follow the bottommost card any
more, since it may fall down before reaching the top (and is thus much more
difficult to track). Thus, instead of following a specific card, we follow a
space between two cards.

Namely, we stick a \emph{bookmark} right above the card that was initially at
the bottom. This bookmark will serve as a marker that will distinguish the
fully mixed part (which is the part below the bookmark) from the rest of the
deck. The bookmark itself is not considered to be a card in the deck, so the
only way it moves is when a card that was above it is inserted below
it.\footnote{We agree that if a card moves into the space that contains the
bookmark, then it is inserted below (not above) the bookmark.} Thus, the
bookmark never moves down but occasionally moves up the deck. The deck is
mixed once the bookmark is at the top.

The following theorem follows:

\begin{theorem}
\label{thm.times.stat-osc}If $P(1)\neq0$, then the one-sided cycle shuffle
$\osc(P,n)$ admits a stopping time $\tau$ corresponding to the first time that
all cards have been inserted below a bookmark initially placed right above the
card at the bottom of the deck before the shuffling process. If $X_{t}$ is the
random variable for $\osc(P,n)$, the distribution of $X_{t}$ is uniform for
all $t\geq\tau$, meaning that $\tau$ is a strong stationary time.
\end{theorem}

If $P(1) = 0$, then the top card never moves, and the stationary distribution
is not the uniform distribution over all permutations.

\subsection{The waiting time for the strong stationary time of the
random-to-below shuffle}

Knowing the existence of a strong stationary time for the one-sided cycle
shuffle (with $P\left(  1\right)  \neq0$), one might be interested to know
when it is reasonable to expect this phenomenon to occur. We shall compute
this waiting time for the random-to-below shuffle; the computations for other
one-sided cycle shuffles would result in other numbers.

\begin{itemize}
\item If the bookmark is below the $i$-th card from the bottom, the
probability for it to move in one iteration of the random-to-below shuffle is
the sum of the probabilities for cards above it to move below it. The card at
position $j$ (counting from the bottom) is selected with probability $P\left(
j\right)  =\dfrac{1}{n}$, and (assuming that $j\geq i$) is inserted below the
bookmark with probability $\dfrac{i}{j}$ (this includes the case when it is
moved inbetween positions $i$ and $i-1$, because in this case we insert it
below the bookmark). Hence, the bookmark climbs up one position in the deck
with probability
\[
\sum_{j=i}^{n}\dfrac{1}{n}\cdot\dfrac{i}{j}=\frac{i}{n}\sum_{j=i}^{n}\frac
{1}{j}=\frac{i}{n}\left(  H_{n}-H_{i-1}\right)  ,
\]
where $H_{i}:=\sum_{k=1}^{i}\dfrac{1}{k}$ is the $i$-th harmonic number.

Thus, the probability of the bookmark climbing from position $i$ to $i+1$ at
any single step follows a geometric distribution with parameter $\dfrac{i}%
{n}\left(  H_{n}-H_{i-1}\right)  $, and therefore the expected time needed for
the event to happen is
\[
\dfrac{1}{\dfrac{i}{n}\left(  H_{n}-H_{i-1}\right)  }=\frac{n}{i\left(
H_{n}-H_{i-1}\right)  }.
\]
(Recall that the expected time for an event with probability $p$ to happen is
$\frac{1}{p}$.)

\item The stopping time is the time required for the bookmark to reach the top
of the deck (position $n$). This is achieved in an expected time corresponding
to
\[
\sum_{i=2}^{n}\frac{n}{i\left(  H_{n}-H_{i-1}\right)  }.
\]

\end{itemize}

\begin{theorem}
\label{thm.times.expect-r2b}Let $n \geq2$. The expected number of steps to get
to the strong stationary time for the random-to-below shuffle is
\[
\mathbb{E}(\tau) = \sum_{i=2}^{n}\frac{n}{i\left(  H_{n}-H_{i-1}\right)  }.
\]
Moreover, this time satisfies the following bound:
\[
\sum_{i=2}^{n}\frac{n}{i\left(  H_{n}-H_{i-1}\right)  } \leq n\log
n+n\log\left(  \log n\right)  + n\log(2) +1.
\]
Here, $\log$ denotes the natural logarithm $\ln$.
\end{theorem}

\begin{proof}
The statement that the expected number of steps is $\sum_{i=2}^{n}\dfrac
{n}{i\left(  H_{n}-H_{i-1}\right)  }$ follows from the discussion above.
Hence, we only need to prove the upper bound.

For this purpose, we shall show several analytic lemmas. The first is a known
property of logarithms:\footnote{Throughout this proof, the notations $\left[
a,b\right]  $, $\left[  a,b\right)  $, $\left(  a,b\right]  $ and $\left(
a,b\right)  $ are used in their familiar meanings from real analysis. In
particular, $\left[  a,b\right]  $ means the set of all real numbers $x$
satisfying $a \leq x \leq b$, contrary to our convention from
Subsection~\ref{subsec.notations}.}

\begin{lemma}
\label{lem.log.1}Let $a$ and $b$ be two positive reals. Then:

\textbf{(a)} We have $\log\dfrac{a+b}{a}\leq\dfrac{b}{a}$.

\textbf{(b)} We have $\log\dfrac{a+b}{a}\geq\dfrac{b}{a+b}$.
\end{lemma}

\begin{proof}
[Proof of Lemma \ref{lem.log.1}.]Since the logarithm function is the
antiderivative of the function $f\left(  x\right)  =\dfrac{1}{x}$, we have
$\int_{a}^{a+b}\dfrac{1}{x}dx=\log\left(  a+b\right)  -\log a=\log\dfrac
{a+b}{a}$. Hence,
\[
\log\dfrac{a+b}{a}=\int_{a}^{a+b}\underbrace{\dfrac{1}{x}}_{\leq\dfrac{1}{a}%
}dx\leq\int_{a}^{a+b}\dfrac{1}{a}dx=\dfrac{b}{a},
\]
which proves part \textbf{(a)}. Furthermore,%
\[
\log\dfrac{a+b}{a}=\int_{a}^{a+b}\underbrace{\dfrac{1}{x}}_{\geq\dfrac{1}%
{a+b}}dx\geq\int_{a}^{a+b}\dfrac{1}{a+b}dx=\dfrac{b}{a+b},
\]
which proves part \textbf{(b)}.
\end{proof}

\begin{lemma}
\label{lem.log.2}Let $m$ be a positive real. Then, the function $f:\left(
0,m\right)  \rightarrow\mathbb{R}$ given by%
\[
f\left(  x\right)  =\dfrac{1}{x\log\dfrac{m}{x}}\ \ \ \ \ \ \ \ \ \ \text{for
all }x\in\left(  0,m\right)
\]
is convex.
\end{lemma}

\begin{proof}
[Proof of Lemma \ref{lem.log.2}.]The second derivative $f^{\left(  2\right)
}$ of this function is easily computed as%
\[
f^{\left(  2\right)  }\left(  x\right)  =\dfrac{2\left(  \log\dfrac{m}%
{x}\right)  ^{2}-3\log\dfrac{m}{x}+2}{x^{3}\left(  \log\dfrac{m}{x}\right)
^{3}},
\]
and this is $\geq0$ because the numerator can be rewritten as $2y^{2}%
-3y+2=2\left(  y-1\right)  ^{2}+y$ for $y=\log\dfrac{m}{x}\geq0$.
\end{proof}

\begin{lemma}
\label{lem.log.3}If $n\geq3$, then $\dfrac{n+1}{n}-\dfrac{n+1}{3}< \dfrac{\log
n}{2n}$.
\end{lemma}

\begin{proof}
[Proof of Lemma \ref{lem.log.3}.]Consider the function $f:\left(
0,\infty\right)  \rightarrow\mathbb{R}$ given by $f\left(  x\right)
:=\dfrac{\log x}{2}-\left(  x+1\right)  \left(  1-\dfrac{x}{3}\right)  $. This
function $f$ is weakly increasing on $\left(  2,\infty\right)  $ (since its
derivative is $f^{\prime}\left(  x\right)  = \dfrac{-4x+4x^{2}+3}{6x} =
\dfrac{\left(  2x-1\right)  ^{2} +2}{6x} \geq0$). Thus, for $n\geq3$, we have
$f\left(  n\right)  \geq f\left(  3\right)  = \dfrac{\log3}{2} >0$. Since
$f\left(  n\right)  =\dfrac{\log n}{2}-\left(  n+1\right)  \left(  1-\dfrac
{n}{3}\right)  $, we can rewrite this as%
\[
\left(  n+1\right)  \left(  1-\dfrac{n}{3}\right)  <\dfrac{\log n}{2}.
\]
Dividing both sides by $n$ and expanding the left hand side, we transform this
into%
\[
\dfrac{n+1}{n}-\dfrac{n+1}{3}<\dfrac{\log n}{2n}.
\]
This proves Lemma \ref{lem.log.3}.
\end{proof}

\begin{lemma}
\label{lem.log.Hn}Let $i \leq n$ be a positive integer. Then,%
\[
H_{n}-H_{i-1}\geq\log\frac{n+1}{i}.
\]

\end{lemma}

\begin{proof}
[Proof of Lemma \ref{lem.log.Hn}.]The definition of $H_{m}$ yields%
\begin{align*}
H_{n}-H_{i-1}  &  =\dfrac{1}{i}+\dfrac{1}{i+1}+\cdots+\dfrac{1}{n}\\
&  \geq\int_{i}^{i+1}\dfrac{1}{x}dx+\int_{i+1}^{i+2}\dfrac{1}{x}dx+\cdots
+\int_{n}^{n+1}\dfrac{1}{x}dx\\
&  \ \ \ \ \ \ \ \ \ \ \ \ \ \ \ \ \ \ \ \ \left(
\begin{array}
[c]{c}%
\text{indeed, }\dfrac{1}{j}\geq\int_{j}^{j+1}\dfrac{1}{x}dx\text{ for each
}j>0\text{,}\\
\text{since the function }\dfrac{1}{x}\text{ is decreasing}%
\end{array}
\right) \\
&  =\int_{i}^{n+1}\dfrac{1}{x}dx=\log\left(  n+1\right)  -\log i=\log
\dfrac{n+1}{i}.
\end{align*}

\end{proof}

\begin{lemma}
\label{lem.convex-riemann}Let $a$ and $b$ be two integers satisfying $a\leq
b$. Let $f:\left(  a-1,b+1\right)  \rightarrow\mathbb{R}$ be a convex
function. Then,%
\[
\sum_{i=a}^{b}f\left(  i\right)  \leq\int_{a-1/2}^{b+1/2}f\left(  x\right)
dx.
\]

\end{lemma}

\begin{proof}
[Proof of Lemma \ref{lem.convex-riemann}.]The interval $\left[
a-1/2,\ b+1/2\right)  $ can be decomposed as a disjoint union%
\begin{align*}
&  \left[  a-1/2,\ a+1/2\right)  \sqcup\left[  a+1/2,\ a+3/2\right)
\sqcup\left[  a+3/2,\ a+5/2\right)  \sqcup\cdots\sqcup\left[
b-1/2,\ b+1/2\right) \\
&  =\bigsqcup_{i=a}^{b}\left[  i-1/2,\ i+1/2\right)  .
\end{align*}
Hence,%
\begin{align*}
\int_{a-1/2}^{b+1/2}f\left(  x\right)  dx  &  =\sum_{i=a}^{b}\underbrace{\int%
_{i-1/2}^{i+1/2}f\left(  x\right)  dx}_{\substack{=\dfrac{1}{2}\left(
\int_{i-1/2}^{i+1/2}f\left(  x\right)  dx+\int_{i-1/2}^{i+1/2}f\left(
x\right)  dx\right)  \\\text{(since }p=\dfrac{1}{2}\left(  p+p\right)  \text{
for any }p\text{)}}}\\
&  =\sum_{i=a}^{b}\dfrac{1}{2}\left(  \int_{i-1/2}^{i+1/2}f\left(  x\right)
dx+\int_{i-1/2}^{i+1/2}f\left(  x\right)  dx\right) \\
&  =\sum_{i=a}^{b}\dfrac{1}{2}\left(  \int_{i-1/2}^{i+1/2}f\left(  x\right)
dx+\int_{i-1/2}^{i+1/2}f\left(  2i-x\right)  dx\right) \\
&  \ \ \ \ \ \ \ \ \ \ \ \ \ \ \ \ \ \ \ \ \left(
\begin{array}
[c]{c}%
\text{here, we have substituted }2i-x\text{ for }x\\
\text{in the second integral}%
\end{array}
\right) \\
&  =\sum_{i=a}^{b}\int_{i-1/2}^{i+1/2}\underbrace{\dfrac{1}{2}\left(  f\left(
x\right)  +f\left(  2i-x\right)  \right)  }_{\substack{\geq f\left(  i\right)
\\\text{(since }f\text{ is convex, and}\\\text{since }i\text{ is the
midpoint}\\\text{between }x\text{ and }2i-x\text{)}}}dx\\
&  \geq\sum_{i=a}^{b}\underbrace{\int_{i-1/2}^{i+1/2}f\left(  i\right)
dx}_{=f\left(  i\right)  }=\sum_{i=a}^{b}f\left(  i\right)  .
\end{align*}
This proves Lemma \ref{lem.convex-riemann}.
\end{proof}

Now, we return to the proof of the upper bound
\begin{equation}
\sum_{i=2}^{n}\frac{n}{i\left(  H_{n}-H_{i-1}\right)  }\leq n\log
n+n\log\left(  \log n\right)  +n\log2+1 \label{pf.thm.times.expect-r2b.goal1}%
\end{equation}
claimed in Theorem \ref{thm.times.expect-r2b}.

Indeed, this upper bound can be checked by straightforward computations for
$n=2$. So let us WLOG assume that $n\geq3$.

Let $m:=n+1$. Define a function $f:\left(  0,m\right)  \rightarrow\mathbb{R}$
as in Lemma \ref{lem.log.2}. Then, Lemma \ref{lem.log.2} says that this
function $f$ is convex. We note also that the function $f$ has antiderivative
$F:\left(  0,m\right)  \rightarrow\mathbb{R}$ given by%
\[
F\left(  x\right)  =-\log\left(  \log\dfrac{m}{x}\right)  .
\]
(This can be easily verified by hand.)

From Lemma \ref{lem.log.Hn}, we obtain
\begin{align*}
\sum_{i=2}^{n}\frac{n}{i(H_{n}-H_{i-1})}  &  \leq\sum_{i=2}^{n}\frac{n}%
{i\log\dfrac{n+1}{i}} = \sum_{i=2}^{n}\frac{n}{i\log\dfrac{m}{i}}
\ \ \ \ \ \ \ \ \ \ \left(  \text{since $n+1=m$}\right) \\
&  =n\cdot\sum_{i=2}^{n}\underbrace{\frac{1}{i\log\dfrac{m}{i}}}%
_{\substack{=f\left(  i\right)  \\\text{(by the definition of }f\text{)}}}
=n\cdot\sum_{i=2}^{n}f\left(  i\right)  .
\end{align*}
Hence, in order to prove (\ref{pf.thm.times.expect-r2b.goal1}), we only need
to show that%
\begin{equation}
\sum_{i=2}^{n}f\left(  i\right)  \leq\log n+\log\left(  \log n\right)
+\log2+\dfrac{1}{n}. \label{pf.thm.times.expect-r2b.goal2}%
\end{equation}
So let us prove this inequality now.

Since $f$ is convex on $\left(  0,m\right)  $, we can apply Lemma
\ref{lem.convex-riemann} to $a=2$ and $b=n=m-1$. We thus obtain%
\begin{align*}
\sum_{i=2}^{n}f\left(  i\right)   &  \leq\int_{3/2}^{n+1/2}f\left(  x\right)
dx\\
&  =\left(  -\log\left(  \log\dfrac{m}{n+1/2}\right)  \right)  -\left(
-\log\left(  \log\dfrac{m}{3/2}\right)  \right) \\
&  \ \ \ \ \ \ \ \ \ \ \ \ \ \ \ \ \ \ \ \ \left(
\begin{array}
[c]{c}%
\text{since }f\text{ has antiderivative }F\text{ given}\\
\text{by }F\left(  x\right)  =-\log\left(  \log\dfrac{m}{x}\right)
\end{array}
\right) \\
&  =\log\left(  \log\dfrac{m}{3/2}\right)  -\log\underbrace{\left(  \log
\dfrac{m}{n+1/2}\right)  }_{\substack{=\log\dfrac{n+1/2+1/2}{n+1/2}%
\\\text{(since }m=n+1=n+1/2+1/2\text{)}}}\\
&  =\log\left(  \log\dfrac{m}{3/2}\right)  -\log\underbrace{\left(  \log
\dfrac{n+1/2+1/2}{n+1/2}\right)  }_{\substack{\geq\dfrac{1/2}{n+1/2+1/2}%
\\\text{(by Lemma \ref{lem.log.1} \textbf{(b)},}\\\text{applied to
}a=n+1/2\text{ and }b=1/2\text{)}}}\\
&  \leq\log\left(  \log\dfrac{m}{3/2}\right)  -\log\dfrac{1/2}{n+1/2+1/2}\\
&  =\log\left(  \log\dfrac{m}{3/2}\right)  -\log\dfrac{1/2}{m}%
\ \ \ \ \ \ \ \ \ \ \left(  \text{since }n+1/2+1/2=n+1=m\right) \\
&  =\log\left(  \left(  \log\dfrac{m}{3/2}\right)  \diagup\dfrac{1/2}%
{m}\right)  =\log\left(  2m\log\dfrac{m}{3/2}\right)  .
\end{align*}
Thus, in order to prove (\ref{pf.thm.times.expect-r2b.goal2}), it will suffice
to show that%
\[
\log\left(  2m\log\dfrac{m}{3/2}\right)  \leq\log n+\log\left(  \log n\right)
+\log2+\dfrac{1}{n}.
\]
After exponentiation, this rewrites as%
\begin{equation}
2m\log\dfrac{m}{3/2}\leq2n\log n\cdot e^{1/n}.
\label{pf.thm.times.expect-r2b.goal4}%
\end{equation}
Upon division by $2$, this rewrites as
\begin{equation}
m\log\dfrac{m}{3/2}\leq n\log n\cdot e^{1/n}.
\label{pf.thm.times.expect-r2b.goal5}%
\end{equation}

However,
\begin{align*}
\log\dfrac{m}{3/2} =\log\left(  n \cdot\dfrac{m}{n}\diagup\dfrac{3}{2}\right)
= \log n + \underbrace{\log\dfrac{m}{n}}_{\substack{=\log\dfrac{n+1}{n}%
\leq\dfrac{1}{n}\\\text{(by Lemma \ref{lem.log.1} \textbf{(a)},}%
\\\text{applied to }a=n\text{ and }b=1\text{)}}} - \underbrace{\log\dfrac
{3}{2}}_{\geq\dfrac{1}{3}} \leq\log n+\dfrac{1}{n}-\dfrac{1}{3},
\end{align*}
so that
\begin{align*}
m\log\dfrac{m}{3/2}  &  \leq\underbrace{m}_{=n+1}\left(  \log n+\dfrac{1}%
{n}-\dfrac{1}{3}\right)  =\left(  n+1\right)  \left(  \log n+\dfrac{1}%
{n}-\dfrac{1}{3}\right) \\
&  =\underbrace{\left(  n+1\right)  \log n}_{=n\log n+\log n}
+\underbrace{\dfrac{n+1}{n}-\dfrac{n+1}{3}}_{\substack{< \dfrac{\log n}%
{2n}\\\text{(by Lemma \ref{lem.log.3})}}}\\
&  < n\log n+\log n+\dfrac{\log n}{2n} =n\log n \cdot\underbrace{\left(
1+\dfrac{1}{n}+\dfrac{1}{2n^{2}}\right)  }_{\substack{=\sum_{k=0}^{2}\dfrac
{1}{k!}\left(  \dfrac{1}{n}\right)  ^{k}\\\leq\sum_{k=0}^{\infty}\dfrac{1}%
{k!}\left(  \dfrac{1}{n}\right)  ^{k}\\=e^{1/n}}} \leq n\log n\cdot e^{1/n}.
\end{align*}
This proves (\ref{pf.thm.times.expect-r2b.goal5}). Thus, the proof of Theorem
\ref{thm.times.expect-r2b} is complete.
\end{proof}

One might ask if this is a good upper bound, or, in other terms, if the order
of magnitude of the bound given in Theorem \ref{thm.times.expect-r2b} is also
the order of magnitude of $\mathbb{E}(\tau)$. Numerical checks suggest that
this is indeed the case, allowing us to make the following conjecture.

\begin{conjecture}
\label{conj.lower_bound_tau} Let $n \geq2$. The expected number of steps to
get to the strong stationary time for the random-to-below shuffle satisfies
the following lower bound:
\[
\mathbb{E}(\tau) = \sum_{i=2}^{n}\frac{n}{i\left(  H_{n}-H_{i-1}\right)  }
\geq n\log n+n\log\left(  \log n\right)  .
\]
Here, $\log$ denotes the natural logarithm $\ln$.
\end{conjecture}

\subsection{Optimality of our strong stationary time}

A legitimate question to ask is whether there is a strong stationary time that
occurs faster than $\tau$ for the one-sided cycle shuffles. Our stopping time
$\tau$ is the waiting time for the bookmark to reach the top of the deck. We
now shall explain why there is no faster stopping time, i.e., why we need to
wait for the bookmark to reach the top. To do so, we claim that some
permutations cannot be reached until the bookmark reaches the top.

Consider the card that was initially at the bottom. This card was initially
the only card to be below the bookmark. For this card to go up, a card needs
to be inserted below it, and thus below the bookmark. Hence, all the cards
that are above the bookmark are atop of the card that was initially at the
bottom. Note that cards that are below the bookmark can still be above the
card initially at the bottom. As long as there are $k$ cards above the
bookmark, the card initially at the bottom cannot be among the top $k$ cards.
Hence, for any permutation of our deck to be likely, we need the bookmark to
reach the top, showing that our stopping time is optimal.\newline

A consequence of this fact is that, assuming Conjecture
\ref{conj.lower_bound_tau}, the random-to-below shuffle would be slower than
top-to-random, for which the strong stationary time approaches $n \log n$. We
attribute the fact that random-to-below is slower to its greater
\textit{laziness}, in other words, to the fact that the probability of
applying the identity permutation is higher for random-to-below than for top-to-random.

\section{Further remarks and questions}

\subsection{Some identities for $t_{1},t_{2},\ldots,t_{n}$}

We have now seen various properties of the somewhere-to-below shuffles
$t_{1},t_{2},\ldots,t_{n}$. In particular, from Theorem \ref{thm.Rcomb-main},
we know that they can all be represented as upper-triangular matrices of size
$n!\times n!$. Thus, the Lie subalgebra of $\mathfrak{gl}\left(
\mathbf{k}\left[  S_{n}\right]  \right)  $ they generate is solvable. In a
sense, this can be understood as an \textquotedblleft
almost-commutativity\textquotedblright: It is not true in general that
$t_{1},t_{2},\ldots,t_{n}$ commute, but one can think of them as commuting
\textquotedblleft up to error terms\textquotedblright. There might be several
ways to make this rigorous. One striking observation is that the commutators
$\left[  t_{i},t_{j}\right]  :=t_{i}t_{j}-t_{j}t_{i}$ satisfy $\left[
t_{i},t_{j}\right]  ^{2}=0$ whenever $n\leq5$ (but not when $n=6$ and $i=1$
and $j=3$). This can be generalized as follows:

\begin{theorem}
\label{conj.comm.j-i+1}We have $\left[  t_{i},t_{j}\right]  ^{j-i+1}=0$ for
any $1\leq i<j\leq n$.
\end{theorem}

\begin{theorem}
\label{conj.comm.n-j+1}We have $\left[  t_{i},t_{j}\right]  ^{\left\lceil
\left(  n-j\right)  /2\right\rceil +1}=0$ for any $1\leq i<j\leq n$.
\end{theorem}

Both of these theorems are proved in the preprint \cite{s2b2}. The proofs are
surprisingly difficult, even though they rely on nothing but elementary
manipulations of cycles and sums. Actually, the following two more general
results are proved in \cite{s2b2}:

\begin{theorem}
Let $j\in\left[  n\right]  $, and let $m$ be a positive integer. Let
$k_{1},k_{2},\ldots,k_{m}$ be $m$ elements of $\left[  j\right]  $ (not
necessarily distinct) satisfying $m\geq j-k_{m}+1$. Then,
\[
\left[  t_{k_{1}},t_{j}\right]  \left[  t_{k_{2}},t_{j}\right]  \cdots\left[
t_{k_{m}},t_{j}\right]  =0.
\]

\end{theorem}

\begin{theorem}
Let $j\in\left[  n\right]  $ and $m\in\mathbb{N}$ be such that $2m\geq n-j+2$.
Let $i_{1},i_{2},\ldots,i_{m}$ be $m$ elements of $\left[  j\right]  $ (not
necessarily distinct). Then,
\[
\left[  t_{i_{1}},t_{j}\right]  \left[  t_{i_{2}},t_{j}\right]  \cdots\left[
t_{i_{m}},t_{j}\right]  =0.
\]

\end{theorem}

The following identities are proved in \cite{s2b2} as well:

\begin{proposition}
We have $t_{i}=1+s_{i}t_{i+1}$ for any $i\in\left[  n-1\right]  $.
\end{proposition}

\begin{proposition}
We have $\left(  1+s_{j}\right)  \left[  t_{i},t_{j}\right]  =0$ for any
$1\leq i<j\leq n$.
\end{proposition}

\begin{proposition}
We have $t_{n-1}\left[  t_{i},t_{n-1}\right]  =0$ for any $1\leq i\leq n$.
\end{proposition}

\begin{proposition}
We have $\left[  t_{i},t_{j}\right]  =\left[  s_{i}s_{i+1}\cdots s_{j-1}%
,t_{j}\right]  t_{j}$ for any $1\leq i<j\leq n$.
\end{proposition}

\begin{proposition}
\label{prop.ti+1ti}We have $t_{i+1}t_{i}=\left(  t_{i}-1\right)  t_{i}$ for
any $1\leq i<n$.
\end{proposition}

\begin{proposition}
\label{prop.ti+2ti-1}We have $t_{i+2}\left(  t_{i}-1\right)  =\left(
t_{i}-1\right)  \left(  t_{i+1}-1\right)  $ for any $1\leq i<n-1$.
\end{proposition}

\subsection{Open questions}

The above results (particularly Propositions \ref{prop.ti+1ti} and
\ref{prop.ti+2ti-1}) might suggest that the $\mathbf{k}$-subalgebra
$\mathbf{k}\left[  t_{1},t_{2},\ldots,t_{n}\right]  $ of $\mathbf{k}\left[
S_{n}\right]  $ can be described by explicit generators and relations. This is
probably overly optimistic, but we believe that it has some more properties
left to uncover. In particular, one can ask:

\begin{question}
What is the representation theory (indecomposable modules, etc.) of this
algebra? What power of its Jacobson radical is $0$? (These likely require
$\mathbf{k}$ to be a field.) What is its dimension (as a $\mathbf{k}$-vector space)?
\end{question}

Any reader acquainted with the standard arsenal of card-shuffling will spot
another peculiarity of the above work: We have not once used any result about
$\mathbf{k}\left[  S_{n}\right]  $-modules (i.e., representations of the
symmetric group $S_{n}$). The subject is, of course, closely related: Each of
the $F\left(  I\right)  $'s and thus also the $F_{i}$'s is a left
$\mathbf{k}\left[  S_{n}\right]  $-module, and it is natural to ask for its
isomorphism type:

\begin{question}
How do the $F\left(  I\right)  $ and the $F_{i}$ decompose into Specht modules
when $\mathbf{k}$ is a field of characteristic $0$ ?
\end{question}

We have been able to answer this question (see \cite{fps2024sn}),
and will prove our answer in forthcoming work.

A different direction in which our results seem to extend is the \emph{Hecke
algebra}. In a nutshell, the \emph{type-A Hecke algebra} (or\emph{
Iwahori-Hecke algebra}) is a deformation of the group algebra $\mathbf{k}%
\left[  S_{n}\right]  $ that involves a new parameter $q\in\mathbf{k}$. It is
commonly denoted by $\mathcal{H}=\mathcal{H}_{q}\left(  S_{n}\right)  $; it
has a basis $\left(  T_{w}\right)  _{w\in S_{n}}$ indexed by the permutations
$w\in S_{n}$, but a more intricate multiplication than $\mathbf{k}\left[
S_{n}\right]  $. A definition of the latter multiplication can be found in
\cite{Mathas99}. We can now define the $q$\emph{-deformed somewhere-to-below
shuffles} $t_{1}^{\mathcal{H}},t_{2}^{\mathcal{H}},\ldots,t_{n}^{\mathcal{H}}$
by%
\[
t_{\ell}^{\mathcal{H}}:=T_{\operatorname*{cyc}\nolimits_{\ell}}%
+T_{\operatorname*{cyc}\nolimits_{\ell,\ell+1}}+T_{\operatorname*{cyc}%
\nolimits_{\ell,\ell+1,\ell+2}}+\cdots+T_{\operatorname*{cyc}\nolimits_{\ell
,\ell+1,\ldots,n}}\in\mathcal{H}.
\]
Surprisingly, these $q$-deformed shuffles appear to share many properties of
the original $t_{1},t_{2},\ldots,t_{n}$; for example:

\begin{conjecture}
Theorem \ref{thm.Rcomb-main} seems to hold in $\mathcal{H}$ when the $t_{\ell
}$ are replaced by the $t_{\ell}^{\mathcal{H}}$.
\end{conjecture}

Attempts to prove this conjecture are underway. \medskip

Thus ends our study of the somewhere-to-below shuffles $t_{1},t_{2}%
,\ldots,t_{n}$ and their linear combinations. From a bird's eye view, the most
prominent feature of this study might have been its use of a strategically
defined filtration of $\mathbf{k}\left[  S_{n}\right]  $ (as opposed to, e.g.,
working purely algebraically with the operators, or combining them into
generating functions, or finding a joint eigenbasis). In the language of
matrices, this means that we found a joint triangular basis for our shuffles
(i.e., a basis of $\mathbf{k}\left[  S_{n}\right]  $ such that each of our
shuffles is represented by an upper-triangular matrix in this basis). In our
case, this method was essentially forced upon us by the lack of a joint
eigenbasis (as we saw in Remark \ref{rmk.Rcomb}). However, even when a family
of linear operators has a joint eigenbasis, it might be easier to find a
filtration than to find such an eigenbasis. Thus, a question naturally appears:

\begin{question}
Are there other families of shuffles for which a filtration like ours (i.e.,
with properties similar to Theorem \ref{thm.t-simultri}) exists and can be
used to simplify the spectral analysis?
\end{question}

\bibliographystyle{halpha-abbrv}
\bibliography{biblio.bib}

\end{document}